\documentclass[twoside,a4paper,12pt]{book}

\usepackage[titletoc]{appendix}
\usepackage{multirow}
\usepackage{amsmath,amssymb,verbatim,amsthm,bbm, mathrsfs,amscd,pb-diagram}
\usepackage{aliascnt} 
\usepackage{appendix}
\DeclareMathAlphabet{\mathpzc}{OT1}{pzc}{m}{it}
\usepackage[all]{xy}
\usepackage{multirow,longtable}
\usepackage{diagbox}
\usepackage{makecell}
\usepackage{float} 
\usepackage{hyperref}
\usepackage{cleveref}
\usepackage{marginnote}
\usepackage{amsmath}
\usepackage{amsfonts}
\usepackage{framed}
\usepackage{dsfont}

\usepackage{csquotes}

\usepackage{tikz}

\usepackage{multicol}

\usepackage{lscape} 

\usepackage{framed}

\usepackage{cite}

\usepackage{todonotes}

\newcounter{todocounter}

\makeatletter
\providecommand\@dotsep{5}
\renewcommand{\listoftodos}[1][\@todonotes@todolistname]{%
	\@starttoc{tdo}{#1}}
\makeatother
\definecolor{Gray}{gray}{0.85}
\definecolor{LightCyan}{rgb}{0.88,1,1}

\newcolumntype{a}{>{\columncolor{Gray}}c}

\usepackage{float}
\restylefloat{table}

\crefname{table}{table}{tables}
\crefname{listing}{Program-code}{Program-codes}  
\Crefname{listing}{Program-code}{Program-codes}
\crefname{subsection}{subsection}{subsections}

\theoremstyle{plain}
\newtheorem{Conj}{Conjecture}

\newtheorem{Thm}{Theorem}[section]
\newtheorem{Cor}[Thm]{Corollary}
\newtheorem{Prop}[Thm]{Proposition}
\newtheorem{Lem}[Thm]{Lemma}

\theoremstyle{definition}
\newtheorem{Remark}[Thm]{Remark}
\newtheorem{Def}[Thm]{Definition}

\numberwithin{equation}{section}
\newcommand{\Card}[1]{\left\vert #1\right\vert} 
\newcommand{\Hecke}{\mathcal{H}} 
\newcommand{\coseta}[1]{\left[ #1 \right]}  
\newcommand{\coset}[1]{ #1 }  
\newcommand{\gen}[1]{\left\langle #1 \right\rangle}  
\newcommand{\NN}{\mathcal{N}}

\newcommand{\Ker}{\operatorname{Ker}}
\newcommand{\Image}{\operatorname{Im}}
\newcommand{\Ind}{i}

\newcommand{\Hom}{\operatorname{Hom}}

\newcommand{\Aut}{\operatorname{Aut}}

\newcommand{\SL}{\operatorname{SL}}
\newcommand{\GL}{\operatorname{GL}}

\newcommand{\Res}{\operatorname{Res}}

\newcommand{\id}{\operatorname{Id}}

\newcommand{\disc}{\operatorname{disc}}

\newcommand{\C}{\mathbb{C}}
\newcommand{\A}{\mathbb{A}}
\newcommand{\Q}{\mathbb{Q}}
\newcommand{\F}{\mathbb{F}}

\newcommand{\p}{\mathbf{P}}
\newcommand{\n}{\mathbf{N}}
\newcommand{\m}{\mathbf{M}}

\newcommand{\G}{\mathbf{G}}
\newcommand{\T}{\mathbf{T}}
\newcommand{\B}{\mathbf{B}}
\newcommand{\U}{\mathbf{U}}
\newcommand{\K}{\mathbf{K}}
\newcommand{\MM}{\mathbf{M}}

\newcommand{\Z}{\mathbb{Z}}
\newcommand{\R}{\mathbb{R}}
\newcommand{\N}{\mathbb{N}}

\newcommand{\M}{\mathcal{M}}

\newcommand{\bk}[1]{\left(#1\right)} 
\newcommand{\bm}{\begin{multline*}}
\newcommand{\tu}{\end  {multline*}}

\DeclareMathOperator{\Id}{\mathbf{1}} 
\newcommand{\Ga}{\mathbb{G}_a} 
\newcommand{\Gm}{\mathbb{G}_m} 

\renewcommand{\check}[1]{#1 ^{\vee}} 
\DeclareMathOperator{\Real}{Re} 
\newcommand{\set}[1]{\left\{ #1 \right\}} 
\newcommand{\res}[1]{\Big\vert_{#1}}

\newcommand{\rmod}{/}
\newcommand{\lmod}{\backslash}
\newcommand{\Stab}{\operatorname{Stab}}

\newcommand{\bfX}{\mathbf{X}}
\newcommand{\bfY}{\mathbf{Y}}

\makeatletter
\newcommand*{\rom}[1]{\expandafter\@slowromancap\romannumeral #1@}
\makeatother

\renewcommand*{\arraystretch}{1.5}

\makeatletter
\def\imod#1{\allowbreak\mkern10mu({\operator@font mod}\,\,#1)}
\makeatother

\makeatletter
\renewcommand\section{\@startsection{section}{1}{\z@}%
	{-3.5ex \@plus -1ex \@minus-.2ex}%
	{2.3ex \@plus.2ex}%
	{\center\normalfont\large\bfseries}}
\makeatother

\makeatletter
\renewcommand\subsection{\@startsection{subsection}{2}{\z@}%
	{-3.5ex \@plus -1ex \@minus-.2ex}%
	{2.3ex \@plus.2ex}%
	{\normalfont\large\bfseries}}
\makeatother

\makeatletter
\renewcommand\subsubsection{\@startsection{subsubsection}{3}{\z@}%
	{-3.5ex \@plus -1ex \@minus-.2ex}%
	{2.3ex \@plus.2ex}%
	{\normalfont\large\bfseries}}
\makeatother

\makeatletter
\newtheorem*{rep@theorem}{\rep@title} \newcommand{\newreptheorem}[2]{%
	\newenvironment{rep#1}[1]{%
		\def\rep@title{\bf #2 \ref{##1} }%
		\begin{rep@theorem} }%
		{\end{rep@theorem} } }
\makeatother
\newreptheorem{theorem}{Theorem}
\newreptheorem{lemma}{Lemma}

\protected\def\ignorethis#1\endignorethis{}
\let\endignorethis\relax

\usepackage{colortbl}
\usepackage{zref-savepos}
\usepackage{pict2e}

\newcounter{NoTableEntry}
\renewcommand*{\theNoTableEntry}{NTE-\the\value{NoTableEntry}}

\makeatletter
\newcommand*{\notableentry}{%
	\kern-\tabcolsep
	\stepcounter{NoTableEntry}%
	\vadjust pre{\zsavepos{\theNoTableEntry t}}
	\vadjust{\zsavepos{\theNoTableEntry b}}
	\zsavepos{\theNoTableEntry l}
	\raisebox{%
		\dimexpr\zposy{\theNoTableEntry b}sp
		-\zposy{\theNoTableEntry l}sp\relax
	}[0pt][0pt]{%
		\setlength{\unitlength}{1pt}%
		\edef\w{%
			\strip@pt\dimexpr\zposx{\theNoTableEntry r}sp%
			-\zposx{\theNoTableEntry l}sp\relax
		}%
		\edef\h{%
			\strip@pt\dimexpr\zposy{\theNoTableEntry t}sp%
			-\zposy{\theNoTableEntry b}sp\relax
		}%
		\ifdim\w pt=0pt 
		\else
		\begin{picture}(0,0)%
		\edef\x{%
			\noexpand\put(0,0){\noexpand\line(\w,\h){\w}}%
			\noexpand\put(0,\h){\noexpand\line(\w,-\h){\w}}%
		}\x
		\end{picture}%
		\fi
	}%
	\hspace{0pt plus 1filll}%
	\zsavepos{\theNoTableEntry r}
	\kern-\tabcolsep
}

\makeatletter
\providecommand*{\cupdot}{%
	\mathbin{%
		\mathpalette\@cupdot{}%
	}%
}
\newcommand*{\@cupdot}[2]{%
	\ooalign{%
		$\m@th#1\cup$\cr
		\sbox0{$#1\cup$}%
		\dimen@=\ht0 %
		\sbox0{$\m@th#1\cdot$}%
		\advance\dimen@ by -\ht0 %
		\dimen@=.5\dimen@
		\hidewidth\raise\dimen@\box0\hidewidth
	}%
}

\providecommand*{\bigcupdot}{%
	\mathop{%
		\vphantom{\bigcup}%
		\mathpalette\@bigcupdot{}%
	}%
}
\newcommand*{\@bigcupdot}[2]{%
	\ooalign{%
		$\m@th#1\bigcup$\cr
		\sbox0{$#1\bigcup$}%
		\dimen@=\ht0 %
		\advance\dimen@ by -\dp0 %
		\sbox0{\scalebox{2}{$\m@th#1\cdot$}}%
		\advance\dimen@ by -\ht0 %
		\dimen@=.5\dimen@
		\hidewidth\raise\dimen@\box0\hidewidth
	}%
}
\makeatother

\allowdisplaybreaks

\newcommand{\fun}[1]{\varpi_{#1}}
\newcommand{\jac}[3]{r^{#1}_{#2}\bk{#3}}
\newcommand{\weyl}[1]{\mathit{W}_{#1}}
\newcommand{\w}{w}
\newcommand{\mult}[2]{mult\bk{#1,#2}}
\newcommand{\br}[2]{br\bk{#1,#2}}
\newcommand{\s}[1]{w_{#1}}
\newcommand{\para}[1]{#1}
\newcommand{\inner}[1]{\left\langle #1 \right\rangle}

\newcommand{\Span}{\operatorname{Span}}

 \newcolumntype{H}{>{\setbox0=\hbox\bgroup}c<{\egroup}@{}}

\numberwithin{equation}{section}

\usepackage[printwatermark]{xwatermark}
\usepackage{pifont}
\usepackage{slashbox}

\usepackage[showframe=false]{geometry}
\usepackage{changepage}

\usepackage{float}
\restylefloat{table}
\usepackage{xparse}
\NewDocumentCommand{\ceil}{s O{} m}{%
  \IfBooleanTF{#1} 
    {\left\lceil#3\right\rceil} 
    {#2\lceil#3#2\rceil} 
}

\newcommand{\iwhaori}{J}
\newcommand{\ord}{\operatorname{ord }}
\newcommand{\leadingterm}{\Lambda}

\usepackage{arydshln}

\definecolor{HOLO_BOREL}{rgb}{1, 1, 1}
 \definecolor{HOLO_RESTRICT_1}{rgb}{1,0.75,0}
 \definecolor{HOLO_RESTRICT_2}{rgb}{0.9, 0.5, 0.71}
 \definecolor{HOLO_RESTRICT_3}{rgb}{0.7,0.75,0.71}
 \definecolor{HOLO_RESTRICT_4}{rgb}{0.13,0.67,0.8}
 
\newcolumntype{R}{>{\raggedleft\arraybackslash}p{11cm}}
\newcolumntype{L}{>{\raggedleft\arraybackslash}p{3cm}}

\usepackage[makeroom]{cancel}
\usepackage{color,colortbl}
\definecolor{LightCyan}{rgb}{1,1,1}
\definecolor{GREEN}{rgb}{0.5,0.88,0}
\definecolor{ORANGE}{RGB}{255,165,0}
\definecolor{ashgrey}{rgb}{1, 1, 1}
\definecolor{BLACK}{rgb}{0, 0, 0}
\definecolor{WHITE}{rgb}{1, 1, 1}

\usepackage{multirow}

\usepackage[inline]{enumitem}

\makeatletter
\newcommand{\mylabel}[2]{#2\def\@currentlabel{#2}\label{#1}}
\makeatother

\newcommand{\hezi}[2]{#2}

\usepackage{algorithm,algpseudocode}
\usepackage{amsmath}
\usepackage{amsfonts}
\usepackage{amssymb}

\usepackage{fancyhdr}

\usepackage[font=small,
            skip=1ex
            ]{caption}

\renewcommand*{\arraystretch}{1.5}

\begin{document}

\frontmatter
\thispagestyle{empty}

\noindent

\begin{center}
    \Huge\bfseries
The Residual Spectrum of $F_4$ Arising from Degenerate Eisenstein Series
\end{center}





\vfill
\begin{center}
    \Large
    Thesis Submitted in Partial Fulfillment\\
    of the Requirements for the Degree of \\
    “DOCTOR OF PHILOSOPHY”\\
\end{center}

\vfill
\begin{center}
    \Huge\bfseries
    By \\
	Hezi \hspace{0.2cm} Halawi
\end{center}

\vfill\vfill
\begin{center}
    \large
    Submitted to the Senate of Ben-Gurion University of the Negev \\
\end{center}

\vfill
\begin{center}
    \huge\bfseries
    April  2021
\end{center}

\vfill
\begin{center}
\large
    Beer Sheva
\end{center}

\newpage
\thispagestyle{empty}
\mbox{}

\newpage
\thispagestyle{empty}

\noindent

\begin{center}
    \Huge\bfseries
The Residual Spectrum of $F_4$ Arising From Degenerate Eisenstein Series
\end{center}

\vfill
\begin{center}
    \Large
    Thesis Submitted in Partial Fulfillment\\
    of the Requirements for the Degree of \\
    “DOCTOR OF PHILOSOPHY”\\
\end{center}

\vfill
\begin{center}
    \Huge\bfseries
    By \\
	Hezi \hspace{0.2cm} Halawi
\end{center}

\vfill\vfill
\begin{center}
    \large
    Submitted to the Senate of \\
    Ben-Gurion University of the Negev \\
\end{center}

\vfill
\noindent
    Approved by the advisor: \underline{\qquad\qquad\qquad\qquad} \\
    Approved by the Dean of the Kreitman School of Advanced Graduate Studies: \underline{\qquad\qquad\qquad\qquad} \\

\vfill
\begin{center}
    \huge\bfseries
    April 2021
\end{center}

\vfill
\begin{center}
\large
    Beer Sheva
\end{center}

\newpage
\thispagestyle{empty}
\mbox{}

\newpage
\thispagestyle{empty}
\noindent
This work was carried out under the supervision of Prof. Nadya Gurevich \\
In the Department of Mathematics \\
Faculty of Natural Sciences \\

\newpage
\thispagestyle{empty}
\mbox{}

\newpage
\thispagestyle{empty}

\begin{center}
\underline{Research-Student's Affidavit when Submitting the Doctoral Thesis for Judgment}
\end{center}

I, Hezi Halawi, whose signature appears below, hereby declare that (Please mark the appropriate statements):

\underline{X} I have written this Thesis by myself, except for the help and guidance offered by my Thesis Advisors.

\underline{X} The scientific materials included in this Thesis are products of my own research, culled from the period during which I was a research student.

\underline{\quad} This Thesis incorporates research materials produced in cooperation with others, excluding the technical help commonly received during experimental work. Therefore, I am attaching another affidavit stating the contributions made by myself and the other participants in this research, which has been approved by them and submitted with their approval.

\begin{center}
Date: \underline{\qquad 07.04.2021 \qquad}  \hfill
Student's name: \underline{Hezi Halawi} \hfill
Signature:\underline{\qquad\qquad\qquad}
\end{center}

\newpage
\thispagestyle{empty}
\mbox{}

\newpage
\chapter*{\centering Acknowledgments}
\thispagestyle{empty}

First I would like to thank my advisor Nadya Gurevich for endless hours of fruitful discussions and for her consistent guidance. With her support, personal help and exceptional patience, I made it through the thesis.

I would like to thank my dearest friend Nir Schwartz for tolerating all my breaking points and for being there for me personally and professionally. He helped me a lot with the brainstorming and the editorial discussions.

I would also like to thank my friend and colleague Avner Segal for our mathematical discussions and for helping me taking my firsts steps in the field.

In addition I would like to show my gratitude to my loving family for being there and believing in me all those years. Without them, I couldn't do anything.

I also would like to thank the referees for reading through this thesis and for making
many helpful remarks that improved this thesis.

Last but not least I wish to thank from the bottom of my heart to my beloved wife for giving me a supportive shoulder and helping me pushing through until I finish this project.

\newpage
\thispagestyle{empty}
\mbox{}

\newpage
\thispagestyle{empty}

\newpage
\thispagestyle{empty}

\newpage
\tableofcontents


\chapter{Abstract}

Let $G$ be a split group of type $F_4$ defined over a number field. We study the square-integrable automorphic representations of $G$
that can be realized as leading terms of degenerate Eisenstein series associated to various  maximal parabolic subgroups. These representations appear in the residual spectrum. 
The local representation theory over finite places plays a central role in our work.
\vfill
\underline{Key words:}
Automorphic forms, Eisenstein series, Residual spectrum.

\mainmatter

\chapter{Introduction}

 The degenerate Eisenstein series are an important class of automorphic
 forms. Their residues and more general leading terms have been extensively
 studied for two main reasons:
 \begin{itemize}
   \item
   They are known to produce small automorphic representations,
such as minimal representations (see \cite{MR1469105}).
\item
  The degenerate Eisenstein series  are  used as an analytic ingredient
  in various Rankin-Selberg integral representations.
  Hence their analytic behavior often governs the analytic  behavior of $L$-functions (see \cite{GPSR}).
 \end{itemize}
 
 In our M.Sc. thesis, we studied the poles of degenerate {\bf spherical}
 Eisenstein series for exceptional groups. In this thesis, which is a natural
 continuation of the project, we study the leading terms of degenerate Eisenstein
 series of the split group of exceptional type $F_4$. Our results go far beyond the
 spherical case. To formulate our main theorem, we first introduce the general setting.

\subsection{General setting}

Let $F$ be a number field. We denote by $\A$ the ring of adeles over $F$.
There is a decomposition $\A=\A_f\times F_{\infty}$, where $\A_f$ is the ring
of finite adeles and $F_\infty=\Pi_{\nu |\infty} F_\nu$.

Let $G$ be a split, simply-connected group defined
over a number field $F$, let $G(\A)=G(\A_f)\times G_\infty$ be the group of points
over adeles and let $K=K_f\times K_\infty$ be a maximal compact subgroup of $G(\A)$.   

For a maximal parabolic subgroup $P$ of $G$ and a complex number $z$, 
\textcolor{black}{
we denote by  $\fun{P}$  the fundamental weight that corresponds to the unique simple root in the Lie algebra of the radical of $P$.}   
Consider a degenerate principal series  representation
$\Pi_P(z)=\Ind^{G(\A)}_{P(\A)} \bk{z\fun{P}}$ of the adelic group $G(\A)$.
Denote by  $\Pi_{P,K_\infty}(z)\subset \Pi_P(z)$ the $G(\A)$ representation 
generated by the subspace of $K_\infty$-fixed standard sections.

For any standard $K_\infty$-fixed section $f$, the degenerate Eisenstein
series  $E_P(f,z)$ converges for  $Re(z)\gg 0$  and defines an automorphic form.
Moreover, it admits a meromorphic continuation to the complex plane.
If $E_P(f,z)$ has a pole at $z=z_0$, then the
leading term of its Laurent expansion around the point defines an automorphic form
on $G(\A)$.  It is expected
that for  a real non-negative number $z_0$, the order of the pole of $E_P(f,z)$
at $z=z_0$ is bounded,  as $f$ runs over the standard sections.
Then the leading term of $E_P(z)$ defines a $G(\A_f)$-equivariant  operator
from  $\Pi^{K_\infty}_{P}(z_0)$ to the space of automorphic forms $\mathcal A(G)$.
\hezi{
Consider a space $\mathcal A_{deg, K_\infty} (G)$  to be the space of automorphic forms
generated by images of the leading terms of poles at real positive points
of degenerate Eisenstein series $E_P(z)$ described above.}
     {Let  $\mathcal A_{deg, K_\infty} (G)$ be the space of automorphic forms generated
       by the leading terms of degenerate Eisenstein series $E_P(z)$ associated with maximal
       parabolic subgroups  at real positive points.}
It contains a subspace $\mathcal A_{deg, K_\infty}^2(G)$  of square integrable functions.
The completion of the latter space with respect to the natural inner product,
denoted by $L^2_{deg}(G,K_\infty),$ is a natural representation of $G(\A)$.
Note that    $L^2_{deg}(G,K_\infty)$ is a subspace of the discrete residual spectrum
$L^2_{res}(G(F)\backslash G(\A))$. Each of its  irreducible constituents,
obtained as a leading term of $E_P(z)$ around 
$z=z_0,$ is naturally a quotient of the degenerate principal series $\Pi_{P,K_\infty}(z_0).$ 

\subsection{ The goal}
The goal of the thesis is to describe completely the decomposition of the space
$L^2_{deg}(G, K_\infty)$ for the split group $G$ of exceptional type $F_4$. We enumerate
the maximal parabolic subgroups of $G$ as $P_1,\ldots, P_4$. 
We focus on the data \hezi{}
$\mathcal P=\{ (\para{P}_1, 1), (\para{P}_3,\frac{1}{2} ), (\para{P}_4,\frac{5}{2}))\}$. 
For $(P_i,z_i)\in \mathcal P$ and a finite place $\nu$, the local degenerate principal series
$\Pi_{P,\nu}(z_0)$ admits a maximal semi-simple quotient $\pi^i_{1,\nu}\oplus \pi^i_{2,\nu}$,
where $\pi^i_{1,\nu}$ is the unique spherical quotient.
For an infinite place $\nu$, the representation $\pi^i_{1,\nu}$ is the unique spherical
quotient of $\Pi_{P,\nu}(z_0)$. 
Consequently, the maximal semi-simple quotient of $\Pi_{P_i,K_\infty}(z_i)$ has \hezi{a}{the} form 
 
$$\oplus_{|\mathcal{S}|<\infty} \pi^{\mathcal{S},i},\quad \pi^{\mathcal{S},i}=
\left(\otimes_{\nu\in \mathcal S} \pi^i_{2,\nu}\right)\otimes \left(\otimes_{\nu\notin \mathcal S} \pi^i_{1,\nu}\right).$$
Here, $\mathcal{S}$  runs over finite sets of finite places of $F$.

 With these notations, we can formulate our main theorem.
 
 \begin{Thm}\label{intro::main::thm} 
 $ $
   The space $L^2_{deg}(G,K_\infty)$ is multiplicity free.   
   Any irreducible constituent  $\Pi \subset L^{2}_{deg}(G, K_\infty)$
   is realized as the leading term of a degenerate Eisenstein  series
   $E_{\para{P}_i}(z)$ at $z=z_i$ for $(\para{P}_i,z_i)\in
   \{ (\para{P}_1, 1), (\para{P}_3,\frac{1}{2} ), (\para{P}_4,\frac{5}{2}),
   (\para{P}_2,\frac{5}{2} )\}$. 

   \begin{enumerate}
     \item
       The Eisenstein series $E_{\para{P}_2}(z)$ has a simple pole  at $z=\frac{5}{2}$, and
       the image of its leading term is the trivial representation.
     \item
    The Eisenstein series $E_{\para{P}_1}(z)$ has a simple pole  at $z=1$, and  the image of its leading term is   
$\Sigma_1=\oplus_{|\mathcal{S}|\,{\rm is\, even}} \pi^{\mathcal{S},1}$.
 \item
    The Eisenstein series $E_{\para{P}_4}(z)$ has a simple pole  at $z=\frac{5}{2}$, and the image of its leading term is   
$\Sigma_4= \oplus_{|\mathcal{S}|\,{\rm is\, even}} \pi^{\mathcal{S},4}$.
     \item
    The Eisenstein series $E_{P_3}(z)$ has a double pole  at $z=\frac{1}{2}$, and the image of its leading term is   
$\Sigma_3= \oplus_{|\mathcal{S}|\neq 1} \pi^{\mathcal{S},3}$.
   \end{enumerate}
 \end{Thm}

 For every $i$, the representations $\pi^{\mathcal{S},i}$ are nearly equivalent
 for all $\mathcal{S}$ and their local constituents at a place $\nu$
 belonging to the same local Arthur packet, see \cite{ArthurParametrs}. 
 In the first three cases, the  local constituents are expected to constitute the entire local Arthur packet. In 
 the fourth case, the local constituents are  not expected to exhaust it.
 Our description agrees with Arthur's multiplicity formula.

\subsection{Main difficulties} 
 During the work on the thesis, we had to resolve problems
 in the local representation theory as well as  problems of a global nature.

 The local problems include:
 \begin{itemize}
 \item
   Description of the constituents of various degenerate
   principal series and their Jacquet modules.
   For this purpose we invoked the branching rule
   method. 
 \item
  Holomorphicity of local intertwining operators at certain points,
   often restricted to lines, a
   description of their images and a description of the action.
   For this purpose we used computations in the Iwahori-Hecke algebra.
 \end{itemize}

 The global problem includes the study of the poles
 of \hezi{}{the} constant term  of Eisenstein series. The constant term can be written
 as a sum over numerous terms, with  some terms having higher order poles
 than the sum. Using the local theory, we  prove all the cancellations.

 In the next section we shall explain how our results fit into
 the framework of the Langlands spectral decomposition
 and Arthur's conjectures. 
 
 \section{Decomposition of the square integrable spectrum}
 We begin with some notation.
 Let  $F_{\nu}$ denote the completion of the number field
 $F$ with respect to a valuation $\nu$. If $\nu$ is \color{black}non-archimedean, \color{black}
 $\mathcal{O}_{\nu}$ stands for the ring of integers of $F_{\nu}$.

 Let $\G$ be a semi-simple, simply connected  split group with a trivial
 center defined over $F$.
 Then $\G(F_{\nu})$ (resp. $\G(\A)$)
 stands for the group of its $F_{\nu}$ (resp. $\A$) points.
 We fix a maximal split torus $\T$ of $\G$ and a Borel subgroup
 $\T \subset \B $.  We also denote by $\weyl{\G}$  the Weyl group of $\G$.

\subsection{The discrete spectrum $L^2_{disc}(\G)$} 
One of the main tasks in the field of automorphic forms concerns
the decomposition of the space  $L^{2}(\G(F) \lmod \G(\A))$,
which is a unitary representation of $\G(\A)$.
To relax the notation, we denote it by $L^{2}(\G)$. 
Recall that $L^{2}(\G) = L^{2}_{cont}(\G) \oplus L^{2}_{disc}(\G),$
where $L^{2}_{disc}(\G)$
is the sum of all irreducible subrepresentations and 
$L^{2}_{cont}(\G)$  stands for  its orthogonal complement.

The space $L^{2}_{disc}(\G)$ admits further
decomposition into the cuspidal spectrum 
$L^{2}_{cusp}(\G)$  and its orthogonal complement 
$L^{2}_{res}(\G)$, which is called "the residual spectrum." 

For any irreducible representation $\pi$, we denote by $m(\pi)$, $m_{res}(\pi)$, $m_{cusp}(\pi)$ its multiplicity in $L^{2}_{disc}(\G)$,$L^2_{res}(\G)$ and $L^{2}_{cusp}(\G)$ respectively. Set
\begin{align} 
L^{2}_{res}(\G) &= \widehat{\oplus} m_{res}(\pi) \pi, &
L^{2}_{cusp}(\G) &= \widehat{\oplus} m_{cusp}(\pi) \pi.
\end{align}

\subsubsection{The Langlands decomposition}

According to Langlands, \cite{Langlands}, the space $L^{2}_{disc}(\G)$ admits another decomposition with respect to  cuspidal data; namely, 
$$L^{2}_{disc}(\G)= \widehat{\oplus}_{[\MM ,\sigma]} L^{2}_{[\MM,\sigma],disc},$$  
where the sum ranges over all the pairs  of a standard Levi subgroup $\MM$ and a cuspidal representation $\sigma $  of $\MM(\A)$, up to equivalence. Moreover, Langlands proved that $L^{2}_{[\MM,\sigma],disc}(\G)$ is generated by iterated   residues of Eisenstein series associated with $(\MM,\sigma)$. It is worth  noting that
\begin{align*}
L^{2}_{cusp}(\G) &=L^{2}_{[\G,\sigma],disc}(\G) & \text{ and } \quad  
L^{2}_{res}(\G) =  \widehat{\oplus}_{[\MM,\sigma]}L^{2}_{[M,\sigma],disc}(\G) \text{, where } \MM \subsetneq \G.
\end{align*}
Let us focus on the cuspidal data  $[\T,\Id]$. The space
$L^{2}_{[\T,\Id],disc}(\G)$ is generated by iterated residues of
$E_{\para{B}}(\lambda)$, where $\lambda$ is an unramified  character of
$\T(F) \lmod \T(\A)$.
The Eisenstein series
$E_{\para{B}}(\lambda)$  is a meromorphic function of several
complex variables, and  the computations of its iterated residues are
very involved as the $\operatorname{rank}(\G)$ increases.
 
On the other hand, a degenerate Eisenstein series $E_{\para{P}}(z)$ corresponds to
a standard maximal parabolic subgroup $\p$, which is a function of one complex variable $z$.
These series appear naturally as a restriction of $E_{\para{B}}(\lambda)$  to a specific line.
Thus the space  $L^{2}_{deg}(\G)$ described earlier is a subspace of $L^{2}_{[\T,\Id],disc}(\G).$
 
\subsubsection*{Decomposition of the discrete spectrum a la Arthur}
Arthur provided a conjectural decomposition of $L^{2}_{disc}(\G)$ in terms of Arthur parameters.
Recall that an Arthur parameter is a morphism,   
$\psi \: : \:  \mathcal{L}_{F} \times \SL_{2}(\C) \rightarrow  {}^{L}G.$
Here, $\mathcal{L}_{F}$ is the conjectural Langlands group, and the properties of the parameter  
are listed in \cite[Part I. 3]{ArthurParametrs}. 

An Arthur parameter $\psi$ is called {\bf unipotent} if
its restriction to $\mathcal{L}_F$ is trivial
and the image of a non-trivial unipotent element of $\SL_2(\C)$ belongs to  a distinguished
unipotent orbit of ${}^{L}G$.
Thus, conjugacy classes of unipotent parameters stay in one-to-one correspondence with
distinguished unipotent orbits of $^{L}G$. 

From now on, we shall assume that the Arthur parameter $\psi$ is unipotent.
Let $S_\psi$ be the finite group of connected components of
$\operatorname{Cent}_{{}^{L}G}(\Image \psi)$.

The parameter  $\psi$ gives rise, by restriction,  to a family $\set{\psi_{\nu}}$
of local Arthur parameters
$\psi_{\nu} \: : \:  \mathcal{L}_{F_{\nu}} \times  \SL_2\bk{\C} \:  \:  \rightarrow  {}^{L}G$. 
Since $\psi$ is unipotent,  $$C_{\psi_\nu}=\operatorname{Cent}_{{}^{L}G}(\Image \psi_\nu)=C_\psi$$
and the group of components 
\hezi{
$S_{\psi_\nu}=C_{\psi_\nu}/ C^0_{\psi_\nu} Z({^LG})=S_\psi$}{$S_{\psi_\nu}= S_{\psi}$ } for all $\nu$.

The group $\mathbf{S_\psi}$ is a finite group scheme and
$\mathbf{S_\psi}(\A)=\Pi'_{\nu} \mathbf{S_\psi}(F_\nu)=\Pi'_{\nu} S_{\psi_\nu}$.
The group $S_\psi=\mathbf{S_\psi}(F)$ is embedded diagonally into
$\mathbf{S_\psi}(\A)$.

For each irreducible representation $\eta=\otimes \eta_\nu$ of $S_\psi(\A)$, we define
$m(\eta)$ to be the multiplicity of $\eta$ in the space $L^2(S_\psi(F)\backslash S_\psi(\A))$.
Explicitly,
\begin{equation}\label{mult}
  m(\eta) = \frac{1}{|S_{\psi}|}
  \left( \sum_{s \in S_{\psi}} \operatorname{tr} \eta(s)\right). 
  \end{equation}

We can formulate Arthur conjectures for the unipotent parameters. 

\begin{enumerate}
 \item Existence of local A-packets  (\cite[Section 6]{AtrhurConjecture}).
For every $\nu$ there exists a set $\Pi_{\psi,\nu}$, called a local
A-packet, of admissible unitarizable representations of $\G(F_{\nu})$,
parameterized by the set of the irreducible representations of
$S_{\psi_\nu}:$
$$\Pi_{\psi,\nu} =\set{\pi(\psi_{\nu}, \eta_{\nu}) \: :\:  \eta_{\nu} \in \widehat{S_{\psi_\nu}}}.$$
  The representation  $\pi(\psi_{\nu}, Triv)$ is necessarily spherical with the prescribed Satake
  parameter. \color{black} In the case of real group this can be found in \cite{Adams1992}.
  \color{black} 

\item
For any irreducible
representation $\eta=\otimes \eta_\nu$ of $S_\psi(\A)$, one can form a  representation
$\pi(\eta)=\otimes_{\nu}\pi(\psi_{\nu},\eta_{\nu})$  of $G(\A).$ 
The global Arthur packet  $\Pi_{\psi}$ is defined to be
$$\Pi_\psi=\{ \pi(\eta), \quad \eta \in \widehat{S_\psi(\A)}\}.$$

\begin{Conj} (multiplicity formula) 
  The multiplicity of $\pi(\eta)$ in $L^2_{disc}(\G)$ equals
  $m(\eta)$ defined in \ref{mult}.
\end{Conj}

Assuming the conjecture, there exists a  subspace
$L^2_\psi(\G)$ of $L^2_{disc}(\G)$ such that 
$$L^2_\psi(\G)=\oplus_{\eta} m(\eta)\pi(\eta).$$

\end{enumerate}

\subsubsection{The main theorem and Arthur's conjectures}
It is expected that $L^{2}_{[\T,\Id],disc}(\G)$, and hence the space of  interest
$L^2_{deg}(G, K_\infty)$, is contained in
$\oplus_{\psi} L^{2}_{\psi}(\G),$ where the sum runs over the unipotent Arthur parameters $\psi$.
    
We now can interpret \Cref{intro::main::thm}  in terms of Arthur's conjectures.
The complex dual group of  $\G$ is $^LG=F_4(\C)$. It  admits
$4$ distinguished unipotent orbits labeled  as in \cite{ORbitsF4}.
\begin{itemize}
\item
  The principal orbit $F_4$. In this case, $S_{\psi} =1$. The local and global packets
  are singletons and consist of the trivial representation, which is
  contained in $L^2_{[\T,\Id],disc}(\G)$. \color{black} 
  The Satake parameter  of the trivial representation is $-\sum_{i=1}^{4} \fun{i}$
  . \color{black}
\item
  The sub-regular orbit $F_{4}(\alpha_1)$.
  In this case, one has $S_{\psi}\simeq \Z_{2}$. The local packet
  is expected to be $\{\pi_{1,\nu}^{4}, \pi_{2,\nu}^{4}\}$.
  \hezi{
  The global packet consists of representations $\pi^{\mathcal{S},4},$
  where $S$ is a finite set of finite (???) places.}
  {Every $\pi^{\mathcal{S},4}$ where $S$ is a finite set of finite is expected to belong to the global packet.}
  \hezi{The}{}Arthur's conjecture predicts that
  $m(\pi^{\mathcal{S},4})=1$ for $|\mathcal{S}|$ even and zero otherwise. 
  This agrees with our description of  $\Sigma_4$. \color{black} The Satake parameter  of   $\pi^{4}_{\nu}$ is   
  $ -\fun{1} -\fun{3}-\fun{4}$. \color{black}
\item
  The  orbit $F_{4}(\alpha_2)$. In this case, one has $S_{\psi}\simeq \Z_{2}$.
  The description is similar to the above and agrees with our description of  $\Sigma_1$. \color{black}
  The Satake parameter  of   $\pi^{1}_{\nu}$ is   
      $ -\fun{1} -\fun{3}$.\color{black}
\item
  The  orbit $F_{4}(\alpha_3)$. In this case, one has $S_{\psi}\simeq S_{4}$.
  There are five irreducible representations of $S_4$, among them the trivial
  representation and the unique two-dimensional representation.
  Consequently,  the Arthur packet is expected to have $5$ elements,
  among them the spherical representation $\pi^3_{1,\nu}$ and the representation $\pi^3_{2,\nu}.$
  Hence, the global Arthur packet contains the representations $\pi^{\mathcal{S},3}.$
  According to Arthur's multiplicity formula
  $$m(\pi^{\mathcal{S},3})=
  \frac{1}{24} \bk{ 2^{|\mathcal{S}|} + 3 \cdot 2^{|\mathcal{S}|}+ 8 (-1)^{|\mathcal{S}|}}.$$
  This number is zero if  $|S|=1,$ and a positive integer otherwise.
  This again agrees with our description of  $\Sigma_3$.
  Note that we only obtain the realization of  $\pi^{\mathcal{S},3}$ in the residual spectrum.
   \color{black} The Satake parameter  of   $\pi^{3}_{\nu}$ is     
  $ -\fun{3}$. \color{black}.
\end{itemize}
\color{black}
\subsection*{Relation to the Springer correspondence}

The Springer correspondence is  an injective map from the
set of irreducible representations of $\weyl{\G}$, the Weyl group of $\G$,
into the set of pairs $(\mathcal{O}, \eta)$, where
$\mathcal{O}$ is a unipotent orbit of $^{L}G$ and $\eta$ is  an irreducible representation of the finite group $S_{\mathcal{O}}$ (see \cite[Chapter 10]{ORbitsF4} and \cite{Adams1992}).

For a given $\mathcal O$, we denote by $Springer(\mathcal O)$ the set
of all $\eta$ such that $(\mathcal{O} , \eta)$ belongs to the image of the correspondence.

Arthur's conjecture does not indicate which representations in the packet $\Pi_\psi$
occur in the cuspidal spectrum, which do so in the residual and which occur in both.

Given a unipotent Arthur parameter $\psi$, for every $\nu$
we denote by  $\Pi_{\psi_\nu,res}\subset  \Pi_{\psi_{\nu}}$ the set of representations
that occur as  a local component of a representation $$\pi \subset L^{2}_{[\T,\Id],disc}(\G) \subset L^{2}_{res}(\G).$$

Let $\psi$ be a unipotent parameter corresponding to the unipotent orbit $\mathcal O$. Note that  $S_\psi=S_{\mathcal O}$.  It is expected  that the representations of $\Pi_{\psi_{\nu},res}$ are
parameterized by those representations $\eta$ of $S_{\psi_\nu}=S_{\mathcal O}$
that belong to $ Springer (\mathcal O)$. 

This expectation holds in the following cases: 
\begin{itemize}
\item
$\G$ is a classical split orthogonal or  symplectic group (see \cite{MoeginOrbits}).
\item
$\G$ is a split group of type $G_2$ (see \cite{KimG2}).
\end{itemize}

This expectation holds also in the case where $\G$ is  of type $F_4$ and $\psi$ corresponds to the orbits $F_4$ or $ F_{4}(\alpha_1)$ or  $F_{4}(\alpha_2)$. By \cite[p.432]{ Orbits} one has 
$|Springer( F_{4}\bk{\alpha_1})| = |Springer\bk{ F_{4}(\alpha_2)}|=2$. 
 In the case where $\psi$ corresponds to the orbit $F_4(\alpha_3)$, 
the expectation is that $|\Pi_{\psi_{\nu},res}|=4$. 
In this thesis, we only   show that  $\set{\pi_{1,\nu}^{3},\pi_{2,\nu}^{3}} \subset \Pi_{\psi_{\nu},res}$. In such a case, $L^{2}_{deg}(\G)$ is expected not to exhaust all of $L^{2}_{[\T,\Id],disc}(\G)$.

\subsection*{The outline of the thesis}

\begin{itemize}
\item
  In \Cref{chapter::preimalies}  (Preliminaries) we set our notations.
  We define the local and global degenerate principal series.
  We also review the degenerate Eisenstein series $E_{\para{P}}(\cdot)$
  and its constant term, $E_{\para{P}}^{0}(\cdot)$, along $\U$, the unipotent radical of $\B$,
  and we introduce global normalizing factors and normalized local intertwining operators.
 \item 
   In \Cref{chapter::local method} we develop tools for studying
   degenerate principal series over local non-Archimedean fields.
   Specifically, we review  the necessary properties of
   Iwahori-Hecke algebra and explain the concept of the branching rule method. 
\item
In  \Cref{chapter::local::intro} 
  we show holomorphicity of certain local normalized intertwining operators and compute their images. In addition, we completely describe the structure of relevant degenerate principal series.
\item
In \Cref{chapter::global} we  prove our main result \Cref{into::global::main}.  
\item
In \Cref{App:knowndata} we list and prove all the branching rules needed for the computations in \Cref{chapter::local::intro}.  
\item  Some of the proofs in the thesis rely on the data supplied in tables
 that appear in  \Cref{app:localtable} and \Cref{app: globaltable}.
  These tables were obtained using \textit{Sagemath} \cite{sagemath}. Most of them may also be obtained manually. In  
\Cref{app:compute}, we elaborate on the  computational aspects of the thesis.
\end{itemize}
\newpage
\chapter{Preliminaries} \label{chapter::preimalies}

Let $F$ be a number field. For each place $\nu$ of $F$, let $F_{\nu}$ denote the completion of $F$ with respect to $\nu$. The ring of adeles of $F$ is denoted by $\A = \A_{F}$.  If $\nu$ is a finite place, the ring of integers of $F_{\nu}$ is denoted by $\mathcal{O}_{\nu}$. Also, $q_{\nu}$  denotes  the cardinality  of the residue field. We let $\varpi_{\nu}$ denote a uniformizer of $F_{\nu}$.

\paragraph{G as an algebraic group.}
Let $\G$ be a split, simply-connected, reductive group  defined over $F$ with a maximal split torus $\T$. Let $\mathfrak{g}$ denote the Lie algebra of $\G$. 

Let $\bfX(\T) = \Hom(\T, \Gm)$ be the character group of $\T$, and let $\bfY(\T) = \Hom(\Gm, \T)$ be the cocharacter group. Both $\bfX(\T)$ and $\bfY(\T)$ are
free abelian groups of rank $\dim \T$.  Since $\Aut(\Gm) =\Z$, there is a canonical pairing $$\inner{,}\: : \:  \bfX(\T) \times \bfY(\T) \rightarrow \Z.$$

The maximal torus $\T$ acts on the Lie algebra $\mathfrak{g}$ via the adjoint action. Since $\G$
is reductive, the zero eigenspace of $\mathfrak{g}$ with respect to $\T$ is exactly the Lie algebra $\mathfrak{t}$ of $\T$  itself. Therefore we can decompose $\mathfrak{g} = \mathfrak{t} \oplus_{\alpha \in \Phi_{\G}} \mathfrak{g}_{\alpha},$
where  $\Phi_{\G}$ is the set of all $0 \neq \alpha \in \bfX(\T)$ such that $\mathfrak{g}_{\alpha} = \set{x \in \mathfrak{g} \: : \:  t.x =\alpha(t) x \: \forall t \in \T}\neq \set{0}$. The set $\Phi_{\G}$ is called the set of roots of $\G$.

 We choose a Borel subgroup $\B$ such that $\T \subset \B$, and denote its Lie algebra by $\mathfrak{b}$.  There exists a closed, connected, normal subgroup of $\B$ denoted by $\U$, such that $\B = \T \U$ is a semi–direct product. Therefore, since $\T$ acts on $\B$, it
 also acts on $\mathfrak{b}$. Hence, by choosing $\B$, we also obtain a choice of positive roots  $\Phi_{\G}^{+}$  defined by $\alpha \in \Phi_{\G}^{+}$ iff $\mathfrak{g}_{\alpha} \subset \mathfrak{b}$. We define the negative roots to be $\Phi_{\G}^{-} = \Phi_{\G} \setminus \Phi_{\G}^{+}$. A positive root $\alpha$ is called a \textbf{simple root} if it cannot be written as
 a sum of two positive roots. The set of simple roots is denoted by $\Delta_{\G}$.
 
For every root $\alpha \in \Phi_{\G}$, one can associate 
a homomorphism $x_{\alpha} \: : \:  \Ga \rightarrow \G$ such that for every $c \in  \Ga$ and $t \in 
\T$, it holds $tx_{\alpha}(c)t^{-1} = x_{\alpha}(\alpha(t)c)$. Moreover, we can define a morphism $\phi_{\alpha}\: : \:  \SL_2 \rightarrow \G$ such that

\begin{align*}
\phi_{\alpha}\left(\begin{array}{rr}
1 & c \\
0 & 1 
\end{array}\right) &= x_{\alpha}(c)
&\phi_{-\alpha}\left(\begin{array}{rr}
1 & 0 \\
c & 1 
\end{array}\right) &= x_{-\alpha}(c).
\end{align*}
Define $\check{\alpha}\: : \: \Gm \rightarrow \T$ by  $\check{\alpha}(c) =\phi_{\alpha}\left(\begin{array}{rr}
c & 0 \\
0 & c^{-1} 
\end{array}\right)$. The element $\check{\alpha} \in \bfY(\T)$, and it is called the
coroot associated with the root $\alpha$. 

Let $\Omega = \set{\fun{\alpha}}_{\alpha \in \Delta_{\G}}$
be elements of $\mathfrak{t}^{\ast}$
such that $\inner{\fun{\alpha},\check{\beta}} =\delta_{\alpha,\beta}$. Here, $\delta_{\alpha,\beta}$ is the Kronecker delta. These elements are called
fundamental weights.

We set $\mathfrak{a}_{\T,\C}^{\ast} = \bfX(\T) \otimes_{\Z} \C$. We note that $\mathfrak{a}^{\ast}_{\T,\C} = \mathfrak{z}^{\ast} \oplus \Span_{\C}\set{\Omega}$, where
$$\mathfrak{z}^{\ast} =\set{x \in \mathfrak{a}^{\ast}_{\T,\C} \: : \:  \inner{x,\check{\alpha}}=0 \: \forall \alpha \in \Delta_{\G}}.$$
 
Every root $\alpha \in \Phi_{\G}$ defines $\w_{\alpha} \in \operatorname{Aut}(\bfX(\T))$ and $\w_{\check{\alpha}} \in  \Aut(\bfY(\T))$ by the formula 
 \begin{align*}
 \w_{\alpha}(x) &=  x - \inner{x,\check{\alpha}} \alpha,
 &
 \w_{\check{\alpha}}(y) =  y-\inner{\alpha ,y } \check{\alpha}. 
 \end{align*}
 We denote by $\operatorname{Norm}_{\G}(\T)$ the normalizer of $\T$ in $\G$. The Weyl group of $\G$ is defined as
 $\weyl{\G} =: \operatorname{Norm}_{\G}(\T) \rmod \T$. It is known that $\weyl{\G} =\inner{\w_{\alpha} \: : \: \alpha \in \Delta_{\G}}$ and it is a finite group.

\paragraph{Parabolic subgroup.} 
A subgroup $\p$ is called  a \textbf{standard parabolic subgroup} if $\B \subseteq \p$. There exists a correspondence between standard parabolic subgroups and subsets of $\Delta_{\G}$ via the formula $\Theta \leftrightarrow \p_{\Theta}$, where 
$\p_{\Theta}= \inner{\B ,\w_{\alpha} \: : \: \alpha \in \Theta  }.$
In addition, each standard parabolic subgroup $\para{\p}_{\Theta}$ admits a Levi decomposition $\p_{\Theta} =  \m_{\Theta} \n_{\Theta}$, where $\m_{\Theta}$ stands for the Levi subgroup  and $\n_{\Theta}$ denotes the unipotent radical of $\para{\p}_{\Theta}$. Moreover,  the set of simple roots of $\MM_{\Theta}$ is $\Theta$, and also $\Phi_{\MM_{\Theta}} =  \Phi_{\G} \cap \Span_{\Z} \set {\Theta}$. We let $\weyl{\m_{\Theta}} = \inner{\w_{\alpha} \: : \: \alpha \in \Theta}$.  
For a standard Levi subgroup $\m$, we denote by $\bfX(\m)$ (resp. $\bfY(\m)$) the characters  
(resp. cocharacter) group of $\m$. Set
$\mathfrak{a}_{\m,\C}^{\ast} =  \bfX(\m) \otimes_{\Z} \C.$ 
As before, $\mathfrak{a}^{\ast}_{\m,\C} =  \mathfrak{z}^{\ast} \oplus \Span_{\C}\set{ \fun{\alpha} \: : \: \alpha \in \Delta_{\G} \setminus \Delta_{\m }}$. We set
\begin{equation}\label{notations:modular}
\rho_{\m} =  \frac{1}{2} \sum_{\alpha \in \Phi_{\G}^{+} \setminus \Phi^{+}_{\m}} \alpha.
\end{equation}
 It is well known that $\rho_{\m} \in \mathfrak{a}_{\m,\C}^{\ast}$. 
 
Among the standard parabolic subgroups, there are proper maximal  parabolic subgroups $\p_{i}= \p_{\Delta_{G} \setminus \set{\alpha_i}}$. We also denote by $\m_{i}$ the Levi part of $\p_{i}$.

\paragraph*{ The group of type $F_4$.}
From now on we assume that $\G$ is a split group of type $F_4$.  Let us order the elements of $\Delta_\G$ with respect to the labeling of the Dynkin diagram of $\G$, given by
$$
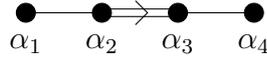
\begin{tikzpicture}[scale=0.5]
\draw (0 cm,0) -- (2 cm,0);
\draw (2 cm, 0.1 cm) -- +(2 cm,0);
\draw (2 cm, -0.1 cm) -- +(2 cm,0);
\draw (4.0 cm,0) -- +(2 cm,0);
\draw[shift={(3.2, 0)}, rotate=0] (135 : 0.45cm) -- (0,0) -- (-135 : 0.45cm);
\draw[fill=black] (0 cm, 0 cm) circle (.25cm) node[below=4pt]{$\alpha_1$};
\draw[fill=black] (2 cm, 0 cm) circle (.25cm) node[below=4pt]{$\alpha_2$};
\draw[fill=black] (4 cm, 0 cm) circle (.25cm) node[below=4pt]{$\alpha_3$};
\draw[fill=black] (6 cm, 0 cm) circle (.25cm) node[below=4pt]{$\alpha_4$};
\end{tikzpicture}
$$
In addition, there are four maximal standard parabolic subgroups. According  to \cite[Section 4]{Asgari}  one has  
\begin{align}\label{f4::parabolic}
\m_{1} &\simeq GSp_{6},  &  \m_2& \simeq \m_3 \simeq \frac{\GL_1 \times \SL_2 \times \SL_3}{A},  &  \m_4& \simeq GSpin_{7}, 
\end{align} 
where $A = \set{ \left(\zeta,\zeta^{3} I_2 ,\zeta^{2} I_{3}  \right) \: : \: \zeta^{6} =1} \simeq \inner{\zeta_{6}}$. Here $\zeta_6$ is a $6^{\textup{th}}$ primitive root of unity, and $I_2,I_3$ are the identity elements in $\SL_2$ and $\SL_3$. For shorthand writing, we let $\w_{i}$ stand for $\w_{\alpha_i}$ where $\alpha_{i} \in \Delta_{\G}$.  
The action of the Weyl group is as follows:
\begin{align} \label{weyl::action} \tag{Weyl Action}
\w_{1} \bk{\fun{\alpha_1}}&=-\fun{\alpha_1} + \fun{\alpha_2}  & \w_{j} \bk{\fun{\alpha_1}} =\fun{\alpha_1}  \quad \text{for} \quad  j \in \set{2,3,4} \\ 
\w_{2} \bk{\fun{\alpha_2}}&=\fun{\alpha_1} - \fun{\alpha_2} + 2\fun{\alpha_3}  & \w_{j} \bk{\fun{\alpha_2}} =\fun{\alpha_2}  \quad \text{for} \quad  j \in \set{1,3,4} \nonumber \\
\w_{3} \bk{\fun{\alpha_3}}&=\fun{\alpha_2} - \fun{\alpha_3} + \fun{\alpha_4}  & \w_{j} \bk{\fun{\alpha_3}} =\fun{\alpha_3}  \quad \text{for} \quad  j \in \set{1,2,4} \nonumber \\
\w_{4} \bk{\fun{\alpha_4}}&=\fun{\alpha_3} - \fun{\alpha_4}  & \w_{j} \bk{\fun{\alpha_4}} =\fun{\alpha_4}  \quad \text{for} \quad  j \in \set{1,2,3}. \nonumber 
\end{align} 
Since $\G$ is semi-simple, it follows that $\mathfrak{a}^{\ast}_{\T,\C} = \Span_{\C}\set{\fun{\alpha} \: : \:  \alpha \in \Delta_{\G}}$.

\paragraph{Representations.}

 We let $G(E)$ denote the group of $E$--points of $\G$ where $E \in \set{F,F_{\nu},\mathcal{O}_{\nu},\A}$.
Let $K_{\nu} = \G(\mathcal{O}_{\nu})$ for a finite place $\nu$ and $K_{\infty}$ be a maximal compact subgroup of $G(F_{\nu})$ in cases where $\nu$ is infinite. We note that $K_{\nu}$ is a maximal compact subgroup of $G(F_{\nu})$ for every $\nu$. Set $\K= \prod_{\nu}K_{\nu}$. Note that $\K$ is a maximal compact subgroup of $G(\A)$.  By abuse of notation, for every $ E \in \set{F,F_{\nu} , \A}$,  the element   $\lambda\in \mathfrak{a}^{\ast}_{\T,\C}$ can be realized as $\lambda \: : \:  T(E) \rightarrow \C$ via  composition with the norm map.

Given a standard parabolic subgroup $\p =\m \n$ and  
 $\lambda \in \mathfrak{a}_{\m,\C}^{\ast}$, we form a representation of $\para{P}(\A)$ by extending $\lambda$ trivially on $N(\A)$.
We consider the normalized induced representation for $\lambda  \in \mathfrak{a}_{\m,\C}^{\ast}$ 
$$\Ind_{M(\A)}^{G(\A)} (\lambda) = \operatorname{Ind}_{\para{P}(\A)}^{G(\A)}(\lambda).$$

It is realized on the space of all $C^{\infty}$ and right $\K$--finite functions on $G(\A)$,
which satisfy
\textcolor{black}{
$$ f(mn g) =  \rho_{\m }(m)\lambda(m) f(g),$$
}
where $m \in M(\A), n \in N(\A), g \in F_{4}(\A)$ and $\rho_{\m}$ as in \eqref{notations:modular}.

Note that 
$\operatorname{Ind}_{\para{P}(\A)}^{G(\A)}(\lambda)= \otimes_{\nu}^{'} \operatorname{Ind}_{\para{P}(F_{\nu})}^{G(F_{\nu})}(\lambda_{\nu})$.

\paragraph{Degenerate Eisenstein series.}
We fix a maximal parabolic subgroup $\p_{i}=\m_i\n_i$. Since $\mathfrak{a}_{\m_i,\C}^{\ast} =\Span\set{\fun{i}}$, we may write $\Ind_{M_i(\A)}^{G(\A)}(z \fun{_{i}})$.   
  A section $f_{z} \in 
\Ind_{\para{M}_i(\A)}^{G(\A)}(z \fun{i})$ is called: 
\begin{itemize}
\item
\textbf{Standard}, if its restriction to $\K$ does not depend on \textcolor{black}{$z$. Note that in the literature  it is also known as  "flat section"}. 
\item
\textbf{Spherical}, if it is standard and $\K$--invariant. In addition,  if $f(1) =1$, we call $f$ a normalized section. Since $\dim \Ind_{\para{M}_i(\A)}^{G(\A)}(z \fun{i})^{\K} = 1$, the normalized spherical section is unique, in which case we denote it as $f^{0}$. 
\end{itemize}

Given a standard section $f_{z} \in \Ind_{M_i(\A)}^{G(\A)} (z\fun{i})$ such that $f = f_{z}\res{\K}$,  we  define the Eisenstein series
\begin{eqnarray}
E_{\para{P}_i}(f,z,g)=  \sum_{\gamma \in \para{P}_i(F) \lmod G(\A)} f_{z}(\gamma g).
\end{eqnarray} 
The Eisenstein series   $E_{\para{P}_i}(f,z, g)$ converges for $z \in  \mathfrak{Z}^{+}_{i},$ where 
\begin{align*}
\mathfrak{Z}^{+}_{i} &=  \set{z \in \C \: : \: \Real{\inner{z\fun{i}-\rho^{\m_i}_{\T}, \check{\alpha}}} >1
\: : \forall \alpha \in \Phi_{\G}^{+} \setminus \Phi_{\m_i}^{+}
},
& \rho_{\T}^{\m_i}  &=  -\rho_{\m_i }  +\rho_{\T}.
\end{align*}
The converges of the degenerate Eisenstein series follows from \cite{Langlands}. 
\color{black}
\begin{Remark}
One has $\mathfrak{Z}^{+}_{i}= \set{z \in \C \: : \: \Real{z} >  \inner{\rho_{\m_i},\check{\alpha}_i}}$.
\end{Remark}
\begin{proof}
For every $\beta \in \Phi_{\G}$, let us write 
$\check{\beta} = \sum_{\alpha \in \Delta_{G} } n_{\alpha} \check{\alpha}$ and set $h(\check{\beta}) =  \sum_{\alpha \in \Delta_{G}} n_{\alpha}$.  By definition $\inner{\rho_{\T}, \check{\beta}} =  h(\check{\beta})$.

Let $w_{\m_i} \in \weyl{\m_i}$ be the longest Weyl element in $\weyl{\m_i}$. Note that 
\begin{itemize}
\item
$w_{\m_i}(\beta) \in  \Phi_{\G}^{+} \setminus \Phi_{\m_i}^{+} $ for every $\beta \in \Phi_{\G}^{+} \setminus \Phi_{\m_i}^{+}$.
\item
$w_{\m_i}(\rho_{\T}^{\m_i}) =  - \rho_{\T}^{\m_i}$.
\item
$w_{\m_i}(\fun{i}) =  \fun{i}$.
\end{itemize}

Given a $\beta \in \Phi_{\G}^{+} \setminus \Phi_{\m_i}^{+}$, 
we write $w_{\m_i}(\check{\beta}) =  \sum_{\alpha\in \Delta_{G} } m_{\alpha} \check{\alpha}$. Since $m_{\alpha_i} \geq 1$, we deduce that 
\begin{align*}
\Real\inner{z\fun{i} -\rho_{\T}^{\m_i}, \check{\beta}} &= \Real\inner{w_{\m_i}\bk {z\fun{i} -\rho_{\T}^{\m_i}}, w_{\m_i}(\check{\beta})} 
\\
&=
\Real\inner{\bk {z\fun{i} +\rho_{\T}^{\m_i}}, w_{\m_i}(\check{\beta})}\\
&= \Real \bk{ m_{\alpha_i}\bk{z -\inner{\rho_{\m_i},\check{\alpha}_i}} + h(\w_{\m_i}(\check{\beta})) }
\\
&\geq \Real \bk{m_{\alpha_i}\bk{z -\inner{\rho_{\m_i},\check{\alpha}_i}}  +1}, 
\end{align*}
showing that if $\Real{z} >  \inner{\rho_{\m_i},\check{\alpha}_i}$  then  $z \in \mathfrak{Z}^{+}_{i}$. 

In the opposite direction recall that 
 $$\inner{z\fun{i} -\rho_{\T}^{\m_i},w_{\m_i}^{-1}(\check{\alpha}_i) } =\inner{z\fun{i} +\rho_{\T}^{\m_i},\check{\alpha_i} }= z- \inner{\rho_{\m_i},\check{\alpha}_i} +1.$$

Thus, if $\Real z \leq \inner{\rho_{\m_i},\check{\alpha}_i}$ it follows that $z \not \in \mathfrak{Z}^{+}_{i}$.
\end{proof}
\color{black}
In addition,  $E_{\para{P}_i}$ admits a meromorphic continuation to $\C$. At points where $E_{\para{P}_i}(f,z,g)$ is holomorphic, it gives an automorphic form of $G(\A)$ as a function of $g$.

We say that the Eisenstein series has a pole of
finite order $l \geq  0$ for $z_0$, if for every $g \in G(\A)$ and 
for every standard section $f_{z}$  such that $f_{z}\res{\K} =f$, one has     
$$\bk{\leadingterm_{\para{P}_i}^{l}(z_0)} (f)(g) = (z-z_0)^{l}  E_{\para{P}_i}(f,z,g)\res{z=z_0}$$ that  is a holomorphic function at $z=z_0$, and there is a standard section such that $
\leadingterm_{\para{P}_i}^{l}(z_0)(f)$ is not zero.

In this case, the mapping $f \mapsto \leadingterm_{\para{P}_i}^{l}(z_0)(f)$ is an automorphic form of $G(\A)$. 

One way to study the poles of Eisenstein series is via $E_{\para{P}_i}^{0}(f,z,g)$ and its constant term  along the Borel subgroup $\B$. Recall that  
\begin{equation}\label{global::constantterm}
E_{\para{P}_i}^{0}(f,z,g) = \int_{U(F) \lmod U(\A)} E_{\para{P}_i}(f,z,ug) du. 
\end{equation}

\begin{Thm}\cite{EisensteinBook} \label{global::eisention::behavior}
The degenerate Eisenstein series $E_{\para{P}_i}(f,z,g)$ admits a pole of order $n$ at $z=z_0$ if and only if $E_{\para{P}}^{0}(f,z,g)$ admits a pole of
order $n$ at $z=z_0$.  
\end{Thm}

Next, we would like to rewrite \eqref{global::constantterm}. 
For this purpose, we recall the global intertwining operators and their normalizing factors. 

\paragraph{Global intertwining operator.}

Given $w \in \weyl{\G}$ and $\lambda \in \mathfrak{a}^{\ast}_{\T,\C}$, we formally define a global intertwining operator
$$ \M_{w}(\lambda)\: : \:  \operatorname{Ind}_{\para{B}(\A)}^{G(\A)}(\lambda) \rightarrow \operatorname{Ind}_{\para{B}(\A)}^{G(\A)}(w\lambda) $$
by the formula 
\begin{equation}\label{def:global::intertwining operator}
\bk{\M_{w}(\lambda)(f)}(g) =  \int_{U(\A) \cap w U(\A) w^{-1} \lmod U(\A)   } f(w^{-1}  ug) du.
\end{equation} 
As a function of $\lambda$, $\M_{\w}(\lambda)$ converges in some right half plane and admits a meromorphic continuation to whole  $\mathfrak{a}^{\ast}_{\T,\C}$. In addition,
the intertwining operator $\M_{\w}(\lambda)$ is factorizable in the sense that  
 if $f=\otimes_{\nu}f_{\nu}$ then 
 \begin{equation}\label{global::ontertwingfactorize}
 \M_{\w}(\lambda)f =  \otimes_{\nu}\M_{\w,\nu}(\lambda_{\nu})f_{\nu}.
 \end{equation}
 
 Here 
\begin{equation}\label{def:local::intertwining operator}
\bk{\M_{w,\nu}(\lambda_{\nu})(f)}(g) =  \int_{U(F_{\nu}) \cap w U(F_{\nu}) w^{-1} \lmod U(F_{\nu})   } f(w^{-1}  ug) du.
\end{equation}
\color{black}
More details on the local intertwining operator can be found in \Cref{local::intertwining operator}.
\color{black}
Additionally, $\M_{w}(\lambda)$ satisfies the following cocycle relation \cite[II.1.6]{EisensteinBook}. Assume that $w =u_1 u_2$, where $l(w) = l(u_1) +  l(u_2)$, then
\begin{equation}\label{global:inetr:cocyle}
\M_{w}( \lambda)  = \M_{u_1}(u_2 \lambda) \circ \M_{u_2}(\lambda).
\end{equation} 

In particular, writing $w = w_{\alpha_{i_k}} \cdot w_{\alpha_{i_{k-1}}} \cdot \dots w_{\alpha_{i_1}}$, where $l(w) =k$ and $w_{\alpha_{i_j}}$ are simple reflections, one has
\begin{align*}
 \M_{w}(\lambda) &=  \M_{w_{\alpha_{i_k}}}(w^{'}_{k-1} \lambda) \circ \M_{w_{\alpha_{i_{k-1}}}}(w^{'}_{k-2} \lambda) \circ \dots \M_{w_{\alpha_{i_1}}}(\lambda),
 & w^{'}_{j} =  w_{\alpha_{i_{j}}} \cdot w_{\alpha_{i_{j-1}}} \cdot \dots w_{\alpha_{i_1}}.
\end{align*} 

We recall the properties of $\M_{w}(\lambda)$. For every $\lambda\in \mathfrak{a}_{\T,\C}^{\ast}$ and a place $\nu$, we set $f_{\lambda,\nu}^{0} \in \operatorname{Ind}_{B(F_{\nu})}^{G(F_{\nu})}(\lambda_{\nu})$ to be the normalized spherical vector. Then 
$$\M_{w,\nu}(\lambda_{\nu})f^{0}_{\lambda,\nu} = C_{\w,\nu}(\lambda_{\nu})f_{w\lambda,\nu}^{0},$$
where 
\begin{align} \label{global:gid}
C_{w,\nu}(\lambda_{\nu}) &= \prod_{\alpha \in R(\w)} \frac{\zeta_{\nu}(\inner{\lambda_{\nu},\check{\alpha}})}{\zeta_{\nu}(\inner{\lambda_{\nu},\check{\alpha}}+1)} \quad  \text{ here } 
\quad  R(\w) =  \set{\alpha>0 \: : \: w\alpha <0}.
\end{align}   
This is known as the Gindikin--Karpeleviˇc formula. 

Returning back to our case, let $\p_{i}$ be a  maximal parabolic subgroup with a Levi decomposition $\p_i =  \m_i \n_i$. We introduce the character $\chi_{\para{P}_i,z} \in \mathfrak{a}^{\ast}_{T,\C}$, where   
\begin{equation}
\chi_{\para{P}_i,z} =  z\fun{\alpha_i} -\rho^{\m_i}_{\T}.
\end{equation}
 We omit the subscript $\para{P}_i$ when there is no confusion. Extending $\chi_{\para{P}_i,z} $ trivially on $U(\A)$,  we obtain a character of $B(\A)$. Consider the representation $\operatorname{Ind}_{B(\A)}^{G(\A)} (\chi_{\para{P}_i,z})$. Note that by induction in stages, as an abstract representation, 
  $\operatorname{Ind}_{\para{P}_i(\A)}^{G(\A)}(z \fun{i})$     
  is a subrepresentation of $\operatorname{Ind}_{B(\A)}^{G(\A)}(\chi_{\para{P}_i,z})$.
  Thus, the restriction of $\M_{w}(\chi_{\para{P}_i,z})$ to $\operatorname{Ind}_{\para{P}_i(\A)}^{G(\A)}(z\fun{i})$ defines an intertwining operator 
  $$ \M_{w}(z) \:  :  \: \operatorname{Ind}_{\para{P}_i(\A)}^{G(\A)}(z \fun{i}) \rightarrow  \operatorname{Ind}_{\para{B}(\A)}^{G(\A)}(w\chi_{\para{P}_i,z}).$$
    
By a standard unfolding as in \cite{MR1469105} and \cite{Langlands} \eqref{global::constantterm} can be rewritten as
\begin{equation}\label{global::conatnttrem_a}
E_{\para{P}_i}^{0}(f,g,z) =  \sum_{w \in W(\para{P_i},G)} \M_{w}(z) f(g),
\end{equation}
where $W(\para{P}_i,G)$ is the set of  shortest representatives of $\weyl{\G} \rmod \weyl{\m_i}$. 

By abuse of notation, for every place $\nu$ we let $f^{0}_{z,\nu}$ (resp. $f^{0}_{w.z}$) be the normalized spherical vector of $\operatorname{Ind}_{\para{P}(F_{\nu})}^{G(F_{\nu})}$ (resp.
$\operatorname{Ind}_{B(\A)}^{G(\A)}(w. \chi_{\para{P}_i,z})$).
Given a factorizable standard section $f_{z} \in \operatorname{ Ind}_{\para{P}_i(\A)}^{G(\A)}(z \fun{i})$ such that $f = \otimes_{\nu} f_{\nu}$ for all $\nu \not \in \mathcal{S}$ ($\mathcal{S}$ is a finite set of places, including the
archimedean place) $f_{z,\nu} =f_{z,\nu}^{0}$. Then one has 
\begin{align}
\M_{w}(z)f &=  \bk{\otimes_{\nu \in \mathcal{S}}\M_{w,\nu}(z)f_{\nu}} \otimes_{\nu \not \in \mathcal{S}} C_{w,\nu}(z)f^{0}_{w.z} \nonumber \\
&=  C_{w}(z)  \bk{\otimes_{\nu \in \mathcal{S}} \bk{ C_{w,\nu}(z)}^{-1} \M_{w,\nu}(z)f_{\nu} } \otimes \otimes_{\nu \not \in \mathfrak{S}} f_{w.z,\nu}^{0} \nonumber\\
&= C_{w}(z) \bk{\otimes_{\nu \in \mathcal{S}}\NN_{w,\nu}(z) f_{\nu} }\otimes \otimes_{\nu \not \in \mathfrak{S}} f^{0}_{w.z,\nu}.  \label{global::inter_b}
\end{align}
Here,
\begin{align*}
C_{w,\nu}(z) &=  C_{w}(\chi_{\para{P}_i,z}),  & C_{w}(z) &=  \prod_{\nu} C_{w,\nu}(z), &  \NN_{w,\nu}(z) = \frac{1}{C_{w,\nu}(z)} \M_{w}(z). 
\end{align*}
In other words,
\begin{equation} \label{cw::form}
C_{w}(z) = \prod_{\alpha \in R(w)} \frac{\zeta(\inner{\chi_{\para{P}_i,z},\check{\alpha}}) }{\zeta(\inner{\chi_{\para{P}_i,z},\check{\alpha}}+1)},
\end{equation}
where $\zeta(z)$ is the completed zeta function of $F$, and 
$R(\w) =  \set{\alpha > 0 \: : \: w\alpha <0}$.

Recall that $\zeta(z)$ is a meromorphic function that satisfies the functional equation $\zeta(z) =\zeta(1-z)$. Furthermore,  $\zeta(z)$ admits simple poles at $z=0,1$ and non zero for $\Real(z)>1$.

In conclusion, the analytic behavior of $\M_{w}(z)f_{z}$ depends  on the analytic behavior of $C_w(z)$ and  
$\NN_{w,\nu}(z)$ for $\nu \not \in \mathcal{S}$. In the next chapter, we recall the representation theory over finite places. We also  
   recall the properties of $\NN_{\nu,w}(z)$ over finite places in  \Cref{local::intertwining operator}.

\chapter{Representation theory}\label{chapter::local method}

Throughout the entire chapter, $F$ stands for a non-Archimedean local field
of characteristic zero. Let $\mathcal{O}$ denote its ring of integers and
$\mathcal{P}$ be the maximal ideal of $\mathcal{O}$.
Let $q= |\mathcal{O}\rmod\mathcal{P}|$ and $\mathbb{F}_{q}$
be the residue field. We let $\varpi$ denote a uniformizer of $F$. 

Let $H$ be  the group of $F$ points of a split, reductive, simply-connected group
$\mathbf{H}$  defined over $F$.
Let $T$ be a maximal split torus of $H$, and let $\para{B} =  TU$ be a Borel subgroup of $H$, where $U$ is the unipotent radical of $\para{B}$. Set $K = H(\mathcal{O})$, which is a maximal compact subgroup of $H$.

\paragraph{Characters.}
Recall that for a Levi subgroup $M$ of $H$, we denote the group of rational characters of $M$ as $\bfX(M)$. We also set $\mathfrak{a}_{M,\C}^{\ast} = \bfX(M) \otimes_{Z} \C$. By abuse of notation, we let $\lambda \in \mathfrak{a}_{M,\C}^{\ast}$ denote a complex character of $M$, by $\lambda(m) =  |\lambda(m)|,$ where $|\cdot|$ is an absolute value of $F$.  A character $\lambda \in \mathfrak{a}_{M,\C}^{\ast}$ is called unramified if $\lambda\res{M\cap K}=1$. An unramified character of $M$ has the form $$\lambda = \sum_{\alpha \in \Delta_{G}\setminus \Delta_{M}} z_{\alpha}\fun{\alpha}.$$

Here, $(z_{\alpha})$ are  complex numbers, and $z_{\alpha}\fun{\alpha}$ is interpreted as $|\fun{\alpha}(\cdot)|^{z_{\alpha}}$.

\paragraph{Representations.}

The category of complex,  smooth, admissible representations of finite length of
$H$ is denoted by
$\operatorname{Rep}(H)$. 
For a standard Levi subgroup $M$ of a parabolic subgroup $\para{P}$,
we recall two functors associating representations of $M$ and $H$:
\begin{itemize}
\item
  For a representation $\Omega$ of $M$, we denote
  the normalized parabolic induction of $\Omega$ to $H$ by $\Ind_{M}^{H}(\Omega)$.  In particular,
 $\Ind_{M}^{H} \: : \:  \operatorname{Rep}(H) \rightarrow \operatorname{Rep}(M)$.
 \item
For a representation $\pi$ of $H$, we denote the normalized Jacquet functor
of $\pi$ along $\para{P}$ by $\jac{H}{M}{\pi}$.
In particular, 
$r_{M}^{H} : \: \: \operatorname{Rep}(H) \rightarrow \operatorname{Rep}(M)$.
\end{itemize}

These two functors are adjoint to each other; namely, they satisfy the Frobenius reciprocity

\begin{equation}\label{eq::Frob}
\Hom_{H}(\pi, \Ind_{M}^{H}(\sigma)) \simeq 
\Hom_{M}(\jac{H}{M}{\pi}, \sigma).
\end{equation}

It is worth mentioning that for a $1$--dimensional representation $\Omega$ of $M$, one has $\jac{M}{T}{\Omega} = \Omega\res{T} -  \rho_{T}^{M}$ where 
$\rho_{T}^{M} = \frac{1}{2} \times \sum_{\alpha \in \Phi^{+}_{M}}\alpha$.


\paragraph{Grothendieck ring.}
In parts of this thesis, it is convenient
to consider representations of a group $H$ as elements in the Grothendieck ring $\mathcal{R}(H)$. To avoid confusion, we denote by $\coseta{\pi}$ the class of a representation $\pi$ of $H$.    

Given an  irreducible representation $\sigma$ of $H$ and  $\pi \in \operatorname{Rep}(H)$,
we denote by $\mult{\sigma}{\pi}$ the multiplicity of $\sigma$ in
the Jordan--H\"{o}lder series of $\pi$.
We write  that $\sigma \leq \pi$ if $\mult{\sigma}{\pi}\geq 1$. Note that if $\coseta{\pi} =  \coseta{\pi_1} + \coseta{\pi_2}$, then for every  irreducible representation $\sigma$ of $H$, one has 
$$\mult{\sigma}{\pi} =\mult{\sigma}{\pi_1}+ \mult{\sigma}{\pi_2}.$$
Furthermore, we write $\pi_1 \leq \pi_2$ if for every irreducible representation $\sigma$ of $H$, 
$$ \mult{\sigma}{\pi_1}  \leq \mult{\sigma}{\pi_2}.$$ 

For a standard Levi subgroup $M \subset H$, the functors $\Ind_{M}^{H}, r_{M}^{H}$  are exact. Hence they give rise to  homomorphisms, which by abuse of notation are denoted also by $\Ind_{M}^{H}, r_{M}^{H}$,
between the rings $\mathcal{R}(M)$ and $\mathcal{R}(H)$.

We recall the following fundamental results (see \cite[Lemma. 2.12]{MR0579172} , \cite[Theorem 6.3.6]{Casselman}). 
\begin{Lem}[Geometric Lemma]\label{local::geomteric lemma}
For the Levi subgroups $L,M$ of $H$, let 
$W^{M,L}$ be the set of shortest representatives in $\weyl{H}$ of $\weyl{L} \lmod \weyl{H} \rmod \weyl{M}$. For a smooth representation $\Omega$ of $M$, one has 
$$\coseta{\jac{H}{L}{\Ind_{M}^{H}(\Omega)}} =\sum_{w \in W^{M,L}} \coseta{\Ind_{L^{'}}^{L}\circ  w \circ \jac{M}{M'}{\Omega}}, $$
where for $w \in W^{M,L}$
\begin{align*}
M^{'} &=  M \cap w^{-1} (L), &
L^{'} &=  w(M) \cap L.
\end{align*}  
\end{Lem}

In particular, for an admissible representation $\Omega$ of $M$, the Jacquet functor $\jac{H}{T}{\Ind_{M}^{H}\bk{\Omega}}$, as an element of $\mathcal{R}(T)$, is  a finite sum of one-dimensional representations of $T$. Each such  representation is called \textbf{ an exponent of $\Ind_{M}^{H}\bk{\Omega}$}.

\begin{Remark}
For $\lambda \in \mathfrak{a}_{M,\C}^{\ast}$ the following  holds 
$$\coseta{\jac{H}{T}{\Ind_{M}^{H}(\lambda)}} =\sum_{w \in W^{T,M}} w\cdot \lambda_{0},  \quad  \text{where} \quad \lambda_0 =\jac{M}{T}{\lambda}=\lambda-\rho^{M}_{T}. $$
\end{Remark}

\begin{Lem}
Let $\pi$ be an irreducible representation of $H$ such that $\pi \hookrightarrow \Ind_{T}^{H}(\lambda_0)$. Then $\lambda_0 \leq \jac{H}{T}{\pi}$. 
\end{Lem}
\begin{proof}
Follows immediately from the Frobenius reciprocity.
\end{proof}

The following Lemma plays a crucial role in our work. We refer to it as the central character argument \cite[Lemma 3.12]{E6}.  
\begin{Lem}
Let $\pi$ be an irreducible constituent of a principal series. If  
$\lambda\leq\jac{H}{T}{\pi}$,  then there is an embedding of $\pi$ into $\Ind_{T}^{H}(\lambda)$. 
\end{Lem}  

\section{ Branching rule process}
The main reference for this section is \cite[Section 3]{E7}.
For  completeness of presentation we recall it. The purpose of the branching rule process is to provide partial knowledge of $\jac{H}{T}{\pi}$ for an irreducible representation $\pi$ of $H$. Specifically, given an irreducible representation $\pi$ of $H$  and an exponent $\lambda \leq \jac{H}{T}{\pi}$, the tool produces a set $\mathcal{S}$ of exponents of $\pi$ and numbers $\set{m_{\lambda}}_{\lambda \in \mathcal{S}}$ such that  
\begin{eqnarray}\label{app:barnching::purpose}
\sum_{\lambda \in \mathcal{S}} m_{\lambda}\lambda \leq \jac{H}{T}{\pi}. 
\end{eqnarray}

Note that this also provides a lower bound for $\dim \jac{H}{T}{\pi}$.

Since the Jacquet functor is exact, if $\pi_1 \leq \pi$ then $\jac{H}{T}{\pi_1} \leq \jac{H}{T}{\pi}$. 
In addition, for every representation $\pi \in \mathcal{R}(H)$ one has
$$\jac{H}{T}{\pi} =  \jac{M}{T}{\jac{H}{M}{\pi}}.$$

We say that the representations $\pi_1$ and $\pi_2$ of $H$ are \textbf{Jacquet-equivalent} if $\coseta{\jac{H}{T}{\pi_1}} = \coseta{\jac{H}{T}{\pi_2}}.$
Given an irreducible representation $\sigma$ of $H$, we denote by $\mathcal{S}_{\sigma,H}$ the set of all irreducible representations which are Jacquet--equivalent to $\sigma$.

\begin{Def}
  Let $M$ be a standard Levi subgroup of $H$, $\sigma$ be an irreducible representation of $M$
  and $\lambda_0$ is a character of $T$. The triple $(M,\sigma, \lambda_0)$ is called a \textbf{branching triple} (also known as \textbf{branching rule})  if
  $\sigma$ is the unique, irreducible representation of $M$, up to Jacquet equivalence,  such that
  $\lambda_{0} \leq \jac{M}{T}{\sigma}$. The multiplicity of $\lambda_0$  in
  $ \jac{M}{T}{\sigma}$ is denoted by $n_{\sigma,M}(\lambda_0)$. 
  \end{Def}


\begin{Lem}\label{barnching::lemma}
Let $(M,\sigma,\lambda_0)$ be a branching triple. 
Let $\pi$ be an irreducible representation of $H$ such that $mult(\lambda_0,\jac{H}{T}{\pi})=k>0$.
Then 
$$\coseta{\jac{H}{T}{\pi}} \geq
\left\lceil\frac{k}{n_{\sigma,M}(\lambda_0) }\right\rceil \times \coseta{\jac{M}{T}{\sigma}}.$$
\end{Lem}

\begin{proof}
  Recall that $\coseta{\jac{H}{T}{\pi}} =
  \coseta{\jac{M}{T}{\jac{H}{M}{\pi}}}$. Moreover, one has 
  $$\coseta{\jac{H}{M}{\pi}}=
  \sum m_{\tau}\times \coseta{\tau} \quad m_{\tau}=\mult{\tau}{\jac{H}{M}{\pi}}.$$
Since $(M,\sigma,\lambda_0)$ is a branching triple and 
$k \times \lambda_{0} \leq \jac{H}{T}{\pi}$ for $k>0$, it follows that
$$ \sum_{\tau \in \mathcal{S}_{M,\sigma}} m_{\tau} \tau  \leq \jac{H}{M}{\pi}.$$

Thus, one has 
$$\coseta{\jac{H}{T}{\pi}} \geq \jac{M}{T}{ \sum_{\tau \in \mathcal{S}_{M,\sigma}} m_{\tau} \times \tau } = \bk{\sum_{\tau \in \mathcal{S}_{M,\sigma}} {m_{\tau}}} \times \coseta{\jac{M}{T}{\sigma}}.$$
In particular,  
$$\bk{\sum_{\tau\in \mathcal{S}_{M,\sigma}}m_{\tau}} \geq \left\lceil\frac{k}{n_{\sigma,M}(\lambda_0) }\right\rceil.$$
 Thus, the claim follows.  
\end{proof}

In \Cref{App:knowndata} we list all the  branching triples used in this thesis.

\section{Aubert involution}
The main reference for this section is \cite[Section 3]{Aubert}.
In the Grothendieck ring, there is an involutive action  $D_{H}$ called the Aubert involution. The Aubert involution is defined as follows
\begin{equation}\label{aub::def}
 D_{H} =  \sum_{\Phi \subseteq \Delta_{H}} (-1)^{|\Phi|} \Ind_{M_{\Phi}}^{H} \circ r^{H}_{M_{\Phi}}.   
\end{equation}

In particular,  
\begin{equation}\label{aub::char}
D_{T}(\lambda) =\lambda \quad \forall \lambda  \in \mathfrak{a}_{T,\C}^{\ast}.
\end{equation}

The Aubert involution satisfies the following properties (all the equalities are in the Grothendieck  ring):
\begin{enumerate}[label=$(Aub_\arabic{*})$,ref=$(Aub_\arabic{*})$]
\item\label{aub::1}
$D_{H} \circ \Ind_{M}^{H} =  \Ind_{M}^{H}\circ D_{M}$.
\item\label{aub::2}
For a  standard Levi subgroup $M=M_{\Theta}$, let $w$ be the longest Weyl-element of the set  $\set{w \in \weyl{H} \ :  w^{-1}(\Theta) >0}$. Then 
$$ r_{M}^{H} \circ D_{H} =  w \circ D_{w^{-1}(M)} \circ r_{w^{-1}(M)}^H.$$
Note that $w^{-1}(M)$ is also a standard Levi subgroup.   In the special case where $M=T$, one has 
$$r_{T}^{H} \circ D_{H} =  w \circ r_{T}^{H},$$
where $w \in \weyl{H}$ is the longest Weyl element.
\item\label{aub::3}
$D_{H}$ takes any irreducible representation to an irreducible representation.
\end{enumerate}

Combining \ref{aub::1} and \eqref{aub::char} one has 
\begin{equation}\label{aub::prinipalseries}
D_{H}(\Ind_{T}^{H}(\lambda)) = \Ind_{T}^{H}\bk{D_{T}(\lambda)} = \Ind_{T}^{H}(\lambda).
\end{equation}

\section{Local intertwining operators}\label{local::intertwining operator}
Recall that $\weyl{H} \simeq N_{H}(T)\rmod T$, where $N_{H}(T)$ is the normalizer of $T$. Hence, it also acts on $\mathfrak{a}_{T,\C}^{\ast}$ as follows
$$\bk{w\lambda}(t) =  \lambda\bk{w^{-1}t w}, \quad w \in \weyl{G}, \quad  \lambda \in \mathfrak{a}_{T,\C}^{\ast},\quad t \in T.$$

Let $\lambda \in \mathfrak{a}_{T,\C}^{\ast}$ and $\w \in \weyl{G}$. We define, formally, \textbf{the local intertwining operator}
$\M_{\w}(\lambda) \: : \:  \Ind_{T}^{H}(\lambda) \rightarrow \Ind_{T}^{H}(\w \lambda)$ 
by the formula 
\begin{equation}\label{eq::localiter}
(\M_{\w}(\lambda)f)(g) =  \int_{ \bk{U\cap wUw^{-1}} \lmod  U  } f(w^{-1}ug)du.
\end{equation}
As long  the integral in \eqref{eq::localiter} converges, it defines an intertwining operator between $\Ind_{T}^{H}(\lambda)$ and $\Ind_{T}^{H}(\w \lambda)$. The local intertwining operator $\M_{w}(\lambda)$ satisfies the following properties:
\begin{enumerate}[label=$(A_{\arabic{*}})$, ref=$(A_{\arabic{*}})$]
\item\label{local::inter::prop::1}
$\M_{w}(\lambda)$ converges for $\lambda$ such that $\Real\bk{\inner{\lambda,\check{\alpha}}}>0$ for every $\alpha \in R(w) = \set{\alpha \: : \: w\alpha<0}$.  
\item \label{local::inter::prop::2}
$\M_{w}(\lambda)$ admits  a meromorphic continuation to all $\mathfrak{a}_{T,\C}^{\ast}$. 
\item\label{local::inter::prop::3}
If $w =  u_1 u_2$ and $l(w) =l(u_1) +l(u_2)$, then $\M_{w}(\lambda) =  \M_{u_1}(u_2 \cdot \lambda) \circ \M_{u_2}(\lambda)$.
\item\label{local::inter::prop::4}
Suppose that $Re(\inner{\lambda,\check{\alpha}}) >0$ for some $\alpha \in \Delta_{H}$. Then $\ker \M_{\w_{\alpha}}(\lambda) \neq 0$ if and only if $\inner{\lambda,\check{\alpha}}=1$. 
\item\label{local::inter::prop::5}
Let $f^{0}_{\lambda} \in \Ind_{T}^{G}(\lambda)$ denote the normalized spherical vector. Then 
$$\bk{\M_{w}(\lambda)f^{0}_{\lambda}} =  C_{w}(\lambda) f^{0}_{w\cdot \lambda},$$
where 
$$C_{w}(\lambda) =  \prod_{\alpha \in R(w)} \frac{\zeta(\inner{\lambda,\check{\alpha}})}{\zeta(\inner{\lambda,\check{\alpha}}+1)}.$$
Here, $\zeta(z)=\frac{1}{1-q^{-z}}$ for $z \in \C \setminus\set{0}$.
This property is known as the Gindikin-Karpelevich formula. 
\end{enumerate}

For global reasons, it is customary to use the \textbf{local normalized intertwining operator} 
$$\NN_{w}(\lambda)  = \prod_{\alpha \in R(w)} \frac{\zeta\bk{\inner{\lambda,\check{\alpha}}+1}}{\zeta\bk{\inner{\lambda,\check{\alpha}}}}\M_{w}(\lambda).$$ 

The normalized intertwining operator $\NN_{w}(\lambda)$   also satisfies \ref{local::inter::prop::1}--\ref{local::inter::prop::4}, while \ref{local::inter::prop::3} holds in an  even wider generality, namely,

$$\NN_{u_1 u_2}(\lambda) = \NN_{u_1}(u_2\lambda)\circ \NN_{u_2}(\lambda) \quad \text{ for any } \quad u_1,u_2\in \weyl{H}.$$  

In addition, the analog of \ref{local::inter::prop::5} is that $$\NN_{w}(\lambda)f^{0}_{\lambda} =  f^{0}_{w\lambda}.$$ 
\color{black}
{
\begin{Lem} \label{Nor::holo}
Let $\w =\w_{\alpha}$. Then $\NN_{w}(\lambda)$ is holomorphic for $\Real\bk{\inner{\lambda,\check{\alpha}}} \neq -1$.
\end{Lem}
\begin{proof}
Follows from \cite{MR517138}.
\end{proof}
 }
 \color{black}
Given  $\lambda \in \mathfrak{a}_{M,\C}^{\ast}$, by induction in stages one has  $$\pi = \Ind_{M}^{H}(\lambda) \hookrightarrow \Ind_{T}^{H}(\lambda -\rho_{T}^{M}).$$

Thus, the restrictions of $\M_{w}(\lambda-\rho^{M}_{T})$ and $\NN_{w}(\lambda-\rho^{M}_{T})$ to $\pi$ define the maps $$\M_{w},\NN_{\w} \: : \: \pi \rightarrow \Ind_{T}^{H}(\w \cdot \bk{\lambda-\rho^{M}_{T}}).$$

Note that 
\begin{align*}
\pi &= \Ind_{M}^{H}(\lambda) & \Ind_{T}^{G}(\lambda+ \rho^{M}_{T}) \overset{\NN_{w_M^{0}}}{\twoheadrightarrow} \pi. 
\end{align*}
Here, $w_{M}^{0} \in \weyl{M}$ is the longest Weyl element.

\color{black}{
We end this section with a useful Lemma:

\begin{Lem}\label{Lemma::sph}
Let $\pi$ be an irreducible constituent of a principal  series $\Pi$. Suppose that $\lambda_{a.d} \leq \jac{H}{T}{\pi}$, where 
$\lambda_{a.d} =  \sum_{\alpha \in \Delta_{H}}c_{\alpha}\fun{\alpha}$  with  $c_{\alpha} \leq 0$; then $\pi$ is spherical. 
\end{Lem}
\begin{proof}

By the central character argument, one has $\pi \hookrightarrow \Ind_{T}^{H}(\lambda_{a.d})$.  Let 
$$\Theta = \set{ \alpha \in  \Delta_{H} \: : \:  \inner{\lambda_{a.d} ,\check{\alpha}}=0} =  \set{\alpha \in \Delta_{H} \: : \:  c_{\alpha} =0}.$$
Set $W_{\Theta} = \inner{\w_{\alpha} \: : \: \alpha \in \Theta}$. 
Thus, \eqref{Eq:OR} asserts that 
$$ \Card{W_{M_{\Theta}}} \big \vert \mult{\lambda_{a.d}}{\jac{G}{T}{\pi}}
$$
Geometric Lemma implies 
$$\mult{\lambda_{a.d}} {\jac{H}{T}{\Ind_{T}^{H}(\lambda_{a.d})}}= \operatorname{Stab}_{\weyl{H}}(\lambda_{a.d})= \Card{W_{M_{\Theta}}}.$$   

The last equality holds since $\lambda_{a.d}$ is real, and hence its stabilizer is generated by the generalized reflections $\w_{\beta}$ for $\beta \in \Phi$ being orthogonal to $\lambda_{a.d}$.
As follows, if $\sigma$ is an irreducible constituent of $\Ind_{T}^{H}(\lambda_{a.d})$ such that $\lambda_{a.d} \leq \jac{H}{T}{\sigma}$, then $\sigma =\pi$.   
      
\item
We continue by showing that $\pi$ is spherical. For this purpose, we proceed as follows:  
\begin{itemize}
\item
Let $w_{0}\in \weyl{H}$ be the longest Weyl element. Note that 
$$\lambda_{d} =  w_{0} \lambda_{a.d} =
\sum_{\alpha \in \Delta_{H}}d_{\alpha}\fun{\alpha}
 $$
 for some $d_{\alpha}\geq 0$.
\item
Set $$\Theta' = \set{\alpha \in \Delta_{H} \: : \:  \inner{\lambda_{d},\check{\alpha}}=0} =
\set{\alpha \in \Delta_{H} \: : \:   d_{\alpha} =0 }$$ 
By induction in stages one has 
$$ \Ind_{T}^{H} (\lambda_{d}) \simeq \Ind_{M_{\Theta'}}^{H} \bk{ \Ind_{T}^{M_{\Theta'}}\bk{0} \otimes \lambda_{d} }.$$

The representation $\Ind_{T}^{M_{\Theta'}}\bk{ 0}$ is tempered and irreducible. Furthermore, $\lambda_{d} \in \bfX(M_{\Theta'})$ and for every $\beta \in \Delta_{H} \setminus \Theta'$ one has $\inner{\lambda_{d} , \check{\beta}} >0$. Thus,
$\Ind_{T}^{H} (\lambda_{d})$ is a standard module. Therefore, according to Langlands classification theory, it admits a unique irreducible quotient.    
\item
Write $w_{0}  = w_{0,M'} w_{M'}$ where 
$w_{0,M'} \in \weyl{H} \rmod \weyl{M_{\Theta'}}$ is the shortest representative of $w_{0}$ and 
$w_{M'} \in \weyl{M_{\Theta'}}$ is the longest Weyl element. Note that 
$w_{M'} \lambda_{d} =  \lambda_{d}$. Thus, 
$w_{0,M'} \lambda_{d} =  \lambda_{a.d}$.

Moreover, for every $\beta \in \Phi^{+}_{H} \setminus \Phi_{M_{\Theta}'}$ one has $\inner{\lambda_{d},\check{\beta}}>0$. 
Thus,
the intertwining operator $\M_{w_{0,M'}}(\lambda)$ is  holomorphic  at $\lambda =\lambda_{d}$.
To summarize 
$$\M_{w_{0,M'}}(\lambda_{d}) \: : \:  \Ind_{T}^{H}(\lambda_{d}) \to
\Ind_{T}^{H}(\lambda_{a.d}) $$
 In particular,
$\M_{w_{0,M'}}(\lambda_d)f^{0}_{\lambda_{d}}= C \times f^{0}_{\lambda_{a.d}}$ with $C \neq 0$. Here $C$ is the local Gindikin-Karpelevich factor.   
\item
On the other hand, \cite[Chapter XI. Proposition 2.6]{MR1721403}  shows that the image of $\M_{w_{0,M'}}(\lambda_{d})$ is irreducible. Thus,  $J = \Image \M_{w_{0,M'}}(\lambda_{d})$ is spherical.
\item
Recall that $J \subset \Ind_{T}^{H}(\lambda_{a.d})$; thus, one has $\lambda_{a.d} \leq \jac{H}{T}{J}$, and therefore $\pi =  J$. 
\end{itemize}

\end{proof}
}
\color{black}
\section{Iwahori--Hecke algebra} \label{loacl::section:IH}
A key tool in studying the structure of  unramified principal series $\pi = \Ind_{T}^{H}(\lambda)$, where $\lambda \in \mathfrak{a}_{T,\C}^{\ast}$ is the Iwahori--Hecke algebra. This section is dedicated to presenting the main properties of this algebra. For more information about the structure of the Iwahori--Hecke algebra and its modules, see  \cite{MR2250034}  and \cite{MR2642451}. 
 
Let $\mathbf{H}$ be a split, semi-simple, simply-connected group scheme such that $H =  \mathbf{H}(F)$, and let us assume that $\mathbf{H}$ is defined over $\mathcal{O}$. Let $\mathbf{B},\mathbf{T}$ be a Borel subgroup and a maximal split torus such that $\para{B}= \mathbf{B}(F)$ and $T= \mathbf{T}(F)$. We set  $K = \mathbf{H}(\mathcal{O})$. Let  
$$\Psi\: : \: K \: \rightarrow \mathbf{H}(\F_{q})$$ denote the projection modulo $\mathcal{P}$.      
 
Let $J = \Psi^{-1}(\mathbf{B}(\F_{q}))$ be an Iwahori subgroup of $H$. The subgroup $J$ plays an important role in the study of  unramified principal series representations of $H$. By \cite[Proposition 2.7]{MR571057}, 
if $\lambda$ is unramified, then $\Pi=\Ind_T^H(\lambda)$ is generated by its Iwahori fixed vectors and so are all of its subquotients. 

The Iwahori--Hecke algebra, $\Hecke$, consists of all compactly supported $J$ --bi-invariant complex functions on $H(F)$, namely,
$$\Hecke = \set{ f \in C_{c}(H) \: : \: f(j_1 h j_2) =f(h)  \quad \forall j_1,j_2 \in J, \quad h \in H}.$$ 
The multiplication in $\Hecke$ is given by convolution, and the measure of $J$ is set to be $1$. By \cite{Borel1976},
 there is an equivalence of categories between the category of
 admissible representations of $H$ which are generated by their $J$--fixed vectors and the category of finitely generated $\Hecke$--modules.
 This equivalence of categories sends an admissible representation $\pi$ of $H$ to the $\Hecke$--module $\pi^J$ of $J$--fixed vectors in $\pi$.
 Thus, in order to study the structure of $\Ind_{T}^{H}(\lambda)$, it is sufficient to study the  corresponding finite dimensional $\Hecke$--module.  

\subsection{The Bernstein presentation and unramified principal series}
The Iwahori--Hecke algebra $\Hecke$ can be described in terms of generators and relations. One such presentation is known as the Bernstein presentation. 
 The algebra $\Hecke$ admits two important subalgebras: 
\begin{itemize}
\item
A finite dimensional algebra $\Hecke_{0} = \inner{ T_{\s{\alpha}} \: \: : \: \alpha \in \Delta_{H}} =\Span_{\C} \set{T_{\w} \: : \: \w \in \weyl{H}}$  of dimension $|\weyl{H}|$. 
\item
An infinite dimensional commutative algebra $$\Theta = \inner{\theta_{\check{\alpha}} \: :\:  \alpha \in \Delta_{H}} = \Span_{\C}\set{\theta_{x} \: : \: x \in \Z[\check{\Delta}_{H}] },$$ 
where $\Z[\check{\Delta}_{H}]$ is the coroot lattice.
\end{itemize}
Together, these two subalgebras generate $\Hecke$. The relations between $T_{\w_{\alpha}}$ and $\theta_{\check{\alpha}}$ are listed in  
\cite[Section 3]{MR2250034}. As a vector space, $\Hecke \simeq \Hecke_{0} \otimes_{\C} \Theta$.

Given an unramified principal series $\Pi = \Ind_{T}^{H}(\lambda)$, we describe the left $\Hecke$--module, $\Hecke(\lambda)$, corresponding to it by the equivalence of categories of \cite{Borel1976}. The module $\Hecke(\lambda)$ is given as follows:   
\begin{itemize}
\item
As a vector space, $\Hecke(\lambda)=\Hecke_{0}$.
\item
The action of $\Hecke_{0} \subset \Hecke$ on $\Hecke(\lambda)$ is given by left multiplication.
\item
By the Bernstein presentation, the action of $\Theta$ on $\Hecke(\lambda)$ is determined by the action of the generators $\theta_{\check{\alpha}}\in \Theta$ on $T_e\in \Hecke(\lambda)$. Let 
\[
\theta_{\check{\alpha}} \cdot T_e = q^{\inner{\lambda,\check{\fun{\alpha}}}}T_{e}.
\]
\end{itemize}

\subsection{Intertwining operators} \label{subsection::itertwining}
We recall the normalized intertwining operators $\NN_{\w}(\lambda)$, which were introduced in \Cref{local::intertwining operator} for $\lambda \in \mathfrak{a}_{T,\C}^{\ast}$. Let $\Pi = \Ind_{T}^{H}(\lambda)$; these operators induce a map $\NN_{\w}(\lambda)\res{\Pi^{\iwhaori}}$ of $\Hecke$--modules. By \cite[Section 2]{MR2250034}, the action of $\NN_{\w_{\alpha}}(\lambda)\res{\Pi^{\iwhaori}}$ for $\alpha \in \Delta_{H}$ is given by
\textbf{right}-multiplication by the following element
$$n_{\w_\alpha}(\lambda)  = \frac{q-1}{q^{z+1}-1} T_{e} + \frac{q^{z}-1}{q^{z+1}-1}T_{\s{\alpha}} \in \Hecke_{0},\quad \text{where} \:\:  z=  \inner{\lambda,\check{\alpha}}. $$

Suppose that $\Re z > -1$. Then,
$\NN_{\w_{\alpha}}(\lambda)$ is holomorphic there. Furthermore, considered
as an element of $\operatorname{End}(H_0)$, $\NN_{\w_{\alpha}}(\lambda)\res{\Pi^{\iwhaori}}$ is a diagonalizable linear operator with two eigenvalues given by
\begin{align*}
\lambda_1 &=\frac{q-1}{q^{z+1} -1} + q\frac{q^{z}-1}{q^{z+1} -1}=1 &\text{ and } & &  
\lambda_2 &=\frac{q-1}{q^{z+1} -1} - \frac{q^{z}-1}{q^{z+1} -1}= \frac{q-q^{z}}{q^{z+1}-1},
\end{align*}

with the exception of $z = 0$, where $n_{\w_{\alpha}}(\lambda)=T_{e}$ is the identity element and $\NN_{\w_{\alpha}}(\lambda)=\id$.
Thus, $\NN_{\w_{\alpha}}(\lambda)\res{\Pi^{\iwhaori}}$
has a kernel if and only if $\lambda_2 = 0$, which happens only if $z \in  1+ \frac{2\pi i}{\log(q)}\Z$.

It follows that for $z \in \R$, the injectivity of $\NN_{\w_{\alpha}}(\lambda)\res{\Pi^\iwhaori}$
does not depend on the value of $q$.

\subsection{The submodule $\Hecke_{\para{P}}(\lambda)$ of  $\Hecke(\lambda)$}\label{subsection:: hPmoduule}

Let $\para{P}$ be a parabolic subgroup of $H$ with Levi part $M$, and let $\lambda$ be an unramified character
of $M$ with respect to $M \cap K$. We denote the longest Weyl element of $\weyl{M}$ by $w^0_{M}$. Let $\pi =  \Ind_{M}^{H}(\lambda)$ with $\lambda \in \mathfrak{a}_{M,\C}^{\ast}$
and $\Hecke_{\para{P}}(\lambda) = \pi^{\iwhaori}$.

We recall that
$$\pi  = \Ind_{M}^{H}(\lambda)  =\Image \M_{w_{M}^{0}}(\lambda + \rho^{M}_{T}) = \Image \NN_{w_{M}^{0}}(\lambda + \rho^{M}_{T}).$$

It follows that $\Hecke_{\para{P}}(\lambda)$ is the image of
$\NN_{w_{M}^{0}}(\lambda + \rho^{M}_{T})$, and hence it has a basis given by

\begin{equation}\label{iwahori ::basis::para}
\set{T_u \cdot  triv : u \in \weyl{H}\rmod \weyl{M}} \quad \text{ where }  \quad 
triv =  \sum_{w \in \weyl{M}} T_{w}.
\end{equation}

Another way to derive  \eqref{iwahori ::basis::para} is by \cite[Lemma  2.1.4]{ASENS_1995_4_28_5_527_0}.

Let $\lambda_0 = \jac{M}{T}{\lambda} =  \lambda -\rho^{M}_{T}$ and
$\Hecke_{\para{P}}(\lambda) = \pi^{\iwhaori}$.
By the equivalence of categories of \cite{MR2250034}, for any $w \in \weyl{H}$, the operator $\NN_w(\lambda_0)\res{\pi}$ has  a non-trivial kernel if and only if
$\NN_{w}(\lambda_0)\res{\pi^{\iwhaori}}$ has  a non-trivial kernel.
Here, the benefit of the finite dimension of $\Hecke_{0}$ (resp. $\Hecke_{\para{P}}(\lambda)$) enters into consideration,
namely, 
$$\ker \NN_{w}(\lambda_0)\res{\pi} \neq 0  \iff \operatorname{Rank} \bk{\Lambda} < \dim \Hecke_{\para{P}}(\lambda), $$   
where $\Lambda$  is the   matrix representation of $\NN_{w}(\lambda_0)\res{\pi^{\iwhaori}}$ with respect to the basis of 
$\Hecke_{\para{P}}(\lambda)$ given in \eqref{iwahori ::basis::para} and the basis of $\Hecke(w \cdot \lambda_0)$ given by 
$\set{T_w : w \in \weyl{H}}$.



\chapter{Local part - Results} \label{chapter::local::intro}
\section{Outline}

We fix the following notations for the entire chapter. Let $\nu$ be a finite place of $F$. We let $G$ denote the group of $F_{\nu}$--points of a split, simply-connected group of type $F_4$ defined over $F$. We fix a maximal split torus $T$ of $G$. 
Given  a maximal parabolic subgroup $\para{P} = MN$ associated with $\Delta_{G} \setminus \set{\alpha_i}$, we denote by   $\pi_{\para{P}_i,z} =  \Ind_{M_i}^{G}(z \fun{\alpha_{i}})$. We omit the subscript $\para{P}_i$ when there is no place for confusion. 
Recall that
\begin{Thm}\cite[Theorem 6.1]{F4}\label{janzen::F4}
Let $(\para{P},z_0) \in \set{(\para{P}_1,1) , (\para{P}_3,\frac{1}{2}),(\para{P}_4,\frac{5}{2}) }$. The representation $\pi_{\para{P},z_0}$ admits a unique irreducible subrepresentation and a maximal semi-simple quotient of length two.  
\end{Thm}

For $w \in \weyl{G}$ we set $\NN_{w}^{L}(\lambda_0)$ to be the restriction of $\NN_{w}(\lambda)$ to the line $L \subset \mathfrak{a}_{T,\C}^{\ast}$ at the point $\lambda_0 \in L$.  
If $L=\chi_{\para{P},z}$ and $w \in W(\para{P},G)$, we simply write $\NN_{w}(z) =  \NN_{w}^{\chi_{\para{P},z}}(\chi_{\para{P},z})$. 

We outline the results of this chapter and explain the ideas behind the proofs.
\begin{Prop}\label{local::holo}
Let $(\para{P},z_0) \in \set{(\para{P}_1,1) , (\para{P}_3,\frac{1}{2}), (\para{P}_4,\frac{5}{2})}$. Then for every $w \in W(\para{P},G)$, the normalized intertwining operator $\NN_{w}(z)\res{\pi_z}$ is holomorphic at $z=z_0$. 
\end{Prop}
\begin{proof}
The proof appears in \Cref{local::p4::holo}, \Cref{local::p1::holo}, \Cref{local::p3:holo}. However, the proofs share a  common idea, which we introduce here.

Note that for a line 
$L\subset \mathfrak{a}_{T,\C}^{\ast}$ and $w =  u_{l}u_{l-1} \cdots u_1$, one has 
 $$ \NN_{w}^{L}(\lambda) = \NN_{u_l}^{u_{l-1}^{'}L}(u_{l-1}^{'}\lambda) \circ \NN_{u_{l-1}}^{u_{l-2}^{'}L}(u_{l-2}^{'}\lambda) \circ \dots \circ \NN_{u_2}^{u_{1}L}(u_{1}\lambda) \circ \NN_{u_1}^{L}(\lambda), $$ 
 where $u_{i-1}^{'} =  u_{i-1}\cdot u_{i-2} \cdots u_{i}$. Thus, if for every $i$ the operator $\NN_{u_i}^{u_{i-1}^{'}L}(u_{i-1}^{'}\lambda)$ is holomorphic at $\lambda =\lambda_0$, then $\NN_{w}^{L}(\lambda)$ is also holomorphic at $\lambda=\lambda_0$ as a composition of holomorphic operators.  
 
 In each case, we show that for each $w \in W(\para{P},G)$, the operator $\NN_{w}(z)\res{\pi_z}$  can be written as a composition
 of holomorphic operators of one of the following types and is therefore holomorphic: 
\begin{itemize}\label{local::op::type}
\item
We say that the pair $(w,\lambda_0)$ defines an operator of $(H)$ type if for every $\alpha \in R(w) $ one has $\inner{\lambda_0,\check{\alpha}} \neq -1$. In other words, the operator $\NN_{w}(\lambda)$ is holomorphic at $\lambda=\lambda_0$. In particular, for each line $L \subset \mathfrak{a}_{T,\C}^{\ast}$ such that $\lambda_0 \in L$, the operator $\NN_{w}^{L}(\lambda)$ is holomorphic at $\lambda=\lambda_0$. 
\item
We say that the triple $(u,\chi,L)$ defines an operator of $(Z)$ type if the triple $(u,\chi,L)$ satisfies the conditions of \Cref{Zampera ::lemma_holo}. In particular, by \Cref{Zampera ::lemma_holo}, the operator $\NN_{w}^{L}(\lambda)$ is holomorphic at $\lambda=\chi$.
\item
Let $L\subset  \mathfrak{a}_{T,\C}^{\ast}$ be a line and let $\sigma_{\lambda}$ be a family of representations such that $\sigma_\lambda \hookrightarrow \Ind_{T}^{G}(\lambda)$ for every $\lambda \in L$. We say that the quadruple $(w,L,\lambda_0,\sigma_{\lambda_0})$ defines an operator of $(S)$ type, if $\NN_{w}^{L}(\lambda)\res{\sigma_{\lambda}}$ is holomorphic at $\lambda=\lambda_0$.   
\end{itemize} 
\end{proof}
By \Cref{local::holo}, each normalized intertwining operator $\NN_{w}(z)\res{\pi_z}$ is holomorphic at $z=z_0$. Hence, its image is well defined. Let $\Sigma_{w} = \Image \NN_{w}(z_0)\res{\pi_{z_0}}$. A central problem appearing in the local part is how to describe $\Sigma_{w}$. We start by finding the structure of $\pi_{z_0}$, which is a main ingredient in describing  $\Sigma_{w}$. The structure of $\pi_{z_0}$ is summarized in \Cref{local::collect::structre}.  Henceforth we denote
 $\Pi_{i}=  \Ind_{M_i}^{G}\bk{z_0 \cdot \fun{\alpha_i}}$.  
\newpage
\begin{Thm}\label{local::collect::structre}
\begin{enumerate}
\item[]
\item
If $i=1$ and $z_0=1$, one has  $\bk{\Pi_{1}}_{s.s} =\tau^{1} + \pi_1^{1} +\pi_2^{1}$. Here, $\tau^{1}$ is the unique irreducible subrepresentation, and $\pi_1^{1} \oplus \pi_2^{1}$ is the maximal semi--simple quotient of $\Pi_1$.     
\item
If $i=3$ and $z_0=\frac{1}{2}$, one has  $\bk{\Pi_{3}}_{s.s} =\tau^{3} +\sigma_{1}^{3} + \sigma_{2}^{3} + \pi_1^{3} +\pi_2^{3}$. Here, $\tau^{3}$ is the unique irreducible subrepresentation, and $\pi_1^{3} \oplus \pi_2^{3}$ is the maximal semi--simple quotient of $\Pi_3$ and $\sigma_1^{3} \oplus \sigma_2^{3} \hookrightarrow \Pi_3\rmod \tau^{3}$.     

\item
If $i=4$ and $z_0=\frac{5}{2}$, one has  $\bk{\Pi_{4}}_{s.s} =\tau^{4} + \pi_1^{4} +\pi_2^{4}$. Here, $\tau^{4}$ is the unique irreducible subrepresentation, and $\pi_1^{4} \oplus \pi_2^{4}$ is the maximal semi--simple quotient of $\Pi_4$.
\end{enumerate}
Moreover, for each irreducible constituent $\sigma$ of $\Pi_1,\Pi_3,\Pi_4$, we describe $\bk{\jac{G}{T}{\sigma}}_{s.s}$. In each case, the multiplicity of $\sigma$ is at most one. 
 \end{Thm}

For global reasons, we have a great interest in the following case. 
Recall that by \Cref{janzen::F4},  $\Pi_{i}$ admits a maximal semi-simple quotient of length two $\pi_1^{i} \oplus \pi_2^{i}$. Assume that there are $u_1 ,u_2\in W(\para{P},G)$ with the following properties:
\begin{itemize}
\item
$\Sigma_{u_1} = \Sigma_{u_2} \simeq \pi_1^{i} \oplus \pi_{2}^{i}$.
\item
Write $u_2 =  s u_1$. Note that   $s \in \Stab_{\weyl{G}}(u_1 \chi_{\para{P},z_0})$. 
\end{itemize}
Set $$E_{s}(z) =  \NN_{s}^{u_1 \chi_{\para{P},z}}(u_1 \chi_{\para{P},z})\res{ \Sigma_{u_1,z}}.$$
where $\Sigma_{u_1,z}= \Image \NN_{u_1}(z)\res{\pi_z}$.

In these cases, it is crucial to understand the properties of $E_{s}(z)$ at $z=z_0$. These properties will play a major role in describing the residual representation realized using the leading term of the Eisenstein series.

Their properties are summarized in \Cref{local::scalar}.  
\begin{Prop}\label{local::scalar}
With the above notations, one has
$E_{s}(z)$, which is holomorphic at $z=z_0$. In addition, there are $a_1,a_2 \in \C$ such that $E_{s}(z_0) \res{\pi_j^{i}}= a_j \id$.

\end{Prop}
\begin{proof}
By \Cref{local::holo}, the operator
$ \NN_{u_i}(z)\res{\pi_z}$ is holomorphic at $z=z_0$ for $i\in \set{1,2}$.
Note that
$$ \NN_{u_2}(z)\res{\pi_z} = E_{s}(z) \circ \NN_{u_1}(z)\res{\pi_z}. $$
Since $\NN_{u_2}(z)\res{\pi_z}$ is holomorphic at $z=z_0$, it follows that $E_{s}(z)$ restricted to the image of $\NN_{u_i}(z)\res{\pi_z}$ is holomorphic at $z=z_0$. Since $\Sigma_{u_1} =\Sigma_{u_2} \simeq\pi_1^{i} \oplus \pi_2^{i}$, it follows that $E_{s}(z_0) \in \operatorname{End}_{G}\bk{\Sigma_{u_1}}$. By Schur's lemma, there are $a_1,a_2 \in \C$ such that $$E_{s}(z_0) \res{\pi_j^{i}} =  a_j \id$$.
\end{proof}
The scalars $a_i$ are determined in \Cref{local::p4::stab}, \Cref{local::p1::stab},    \Cref{local::p3::stab}.

\section{$\para{P}=\para{P}_4, z_0 = \frac{5}{2}$}

\subsection{  Holomorphicity of normalized intertwining operators}
\begin{Prop}\label{local::p4::holo}
For each $w \in W(\para{P},G)$, the normalized intertwining operator $\NN_{w}(z)\res{\pi_z}$ is holomorphic at $z=z_0$.
\end{Prop}
\begin{proof}
It is sufficient to show that $\NN_{w}(z)$ is holomorphic at $z=z_0$. With the definitions of \Cref{local::holo} , we show that each $w \in W(\para{P},G)$ can be written as a composition  of holomorphic operators of one of two types:  $(H)$ and $(Z)$. For the reader's convenience, \Cref{Table:: P4 ::images} records the decomposition of each $w \in W(\para{P},G)$. 
\end{proof}

\begin{Remark} \label{local::p4::holo::arch}
Let $u \in W(\para{P},G)$ such that $\dim \bk{\Image \NN_{u}(z_0)}^{\iwhaori} =24 $. Then  $\NN_{u,\nu}(z)$ is holomorphic at $z=z_0$ for Archimedean places as well. This can be proved along the lines of \Cref{local::p4::holo}

\end{Remark} 

\subsection{The structure of $\Pi_4=\Ind_{M_4}^{G}(\frac{5}{2})$}
We specify the following exponents of $\Pi_4$.
\begin{align*}
\lambda_{0} =&  [-1,-1,-1,7], \text { the initial exponent of $\Pi_4$.  } \\
\lambda_{a.d}=& [-1,0,-1,-1],  \text { the anti-dominant exponent in the Weyl orbit of $\lambda_0$.  }\\
\lambda_1 =& \s{\alpha_3} \lambda_{a.d} = [-1,-1,1,-2]. \\
\lambda_2 =& \s{\alpha_4}\lambda_1 = [-1,-1,-1,2].
\end{align*}

Recall that by \cite[Theorem 6.1]{F4}, $\Pi_4$ admits a unique irreducible subrepresentation $\tau^{4}$ and a maximal semi--simple quotient of length two, $\pi_1^{4} \oplus \pi_2^{4}$. In addition, with the notations of \cite[Theorem 6.1]{F4}, $\pi_1^{4}$ is an irreducible constituent of $\Pi_4$ such that  $\lambda_{a.d} \leq \jac{G}{T}{\pi_1^4}$. 
\begin{Prop}
The Jacquet module $\coseta{\jac{G}{T}{\sigma}}$ of $\sigma \in \set{\pi_1^{4},\pi_2^{4},\tau^{4}}$ is described in \Cref{Table::P4}. Precisely, 
$$\coseta{\jac{G}{T}{\sigma}} = \sum n_{\lambda} \lambda,$$
where $\lambda$ and $n_{\lambda}$ are  given in \Cref{Table::P4}.
\end{Prop}
\begin{proof}
In order to simplify notations,  all  calculations are preformed in the relevant Grothendieck ring. This allows us to suppress the notations $\coseta{ \cdot}$.       

Note that
$$ \coset{\jac{G}{T}{\tau^{4}}} + \coset{\jac{G}{T}{\pi_1^{4}}}  + \coset{\jac{G}{T}{\pi_2^{4}}} \leq \coset{\jac{G}{T}{\Pi_4}}$$
and  $\coset{\jac{G}{T}{\Pi_4}}$ can be easily computed by the Geometric Lemma. This calculation is given in the second column of \Cref{Table::P4}. 

Recall that  $\pi_1^{4}$ is an irreducible constituent having $\coset{\lambda_{a.d}} \leq \coset{\jac{G}{T}{\pi_1^{4}}}$. 
An application of the following sequence of branching rules yields:  

\begin{itemize}
\item
Branching rule of type \eqref{Eq:OR} on $\lambda_{a.d}$ yields $2 \times \coset{\lambda_{a.d}} \leq \coset{\jac{G}{T}{\pi_1^{4}}}$.
\item
Branching rule of type \eqref{Eq::C2b} on $\lambda_{a.d}$ with respect to $M_{\alpha_2,\alpha_3}$ yields
$$2 \times \coset{\lambda_{a.d}} +  \coset{\s{\alpha_3} \lambda_{a.d}} \leq \coset{\jac{G}{T}{\pi_1^{4}}}.$$
\item
Branching rule of type \eqref{Eq:A2} on $\lambda_{a.d}$ with respect to $M_{\alpha_1,\alpha_2}$ yields
$$2 \times \coset{\lambda_{a.d}} +  \coset{\s{\alpha_1} \lambda_{a.d}} \leq \coset{\jac{G}{T}{\pi_1^{4}}}.$$
\item
Branching rule of type \eqref{Eq:A1} on $\s{\alpha_3}\lambda_{a.d}$ with respect to $M_{\alpha_4}$ yields
$$ \underbrace{\coset{\s{\alpha_3} \lambda_{a.d}}}_{\lambda_1} + \underbrace{\coset{\s{\alpha_4} \s{\alpha_3} \lambda_{a.d}}}_{\lambda_2} \leq \coset{\jac{G}{T}{\pi_1^{4}}}.$$
\end{itemize}
To summarize, one has 
\begin{equation}\label{F4::P4::spheruical}
2 \times \coset{\lambda_{a.d}} + \coset{\s{\alpha_1}\lambda_{a.d}} +  \coset{\lambda_1}  + \coset{\lambda_2} \leq \coset{\jac{G}{T}{\pi_1^{4}}}.
\end{equation}
 
By our notation, $\tau^{4}$ is the unique irreducible subrepresentation of $\Pi_4$. By induction in stages, one has $\Pi_4 \hookrightarrow \Ind_{T}^{G}(\lambda_0)$. In particular, $\tau^{4}$ is an irreducible subrepresentation of $\Ind_{T}^{G}(\lambda_0)$.  
Hence, by Frobenius reciprocity, $\coset{\lambda_{0}} \leq \coset{\jac{G}{T}{\tau^{4}}}$.  

Since $\mult{\lambda_0}{\jac{G}{T}{\Pi_4}}=1$, one has 
$$ 1 \leq \mult{\lambda_0}{\jac{G}{T}{\tau^{4}}} \leq \mult{\lambda_0}{\jac{G}{T}{\Pi_4}}=1.$$
Thus $\mult{\lambda_0}{\jac{G}{T}{\tau^{4}}}=1$. An application of a sequence of branching rules of type \eqref{Eq:A1} yields 
\begin{eqnarray}\label{P4::subrepresentaion}
\coset{\lambda_0} +  \coset{\s{\alpha_4} \lambda_0} + \coset{\s{\alpha_3}\s{\alpha_4} \lambda_0} + \coset{\s{\alpha_2}\s{\alpha_3}\s{\alpha_4} \lambda_0} & \nonumber\\
+\coset{\s{\alpha_1}\s{\alpha_2}\s{\alpha_3}\s{\alpha_4} \lambda_0} +
\coset{\s{\alpha_3}\s{\alpha_2}\s{\alpha_3}\s{\alpha_4} \lambda_0} 
&\nonumber \\
+
\coset{\s{\alpha_4}\s{\alpha_3}\s{\alpha_2}\s{\alpha_3}\s{\alpha_4} \lambda_0}&  \nonumber\\
+\coset{\s{\alpha_2}\s{\alpha_3}\s{\alpha_2}\s{\alpha_3}\s{\alpha_4} \lambda_0}
+
\coset{\s{\alpha_4}\s{\alpha_1}\s{\alpha_3}\s{\alpha_2}\s{\alpha_3}\s{\alpha_4} \lambda_0}& \nonumber\\
+\coset{\s{\alpha_2}\s{\alpha_1}\s{\alpha_3}\s{\alpha_2}\s{\alpha_3}\s{\alpha_4} \lambda_0} 
 \coset{\s{\alpha_2}\s{\alpha_4}\s{\alpha_1}\s{\alpha_3}\s{\alpha_2}\s{\alpha_3}\s{\alpha_4} \lambda_0}  &  \nonumber \\
+ \coset{\s{\alpha_3}\s{\alpha_2}\s{\alpha_4}\s{\alpha_1}\s{\alpha_3}\s{\alpha_2}\s{\alpha_3}\s{\alpha_4} \lambda_0}& \nonumber\\
+
\coset{\s{\alpha_2}\s{\alpha_3}\s{\alpha_2}\s{\alpha_4}\s{\alpha_1}\s{\alpha_3}\s{\alpha_2}\s{\alpha_3}\s{\alpha_4} \lambda_0}
 & \nonumber\\ +\coset{\s{\alpha_3}\s{\alpha_2}\s{\alpha_1}\s{\alpha_3}\s{\alpha_2}\s{\alpha_3}\s{\alpha_4} \lambda_0}
& \nonumber\\
+ \coset{\s{\alpha_4}\s{\alpha_3}\s{\alpha_2}\s{\alpha_1}\s{\alpha_3}\s{\alpha_2}\s{\alpha_3}\s{\alpha_4} \lambda_0}
 & \nonumber\\ +\coset{\s{\alpha_3}\s{\alpha_4}\s{\alpha_3}\s{\alpha_2}\s{\alpha_1}\s{\alpha_3}\s{\alpha_2}\s{\alpha_3}\s{\alpha_4} \lambda_0}
& \nonumber\\
+ \coset{\s{\alpha_2}\s{\alpha_3}\s{\alpha_4}\s{\alpha_3}\s{\alpha_2}\s{\alpha_1}\s{\alpha_3}\s{\alpha_2}\s{\alpha_3}\s{\alpha_4} \lambda_0} 
& \leq \coset{\jac{G}{T}{\tau^{4}}}.
\end{eqnarray} 

On the other hand, we know that the length of $\Pi_4$ is at least three. However, 
if $\lambda \not \in \set{\lambda_1,\lambda_2}$ and $\lambda \leq \coset{\jac{G}{T}{\Pi_4}}$, it follows that 
$$\mult{\lambda}{\jac{G}{T}{\Pi_4}} = \br{\lambda}{\jac{G}{T}{\tau^{4}}} + \br{\lambda}{\jac{G}{T}{\pi_1^{4}}},$$  
where $\br{\lambda}{\jac{G}{T}{\sigma}}$ is a lower bound of the    multiplicity of $\lambda$ in $\coset{\jac{G}{T}{\sigma}}$.  
In addition, for  $\lambda  \in \set{\lambda_1,\lambda_2}$, one has  $\mult{\lambda}{\jac{G}{T}{\Pi_4}} =2$. 

This, combined with \eqref{F4::P4::spheruical} and \eqref{P4::subrepresentaion}, implies that for the third irreducible constituent $\pi_2^4$, one has  
$$\coset{\jac{G}{T}{\pi_2^{4}}} \leq \coset{\lambda_1} + \coset{\lambda_2}.$$  

An application of \eqref{Eq:A1} branching rule with respect to $M_{\alpha_4}$ yields 
$$\coset{\lambda_1} \leq \coset{\jac{G}{T}{\sigma}} \iff \coset{\lambda_2} \leq \coset{\jac{G}{T}{\sigma}}.$$

Hence, 
$$\coset{\jac{G}{T}{\pi_2^{4}}} =\coset{\lambda_1} + \coset{\lambda_2}.$$

Let us summarize. For every $\lambda \leq \jac{G}{T}{\Pi_4}$, one has 
$$ \br{\lambda}{\jac{G}{T}{\tau^{4}}} + \br{\lambda}{\jac{G}{T}{\pi_1^{4}}} +  \br{\lambda}{\jac{G}{T}{\pi_2^{4}}} = \mult{\lambda}{\jac{G}{T}{\Pi_4}}.$$
Thus, we conclude that 
$\mult{\lambda}{\jac{G}{T}{\tau^{4}}} = \br{\lambda}{\jac{G}{T}{\tau^{4}}}$
and
\begin{align*}
\mult{\lambda}{\jac{G}{T}{\pi_1^{4}}} &= \br{\lambda}{\jac{G}{T}{\pi_1^{4}}} , &
\mult{\lambda}{\jac{G}{T}{\pi_2^{4}}} &= \br{\lambda}{\jac{G}{T}{\pi_2^{4}}}.
\end{align*}  
\end{proof}

\begin{Cor} $ $\label{P4::dim iwahori}
\begin{enumerate} 
\item
$\Pi_4$ is of length three.  In particular,
$\Pi_4 \rmod \tau^{4} \simeq \pi_1^{4} \oplus \pi_2^{4}$.
\item
One has 
\begin{align*}
\dim \jac{G}{T}{\pi_1^{4}} = \dim \bk{\pi_1^{4}}^{\iwhaori} &= 5, & \dim \jac{G}{T}{\pi_2^{4}} =\dim \bk{\pi_2^{4}}^{\iwhaori} &= 2, \\ \dim \jac{G}{T}{\tau^{4}} =\dim \bk{\tau^{4}}^{\iwhaori} &= 17. 
\end{align*}
\end{enumerate}
\end{Cor}
With the same notations as above one has   
\begin{Prop}\label{local::P4::multiplicityone}
For every $\lambda \in \weyl{G} \cdot \lambda_{a.d}$, one has 
$\mult{\pi_{i}^{4}}{\Ind_{T}^{G}(\lambda)} =1$ for $i=1,2$.
\end{Prop}
\begin{proof}
Given $\lambda \in \weyl{G} \cdot \lambda_{a.d}$,  an application of the Geometric Lemma asserts that 
$\bk{\Ind_{T}^{G}(\lambda)}_{s.s} =  \bk{\Ind_{T}^{G}(\lambda_{a.d})}_{s.s}$. Thus, it  is enough to show that $\mult{\pi_i^{4}}{\Ind_{T}^{G}(\lambda_{a.d})}=1$ for $i \in \set{1,2}$. 

Let $\lambda \in \weyl{G} \cdot \lambda_{a.d}$.  Geometric lemma yields that $$\mult{\lambda}{\Ind_{T}^{G}(\lambda_{a.d})} = |\Stab_{\weyl{G}}(\lambda)|.$$
Since $|\Stab_{\weyl{G}}(\lambda)|$ does not depend on the point of the orbit, it follows that 
$$ \mult{\lambda}{\Ind_{T}^{G}(\lambda_{a.d})} = |\Stab_{\weyl{G}}(\lambda_{a.d})|=2.$$

By \Cref{Table:: P4 ::images}, one has 
\begin{align*}
\mult{\lambda_{a.d}}{\jac{G}{T}{\pi_1^{4}}} =& 2, &
\mult{\lambda_{1}}{\jac{G}{T}{\pi_1^{4}}} =& 1, &
\mult{\lambda_{1}}{\jac{G}{T}{\pi_2^{4}}} =& 1.
\end{align*}
Hence, 
$$\mult{\lambda_{a.d}}{\jac{G}{T}{\pi_1^{4}}}= \mult{\lambda_{a.d}}{\jac{G}{T}{\Ind_{T}^{G}(\lambda_{a.d})}}=2.$$
In conclusion, $\mult{\pi_1^{4}}{\Ind_{T}^{G}(\lambda_{a.d})}=1$. Note that,   
$$\underbrace{\mult{\lambda_{1}}{\jac{G}{T}{\pi_1^{4}}}}_{=1} + 
\mult{\lambda_{1}}{\jac{G}{T}{\pi_2^{4}}} = \mult{\lambda_1}{\Ind_{T}^{G}(\lambda_{a.d})}=2.$$
Thus, $\mult{\pi_2^{4}}{\Ind_{T}^{G}(\lambda_{a.d})}=1$.
\end{proof}
\begin{Prop}\label{P4::pi1 ::spherical} 
The representation $\pi_1^{4}$ is spherical.   
\end{Prop}
\begin{proof}
\color{black}A direct consequence of \Cref{Lemma::sph}.
\end{proof}
\color{black}
\begin{Prop}\label{P4::emmdindings}
Let $\lambda \in \set{\lambda_1,\lambda_2}$. Then, 
$\pi_1^{4} \oplus \pi_2^{4} \hookrightarrow \Ind_{T}^{G}(\lambda)$.
\end{Prop}
\begin{proof}
By \Cref{Table::P4}, one has 
\begin{align*}
\mult{\lambda}{\jac{G}{T}{\pi_1^{4}}} =& 1, &
\mult{\lambda}{\jac{G}{T}{\pi_2^{4}}} =& 1.
\end{align*}

Hence, the central character argument yields  that both $\pi_1^{4}$ and $\pi_2^{4}$  have an embedding into $\Ind_{T}^{G}(\lambda_1)$ and $\Ind_{T}^{G}(\lambda_2)$. In particular, $\pi_1^{4} \oplus \pi_2^{4} \hookrightarrow \Ind_{T}^{G}(\lambda)$.  
\end{proof}
\subsection{Images of $\Sigma_{w}$}
\begin{Lem}\label{P4::Lemma::images}
Let $w \in W(\para{P},G)$.
\begin{enumerate}  
\item
 If $\dim \Sigma_{w}^{\iwhaori} =5$, then $\Sigma_{w} \simeq \pi_1^{4}$.  
\item
If $\dim \Sigma_{w}^{\iwhaori} =24$, then $\Sigma_{w} \simeq \Pi_4$.  
\item 
 If $\dim \Sigma_{w}^{\iwhaori} =7$, then $\Sigma_{w} \simeq \pi_1^{4} \oplus \pi_2^{4}$.
\end{enumerate}
\end{Lem}
\begin{proof}
Recall that $\Sigma_w$ is a component of $\Pi_4$. Thus, in the Grothendieck ring, one has  
$$ \Sigma_{w} =  m_{\tau} \times \tau^{4} + m_{1} \times  \pi_1^{4} + m_{2} \times \pi_2^{4}$$
for suitable $m_{\tau} ,m_1 ,m_2 \in \set{0,1}$. 
In particular,
$$ \dim \Sigma_{w}^{\iwhaori} =  m_{\tau} \times \dim  \bk{ \tau^4} ^{\iwhaori} + m_{1}\times \dim \bk{ \pi_1^{4} }^{\iwhaori} + m_{2} \times \dim \bk{\pi_2^{4}}^{\iwhaori}.$$
Thus,  $\dim \Sigma_{w}^{\iwhaori}=5$ (resp. $\dim \Sigma_{w}^{\iwhaori}=24$)  implies that $\Sigma_{w}\simeq \pi_1^{4}$ (resp. $\Sigma_{w}\simeq \Pi_4$).   

Additionally, if   $\dim \Sigma_{w}^{\iwhaori}=7$, then $(\Sigma_{w})_{s.s}= \pi_1^{4} +\pi_2^{4}$. 
By \Cref{Table:: P4 ::images}, $\dim \Sigma_{w}^{\iwhaori} =7$ implies that 
$ \Sigma_{w} \hookrightarrow \Ind_{T}^{G}(\lambda_1)$ or $ \Sigma_{w} \hookrightarrow \Ind_{T}^{G}(\lambda_2)$.

By \Cref{P4::emmdindings} , it follows that there is an embedding of  $\pi_1^{4}\oplus \pi_2^{4}$ into   $\Ind_{T}^{G}(\lambda_1)$ and $\Ind_{T}^{G}(\lambda_2)$. By  \Cref{local::P4::multiplicityone}, the claim follows.
 \end{proof}
\subsection{Action of the stabilizer} \label{local::p4::stab}
In this subsection, we set
\begin{align*}
u_1&=\bk{w_{1}w_{2}w_{3}w_{4}w_{3}w_{2}w_{3}w_{1}w_{2}w_{3}w_{4}}, & u_{1,1}&= \bk{w_{3}w_{2}w_{3}}\bk{w_{1}w_{2}w_{3}w_{4}w_{3}w_{2}w_{3}w_{1}w_{2}w_{3}w_{4}},\\
u_2& = w_{1}w_{2}w_{3}w_{4}w_{2}w_{3}w_{1}w_{2}w_{3}w_{4}, & u_{2,1}&=  \bk{w_{4}}\cdot \bk{w_{3}w_{2}w_{3}}\cdot \bk{w_{1}w_{2}w_{3}w_{4}w_{3}w_{2}w_{3}w_{1}w_{2}w_{3}w_{4}}.\\
\end{align*}
Note that $u_{i}\chi_{\para{P},z_0} =   u_{i,1} \chi_{\para{P},z_0} $ for $i=1,2$. Moreover, using braid relations, one has 
 $$u_{2,1}= 
 \bk{w_{4}}\cdot \bk{w_{3}w_{2}w_{3}}\cdot \bk{w_4}  \cdot  \bk{w_{1}w_{2}w_{3}w_{4}w_{2}w_{3}w_{1}w_{2}w_{3}w_{4}}.
 $$
 \newpage
 \begin{Prop}\label{P4::same::image}
 One has 
 $\Sigma_{u_1} = \Sigma_{u_{1,1}}$ and $\Sigma_{u_2} = \Sigma_{u_{2,1}}$.
 \end{Prop}
 \begin{proof}
 By \Cref{P4::Lemma::images},  for each $u \in \set{u_1,u_{1,1},u_2,u_{2,1}}$ one has $\Sigma_{u} \simeq \pi_1^{4} \oplus \pi_2^{4}$. By \Cref{local::P4::multiplicityone}, it follows that 
 $\Sigma_{u_1} = \Sigma_{u_{1,1}}$ and  
 $\Sigma_{u_2} = \Sigma_{u_{2,1}}$.  
\end{proof}
Observe that, $\lambda_1 =u_1 \chi_{\para{P},z_0}  $ and $\lambda_2 =  u_2\chi_{\para{P},z_0}$. Let $s = w_{3}w_{2}w_{3}$. Put $$E_{s}(z) =  \NN_{s}^{u_1 \chi_{\para{P},z}}(u_1 \chi_{\para{P},z}).$$
\begin{Prop}\label{P4:: stablizer}
One has $E_{s}(z)$ that is holomorphic at $z=z_0$. In addition,
$E_{s}(z_0)$ is an isomorphism such that 
$$E_{s}(z_0)\res{\pi_1^{4}} =\id, \quad E_{s}(z_0)\res{\pi_2^{4}} =-\id.$$    
\end{Prop} 
\begin{proof}
An application of \Cref{Zampera ::lemma_holo}, with the following arguments:
\begin{align*}
u &= s, & \chi &= u_1\chi_{\para{P},z_0}, & L&= u_1\chi_{\para{P},z_0}
\end{align*}
yields that $E_{s}(z)$ is holomorphic at $z=z_0$. Moreover, with the notations of \Cref{local::section::Zampera}, one has $a_1 =1$.  \Cref{Zampera::action} yields that $E_{s}(z_0)$  is a diagonalizable  operator with two eigenvalues $\pm 1$. Note that $E_{s}(z_0) \: :\Ind_{T}^{G}(\lambda_1) \rightarrow  \Ind_{T}^{G}(\lambda_1)$. Thus, by  putting together \Cref{local::P4::multiplicityone},  \Cref{P4::emmdindings}  and Schur's Lemma, there are $b_1,b_2 \in \C$ such that  
$E_s(z_0)\res{\pi_i^4} =  b_i \id$. 
By \Cref{Table::P4}, one has 
\begin{align*}
\mult{\lambda_{a.d}}{\jac{G}{T}{\pi_1^{4}}} &= 2, &\mult{\lambda_{a.d}}{\jac{G}{T}{\pi_2^{4}}}&=0. \\
\mult{\lambda_{1}}{\jac{G}{T}{\pi_1^{4}}} &= 1, 
&\mult{\lambda_{1}}{\jac{G}{T}{\pi_2^{4}}} &= 1.
\end{align*} 
In particular, $\pi_2^{4}$ is not a subrepresentation of $\Ind_{T}^{G}(\lambda_{a.d})$, and hence, $\pi_2^{4} \hookrightarrow \ker \NN_{w_{3}}(u_1 \chi_{\para{P},z_0})$. Thus, \Cref{Zampera::action} implies that $E_{s}(z_0)\res{\pi_2^{4}} =-\id$. On the other hand, by \Cref{P4::pi1 ::spherical}, $\pi_1^{4}$ is the unique irreducible spherical constituent of $\Pi_4$. Hence $$E_{s}(z_0)\res{\pi_1^{4}}=\id.$$ 
\end{proof}

\begin{Prop}\label{P4::stab::2}
Let $u =w_4 w_{3}w_{2}w_{3} w_4$. Let $E^{1}_{u}(z) =  \NN_{s}^{u_2 \chi_{\para{P},z}}(u_2 \chi_{\para{P},z})$. Then
$E^{1}_{u}(z)$ is holomorphic at $z=z_0$. In addition, one has 
$$E_{u}^{1}(z_0)\res{\pi_1^{4}} =\id, \quad E_{u}^{1}(z_0)\res{\pi_2^{4}} =-\id.$$      
\end{Prop}
\begin{proof}
Let 
\begin{align*}
F_{1}(z) &=   \NN_{w_4}^{u_{1,1} \chi_{\para{P},z}}(u_{1,1} \chi_{\para{P},z}) \: : \:  \Ind_{T}^{G}(u_{1,1} \chi_{\para{P},z}) \rightarrow  \Ind_{T}^{G}(u_{1,2} \chi_{\para{P},z}),
\\
F_{2}(z) &=  
 \NN_{w_4}^{u_{2} \chi_{\para{P},z}}(u_{2} \chi_{\para{P},z}) \: : \:  \Ind_{T}^{G}(u_{2} \chi_{\para{P},z}) \rightarrow  \Ind_{T}^{G}(u_{1} \chi_{\para{P},z}).
\end{align*}

With the notations of \Cref{P4:: stablizer},  one has 
$$E^{1}_{u}(z) = F_{1}(z)   \circ E_{s}(z) \circ F_2(z).$$

Since $E_{s}(z)$ is holomorphic at $z=z_0$, it suffices to show that both 
$F_{1}(z)$ and $F_2(z)$ are holomorphic at $z=z_0$. Note that
\begin{align*}
\inner{u_2 \chi_{\para{P},z_0} ,\check{\alpha}_{4}} &=-2, &
\inner{w_3w_2w_3u_1 \chi_{\para{P},z_0} ,\check{\alpha}_{4}} &=2.
\end{align*} 
Hence, both $F_1(z)$ and $F_2(z)$ are holomorphic at $z=z_0$. 
Moreover, both $F_1(z_0)$ and $F_{2}(z_0)$ are isomorphisms.  
Note $ \bk{w_3w_2w_3 u_1} \cdot \chi_{\para{P},z_0} =\bk{w_4 u_2} \cdot \chi_{\para{P},z_0}$. Hence, 
$$F = F_{1}(z_0) \circ F_{2}(z_0) \in \operatorname{End}\bk{\Ind_{T}^{G}(u_2 \chi_{\para{P},z_0}) }.$$ By the properties of the normalized intertwining operator, one has $F =\id$. Thus, the action of $E^{1}_{u}(z_0)$ on each irreducible subrepresentaion $\pi_i^{4}$ is determined by the action of $E_{s}(z_0)$.   By \Cref{P4:: stablizer}, the claims follows.
\end{proof}

In \Cref{Table:: P4 ::images}  we summarize all the information for $\Ind_{M_4}^{G}(z_0)$ that is needed for the global part.

\newpage
\section{$\para{P}=\para{P}_1, z_0 = 1$}

\subsection{  Holomorphicity of normalized intertwining operators}
\begin{Prop}\label{local::p1::holo}
For each $w \in W(\para{P},G)$, the normalized intertwining operator $\NN_{w}(z)\res{\pi_z}$ is holomorphic at $z=z_0$.
\end{Prop}
\begin{proof}
 With the definitions of \Cref{local::holo} , we shall show that each $w \in W(\para{P},G)$ can be written as a composition  of holomorphic operators of one of three types:  $(H)$,  $(Z)$ and $(S)$.
 Let 
 \begin{align*}
 u_1&=w_{2}w_{3}w_{4}w_{2}w_{3}w_{1}w_{2}w_{3}w_{4}w_{1}w_{2}w_{3}w_{2}w_{1},  & u_2&=w_{1}w_{2}w_{3}w_{4}w_{2}w_{3}w_{1}w_{2}w_{3}w_{4}w_{1}w_{2}w_{3}w_{2}w_{1}.
 \end{align*}
 For $w \in W(\para{P},G) \setminus \set{u_1,u_2}$, this decomposition is recorded in \Cref{Table:: P1 ::images}. For $u_1,u_2$, we shall use braid relations  in order to preform a decomposition. Note that
 \begin{align*}
 u_1& = \underbrace{\bk{w_{2}w_{3}w_{4}w_{2}w_{3}w_{2}}}_{u}\underbrace{\bk{w_{1}w_{2}w_{3}w_{4}  w_{2}w_{3}w_{2}w_{1}}}_{w'},
 \\
 u_2&=\bk{w_1}\underbrace{\bk{w_{2}w_{3}w_{4}w_{2}w_{3}w_{2}}}_{u}\underbrace{\bk{w_{1}w_{2}w_{3}w_{4}  w_{2}w_{3}w_{2}w_{1}}}_{w'}.
 \end{align*}
 In \Cref{P1::type::S}  we show that $(u,w'\chi_{\para{P},z},\Image \NN_{w'}(z)\res{\pi_z},w'\chi_{\para{P},z_0})$ defines an operator of $S$ type. It is easy to verify that both $(w',\chi_{\para{P},z_0})$ and $(w_4, u_1 \chi_{\para{P},z_0})$ define an operator of $H$ type.   
 Thus,  
 $u_1$ is of type $SH$ and $u_2$ is of type $HSH$   
\end{proof}
\begin{Remark} \label{local::p1::holo::arch}
Let $u \in W(\para{P},G)$ such that $\dim \bk{\Image \NN_{u}(z_0)}^{\iwhaori} =24 $. Then  $\NN_{u,\nu}(z)$ is holomorphic at $z=z_0$ for Archimedean places as well. This can be proved along the lines of \Cref{local::p1::holo}
\end{Remark} 
\begin{Prop}\label{P1::type::S} 
Let $u=w_{2}w_{3}w_{4}w_{2}w_{3}w_{2} $ and $w' =w_{1}w_{2}w_{3}w_{4}  w_{2}w_{3}w_{2}w_{1} $. The quadruplet $(u,w'\chi_{\para{P},z},\Image \NN_{w'}(z)\res{\pi_z},w'\chi_{\para{P},z_0})$ defines an operator of $S$ type.
\end{Prop}
\begin{proof}
By induction in stages, one has 
\begin{align*}
\pi_z \hookrightarrow \sigma_{z} &=  \Ind_{M_{1}}^{G}\bk{ \Ind_{M_{\alpha_3,\alpha_4}}^{M_1}(-\rho^{M_{1}}_{M_{\alpha_3,\alpha_4}}) \otimes z\fun{\alpha_1}}   
\\
&\simeq \Ind_{M_{\alpha_3,\alpha_4}}^{G}(z\fun{1} + 3\fun{1}-2\fun{2} ),
\end{align*}
where $M_{\alpha_3,\alpha_4}$ stands for the standard Levi subgroup $M$ such that $\Delta_{M}= \set{\alpha_3,\alpha_4}$ and 
$$\rho^{M_{1}}_{M_{\alpha_3,\alpha_4}} = \rho^{M_1}_{T} -  \rho^{M_{\alpha_3,\alpha_4}}_{T} - \rho^{M_1}_{T}=  -3\fun{1} +2\fun{2}.$$
Note that $w'\set{\alpha_3,\alpha_4} = \set{\alpha_3,\alpha_4}$. Thus, as long as $\NN_{w'}(z)\res{\pi_z}$ is holomorphic, the image lies in 
$$\Ind_{M_{\alpha_3,\alpha_4}}^{G}\bk{w'\cdot (z\fun{1} + 3\fun{1}-2\fun{2} )} = \underbrace{\Ind_{M_{\alpha_3,\alpha_4}}^{G}(z (-2\fun{1}+\fun{2}) -\fun{2} )}_{\Sigma_z}.$$   

Thus, 
$$ \NN_{w'}(z)\res{\pi_z} \: : \: \Ind_{M_1}^{G}(z \fun{1}) \rightarrow \Ind_{M_{\alpha_3,\alpha_4}}^{G}(z (-\fun{1}+\fun{2}) -\fun{2} ).$$

Hence, it is enough to show that $(u,w'\chi_{\para{P},z},\Sigma_{z},w'\chi_{\para{P},z_0})$ defines an operator of $S$ type. We transfer the issue to the Iwahori--Hecke algebra. Namely, let 
$n_{u}(z)$ be the operator associated with $\NN_{u}^{w'\chi_{\para{P},z}}(w'\chi_{\para{P},z})$ 
in the Iwahori--Hecke algebra. We show that $n_{u}(z)\res{\Sigma_z^{\iwhaori}}$ is holomorphic at $z=z_0$.

Recall that 
$$\Sigma_{z}^{\iwhaori} =  \Span\set{  Triv\cdot T_{u'} \: : \:  u' \in W(M_{\alpha_3,\alpha_4} ,G)}.$$ 

Here $Triv = 1  
+T_{ w_{4}w_{3} }  
 +T_{ w_{4} } 
 +T_{ w_{3}w_{4}w_{3} }  
 +T_{ w_{3}w_{4} }  
 +T_{ w_{3} }
 $ and $W(M_{\alpha_3,\alpha_4},G)$ in the set of shortest representatives of $\weyl{G} \rmod \weyl{M_{\alpha_3,\alpha_4}}$. 
 
Since $$\Image n_{u}(z)\res{\Sigma_{z} ^{\iwhaori}} = Span\set{T_{u'} \cdot   Triv  \cdot n_{u}(z) \: \: \:  u' \in W(M_{\alpha_3,\alpha_4},G)},   $$   
we reduce the problem  to  show the holomorphicity of one element $Triv \cdot  n_{u}(z)$ at $z=1$. 

Recall that for a simple reflection $w_i$, we let $n_{w_i}(z) =  \frac{q-1}{q^{z+1}-1} + \frac{q^{z}-1}{q^{z+1}-1}.$

Then,  
$$ n_{u}(z) =  n_{w_2}(z) \underbrace{\bk{n_{w_3}(2z-1) n_{w_2}(z-1) n_{w_4}(2z-2) n_{w_3}(2z-3)}}_{e(z)} n_{w_2}(z-2).$$

Direct computations show that 
\begin{align*}
\lim_{z \to z_0} e(z) &= \frac{ 1 }{ 2(q + 1) }\cdot \bk{ T_{ w_{2} }  + T_{w_{3}w_{2} }+ 2\cdot T_{ w_{3}w_{4} } +2\cdot T_{ w_{4} } +4 \cdot T_{ w_{3} } }  \nonumber\\
&-\frac{ 1 }{ 2q (q + 1) }\cdot \bk{ T_{ w_{2}w_{3} } + T_{ w_{3}w_{2}w_{3}  }+ 2 \cdot T_{ w_{4}w_{3} } +2 \cdot T_{ w_{3}w_{4}w_{3} } }
\nonumber\\
&-\frac{   (q - 1) }{  (q + 1) }\cdot T_{ 1 },\\
\lim_{z \to z_0} Triv \cdot n_{w_2}(z)e(z) &=\frac{Triv\cdot \bk{T_{w_2}+1} }{q+1}.
\end{align*}

Note that $\ord\limits_{z=z_0}n_{w_2}(z-2)=1$. Thus, in order to show that 
 $Triv \cdot n_{u}(z)$ is holomorphic at $z=z_0$, it suffices to show that 
 $$\frac{Triv(T_{w_2} +T_{e}) }{q+1} \Lambda(n_{w_2}(z-2)) =0.$$
 Here, $\Lambda(n_{w_2}(z-2))$ is the leading term of $n_{w_2}(z-2)$ around $z=z_0$.

Note that   $$\Lambda(n_{w_2}(z-2)) = \frac{1}{\log q} \bk{ (q-1) + (q^{-1}-1)T_{w_2}}.$$
However,
$$\bk{1 + T_{w_2}} \cdot  \bk{(q-1) + (q^{-1}-1)T_{w_2}}=  \bk{q-1}\bk {1 +T_{w_2}} \bk{q -  T_{w_2}}=0.$$   
Hence, the claim follows.
\end{proof}
\subsection{ Structure of $\Pi_1 =  \Ind_{M_{1}}^{G}(1)$}
We specify the following exponents of $\Pi_1$:
\begin{align*}
\lambda_{0} =&  [4,-1,-1,-1], \text { the initail exponent of $\Pi_1$  }. \\
\lambda_{a.d}=& [-1,0,-1,0],  \text { the anti-dominant exponent in the Weyl orbit of $\lambda_0$  }.\\
\lambda_1 =& \s{\alpha_2}\s{\alpha_3} \lambda_{a.d} = [-2,1,-1,-1]. \\
\lambda_2 =& \s{\alpha_1}\lambda_1 = [2,-1,-1,-1].
\end{align*}

Recall that by \cite[Theroem 6.1]{F4}, $\Pi_1$ admits a unique irreducible subrepresentation $\tau^{1}$ and a maximal semi--simple quotient of length two, $\pi_1^{1} \oplus \pi_2^{1}$. In addition, with their notations, $\pi_1^{1}$ is an irreducible constituent of $\Pi_1$ such that  $\lambda_{a.d} \leq \jac{G}{T}{\pi_1^{1}}$. 
\begin{Prop}
The Jacquet module $\jac{G}{T}{\sigma}_{s.s}$ of $\sigma \in \set{\pi_1^{1},\pi_2^{1},\tau^{1}}$ is described in \Cref{Table::P1}. Specifically, 
$$\jac{G}{T}{\sigma}_{s.s} = \sum n_{\lambda} \lambda,$$
where $\lambda$ and $n_{\lambda}$ are  given in \Cref{Table::P1}.
\end{Prop}

\begin{proof}
In order to simplify the notations,  all  calculations are preformed in the relevant Grothendieck ring. This allows us to suppress the notation $\coseta{ \cdot}$.       

Note that
$$ \coset{\jac{G}{T}{\tau^{1}}} + \coset{\jac{G}{T}{\pi_1^{1}}}  + \coset{\jac{G}{T}{\pi_2^{1}}} \leq \coset{\jac{G}{T}{\Pi_1}}$$
and  $\coset{\jac{G}{T}{\Pi_1}}$ can be easily computed by the Geometric Lemma. This calculation is given in the second column of \Cref{Table::P1}. 

Recall that  $\pi_1^{1}$ is an irreducible constituent having $\coset{\lambda_{a.d}} \leq \coset{\jac{G}{T}{\pi_1^{1}}}$. 
An application of the following sequence of branching rules yields  
\begin{itemize}
\item
Branching rule of type \eqref{Eq:OR} on $\lambda_{a.d}$ yields $4 \times \coset{\lambda_{a.d}} \leq \coset{\jac{G}{T}{\pi_1^{1}}}$.
\item
Branching rule of type \eqref{Eq::C3a} on $\lambda_{a.d}$ with respect to $M_{\alpha_2,\alpha_3,\alpha_4}$ yields
$$4 \times \coset{\lambda_{a.d}} +  2 \times \coset{\s{\alpha_3} \lambda_{a.d}} +\coset{\s{\alpha_2}\s{\alpha_3} \lambda_{a.d}}  \leq \coset{\jac{G}{T}{\pi_1^{1}}}.$$
\item
Branching rule of type \eqref{Eq:A2} on $\lambda_{a.d}$ with respect to $M_{\alpha_1,\alpha_2}$ yields
$$4 \times \coset{\lambda_{a.d}} +  2\times \coset{\s{\alpha_1} \lambda_{a.d}} \leq \coset{\jac{G}{T}{\pi_1^{1}}}.$$
\item
Branching rule of type \eqref{Eq:A2} on $\s{\alpha_1}\lambda_{a.d}$ with respect to $M_{\alpha_3,\alpha_4}$ yields
$$2 \times \coset{\s{\alpha_1}\lambda_{a.d}} +  \coset{\s{\alpha_3}\s{\alpha_1} \lambda_{a.d}} \leq \coset{\jac{G}{T}{\pi^{1}_1}}.$$
\end{itemize}
Continuing by an application of all possible \eqref{Eq:A1} rules yields 
\begin{eqnarray}\label{F4::P1::spheruical}
4 \times \coset{\lambda_{a.d}} + 2 \times \coset{\s{\alpha_3}\lambda_{a.d}} + 2 \times \coset{\s{\alpha_1}\lambda_{a.d}} \nonumber \\
+ \coset{\s{\alpha_3}\s{\alpha_1}\lambda_{a.d}}
+\coset{\s{\alpha_2}\s{\alpha_3}\s{\alpha_1}\lambda_{a.d}}
+  \coset{\lambda_1}  + \coset{\lambda_2} \nonumber\\
+\coset{\s{\alpha_3}\s{\alpha_2}\s{\alpha_3}\s{\alpha_1}\lambda_{a.d}}
+\coset{\s{\alpha_4}\s{\alpha_3}\s{\alpha_2}\s{\alpha_3}\s{\alpha_1}\lambda_{a.d}}
 &\leq \coset{\jac{G}{T}{\pi_1^{1}}}. 
\end{eqnarray}

By induction in stages, one has $\Pi_1 \hookrightarrow \Ind_{T}^{G}(\lambda_0)$. 
By our notations, $\tau^{1}$ is the unique  irreducible subrepresentation of $\Pi_1$. In particular, $\tau^{1} \hookrightarrow \Ind_{T}^{G}(\lambda_0)$.
 Hence, by Frobenius reciprocity, $\coset{\lambda_{0}} \leq \coset{\jac{G}{T}{\tau^{1}}}$.  

Since $\mult{\lambda_0}{\jac{G}{T}{\Pi_!}}=1$, one has 
$$ 1 \leq \mult{\lambda_0}{\jac{G}{T}{\tau^{1}}} \leq \mult{\lambda_0}{\jac{G}{T}{\Pi_1}}=1.$$
Thus, $\mult{\lambda_0}{\jac{G}{T}{\tau^{1}}}=1$. An application of a sequence of branching rules of type \eqref{Eq:A1} yields
\begin{align}\label{P1::pi2::jac}
\coset{\lambda_0} + \coset{\s{\alpha_1}\lambda_0} + \coset{\s{\alpha_2}\s{\alpha_1}\lambda_0} + \coset{\s{\alpha_3}\s{\alpha_2}\s{\alpha_1}\lambda_0} + \coset{\s{\alpha_2}\s{\alpha_3}\s{\alpha_2}\s{\alpha_1}\lambda_0} & \\
+  \coset{\s{\alpha_4}\s{\alpha_3}\s{\alpha_2}\s{\alpha_1}\lambda_0}
+ \coset{\s{\alpha_4}\s{\alpha_2}\s{\alpha_3}\s{\alpha_2}\s{\alpha_1}\lambda_0}
\coset{\s{\alpha_1}\s{\alpha_4}\s{\alpha_2}\s{\alpha_3}\s{\alpha_2}\s{\alpha_1}\lambda_0} & \leq \coset{\jac{G}{T}{\tau^{1}}}. \nonumber
\end{align}


On the other hand, we know that the length of $\Pi_1$ is at least three. 
Noting that $\lambda \not \in \set{\lambda_1,\lambda_2}$ and $\lambda \leq \jac{G}{T}{\Pi_1}$, it follows that 
$$\mult{\lambda}{\jac{G}{T}{\Pi_1}} = \br{\lambda}{\jac{G}{T}{\tau^{1}}} + \br{\lambda}{\jac{G}{T}{\pi_1^{1}}},$$
where $\br{\lambda}{\jac{G}{T}{\sigma}}$ is a lower bound for the    multiplicity of $\lambda$ in $\coset{\jac{G}{T}{\sigma}}$.

In addition, for  $\lambda  \in \set{\lambda_1,\lambda_2}$ one has  $\mult{\lambda}{\jac{G}{T}{\Pi_1}} =2$. This, combined with   \eqref{F4::P1::spheruical} and \ref{P1::pi2::jac}, implies that the third irreducible constituent $\pi_2^{1}$ has  
$$\coset{\jac{G}{T}{\pi^{1}_2}} \leq \coset{\lambda_1} + \coset{\lambda_2}.$$  

An application of \eqref{Eq:A1} branching rule with respect to $M_{\alpha_1}$ yields 
$$\coset{\lambda_1} \leq \jac{G}{T}{\sigma} \iff \coset{\lambda_2} \leq \jac{G}{T}{\sigma}.$$

Hence, 
$$\coset{\jac{G}{T}{\pi_2^{1}}} =\coset{\lambda_1} + \coset{\lambda_2}.$$
Thus, we conclude that 
$\mult{\lambda}{\jac{G}{T}{\tau^{1}}} = \br{\lambda}{\jac{G}{T}{\tau^{1}}}$
and
\begin{align*}
\mult{\lambda}{\jac{G}{T}{\pi_1^{1}}} &= \br{\lambda}{\jac{G}{T}{\pi_1^{1}}} , &
\mult{\lambda}{\jac{G}{T}{\pi_2^{1}}} &= \br{\lambda}{\jac{G}{T}{\pi_2^{1}}}.
\end{align*}  
\end{proof}

\begin{Cor}\label{P1::dim iwahori}
\begin{enumerate}
\item
$\Pi_1$ is of length three.  In particular
$\Pi_1 \rmod \tau^{1} \simeq \pi_1^{1} \oplus \pi_2^{1}$.
\item
One has 
\begin{align*}
\dim \jac{G}{T}{\pi_1^{1}} = \dim \bk{\pi_1}^{\iwhaori} &= 14, & \dim \jac{G}{T}{\pi^{1}_2} =\dim \bk{\pi_2^{1}}^{\iwhaori} &= 2, \\ \dim \jac{G}{T}{\tau^{1}} =\dim \bk{\tau^{1}}^{\iwhaori} &= 8. 
\end{align*}
\end{enumerate}
\end{Cor}
\newpage
\begin{Prop}\label{P1::pi1 ::spherical} 
The representation $\pi_1^{1}$ is spherical.   
\end{Prop}
\begin{proof}
\color{black}{
A direct consequence of \Cref{Lemma::sph}}.
\end{proof}
\color{black}

With the same notations as above, one has   
\begin{Prop}\label{local::P1::multiplicityone}
For every $\lambda \in \weyl{G} \cdot \lambda_{a.d}$, one has 
$\mult{\pi_{i}^{1}}{\Ind_{T}^{G}(\lambda)} =1$ for $i=1,2$.
\end{Prop}
\begin{proof}
The proof of  $\mult{\pi_{1}^{1}}{\Ind_{T}^{G}(\lambda_{a.d}}=1$ follows through the same lines as in \Cref{local::P4::multiplicityone}. 
In order to show that $\mult{\pi_2^{1}}{\Ind_{T}^{G}(\lambda_{a.d})}=1$, we proceed as follows:

 Let $\lambda_3=\w_{\alpha_3}\lambda_1 =[-2,0,1,-2]$. Note that $\lambda_3 \not \leq \jac{G}{T}{\pi_1^{1}}, \jac{G}{T}{\pi_2^{1}}$. 
Thus, there exists an irreducible constituent  $\sigma$ of $\Ind_{T}^{G}(\lambda_{a.d})$ such that $\lambda_3 \leq \jac{G}{T}{\sigma}$.

Applying \eqref{Eq:OR} and \eqref{Eq::C2b}  implies that 
$$2\times  \lambda_3 +  \lambda_1 \leq \jac{G}{T}{\sigma}  .$$
Note that:
\begin{itemize}
\item
 If $\mult{\lambda_3}{\jac{G}{T}{\sigma}}=2$, then there is another irreducible constituent $\sigma_1$ with 
$$2\times  \lambda_3 +  \lambda_1 \leq \jac{G}{T}{\sigma_1}  .$$
Hence, 
\begin{align*}
 \mult{\lambda_1}{\jac{G}{T}{\Ind_{T}^{G}(\lambda_{a.d})}}&\geq 
\mult{\lambda_1}{\jac{G}{T}{\pi_1^{1}}}
+
\mult{\lambda_1}{\jac{G}{T}{\pi_2^{1}}}
\\&+ 
\mult{\lambda_1}{\jac{G}{T}{\sigma}}
+
\mult{\lambda_1}{\jac{G}{T}{\sigma^{1}}}
\\
&\geq 
1
+
1+
1+1=\mult{\lambda_1}{\jac{G}{T}{\Ind_{T}^{G}(\lambda_{a.d})}}.
\end{align*}
Hence, $\pi^1_{2}$ appears with multiplicity one in $\Ind_T^{G}(\lambda_{a.d})$.
\item
If $\mult{\lambda_3}{\jac{G}{T}{\sigma}}=4$, then
$$4\times  \lambda_3 +  2\times \lambda_1 \leq \jac{G}{T}{\sigma}  .$$
As before, we show that $\pi^1_{2}$ appears with multiplicity one in $\Ind_T^{G}(\lambda_{a.d})$.
\end{itemize}
\end{proof}

\begin{Prop}\label{P1::emmdindings}
Let $\lambda \in \set{\lambda_1,\lambda_2}$. Then, 
$\pi_1^{1} \oplus \pi_2^{1} \hookrightarrow \Ind_{T}^{G}(\lambda)$.
\end{Prop}
\begin{proof}
As in \Cref{P4::emmdindings}.
\end{proof}
\subsection{Determine $\Sigma_{w}$}
\begin{Lem}\label{P1::Lemma::images}
Let $w \in W(\para{P},G)$.
\begin{enumerate}  
\item
 If $\dim \Sigma_{w}^{\iwhaori} =14$, then $\Sigma_{w} \simeq \pi_1^{1}$.  
\item
If $\dim \Sigma_{w}^{\iwhaori} =24$, then $\Sigma_{w} \simeq \Pi_1$.  
\item 
 If $\dim \Sigma_{w}^{\iwhaori} =16$, then $\Sigma_{w} \simeq \pi_1^{1} \oplus \pi_2^{1}$.
\end{enumerate}
\end{Lem}
\begin{proof}
The proof is the same as in \Cref{P4::Lemma::images}.
\end{proof}

\subsection{Action of the stabilizer} \label{local::p1::stab}
In this subsection, we set
\begin{align*}
u_1&=\bk{w_{1}w_{2}w_{3}w_{4}  w_{2}w_{3}w_{2}w_{1}}, & u_{1,1}&= 
\bk{w_{2}w_{3}w_{4}w_{2}w_{3}w_{2}}\bk{w_{1}w_{2}w_{3}w_{4}  w_{2}w_{3}w_{2}w_{1}},
\\
u_2& = \bk{w_{2}w_{3}w_{4}  w_{2}w_{3}w_{2}w_{1}}, & u_{2,1}&= \bk{w_1} 
\bk{w_{2}w_{3}w_{4}w_{2}w_{3}w_{2}}\bk{w_{1}}\bk{w_{2}w_{3}w_{4}  w_{2}w_{3}w_{2}w_{1}}.
\end{align*}
Note that $\lambda_1 =u_1 \chi_{\para{P},z_0} = u_{1,1} \chi_{\para{P},z_0}$ and
$\lambda_2 =u_2 \chi_{\para{P},z_0} = u_{2,1} \chi_{\para{P},z_0}$.    
\begin{Prop}\label{P1::same::image}
 One has 
 $\Sigma_{u_1} = \Sigma_{u_{1,1}}$ and $\Sigma_{u_2} = \Sigma_{u_{2,1}}$.
 \end{Prop}
\begin{proof}
Same as the proof of \Cref{P4::same::image}.
\end{proof}
Let $s=w_{2}w_{3}w_{4}w_{2}w_{3}w_{2} $. Set $E_{s}(z)=\NN_{s}^{u_1 \chi_{\para{P},z}}(u_1 \chi_{\para{P},z})\res{\NN_{u_1}(\pi_z)}$.
\begin{Prop}
Then $E_s(z)$ is holomorphic at $z=z_0$ and
\begin{align*} 
E_{s}\res{\pi_1^{1}} &=\id, & E_{s}\res{\pi_2^{1}} &=-\id.  
\end{align*}
\end{Prop} 
\begin{proof}
The holomorphicity in the content of \Cref{P1::type::S}. Note that 
$E_{s}(z_0) \in \operatorname{End}(\Sigma_{u_1})$. Hence by Schur's Lemma, it acts by a scalar on each summand. Since $\pi_1^{1}$ is spherical, it follows that $E_{s}\res{\pi_1^{1}} =\id$.

We proceed as follows: 
\begin{itemize}
\item 
By the proof of \cite[Theorem 6.1]{F4} and \Cref{P1::type::S}, it follows that 
$$\Image \NN_{u_1}(z_0)\res{\pi_{z_0}} \subset \Ind_{M_{\alpha_3,\alpha_4}}^{G}(-2\fun{2}) \simeq \Ind_{M_{1}}^{G}\bk{ \Ind_{M_{\alpha_3,\alpha_4}}^{M_1}(\tau)},$$
where $\tau =1_{GL_3}\rtimes |\cdot|^{-2}$.
\item
By  \cite[Lemma 5.1(5)]{F4}, one has  $\Ind_{M_{\alpha_3,\alpha_4}}^{M_1}(\tau) = \sigma_1 \oplus \sigma_2$.
\item
 The proof of  \cite[Theorem 6.1]{F4} yields $\pi_i^{1} \hookrightarrow \Ind_{M_1}^{G}\bk{\sigma_i}$. Hence it is enough to understand 
$E_{s}(z_0) \res{   \Ind_{M_1}^{G}(\sigma_i)}$.
\item 
With the notations of \Cref{app::c3::lemma1}, one has   
$\Sigma_1 = \operatorname{Res}^{M_1}_{M_1^{der}}(\Ind_{M_{\alpha_3,\alpha_4}}^{M_1}(\tau))$. In particular, it is of length two. 
Since $\Ind_{M_{\alpha_3,\alpha_4}}^{M_1}(\tau)$ is of length two also, we deduce that 
$$\Sigma_{1} = \operatorname{Res}^{M_1}_{M_1^{der}}(\sigma_1) \oplus   
\operatorname{Res}^{M_1}_{M_1^{der}}(\sigma_2).
$$
\item
Going through the proof of \cite[Proposition 4.1]{MR2103658}, it follows that the action of $E_{s}(z_0)$ is determined via its action on 
$\operatorname{Res}^{M_1}_{M_1^{der}}(\sigma_i)$. By \cite{Hanzer2015a}, the operator admits two eigenvalues $\pm1$. On each summand it acts differently. Hence, 
  $$E_{s}\res{\pi_2^{1}} =-\id.$$
\end{itemize}
 
\end{proof}
\begin{Prop}
Let $u = w_1 w_{2}w_{3}w_{4}w_{2}w_{3}w_{2}w_{1}$. We set  $$E^{1}_{u}(z) =  \NN_{s}^{u_2 \chi_{\para{P},z}}(u_2 \chi_{\para{P},z})\res{\Image \NN_{u_2}(\pi_z)}.$$ Then
$E^{1}_{u}(z)$ is holomorphic at $z=z_0$. In addition, one has 
$$E_{u}^{1}(z_0)\res{\pi_1^{1}} =\id, \quad E_{u}^{1}(z_0)\res{\pi_2^{1}} =-\id.$$      
\end{Prop}
\begin{proof}
As in \Cref{P4::stab::2}.
\end{proof}
In \Cref{Table:: P1 ::images}  we summarize all the information for $\Ind_{M_1}^{G}(z_0)$ needed for the global part.

\newpage
\section{$\para{P}=\para{P}_3, z_0=\frac{1}{2} $ }\label{local::p3::section}

\subsection{  Holomorphicity of normalized intertwining operators} \label{local::p3:: holo}

\begin{Prop}\label{local::p3:holo}
For any $w \in W(\para{P},G)$, the operator $\NN_{w}(z)$ is holomorphic at $z=z_0$. 
\end{Prop}

\begin{proof}

It is sufficient to show that $\NN_{w}(z)$ is holomorphic at $z=z_0$. With the definitions of \Cref{local::holo}, we shall show that each $w \in W(\para{P},G)$ can be written as a composition  of holomorphic operators of one of two types:  $(H)$ and $(Z)$. For the convenience of the reader, we record the decomposition of each $w \in W(\para{P},G)$ in \Cref{Table:: P3 ::images}. 

Note that the above decomposition is not necessarily of a reduced form. However, except for the word  listed below, the reduced word admits such a decomposition. The unique word for which such a decomposition does  not exist is  
$$w = w_{2} w_{3}w_{1}w_{2}w_{3} w_{4}w_{3}w_{2}w_{3}w_{1}w_{2}w_{3}w_{4}w_{2}w_{3}w_{1}w_{2}w_{3}.$$
In this case we shall show that it is of the type $HZHZHZH$.

It is easy to verify that $w=  w_4u,$ where $$u=\left(w_{2}w_{3}w_{1}w_{2}\right) \cdot\left(w_{3}w_{4}\right) \cdot\left(w_{3}w_{2}w_{3}\right) 
   \cdot \left(w_{1}w_{2}w_{3}w_{4}\right) 
   \cdot \left(w_{3}w_{2}w_{3}\right) 
   \cdot \left(w_{1}w_{2}w_{3}.\right)$$
Note that $u$ is of type $ZHZHZH$ and $(w_4,u\lambda_0)$ is of type $H$. Thus the claim follows.  
\end{proof}

We end this section with the following observation:
\begin{Remark} \label{local::p3::holo::arch}
Let $u \in W(\para{P},G)$ such that $\dim \bk{\Image \NN_{u}(z_0)}^{\iwhaori} >61 $. Then  $\NN_{u,\nu}(z)$ is holomorphic at $z=z_0$ for Archimedean places as well. It can be proved along the lines of \Cref{local::p3:holo}.
\end{Remark} 
\subsection{ The structure of $\Pi_3= \Ind_{M_3}^{G}(\frac{1}{2})$}\label{local::p3::structure}

For the remaining part, we shall specify the following exponents of $\Pi_3= \Ind_{M_3}^{G}(\frac{1}{2})$. Let 
\begin{align*}
\lambda_{a.d} &= [0,0,-1,0],
&\lambda_{0} &= [-1,-1,3,-1],
&\lambda_{1} &= [0,-1,1,-1],
&\lambda_2 &=  [1,0,-1,-1].
\end{align*}
Note that $\lambda_{a.d}$ (resp. $\lambda_0$) is the anti--dominant (resp. initial exponent) of $\Pi_3$. According to  \cite[Thm 6.1 ]{F4}, the representation $\Pi_3$ admits a unique irreducible subrepresentation $\tau^{3}$ and a maximal semi-simple quotient $\pi_1^{3} \oplus \pi_2^{3}$. Moreover, with the notations of \cite[Thm 6.1 ]{F4},  $\pi_1^{3}$ is the unique irreducible constituent of $\Pi_3$, having $\lambda_{a.d}$ in its Jacquet module. On the other hand, $\pi_1^{3},\pi_2^{3}$ are the only irreducible constituents having $\lambda_1$ in their Jacquet module.

We start this subsection by giving a lower bound for $\dim \jac{G}{T}{\pi_1^{3}}$ and  $\dim \jac{G}{T}{\pi_2^{3}}$. Then, using several computations in the Iwahori--Hecke algebra, we shall show that the length of $\Pi_3$ is at least five. We end this subsection by proving that its length is exactly $5$ and determining its structure.

 By  \Cref{local::p3:holo}, each normalized intertwining operator $\NN_{w}^{\chi_{\para{P},z}}(\lambda)$ is holomorphic at $\lambda=\lambda_0 = \chi_{\para{P},z_0}$. Thus, one can consider its restriction to $\Pi_3$. For shorthand writing,  for every $w \in W(\para{P},G)$ we denote by $\Sigma_{w} = \Image \left(\NN_{w}^{\chi_{\para{P},z}} (\lambda_0) \right)\res{\Pi_3}$.
 
 All the calculations in this section are preformed in the relevant Grothendieck ring. This allows us to suppress the notation $\coseta{ \cdot}$.
 
\subsubsection{ Determine $\jac{G}{T}{\pi_1^{3}}$}   
In this section we give a lower bound for $\dim \jac{G}{T}{\pi_1^{3}}$.     

\begin{Lem}\label{local::p3::pi1::common}
The representation $\pi_1^{3}$ is a common irreducible constituent of  $\Pi_3 = \Ind_{M_{3}}^{G}(\frac{1}{2})$ and $\Pi' =  \Ind_{M_{2}}^{G}(\frac{1}{2})$. Moreover, $\mult{\lambda_2}{\jac{G}{T}{\pi_1^{3}}}=2$ and  $$\mult{\w_{\alpha_1}\w_{\alpha_3}\lambda_2}{\jac{G}{T}{\pi_1^{3}}}\leq 2.$$  
\end{Lem}
\begin{proof}
In order to show that $\pi_1^{3}$ is a common irreducible constituent of $\Pi',\Pi_3$, we apply Tadic's criterion \cite[Lemma 3.1]{Tadica}  with the following arguments: $\Ind_{T}^{G}(\lambda_{a.d}), \Pi_3,\Pi',\lambda_{a.d}$. Indeed, 
$$\mult{\lambda_{a.d}}{ \Ind_{T}^{G}(\lambda_{a.d}}=
\mult{\lambda_{a.d}}{ \Pi'}=
\mult{\lambda_{a.d}}{ \Pi_3} =12.$$
Hence, $\Pi',\Pi_3$ share a common irreducible constituent such that $\lambda_{a.d}$ appears in their Jacquet module. However, since $\pi_1^{3}$ is the unique irreducible constituent of $\Pi_3$ for which $\lambda_{a.d}$ appears in its Jacquet module, the first claim follows. In order to prove the second part, we proceed as follows:
\begin{itemize}
\item
Recall that $\lambda_{a.d} \leq \jac{G}{T}{\pi_1^{3}}$. Therefore, by the central character argument, $\pi_1^{3} \hookrightarrow \Ind_{T}^{G}(\lambda_{a.d})$. Since $\Ind_{T}^{G}(\lambda_{a.d})$ is a standard module, by Langlands, it admits a unique irreducible subrepresentation. Namely, $\pi_1^{3}$ is the unique irreducible subrepresentation of $\Ind_{T}^{G}(\lambda_{a.d})$.  
\item
Note that $\lambda_{a.d} = u \cdot \lambda_{0}$. Where $u=w_{3}w_{4}w_{2}w_{3}w_{1}w_{2}w_{3}w_{4}w_{3}w_{1}w_{2}w_{3}$, one has $\Sigma_{u} \subset \Ind_{T}^{G}(\lambda_{a.d})$. Since $\pi_1^{3}$ is the unique irreducible subrepresentation of $\Ind_{T}^{G}(\lambda_{a.d})$, it follows that $\pi_1^{3}$ is also a subrepresentation of $\Sigma_{u}$. 
\item
 Write $u = \left(w_{3}w_{2}w_{1} \right) \cdot u'$ where $u' =w_{4}w_{3}w_{2}w_{3}w_{4}w_{3}w_{1}w_{2}w_{3}$.   \Cref{Table:: P3 ::images} implies that $\Sigma_{u'} \simeq \Sigma_{u}$. In particular, $\pi_{1}^{3}$ is  a subrepresentation of $\Sigma_{u'} \subset \Ind_{T}^{G}(u' \lambda_{0})$. Note that $u'\lambda_0 =\lambda_2$, and thus,  $\pi_1^{3} \hookrightarrow \Ind_{T}^{G}(\lambda_2)$. Therefore by Frobenius reciprocity, $\lambda_2 \leq \jac{G}{T}{\pi_1^{3}}$.
\item
By \eqref{Eq:OR} it follows that $2 \vert \mult{\lambda_2}{\jac{G}{T}{\pi_1^{3}}}$. However, $\mult{\lambda_2}{\jac{G}{T}{\Pi_3}}=2$. Thus the second part follows.     
\end{itemize}
In order to prove the third part we note that
$$\mult{\w_{\alpha_3}\w_{\alpha_1}\lambda_2}{\jac{G}{T}{\pi_1^{3}}} \leq
   \mult{\w_{\alpha_3}\w_{\alpha_1}\lambda_2}{\jac{G}{T}{\Pi'}}=2.$$ 
\end{proof}

\begin{Lem}\label{local::p3::shreical::rep}
$\dim \jac{G}{T}{\pi_1^{3}} \geq 42$.
\end{Lem}   
\begin{proof}

By \Cref{local::p3::pi1::common}, $\coset{\lambda_{a.d}} + \coset{\lambda_2} \leq \jac{G}{T}{\pi_1^{3}}$. 
We proceed by applying the following sequence of branching rules:
\begin{itemize}
\item
Applying \eqref{Eq:OR} with respect to $\lambda_{a.d}$ implies that $12 \times \coset{\lambda_{a.d}} \leq \jac{G}{T}{\pi_1^{3}}$. 
\item
Applying  \eqref{Eq::B3a} on $\lambda_{a.d}$ with respect to $M_{\alpha_1,\alpha_2,\alpha_3}$ yields
\begin{equation*}
12 \times \coset{\lambda_{a.d}}+ 8 \times\coset{\s{\alpha_3}\lambda_{a.d}} + 4\times \coset{\s{\alpha_2}\s{\alpha_3}\lambda_{a.d}}\leq \jac{G}{T}{\pi_1^{3}}.
\end{equation*} 
\item
Applying \eqref{Eq::C3a} on $\lambda_{a.d}$ with respect to $M_{\alpha_2,\alpha_3,\alpha_4}$  gives 
\begin{equation*}
12 \times \coset{\lambda_{a.d}}+ 6 \times\coset{\s{\alpha_3}\lambda_{a.d}} + 3\times \coset{\s{\alpha_2}\s{\alpha_3}\lambda_{a.d}} \leq \jac{G}{T}{\pi_1^{3}}.
\end{equation*}
However, we have shown that $\mult{\s{\alpha_3}\lambda_{a.d}}{\jac{G}{T}{\pi_1^{3}}}=8$. \Cref{Table::C3:: a} yields
$$\operatorname{Res}^{GSp_6}_{Sp_6} \bk{\jac{G}{M_{\alpha_{2},\alpha_3,\alpha_4}}{\pi_1}} \geq 2 \times \coset{D_{Sp_6}(\sigma_{0})} + 3 \times \coset{\sigma_{0}}.$$
Thus,    
	\begin{equation*}
	\begin{split}
	4 \times \coset{\s{\alpha_4}\s{\alpha_3}\lambda_{a.d}}+  2\times\coset{\s{\alpha_3}\lambda_{a.d}} + 2\times \coset{\s{\alpha_2}\s{\alpha_4}\s{\alpha_3}\lambda_{a.d}}
	&\\
	 +2\times \coset{\s{\alpha_3}\s{\alpha_2}\s{\alpha_4}\s{\alpha_3}\lambda_{a.d}}
	+12 \times \coset{\lambda_{a.d}}+ 6 \times\coset{\s{\alpha_3}\lambda_{a.d}} + 3\times \coset{\s{\alpha_2}\s{\alpha_3}\lambda_{a.d}}   & 
	\leq \jac{G}{T}{\pi_1^{3}}.  
	\end{split}
	\end{equation*}
	In other words, we also invoke   \eqref{Eq::C3b} twice.
\item
Applying \eqref{Eq:OR} and \eqref{Eq::C3d} on $\lambda_{2}$ implies 
$$
2 \times \coset{\lambda_2} + \coset{\s{\alpha_3}\lambda_2} +  \coset{\s{\alpha_2}\s{\alpha_3}\lambda_2}\leq   \jac{G}{T}{\pi_1^{3}}.
$$
\item
By the Geometric Lemma,  one has $\mult{\s{\alpha_2}\s{\alpha_3}\lambda_2}{\jac{G}{T}{\Pi_3}}=0$. In particular, $\mult{\s{\alpha_2}\s{\alpha_3}\lambda_2}{\jac{G}{T}{\pi_1^{3}}}=0$.  Hence, by \Cref{App:B3::1}, it follows that we can invoke  \eqref{Eq::B3::c} once, namely, 
$$  2\times \coset{\lambda_2} +  \coset{\s{\alpha_1} \lambda_2} + \coset{\s{\alpha_3} \lambda_2} + 2 \times \coset{\s{\alpha_1}\s{\alpha_3}\lambda_2}\leq \jac{G}{T}{\pi_1^{3}}.$$

Note that $\s{\alpha_1}\lambda_2 =  \s{\alpha_2}\s{\alpha_3}\lambda_{a.d}$. As a result, $\mult{\s{\alpha_2}\s{\alpha_3}\lambda_{a.d}}{\jac{G}{T}{\pi_1^{3}}} \geq 5$  one occurrence from \eqref{Eq::B3::c} and four  occurrences from \eqref{Eq::B3a}. 
\item
Note that $\coset{\lambda_3} \leq \jac{G}{T}{\pi_1^{3}},$ where $\lambda_3 = \s{\alpha_1}\s{\alpha_3}\lambda_2 =[-1,0,1,-2]$. Applying \eqref{Eq::C3c} on $\lambda_3$ with respect to $M_{\alpha_2,\alpha_3,\alpha_4}$ implies that 
$$ 2 \times \coset{\lambda_3} + 2 \times \coset{\s{\alpha_4}\lambda_3} + \coset{\s{\alpha_3}\lambda_3} + \coset{\s{\alpha_3}\s{\alpha_4}\lambda_3}\leq   \jac{G}{T}{\pi_1^{3}}.
$$
\end{itemize}

To summarize, we applied the following rules of type $C_3$
\begin{center}

\begin{tabular}{|c|c|c|}
\hline  
Rule Type $(\sigma)$&  Multiplicity (m)   & $m \times \dim \jac{M_{\alpha_{2},\alpha_3,\alpha_4}}{T}{\sigma}$   \\ 
\hline  
\eqref{Eq::C3a}& 3 & 21  \\ \hline
\eqref{Eq::C3b}& 2 & 10  \\ \hline
\eqref{Eq::C3c}& 1 & 6  \\ \hline
\eqref{Eq::C3d}& 1 & 4  \\ \hline
\end{tabular} 
\end{center}
\vspace{2pt}
Thus, we have shown that $\dim \jac{G}{T}{\pi_1^{3}} \geq 41$. However, if $\dim \jac{G}{T}{\pi_1^{3}} = 41$, the implication is that $\mult{\s{\alpha_2}\s{\alpha_3}\lambda_{a.d} }{\jac{G}{T}{\pi_1^{3}}} =4$. But we have shown that  $$\mult{\s{\alpha_2}\s{\alpha_3}\lambda_{a.d} }{\jac{G}{T}{\pi_1^{3}}} \geq 5.$$ Thus the claim follows.

\end{proof}  

\subsubsection{Determine $\jac{G}{T}{\pi_2^{3}}$}
The aim of this subsection is to demonstrate that $\jac{G}{T}{\pi_2^{3}} \geq 19$. We fix the following notations.

\begin{align}
u_1 &= w_{1}w_{2}w_{3}w_{4}w_{1}w_{2}w_{3},
&
u_1^{'}&= \bk{w_{2}w_{3}}u_1, \nonumber \\  
u_2 &= w_{3}w_{2}w_{3}w_{1}w_{2}w_{3}w_{4}w_{3}w_{2}w_{3}w_{1}w_{2}w_{3}, \nonumber
& u_2^{'} &=\bk{w_{2}w_{3}w_{4}}u_2,  \nonumber\\
\lambda_{s}&=[-1,-1,1,1], & \lambda_{s}^{'}&=[0,-1,0,1], \nonumber  \\
u&=w_{1}w_{2}w_{3}w_{4}w_{3}w_{2}w_{3}w_{1}w_{2}w_{3}, &
s&= w_3 w_2 w_3, \nonumber \\
L&=u \chi_{\para{P},z}, & 
E(z) &= \NN_{w_3 w_2 w_3}^{L}(u\chi_{\para{P},z}). \label{localp3not}
\end{align} 
 
\begin{Prop}\label{p2::multi2}
One has $\pi_2^{3} \oplus \pi_2^{3} \hookrightarrow \Ind_{T}^{G}(\lambda_{s})$. In particular $$\mult{\lambda_{s}}{\jac{G}{T}{\pi_2^{3}}}\geq 2.$$
\end{Prop}
\begin{proof}
The proof goes as follows: In \Cref{local::p3 :: decompostion}, we show that $E(z)$ is holomorphic at $z=z_0$. In addition, $E(z_0) \in \operatorname{End}_{G}\bk{ \Ind_{T}^{G}(\lambda_s)}$. We also show that  
$$ \Ind_{T}^{G}(\lambda_s) =  V_{1} \oplus V_{-1},$$
where $V_1,V_{-1}$ are eigenspaces of $E(z_0)$. 

Thus, it suffices to show that $\pi_2^{3} \hookrightarrow V_{1}$ and $\pi_2^{3} \hookrightarrow V_{-1}$. In  \Cref{local::p3 ::pi2 ::decomp},  we show that $\pi_1^{3} \oplus \pi_2^{3} \hookrightarrow V_1$. Hence $\pi_2^{3} \hookrightarrow V_{1}$. \Cref{local::p3::pi2::vm1}  implies that $\pi_2^{3} \hookrightarrow V_{-1}$ .  The second part follows immediately from Frobenius reciprocity.
\end{proof}

\begin{Lem}\label{local::p3 :: decompostion}
One has $u \chi_{\para{P},z_0} =  su \chi_{\para{P},z_0} = \lambda_s$. In addition,
the operator $E(z)$ is holomorphic at $z=z_0$. Furthermore, one has 
$$ \Ind_{T}^{G} (\lambda_s) = V_1 \oplus V_{-1},$$
where $V_{i}$ are eigenspaces of $E(z_0)$.
\end{Lem}
\begin{proof}
A direct computation shows that $u\chi_{\para{P},z_0}= su \chi_{\para{P},z_0}$. The other parts follow  immediately from an application of \Cref{Zampera::action} with the following parameters: 
\begin{align*}
\chi &= \lambda_s, & \alpha &=\alpha_3,  \\ m&=u\fun{3}, & s&=  w_3w_2w_3, & L=& u \chi_{\para{P},z} = 
(z-z_0)m +\chi.  
\end{align*}
This 
shows that $\Ind_{T}^{G}(\lambda_{s}) = V_1 \oplus V_{-1}$. Here, $V_1,V_{-1}$ are eigenspaces for the action of $E(z_0)= \NN_{s}^{L}(u\chi_{\para{P},z_0})$.
\end{proof}

Let  $E=E(z_0)$. 
\begin{Lem} \label{local::p3::pi2 :ker}
One has 
$\Sigma_{u_1^{'}} = \Sigma_{u_2^{'}}$
and 
$\pi_1^{3} \oplus \pi_2^{3} \hookrightarrow \Sigma_{u_1^{'}}$.
\end{Lem}
\begin{proof}
The first item follows by a direct computation in the Iwahori-Hecke algebra. Namely, we show that $\Sigma_{u_1'}^{\iwhaori} = \Sigma_{u_2'}^{\iwhaori}$.  For the second part, we proceed as follows: Let  $w =  w_4 u_{2}^{'}$. Note that $\lambda_{s}' =u_{2}' \lambda_{0}$.
 Based on \Cref{Table:: P3 ::images},   $\ker \NN_{w_{4}} ( \lambda_{s}^{'})$ and $\Sigma_{u^{'}_{2}}$ share a common irreducible subrepresentation $\sigma$. 
Moreover, by dimension considerations, one has 
$$\dim \jac{G}{T}{\sigma} =\dim \sigma^{\iwhaori} \leq \dim \Sigma_{u_{2}^{'}}^{\iwhaori} - \dim \Sigma_{w}^{\iwhaori} =  61-42=19.$$
In particular, $\sigma  \neq \pi_1^{3}$. We continue by showing that $\sigma  \simeq \pi_2^{3}$. By construction, $\sigma \hookrightarrow \Ind_{T}^{G}(\lambda_{s}')$. Thus,   by Frobenius $\mult{\lambda_{s}'}{\jac{G}{T}{\sigma}} \geq 1$,  and application of \eqref{Eq:OR} and \eqref{Eq:A2} with respect to $M_{\alpha_3,\alpha_4}$ implies that $\mult{\s{\alpha_4}\lambda_{s}'}{\jac{G}{T}{\sigma}} \neq 0$. However, $\lambda_2= \s{\alpha_4}\lambda_{s}'$, and by definition $\pi_1^{3}, \pi_2^{3}$ are the only irreducible constituents of $\Pi_3$ with the property $\mult{\lambda_2}{\jac{G}{T}{\cdot}} \neq 0$. Since $\sigma \neq \pi_1^{3}$, it follows that $\sigma \simeq \pi_2^{3}$.   

Next, we  show that $\pi_1^{3} \oplus \pi_2^{3} \hookrightarrow \Sigma_{u_{1}'}$. Both $\Sigma_{u_1'} $ and $\pi_1^{3}$ are constituents of $\Pi_3$. By dimension considerations, $\pi_1^{3}$ is a constituent of $\Sigma_{u_1'}$. Otherwise the implication is that 
$$96 =\dim \jac{G}{T}{\Pi_3} \geq \dim \jac{G}{T}{\pi_1^{3}} + \dim \jac{G}{T}{\Sigma_{u_1^{'}}} \geq 42 +61.$$  
 
Since $\s{\alpha_2}\s{\alpha_3} \lambda_{a.d} = u_1^{'} \lambda_{0} =  \lambda_{s}'$, it follows that $\Sigma_{u_1^{'}} \hookrightarrow  \Ind_{T}^{G}(\lambda_{s}^{'})$. Moreover, by \Cref{local::p3::shreical::rep}, it follows that $\mult{\lambda_{s}'}{\jac{G}{T}{\pi_1^{3}}} \geq 4$. Thus, by the central character argument, one has $\pi_1^{3} \hookrightarrow \Ind_{T}^{G}(\lambda_{s}')$. Since $\pi_1^{3}$ appears with multiplicity one in $\Ind_{T}^{G}(\lambda_{s}'),$ it follows that $\pi_1^{3} \hookrightarrow \Sigma_{u_1'}$. Hence $\pi_1^{3} \oplus \pi_2^{3} \hookrightarrow \Sigma_{u_1'}$. 
\end{proof}

\begin{Lem}\label{local::p3 ::pi2 ::decomp}
Keeping the above notations, one has $\Sigma_{u_1} \subset V_1$. In particular $$\pi_1^{3} \oplus \pi_2^{3} \hookrightarrow V_1.$$ 
\end{Lem}
\begin{proof}
Note that $u_1 \lambda_{0} = \lambda_{s}$. Hence $\Sigma_{u_1} \hookrightarrow \Ind_{T}^{G}(\lambda_{s})$. Based on \Cref{Table:: P3 ::images}, one has $\Sigma_{u_1} \simeq \Sigma_{u_1^{'}}$. Thus, by \Cref{local::p3::pi2 :ker}, it follows that $\pi_1^{3} \oplus \pi_2 ^{3}\hookrightarrow \Ind_{T}^{G}(\lambda_s)$. We conclude this proof by showing that $\Sigma_{u_1} \hookrightarrow V_1$. By \Cref{local::p3 :: decompostion}, $V_1$ is the eigenspace corresponding to the eigenvalue $1$ of $E$. Thus it suffices to show  that for every $v \in \Sigma_{u_1}$, one has $E(v) =  v$. This check can be performed  in the Iwahori-Hecke algebra as follows: Let $n_{u_1}(z_0)$ (resp. $n_{s}(z_0)$) be a represnatative of $\NN_{u_1}(z_0)$ (resp. $E(z_0)$) in the Iwahori--Hecke algebra, as in \Cref{subsection:: hPmoduule}. Recall that  
\begin{align*}
\Sigma_{u_1}^{\iwhaori} &= \Span \set{ T_{w^{'}} \cdot  Triv \cdot n_{u_1}(z_0) \: : w' \in W(P,G)}, & Triv =  \sum_{x \in \weyl{M_{3}}} T_{x}.  
\end{align*}

Thus, it is enough to show that 
$$ Triv \cdot n_{u_1}(z_0)  =Triv \cdot n_{u_1}(z_0) \cdot n_s(z_0). $$  
Using the computer, it follows that $$ Triv \cdot n_{u_1}(z_0)  =Triv \cdot n_{u_1}(z_0)\cdot n_s(z_0).$$
Hence, the claim follows.
\end{proof}
\begin{Lem}\label{local::p3::pi2::vm1}
With the same notations as above, one has $\Sigma_{u_2} \not \subset V_1$ and $\Sigma_{u_2} \not \subset V_{-1}.$  Moreover, $\pi_2^{3}$ is the unique common irreducible  subrepresentation of $\Sigma_{u_2}$ and $V_{-1}$.
\end{Lem}
\begin{proof}
Note that $u_2 \lambda_{0} =\lambda_s$. Thus, as before, it suffices to show that $$Triv \cdot n_{u_2}(z_0)  \not \in \Span\{ Triv \cdot n_{u_2}(z_0) \cdot n_s(z_0) \} .$$
This is done using the computer, which implies the first part of this Lemma. In particular, the operator 
$ (E(z_0) +\id)\res{ \Sigma_{u_2}}$ admits a non-trivial kernel $K$. Let $\sigma \subset K$ be an irreducible subrepresentation. Note that $K \subset   V_{-1}$.   

 Based on \Cref{Table:: P3 ::images}, one has $\Sigma_{u_2} \simeq \Sigma_{u_2^{'}}$. Thus, by \Cref{local::p3::pi2 :ker}, it follows that $\pi_1^{3} \oplus \pi_2^{3} \hookrightarrow \Sigma_{u_2^{'}}$.  
 
 Suppose  $\sigma \not \simeq \pi_1^{3} $ and $\sigma \not \simeq \pi_2^{3} $. Then it follows that $$\pi_1^{3} \oplus \pi_2^{3} \oplus  \sigma \hookrightarrow \Sigma_{u_2}.$$ 
 Thus,  by \Cref{Table:: P3 ::images}, one has $\sigma \oplus \pi_1^{3} \oplus \pi_2^{3} \hookrightarrow \Sigma_{u_2^{'}}$. 
 In particular,
 $$\mult{ \lambda_{s}'}{\jac{G}{T}{\pi_1^{3}}}, \mult{ \lambda_{s}'}{\jac{G}{T}{\pi_2^{3}}}, \mult{ \lambda_{s}'}{\jac{G}{T}{\sigma}} \geq 1. $$
 
 However, \eqref{Eq:OR} implies that $$\mult{ \lambda_{s}'}{\jac{G}{T}{\pi_1^{3}}}, \mult{ \lambda_{s}'}{\jac{G}{T}{\pi_2^{3}}}, \mult{ \lambda_{s}'}{\jac{G}{T}{\sigma}} \geq 4.$$ Thus,
 \begin{align*}
 \mult{ \lambda_{s}'}{\jac{G}{T}{\Pi_3}} & \geq \mult{ \lambda_{s}'}{\jac{G}{T}{\pi_1^{3}}} + \mult{ \lambda_{s}'}{\jac{G}{T}{\pi_2^{3}}} +  \mult{ \lambda_{s}'}{\jac{G}{T}{\sigma}} 
 \\&\geq  4+4 +4=12.
\end{align*}
On the other hand, Geometric lemma implies that $\mult{\lambda_{s}'}{\Pi_{3}} =8$, which is a contradiction. Hence $\sigma \simeq \pi_1^{3}$ or $\sigma \simeq \pi_2^{3}$. Note that $\mult{\pi_1^{3}}{\Ind_{T}^{G}(\lambda_{s})}=1$. By \Cref{local::p3 ::pi2 ::decomp}, one has $\pi_1^{3} \hookrightarrow V_1$. Thus, $\pi_1^{3} \not \hookrightarrow V_{-1}$, and therefore $\pi_2^{3} \hookrightarrow V_{-1}$.   
\end{proof}
\newpage
\begin{Lem}\label{local::p3::pi2::dim}
$\dim \jac{G}{T}{\pi_2^{3}} \geq 19.$
\end{Lem}
\begin{proof}
\begin{itemize}
\item
Recall that  $\mult{\lambda_{a.d}}{\jac{G}{T}{\Pi_3}}= \mult{\lambda_{a.d}}{\jac{G}{T}{\pi_1^{3}}}=12$. Hence 
$$\mult{\lambda_{a.d}}{\jac{G}{T}{\pi_2^{3}}}=0.$$ On the other hand, \Cref{local::p3::pi2 :ker} also implies that $\mult{\lambda_{1}}{\jac{G}{T}{\pi_2^{3}}}=2$. Note that $\s{\alpha_3}\lambda_{1}=\lambda_{a.d}$.   Thus ,\Cref{App:B3::1} and \Cref{Table::C3:: a} yield 
\begin{eqnarray} 
2\times \coset{\lambda_{1}} +  4 \times \coset{\s{\alpha_4} \lambda_{1}} + 2 \times\coset{\s{\alpha_2}\s{\alpha_4} \lambda_{1}} +  2 \times \coset{\s{\alpha_3}\s{\alpha_2}\s{\alpha_4} \lambda_{1}}&\leq 
  \jac{G}{T}{\pi_2^{3}}, \nonumber \\
2\times \coset{\lambda_{1}} +  \coset{\s{\alpha_2} \lambda_{1}} + \coset{\s{\alpha_3}\s{\alpha_2} \lambda_{1}} +  2 \times \coset{\s{\alpha_1}\s{\alpha_3}\s{\alpha_2} \lambda_{1}}&\leq 
 \jac{G}{T}{\pi_2^{3}}. \nonumber
\end{eqnarray} 
\item
Invoking \eqref{Eq::C3c} on $\lambda_{2}= \s{\alpha_3}\s{\alpha_2}\lambda_{1} =[-1,0,-1,2]$ with respect to $M_{\alpha_2,\alpha_3,\alpha_4}$ implies that 
$$
2 \times \coset{\lambda_{2}} + 2\times \coset{\s{\alpha_4}\lambda_{2}} + \coset{\s{\alpha_3}\lambda_{2}} + \coset{\s{\alpha_3}\s{\alpha_4}\lambda_{2}} \leq \jac{G}{T}{\pi_2^{3}}.
$$
\item
Note that $\s{\alpha_3}\s{\alpha_4}\lambda_{2} = \lambda_{s}$ and $\lambda_{3} =  \s{\alpha_4}\lambda_{2} =  [-1,0,1,-2]$. Thus, an application  of  \eqref{Eq:A2} on $\lambda_{3}$ with respect to $M_{\alpha_1,\alpha_2}$ and then \eqref{Eq:A1} implies that 
$$ 2\times \coset{\lambda_{3}} + \coset{\s{\alpha_1}\lambda_{3}} + \coset{\s{\alpha_4}\s{\alpha_1}\lambda_{3}} \leq \jac{G}{T}{\pi_2^{3}}.$$   
\item
The above process yields  
\begin{equation} \label{ loacl:: P3 ::pi_2jac_a}
\begin{split}
2\times \coset{\lambda_{2}} +  4 \times \coset{\s{\alpha_4} \lambda_{2}} + 2 \times\coset{\s{\alpha_2}\s{\alpha_4} \lambda_{2}} +  2 \times \coset{\s{\alpha_3}\s{\alpha_2}\s{\alpha_4} \lambda_{2}}
& \\ 
+
2 \times \coset{\s{\alpha_3}\s{\alpha_2}\lambda_{2}} + 2\times \coset{\s{\alpha_4}\s{\alpha_3}\s{\alpha_2}\lambda_{2}} + \coset{\s{\alpha_2}\lambda_{2}} + \coset{\s{\alpha_3}\s{\alpha_4}\s{\alpha_3}\s{\alpha_2}\lambda_{2}} &
\\
 + \coset{\s{\alpha_1}\s{\alpha_4}\s{\alpha_3}\s{\alpha_2}\lambda_{2}} + \coset{\s{\alpha_1}\s{\alpha_3}\s{\alpha_2}\lambda_{2}}& 
\\
 \leq \jac{G}{T}{\pi_2^{3}}.  
\end{split}
\end{equation}    
Thus, $\dim \jac{G}{T}{\pi_2^{3}} \geq 18$. However, this process also shows that 
$\mult{ \lambda_{s}}{\jac{G}{T}{\pi_2^{3}}}= 1$, but \Cref{p2::multi2} shows that $\mult{ \lambda_{s}}{\jac{G}{T}{\pi_2^{3}}}\geq 2$. 
Thus 
$$\dim \jac{G}{T}{\pi_2^{3}} \geq 19.$$
\end{itemize}
\end{proof}

\subsubsection{Jordan-H\"older series of $\Pi_3$}
The aim of this part is to prove the following proposition:
\begin{Prop}
Let 
\begin{align*}
u_1 &=  w_1 w_2 w_3, &
u_2 &=  \left(w_3\right) \cdot u_1, &
u_3 &= \left(w_{1}w_{2}w_{3}w_{4}\right)\cdot u_2, &
u_4 & = \left( w_4w_2w_3\right) u_3.
\end{align*}
Denote by $K_i =  \ker \left( \NN_{u_i}^{\chi_{\para{P},z}}(\lambda_0)\res{\Pi_3}\right)$, and then 
$$ 0  = K_{0}\subsetneq  K_2 \subsetneq K_3 \subsetneq K_4 \subsetneq \Pi_3$$
is a Jordan--H\"older series for $\Pi_3$. In particular, $\Pi_3$ is of length $5$.
\end{Prop}

\begin{proof}
Based on \Cref{Table:: P3 ::images}, one has 
\begin{align*}
\dim K_1^{\iwhaori} &=6, & \dim K_2^{\iwhaori}&=10, & \dim K_3^{\iwhaori} &= 35, \\
\dim K_4^{\iwhaori} &=54, & \dim \Pi_3^{\iwhaori}&=96. 
\end{align*}
By definition, $K_{i-1} \subset K_{i}$. In particular $K_{i-1} \neq K_{i}$. It remains to show that $K_{i}\rmod K_{i-1}$ is irreducible.
\begin{itemize}
\item
 By
\Cref{local::p3::shreical::rep}, $\dim \jac{G}{T}{\pi_1^{3}} \geq 42$. Hence by  dimension considerations, $\pi_1^{3}$ is not a constituent of $K_3$. Moreover,  $\dim K_4^{\iwhaori} <  \dim K_3^{\iwhaori}  + \dim \jac{G}{T}{\pi_1^{3}}$. It follows that $\pi_1^{3}$ is not a constituent of $K_4$ also. In conclusion, $\pi_1^{3}$ is a constituent of $\Pi_3 \rmod K_4$, but 
$$96 = \dim \Pi_3^{\iwhaori} \geq \dim K_4^{\iwhaori} + \dim \jac{G}{T}{\pi_1^{3}} =54+42=96.       
$$
Hence $\Pi_3 \rmod K_4 \simeq \pi_1^{3}$.
\item
 \Cref{local::p3::k1} shows that $K_1=\tau^{3}$.
\item
 \Cref{local::p3::k2} shows that $K_2 \rmod K_1 \simeq \sigma_1^{3}$ is irreducible.
\item
 \Cref{local::p3::k3} shows that $K_3 \rmod K_2 \simeq \sigma_2^{3}$ is irreducible.
\item
\Cref{local::p3::k4} shows that $K_4 \rmod K_3 \simeq \pi_2^{3}$.
\end{itemize}

\end{proof}
\begin{Lem} \label{local::p3::k1}
The length of $K_1$ is one. In particular,  $K_1 =\tau_1^{3}$.
\end{Lem}
\begin{proof}
By our notations, $\tau^{3}$ is the unique irreducible subrepresentation of $\Pi_3$, and therefore $\tau^{3} \hookrightarrow K_1$. Thus, one has 
$\dim \bk{\tau^{3}}^{\iwhaori}   \leq \dim K_1^{\iwhaori} =6.$   Hence, it is enough to show that 
$6 \leq \dim \bk{\tau^{3}}^{\iwhaori} =\dim \jac{G}{T}{\tau^{3}}.$
This implies that $ \dim \bk{\tau^{3}}^{\iwhaori}= \dim K_1^{\iwhaori} =6$. In particular, $K_1 =\tau^{3}$ and is thereby irreducible. 

Let $\lambda_0$ denote the initial exponent of $\Pi_3$, namely, 
$\lambda_0 = [-1,-1,3,-1]$. Induction in stages yields  
$$\tau^{3} \hookrightarrow \Pi_3 \hookrightarrow \Ind_{T}^{G}(\lambda_0).$$
Thus, Frobenius reciprocity implies that  $\coset{\lambda_{0}} \leq \jac{G}{T}{\tau^{3}}$.
An application of a sequence of branching rules of type $A_1$ yields 
\begin{equation} \label{ loacl:: P3 ::tau}
\begin{split}
\coset{\lambda_0} + \coset{\s{\alpha_3} \lambda_0} + \coset{\s{\alpha_2}\s{\alpha_3} \lambda_0} +    
\coset{\s{\alpha_4}\s{\alpha_3} \lambda_0}& \\ 
+\coset{\s{\alpha_4}\s{\alpha_2}\s{\alpha_3} \lambda_0}+
\coset{\s{\alpha_3}\s{\alpha_4}\s{\alpha_2}\s{\alpha_3} \lambda_0}& \leq \coset{\jac{G}{T}{\tau^{3}}}.  
\end{split}
\end{equation}  
Hence, $K_1$ is irreducible.
\end{proof}
\begin{Lem} \label{local::p3::k2}
$K_2$ is of length two.
\end{Lem}
\begin{proof}

Let $\lambda_2 =  u_1\chi_{\para{P},z_0}=[-1,-1,1,2]$. The operator $\NN_{w_3}(\lambda_2)$ admits a non-trivial kernel  $K$.
According to \Cref{Table:: P3 ::images}, it follows that $K$ and $\Sigma_{u_1}$ share a common irreducible subrepresentation $\sigma_1^{3}$. Since $\sigma_1^{3} \hookrightarrow K \hookrightarrow \Ind_{T}^{G}(\lambda_2)$, one has by Frobenius reciprocity that $\coset{\lambda_2} \leq\coset{\jac{G}{T}{\sigma_1^{3}}}$.  

An application of branching rules of type $A_1$ yields 
\begin{equation}
\begin{split}
 \coset{\lambda_2} + \coset{\s{\alpha_4}\lambda_2} + \coset{\s{\alpha_3}\s{\alpha_4}\lambda_2} + \coset{\s{\alpha_2}\s{\alpha_3}\s{\alpha_4}\lambda_2} \leq \jac{G}{T}{\sigma_{1}^{3}}  \label{local :: P3 : pi3 jac}.       
\end{split}
\end{equation}
In particular, $\dim \jac{G}{T}{\sigma_1^{3}} \geq 4$. 

On the other hand, $\sigma_1^{3}$ is also a constituent of $K_2$. Hence,
\begin{equation} \label{local :: P3 : pi3 dim}
\dim \jac{G}{T}{\sigma_{1}}= \dim \sigma_{1}^{\iwhaori} \leq \dim K_2^{\iwhaori}-\dim K_1^{\iwhaori} =4.
\end{equation}

Thus, combining  \eqref{local :: P3 : pi3 jac} and \eqref{local :: P3 : pi3 dim},  we see that there is an equality in \eqref{local :: P3 : pi3 jac}.
In particular, $\left(K_2\right)_{s.s} = K_1 \oplus \sigma_{1}^{3}$.
\end{proof}

\newpage
\begin{Lem} \label{local::p3::k3}
$K_3$ is of length 3.
\end{Lem}
\begin{proof}
Note that by \Cref{Table:: P3 ::images}, one has  $K_2 \subsetneq K_3$. Since $K_2$ is of length $2$, it follows that $K_3$ is at least of length  $3$. Moreover, if $\left( K_3\right)_{s.s} =  K_2 \oplus \sigma$, then $$\dim \jac{G}{T}{\sigma} = \dim \sigma^{\iwhaori} =\dim K_3^{\iwhaori}-\dim K_2^{\iwhaori}= 25.$$

Showing the existence of an irreducible constituent $\sigma_2^{3}$ of $K_3$ satisfying $\dim \jac{G}{T}{\sigma_2^{3}} = 25$ ends the proof. For this purpose we proceed as follows:
\begin{itemize}
\item
Let $u^{'}=\left(w_{2}w_{3}w_{4}\right)\cdot u_2$. Then, \Cref{Table:: P3 ::images} implies that $ \Sigma_{u'} \simeq \Sigma_{u_2}$. In particular, 
$$\left(\ker \NN_{u'}(z_0)\res{\Pi_3}\right)^{\iwhaori} \simeq K_{2}^{\iwhaori}.$$   
\item
Note that $u_3 =  w_1 u^{'}$ and $\dim \left(\ker \NN_{u'}(z_0)\res{\Pi_3}\right)^{\iwhaori} < \dim K_3^{\iwhaori}$. This implies that the operator $\NN_{w_1}(u^{'}\chi_{\para{P},z_0})\res{\Sigma_{u'}}$ has a non-trivial kernel. Let $\sigma_2^{3}$ be an irreducible subrepresentation of
$\ker \NN_{w_1}(u^{'}\chi_{\para{P},z_0})\res{\Sigma_{u'}}$. 
\item
Observe that $$\sigma_2^{3} \hookrightarrow \ker \NN_{w_1}(u^{'}\chi_{\para{P},z_0})\res{\Sigma_{u'}} \hookrightarrow \Ind_{T}^{G}(u^{'} \chi_{\para{P},z_0} ).$$
If $\lambda_2 =u^{'} \chi_{\para{P},z_0} = [1,-2,2,-1]$, then by Frobenius reciprocity  
$$\coset{\lambda_2} \leq \coset{\jac{G}{T}{\sigma_2^{3}}}.$$
\item
An application of \eqref{Eq::B3b} on $\lambda_2$ with respect to $M_{\alpha_1,\alpha_2,\alpha_3}$ implies that 
\begin{align*}
	2 \times \coset{\lambda_2} +
	2 \times \coset{\s{\alpha_2}\lambda_2}+
	2 \times \coset{\s{\alpha_3}\lambda_2}
	+\coset{\s{\alpha_1}\lambda_2}	
	+2 \times \coset{\s{\alpha_2}\s{\alpha_1}\lambda_2}
	+\coset{\s{\alpha_1}\s{\alpha_2}\lambda_2} &\\
	+2 \times \coset{\s{\alpha_3}\s{\alpha_2}\lambda_2}
	+\coset{\s{\alpha_1}\s{\alpha_3}\lambda_2}
	+2 \times \coset{\s{\alpha_1}\s{\alpha_2}\s{\alpha_1}\lambda_2}
	+\coset{\s{\alpha_1}\s{\alpha_3}\s{\alpha_2}\lambda_2}& \leq \coset{\jac{G}{T}{\sigma_{2}^{3}}}. 
\end{align*}
\item
Set 
$\lambda_3= \s{\alpha_3}\s{\alpha_2} \lambda_2 =  [-1,0,2,-3]$.
Thus, using the following sequence of branching rules yields: 
\item
\eqref{Eq:A1} shows that $2 \times \coset{\s{\alpha_4}\lambda_3} \leq \jac{G}{T}{\sigma_{2}^{3}}$. 
\item
An application of \eqref{Eq:A2} on $\s{\alpha_4}\lambda_3$ with respect to $M_{\alpha_1,\alpha_2}$ implies 
$$ 2 \times \coset{\s{\alpha_4}\lambda_3} + \coset{\s{\alpha_1}\s{\alpha_4}\lambda_3} \leq \coset{\jac{G}{T}{\sigma_{2}^{3}}}.$$ 
\item
An application of \eqref{Eq::C2b} on $\s{\alpha_4}\lambda_3$with respect to $M_{\alpha_2,\alpha_3}$ implies 
$$ 2 \times \coset{\s{\alpha_4}\lambda_3} + \coset{\s{\alpha_3}\s{\alpha_4}\lambda_3} \leq \coset{\jac{G}{T}{\sigma_{2}^{3}}}.$$ 
\item
An application of all possible branching rules of type \eqref{Eq:A1} on
$\s{\alpha_3}\s{\alpha_4}\lambda_3$ yields 
$$
\coset{\s{\alpha_3}\s{\alpha_4}\lambda_3} +
\coset{\s{\alpha_4}\s{\alpha_3}\s{\alpha_4}\lambda_3} + 
\coset{\s{\alpha_4}\s{\alpha_3}\lambda_3} + 
\coset{\s{\alpha_2}\s{\alpha_4}\s{\alpha_3}\lambda_3} \leq \coset{\jac{G}{T}{\sigma_{2}^{3}}}.
$$ 
\item
Note that $\lambda_4 =\bk{\s{\alpha_1}\s{\alpha_2}\s{\alpha_1}} \lambda_2 =  [2,-1,0,-1]$. Hence an application of \eqref{Eq:A2} with respect to $M_{\alpha_3,\alpha_4}$ shows 
$$2 \times \coset{\lambda_4} + \coset{\s{\alpha_4}\lambda_4} \leq \jac{G}{T}{\sigma_{2}^3}.$$
\item
Hence $$\coset{\s{\alpha_1}\s{\alpha_4} \lambda_4} \leq \coset{\jac{G}{T}{\sigma_{2}^{3}}}.$$  

In summary,
\begin{equation}\label{local::p3::pi4}
\begin{split}
	2 \times \coset{\lambda_2} +
	2 \times \coset{\s{\alpha_2}\lambda_2}+
	2 \times \coset{\s{\alpha_3}\lambda_2}
	+\coset{\s{\alpha_1}\lambda_2}	
	+2 \times \coset{\s{\alpha_2}\s{\alpha_1}\lambda_2} & \\
	+\coset{\s{\alpha_1}\s{\alpha_2}\lambda_2} 
	+2 \times \coset{\s{\alpha_3}\s{\alpha_2}\lambda_2}
	+\coset{\s{\alpha_1}\s{\alpha_3}\lambda_2}
	+2 \times \coset{\s{\alpha_1}\s{\alpha_2}\s{\alpha_1}\lambda_2} &\\
	+\coset{\s{\alpha_1}\s{\alpha_3}\s{\alpha_2}\lambda_2}
	+2 \times \coset{\s{\alpha_4}\lambda_3} + \coset{\s{\alpha_3}\s{\alpha_4}\lambda_3}
	+
	\coset{\s{\alpha_3}\s{\alpha_4}\lambda_3} &\\ 
	+
	\coset{\s{\alpha_4}\s{\alpha_3}\s{\alpha_4}\lambda_3} + 
	\coset{\s{\alpha_4}\s{\alpha_3}\lambda_3} + 
	\coset{\s{\alpha_2}\s{\alpha_4}\s{\alpha_3}\lambda_3}
	+\coset{\s{\alpha_4} \lambda_4}
	\\
	+\coset{\s{\alpha_1}\s{\alpha_4} \lambda_4}
	&\leq \coset{\jac{G}{T}{\sigma_{2}^{3}}}. 
\end{split}
\end{equation}

In particular $\dim \jac{G}{T}{\sigma_{2}^{3}} \geq 25$.
\item
By our introduction, $\sigma_2^{3}$ is an irreducible constituent of $K_3$.  
Hence, there is an equality in \eqref{local::p3::pi4}. In other words
$$\left(K_3\right)_{s.s} = K_2 \oplus \sigma_{2}^{3}= \tau^3 \oplus \sigma_{1}^3 \oplus \sigma_{2}^3. $$
\vspace{-2cm}
\end{itemize} 
\end{proof}
\begin{Lem}\label{local::p3::k4}
One has
\begin{align*}
\dim \bk{\pi_1^3}^{\iwhaori} & = 42, & \dim \bk{\pi_2^3}^{\iwhaori} & = 19, &\dim \bk{\sigma_1^{3}}^{\iwhaori} & = 4, \\
\dim \bk{\sigma_2^3}^{\iwhaori} & = 25, & \dim \bk{\tau^3}^{\iwhaori} & = 6, &\dim \Pi_3^{\iwhaori} & = 96. \\
\end{align*} 
In particular, $K_4 \rmod K_3 \simeq \pi_2^3$ is irreducible.
\end{Lem}
\begin{proof}
It remains to show that $\dim \bk{\pi_2^{3}}^{\iwhaori} =19$; however, this follows from a direct computation. Note that $\dim (\Pi_3 \rmod K_3)^{\iwhaori} =  61$ and both $\pi_1^3,\pi_2^3$ are constituents of it. Thus, by dimension considerations, the claim follows. 
\end{proof}
With the above notations one has  
\begin{Prop}
The structure of $\Pi_3$ is as follows:
\begin{itemize}
\item
$\tau^{3}$ is the unique irreducible subrepresentation. 
\item
$\sigma_1^{3} \oplus \sigma_2^{3} \hookrightarrow \Pi_3 \rmod \tau^3$.
\item
$\pi_1^{3} \oplus \pi_2^{3}$ is the maximal semi-simple quotient of $\Pi_3$.
\end{itemize}
\end{Prop}
\begin{proof}
Note that it remains to show that $\sigma_1^{3} \oplus \sigma_2^{3} \hookrightarrow \Pi_3 \rmod \tau^3$. For this we proceed as follows:
\begin{itemize}
\item
Note that $\Sigma_{u_1} \simeq \Pi_3 \rmod K_1$. Thus it is enough to show that $\sigma_1^{3} \oplus \sigma_2^{3} \hookrightarrow \Sigma_{u_1}$.
\item
Observe that $\chi_1 =u_1 \lambda_0= [-1,-1,1,2]$. Hence $\Sigma_{u_1} \hookrightarrow \Ind_{T}^{G}(\chi_1)$. 
\item
Invoking  \Cref{Zampera::action} with the following parameters:
\begin{align*}
\chi &= \chi_1, & \alpha &=\alpha_3,  \\ m&=u_1\fun{3}, & u &=  w_3w_2w_3, & L=& u_1 \chi_{\para{P},z} = 
(z-z_0)m +\chi , 
\end{align*}
shows that $\Ind_{T}^{G}(\chi_1) = V_1 \oplus V_{-1}$. In that case, $V_1,V_{-1}$ are eigenspaces for the action of $E= \NN_{u}(\chi_1)\res{L}$.
\item
Transferring the problem to the Iwahori--Hecke algebra shows that $\Sigma_{u_1}$ shares a constituent with $V_1$ and also with $V_{-1}$. Thus, $\Sigma_{u_1}$ admits a maximal semi-simple subrepresentation of a length of at least $2$. By Frobenius reciprocity, each irreducible subrepresentation $\tau$ of $\Sigma_{u_1}$ admits $\coset{\chi_1} \leq \coset{\jac{G}{T}{\tau}}$.  However, only $\sigma_1^{3},\sigma_2^{3}$ have this property, and thus 
$\sigma_1^{3} \oplus \sigma_2^{3} \hookrightarrow \Pi_3\rmod K_1$. 
\end{itemize}
\end{proof}
\begin{Remark}\label{local::pi1 ::spherical}
With our notations, $\pi^{3}_{1}$ is spherical.
\end{Remark}
\subsection{Images of normalized operators}
\begin{Prop}\label{local::p3::images}
Let $w \in W(\para{P},G)$. Then 
$$(\Sigma_{w})_{s.s} = 
\begin{cases}
\Pi_3 & \dim (\Sigma_{w})^{\iwhaori} = 96\\ 
\sigma_1^{3} \oplus \sigma_2^{3} \oplus \pi_1^{3} \oplus \pi_2^{3} & \dim (\Sigma_{w})^{\iwhaori} = 90 \\  
\sigma_2^{3} \oplus \pi_1^{3} \oplus \pi_2^{3} & \dim (\Sigma_{w})^{\iwhaori} = 86 \\
 \pi_1^{3} \oplus \pi_2^{3} & \dim (\Sigma_{w})^{\iwhaori} = 61 \\
 \pi_1^{3}  & \dim (\Sigma_{w})^{\iwhaori} = 42 \\
\end{cases} 
$$
In addition, if $\dim \Sigma_{w}^{\iwhaori} = 61$, then $\Sigma_{w}\simeq \pi_1^{3} \oplus \pi_2^{3}$. 
\end{Prop}
\begin{proof}
The first part is the same as in \Cref{P4::Lemma::images} . For the second part, we argue as follows: By 
\Cref{local::p3::pi2 :ker}, one has 
$ \pi_1^{3} \oplus \pi_2^{3} \hookrightarrow \Sigma_{u}$, where 
$u =  \bk{w_{2}w_{3}} \bk{  w_{1}w_{2}w_{3}w_{4}w_{1}w_{2}w_{3}}$.
On the other hand, by dimension considerations, it follows that 
$$\Sigma_{u} \simeq \pi_1^{3} \oplus \pi_2^{3}.$$

For each $ \w \in W(\para{P},G)$ such that $\dim \Sigma_{w}^{\iwhaori} =61$, write $w =  s u$. Let 
$$E_{s}(z) =  \NN_{s}^{u\chi_{\para{P},z}}(u\chi_{\para{P},z})\res{\Sigma_{u,z}},$$
where $\Sigma_{u,z} = \Image \NN_{u}(z)\res{\pi_z}$. Following the same lines as in \Cref{local::scalar} , it follows that $E_{s}(z)$ is holomorphic at $z=z_0$. Since $\dim \Sigma_{u}^{\iwhaori} = \dim \Sigma_{w}^{\iwhaori}$, it follows that $E(z_0)$ is an isomorphism. Hence, $\pi_1^{3} \oplus \pi^{3}_{2} \hookrightarrow \Sigma_{w}.$
By dimension considerations, the claim follows.    
\end{proof}
\begin{Def}
Let $u,w \in W(\para{P},G)$. We say that $u\sim_{z_0} w$ if $u \chi_{\para{P},z_0}= w \chi_{\para{P},z_0}$.
We denote the equivalence classes by $[u]_{z_0}$. 
\end{Def}
\subsection{Action of stabilizer}\label{local::p3::stab}

In this subsection, we fix an equivalence class $[u_1]_{z_0}$, where $u_1$ is the shortest representative. For any $u_2 \in [u_1]_{z_0}$, we write 
$u_2 =  s\cdot  u_1$, where $s \in \Stab_{\weyl{G}}(u_1 \chi_{\para{P},z_0})$. Note that
\begin{equation}\label{local::p3::opertor::eq}
\NN_{u_2}^{\chi_{\para{P},z}}(\lambda)  = \NN^{u_1 \chi_{\para{P},z_0}}_{s}(u_1 \lambda) \circ \NN_{u_1}^{\chi_{\para{P},z}}(z).
\end{equation} 

Let $E_{s}(\lambda) =  \NN^{u_1 \chi_{\para{P},z}}_{s}(u_1 \lambda)\res{\Image \NN_{u_1}^{\chi_{\para{P},z}}}$. Then one has 

\begin{Prop}\label{Local::p3::same}
\begin{enumerate}
\item
The operator $E = E_{s}(\lambda)$ is holomorphic at $\lambda_0 =  \chi_{\para{P},z_0}$.
\item
Assume that $\Image \NN_{u_1}^{\chi_{\para{P},z}}(\lambda_0)  = \Image \NN_{u_2}^{\chi_{\para{P},z}}(\lambda_0) \simeq \pi_1^{3} \oplus \pi_2^{3}$. Then there exists $a \in \mathbb{C}$ such that  
$E \res{\pi_2^{3}} = a\id$. Moreover, 
$E \res{\pi_2^{3}} = a\id$ if and only if 
 $$\Image \left(-a \NN_{u_1}^{\chi_{\para{P},z}}(\lambda_0) +  \NN_{u_2}^{\chi_{\para{P},z}}(\lambda_0)\right)  \subseteq \pi_1^{3}.$$
\end{enumerate}
\end{Prop}
\begin{proof}
\begin{itemize}
\item
According to \Cref{local::p3:holo}, both $\NN_{u_1}^{\chi_{\para{P},z}}(\lambda), \NN_{u_2}^{\chi_{\para{P},z}}(\lambda)$ are holomorphic at $\lambda = \chi_{\para{P},z_0}$, and by \eqref{local::p3::opertor::eq} it follows that $E_{s}(\lambda)$ is holomorphic at $\lambda=\lambda_0$.
\item
By our assumption, $E \in \operatorname{End}_{G}(\pi_1^{3} \oplus \pi_2^{3})$. Thus by Schur's lemma, it follows that $E$ acts by a scalar on each summand. By \Cref{local::pi1 ::spherical}, it follows that $E\res{\pi_1^{3} }=\id$. Note that 
\begin{align*}
E\res{\pi_2^{3}} =  a\id & \iff \left(a \id -E\right)\res{\pi_2^{3}} =0 \\
& \iff \left(a\id -  E\right) \circ \NN_{u_1}^{\chi_{\para{P},z}}(\lambda_0) \subseteq \pi_1^{3} \\
& \iff\Image \bk{ a\NN_{u_1}^{\chi_{\para{P},z}}(\lambda_0) - \NN_{u_2}^{\chi_{\para{P},z}}(\lambda_0) } \subseteq \pi_1^{3}. 
\end{align*}
\end{itemize}
\end{proof}
Note that in order to find the scalar $a$ as in the last proposition, it suffices to find for which $a \in \mathbb{C}$ one has  
$\dim \Image \left(a\NN_{u_1}^{\chi_{\para{P},z}} (\lambda_0) +\NN_{u_2}^{\chi_{\para{P},z}} (\lambda_0) \right)\res{\Pi_3^{\iwhaori}} \in \set{0,42}$. These scalars are summarized in \Cref{Table:: P3 ::images}.

\subsection{Special case}
In this subsection, we  prove a key lemma, \Cref{local:p3::special::1}, that helps us in the global calculation.  \Cref{local:p3::special::1} concerns the case of $u_1 \sim_{z_0} u_2$, but while the spaces $\Sigma_{u_1},\Sigma_{u_2}$ are 
isomorphic, they are not equal as  subspaces of the appropriate principal series. Note that this case is not covered in \Cref{Local::p3::same}. There exists a unique equivalence class with this property.  
Specifically, consider the following equivalence class $[u]_{z_0}= \set{s_1,s_2,s_3}$, where  
\begin{align*} 
s_1&=w_{1}w_{2}w_{3}w_{4}w_{1}w_{2}w_{3} &s_2&=w_{1}w_{2}w_{3}w_{4}w_{3}w_{2}w_{3}w_{1}w_{2}w_{3} \\ 
s_3&=w_{3}w_{2}w_{3}w_{1}w_{2}w_{3}w_{4}w_{3}w_{2}w_{3}w_{1}w_{2}w_{3}  
\end{align*}
A direct calculation shows that
$$s_1 \chi_{\para{P},z_0} =s_2 \chi_{\para{P},z_0}= s_3 \chi_{\para{P},z_0} = -\fun{1} -\fun{2} +\fun{3} + \fun{4}.$$
We denote this weight by $\lambda$.    
Note that by  \Cref{Table:: P3 ::images}  and \Cref{local::p3::images},  one has $$\Sigma_{s_1} \simeq \Sigma_{s_2} \simeq \Sigma_{s_3} \simeq \pi_1^{3} \oplus \pi_2^{3}.$$

\begin{Lem} \label{local:p3::special::1}
There exists a vector $v \in \Pi_3$ that satisfies:
\begin{itemize}
\item
$ \NN_{s_2}(z_0) v$ (resp. $\NN_{s_3}(z_0) v$)  generates a copy of $\pi_{2}^{3}$ in $\Sigma_{s_2} $ (resp. $\Sigma_{s_3}$).
\item
$\set{ \NN_{s_2}(z_0)v, \NN_{s_3}(z_0)v}\subset \Ind_{T}^{G}(\lambda)$ is a linear independent set.
\end{itemize} 
\end{Lem}
\begin{proof}
\begin{itemize}
\item
Since $\Sigma_{s_2} \simeq \pi_{1}^{3}\oplus \pi_{2}^3$,
 there exists  a vector $v \in \Pi_3$ such that $\NN_{s_2}(z_0) v$ generates   $\pi_2^{3}$ in $\Sigma_{s_2}$.
\item
Write 
$\NN_{s_3}(z) = E(z) \circ \NN_{s_2}(z)$, where
$E(z)=\NN_{w_3w_2w_3}^{s_2 \chi_{\para{P},z}}(z)$. \Cref{local::p3 :: decompostion}  shows that at $z=z_0$ $E(z_0) \in \operatorname{End}_{G}\bk {\Ind_{T}^{G}(\lambda)}$. In addition, $E(z_0)$ is an isomorphism. In particular,  $\NN_{s_3}(z_0) v\neq 0$.
Thus, $\NN_{s_3}(z_0) v$ generates a copy of   $\pi_2^{3}$ in $\Sigma_{s_3}$. 
\item
Suppose that there exists a scalar $c \in \C$ such that $ \NN_{s_2}(z_0) v =  c \NN_{s_3}(z_0) v$. Then, $\NN_{s_2}(z_0) v$ and $\NN_{s_3}(z_0) v$ generate  the same copy of $\pi_{2}^{3}$ in $\Ind_{T}^{G}(\lambda)$.
\item
 Since $\mult{ \pi_1^{3}}{ \Ind_{T}^{G}(\lambda)} =1$, it  follows that $\Sigma_{s_2} = \Sigma_{s_3}$. In particular, $\Sigma_{s_2}^{\iwhaori} = \Sigma_{s_3}^{\iwhaori}$. However, a         
calculation in the Iwahori-Hecke algebra shows that 
$\Sigma_{s_2}^{\iwhaori} \neq \Sigma_{s_3}^{\iwhaori}$.
Thus, 
$\set{ \NN_{s_2}(z_0)v, \NN_{s_3}(z_0)v}$ is a linear independent set.
\end{itemize}
\end{proof}
\begin{Lem}\label{local:p3::special::2}
Let us denote $\NN_{s_{1},s_2,s_3}(z_0) = \NN_{s_1}(z_0) +\frac{1}{2} \NN_{s_2}(z_0) +  \frac{1}{2} \NN_{s_3}(z_0)$. Then
\begin{enumerate}
\item 
$$\Image \bk{\NN_{s_1,s_2,s_3}(z_0) \res{\Pi_3}} \simeq  \pi^{3}_{1}.$$
\item 
Let  $v \in \Pi_3$ as in \Cref{local:p3::special::1}; then 
$  \NN_{s_1,s_2,s_3}(z_0)v =0.$
\end{enumerate}
\end{Lem}
\begin{proof}
The first part is  a straightforward computation in the Iwahori-Hecke algebra.  In order to prove the second part, we recall that
\begin{equation} \label{eq:: 1}
\Image \bk{\NN_{s_1}(z_0)\res{\Pi_3}} \simeq \Image\bk{ \NN_{s_2}(z_0)\res{\Pi_3}} \simeq \Image \bk{\NN_{s_3}(z_0)\res{\Pi_3}} \simeq \pi_1^{3} \oplus \pi_2^{3}.
\end{equation} 

We continue  by showing:
\begin{eqnarray}\label{local::p3::same}
\Image  \bk{\NN_{s_1}(z_0)\res{\Pi_3}} =  \Image \bk{\bk{ \NN_{s_2}(z_0) +  \NN_{s_3}(z_0)  }\res{\Pi_3}}.
\end{eqnarray} 
This can be proved by performing  computations in the Iwahori-Hecke algebra. However, we present here a more conceptual proof.

Recall that $\pi_2^{3} \oplus \pi_{2}^{3} \hookrightarrow \Ind_{T}^{G}(\lambda)$.
\eqref{eq:: 1} implies that
$$  \Image \bk{\bk{ \NN_{s_2}(z_0) + \NN_{s_3}(z_0)   }\res{\Pi_3}} \subset \pi_1^{3} \oplus \pi_2^{3} \oplus \pi_2^{3}.$$
Here we use the fact that $\pi_1^{3}$ appears with multiplicity one in $\Ind_{T}^{G}(\lambda)$.

Let $E^{1}(z) =  \NN_{w_3w_2w_3}^{s_3 \chi_{\para{P},z}}(z)$. \Cref{local::p3 :: decompostion}  shows that $E^1(z)$ is holomorphic at $z=z_0$. In addition, $E^{1}(z_0)$ is diagonalizable  with two eigenvalues $\pm1$. 
Thus, 
$$\Ind_{T}^{G}(\lambda) = V_1 \oplus V_{-1}.$$
By Jacquet module considerations, $\pi_2^{3}$ appears with multiplicity two in $\Ind_{T}^{G}(\lambda)$. Furthermore, \Cref{local::p3 ::pi2 ::decomp} and \Cref{local::p3::pi2::vm1} yield that $\pi_2^{3}$ admits an embedding  in $V_1$ and  in $V_{-1}$. Thus, $\mult{\pi_{2}^{3}}{V_{\pm1}}=1$.  \Cref{local::p3 ::pi2 ::decomp}  also  implies that
$$\pi_1^{3} \oplus \pi_2^{3} \simeq  \Image \bk{\NN_{s_1}(z_0)\res{\Pi_3}} \subset V_1.$$

 Thus, in order to prove \eqref{local::p3::same}, it suffices to show that:
 \begin{enumerate}
 \item
 $\Image \left(\ \NN_{s_2}(z_0) + \NN_{s_3}(z_0) \right) \subset V_1$.
 Since $\pi_2^{3}$ is contained in $V_1$ with multiplicity one, this shows that as an abstract representation 
 $\Image  \bk{ \bk{  \NN_{s_2}(z_0) + \NN_{s_3}(z_0)}\res{\Pi_3}}$  is contained in $\pi_1^{3}\oplus \pi_2^{3}$. 
 
 \item
 $\pi_1^{3},\pi_2^3$ are constituents of $ \Image \bk{\bk{ \NN_{s_2}(z_0) +  \NN_{s_3}(z_0)  }\res{\Pi_3}}.$  
 \end{enumerate} 

Since $\NN_{s_2}(z_0) =  E^{1} \circ \NN_{s_3}(z_0)$, one can see that 
$$ \Image \left( \NN_{s_2}(z_0) +  \NN_{s_3}(z_0)\right) = \Image \left (\id + E^{1}(z_0) \right) \circ  \NN_{s_3}(z_0) \subset V_1.$$   

Let $f^{0} \in \Pi_3$ be the normalized spherical vector in $\Pi_3$.  Since $0 \neq \left(\NN_{s_2}(z_0) + \NN_{s_3} (z_0) \right) f^{0}$ and $\pi_1^{3}$ is spherical, it follows that $\pi_1^{3}$ is a constituent of $$\Image \bk{\bk{ \NN_{s_2}(z_0) +  \NN_{s_3}(z_0)  }\res{\Pi_3}}.$$

By \Cref{local:p3::special::1},  there is a $v \in \Pi_3$ such that  $0 \neq  \left(\NN_{s_2}(z_0) + \NN_{s_3}(z_0) \right) v$  and generates a copy of $\pi_2^{3}$ inside $\Ind_{T}^{G}(\lambda)$. In particular, $\pi_2^{3}$ is a constituent of 
$$ \Image \bk{ \bk{ \NN_{s_2}(z_0) +\NN_{s_3}(z_0) } \res{\Pi_3}}.$$
This ends the proof of \eqref{local::p3::same}.

Write  $s_2 =  ws_1, s_3 =  w's_1$ for $w,w' \in \weyl{G}$. Then one has 
$$ \bk{\frac{1}{2} \bk{\NN_{s_2}(z) + \NN_{s_3}(z) }} =  \frac{1}{2} \left(\NN_{w}^{s_1 \chi_{\para{P},z}}(z)+ \NN_{w'}^{s_1 \chi_{\para{P},z}}(z)  \right) \circ \NN_{s_1}(z).$$
Let $E(z) =  \frac{1}{2} \left(\NN_{w}^{s_1 \chi_{\para{P},z}}(z)+ \NN_{w'}^{s_1 \chi_{\para{P},z}}(z)  \right) \res{ \Image \NN_{s_1}(z)}$.  Since $\NN_{s_1}(z) , \NN_{s_2}(z), \NN_{s_3}(z)$ are holomorphic at $z=z_0$, it follows that $E(z)$ is holomorphic at $z=z_0$. In addition, \eqref{local::p3::same} yields that $E = E(z_0) \in \operatorname{End}_{G}(\pi_1^{3} \oplus \pi_2^{3})$.  
Note that $E$ is an isomorphism, otherwise  
    $$\Image  \NN_{s_1}(z_0)\res{\Pi_3} \neq   \Image \bk{\frac{1}{2} \bk{\NN_{s_2}(z_0) + \NN_{s_3}(z_0) }  }\res{\Pi_3} \simeq \pi_1^{3} \oplus \pi_2^{3}.$$

Given $v$ as in \Cref{local:p3::special::1}, one has $\frac{1}{2} \bk{ \NN_{s_2}(z_0) + \NN_{s_3}(z_0) }  v  \neq 0$. 
Note that 
$$0 \neq \frac{1}{2} \bk{ \NN_{s_2}(z_0) + \NN_{s_3}(z_0) } v =  E \circ \NN_{s_1}(z_0) v $$ 
and generates a copy of $\pi_2^{3}$.
Since $E$ is an isomorphism and $G$--equivariant,   it follows that $0 \neq \NN_{s_1}(z_0) v$ and generates the same copy of $\pi_2^{3}$.  Thus, if  
  $\NN_{s_1,s_2,s_3}(z_0)v\neq 0$ it generates a copy  of $\pi_2^{3}$. But $\Image \bk{\NN_{s_1,s_2,s_3}(z_0) \res{\Pi_3}} \simeq  \pi^{3}_{1}$. Hence $$\NN_{s_1,s_2,s_3}(z_0)v=0.$$
\end{proof}

\subsection{ Derivative  of normalized  operators} \label{local::p3::der}
We fix the following notations for this section. 
Let $$\lambda_1 =-\fun{1} -\fun{2} + 2 \fun{3} -\fun{4}, \quad  \lambda_2=-\fun{1} +\fun{2} - 2 \fun{3} +\fun{4}.$$ 
Let $u \in W(\para{P},G)$ be an element such that $u \chi_{\para{P},z_0} \in \set{\lambda_1,\lambda_2}$.

From \Cref{Table:: P3 ::images}, there exist  $u_1 ,\dots ,u_5 \in W(\para{P},G)$ such that $u_i \chi_{\para{P},z_0} =\lambda_1$, and $u_1' ,\dots ,u_5' \in W(\para{P},G)$ such that $u_i^{'} \chi_{\para{P},z_0} =\lambda_2$.  In both cases the treatment is identical. 
We enumerate these elements as $u_1 ,\dots , u_5$, where    
$$l(u_1) < l(u_2) =  l(u_3) < l(u_4) <  l(u_5) .$$

For each $u \in \set{u_1,\dots, u_5}$, we write the Laurent expansion of $\NN_{u}(z)$ around $z=z_0$. Namely, $$\NN_{u}(z) = \sum_{j=0}^{\infty} (z-z_0)^{j} D_{u}^{(j)},$$
where $D_{u}^{j} \in \Hom\bk{\Pi_3,  \Ind_{T}^{G}(u \chi_{\para{P},z_0}) }$.  Note that $D_{u}^{j} $ is not necessarily $G$--equivariant. 

In addition, let $K_4 \subset \Pi_3$ such that $\bk{K_{4}}_{s.s} = \tau^{3} \oplus \sigma_{1} ^{3} \oplus \sigma_2^{3} \oplus \pi_{2}^{3}$. 

The following Lemma will play a key ingredient in our global calculations. The coefficients in this Lemma arise from global considerations.   
\begin{Lem}
Keeping the above notations, one has 
$$\bk{ D_{u_1}^{(1)} -\frac{1}{2}D_{u_2}^{(1)} -\frac{1}{2}D_{u_3}^{(1)} +  \frac{1}{6} D_{u_4}^{(1)} -\frac{1}{6} D_{u_5}^{(1)} } \res{K_4} =0.$$  
\end{Lem}  
\begin{proof}
In order to prove this we proceed as follows:
\begin{itemize}
\item
We start by finding a basis for $K_4^{\iwhaori}$. For this purpose, 
let  $w \in W(\para{P},G)$ be the shortest Weyl element such that 
$w \chi_{\para{P},z_0}$ is anti-dominant. By \Cref{Table:: P3 ::images},  the image of $\NN_{w}(z_0)$ is $\pi_1^{3}$. Thus, its kernel is $K_4$.
Let $n_{w}(z_0)$ be the element that encodes the action of $\NN_{w}(z_0)$ in the Iwahori--Hecke algebra, as in \Cref{subsection::itertwining}. Write the matrix representation of $n_{w}(z_0)$ with respect to $\set{T_w \: : \:  w \in \weyl{G}}$ and a basis of $\Hecke_{\para{P}}(z_0)$ as in \Cref{subsection:: hPmoduule}. The kernel of $n_{w}(z_0)$ is $K_{4}^{\iwhaori}$.  
This gives us a   basis $B_{K_4} = \set{v_1 ,\dots, v_l}$ of $K_{4}^{\iwhaori}$.   
\item
We continue by writing the Laurent expansion of $n_{u_i}(z)$ around $z=z_0$. This can be done in the following way: Recall that for a simple reflection $w_{\alpha}$, one has 
$$ n_{w_{\alpha}}(\lambda) = \frac{1-q}{1- q^{z+1}} T_{e} + \frac{1-q^{z}}{1- q^{z+1}}T_{w_{\alpha}}, \text{ where } z=\inner{\lambda,\check{\alpha}}.$$ 
For each summand, we write the Laurent expansion of
\begin{align*}
\frac{1}{1- q^{z+1}} &=  \sum_{i=-1}^{\infty} a_i(z)(z-z_0)^{i},
&{1- q^{z}} &=  \sum_{j=0}^{\infty} b_j(z)(z-z_0)^{j}.
\end{align*}
Thus, around $z=z_0$, one has 
\begin{align*}
n_{w_{\alpha}}(z) &= 
\bk{ (1-q) \sum_{i=-1}^{\infty} a_{i}(z) (z-z_0)^{i+j}} T_{e} \\&+    
\bk{ \sum_{i=-1,j=0}^{\infty} a_{i}(z)b_{j}(z) (z-z_0)^{i+j}}T_{w_{\alpha_i}}.
\end{align*}

Multiplying all the terms, one gets around $z=z_0$, 
$$n_{u_i}(z) =  \sum_{j=0}^{\infty} D_{u_i}^{(j)}(z-z_0)^{j},$$ 
where $D_{u_i}^{j} \in \Span\set{T_{w} \: : \:  w \in \weyl{G}}$ and the coefficients are rational functions in $q,\log q$.     
\item
Recall that the action of the operator $\NN_{u}(z)$ is interpreted as right multiplication by $n_{u}(z)$ in the realm of Iwahori-Hecke algebra.  
Thus, the claim follows once we check that for each $v \in B_{K_{4}}$ one has 
$$ 
v \bk{ D_{u_1}^{(1)} -\frac{1}{2}D_{u_2}^{(1)} -\frac{1}{2}D_{u_3}^{(1)} +  \frac{1}{6} D_{u_4}^{(1)} -\frac{1}{6} D_{u_5}^{(1)} }=0. $$

\end{itemize}
\vspace{-1cm} \qedhere
\end{proof}

\chapter{  Global results } \label{chapter::global}
In this chapter, we state and prove our main result.  We start by recalling our notations. 

\section{Notations - recollection}\label{global::res::nota}
For the entire chapter, $F$ denotes a number field of characteristic zero.  For every place $\nu$, we let $F_\nu$ stand for the completion of $F$ with respect to $\nu$. If $\nu$ is a finite place, then $\mathcal{O}_{\nu}$ denotes its ring of integers. Set $\A$ to be  the ring of adeles of $F$.

Here, $\G$ is a simply-connected, split group defined over $F$ of type $F_4$. We also fix a maximal split torus $\T$ and a Borel subgroup $\B =  \T \U$ of $\G$, where $\U$ is the unipotent radical of $\B$. The Weyl group of $\G$ is denoted by $\weyl{G}= \weyl{\G}$.

We let $G(F), G(F_{\nu}) , G(\A)$ be the groups of $F, F_{\nu},\A$-- points  of $\G$.  For every place $\nu$, we denote by $K_{\nu} \subset G(F_{\nu})$ a maximal compact subgroup of $G(F_{\nu})$. If $\nu$ is a finite place, then $K_{\nu} =  G(\mathcal{O}_{\nu})$.
 
Let  $L^2_{deg}(G) \subset L^{2}_{[\T,\Id],disc}(G)$
be the space of square--integrable automorphic functions
on $G$,  obtained as leading terms of the Laurent expansion
of various degenerate Eisenstein series $E_{\para{P}}(z)$  associated
with maximal parabolic subgroups at real points $z>0$.
\begin{Def}\label{defn::l2k}
Define $L^2_{deg}(G,K_\infty) \subset L_{deg}^2(G)$ to be the subspace generated by $K_\infty$-fixed vectors. 
\end{Def}

The main goal of this part is to describe the irreducible components of  
$L^{2}_{deg}(G,K_{\infty})$.
Clearly, such an irreducible component  $\Pi$ is a quotient of a degenerate principal series $\Ind_{M(\A)}^{G(\A)}(z)$. Therefore, as an intermediate step, one should understand the quotients of $\Ind_{M(F_{\nu})}^{G(F_{\nu})}(z)$ for every place $\nu$. 
If $\nu$ is a finite place, this information is available by \cite{F4}.
Specfically, for a finite place  $\nu$, the representation
$\Ind_{M_i(F_{\nu})}^{G( F_{\nu})}(z_0)$ admits a unique irreducible quotient,
unless $(\para{P}_i,z_0) \in
\set{ (\para{P}_1,1), (\para{P}_3 , \frac{1}{2}) , (\para{P}_{4},\frac{5}{2})}$.
In the latter case,  it admits a  maximal semi-simple quotient of length two
$\pi^{i}_{1, \nu} \oplus  \pi^{i}_{2,\nu}$ for $i \in \set{1,3,4}$, where $\pi^{i}_{1, \nu}$ is spherical.  For Archimedean places, we let $\pi^i_{1,\nu}$ be the unique irreducible  spherical quotient of $\Ind_{M_i(F_{\nu})}^{G(F_{\nu})}(z_0)$ also.  
Given a finite set of finite places $\mathcal{S}$, we let 
$$\pi^{\mathcal{S},i} =
\otimes_{\nu \in \mathcal{S}} \pi_{2,\nu}^{i} \otimes_{\nu \not \in \mathcal{S}} \pi^{i}_{1,\nu}.$$

With these notations, we are able to state our main result. 
\section{Main result}
\begin{Thm}\label{into::global::main}
  Let $G$ be as above. Then 
$$ L^{2}_{deg}(G,K_{\infty})= Triv_{F_4} \oplus \Sigma_1 \oplus \Sigma_3 \oplus  \Sigma_4,$$
where 
\begin{itemize}
\item
$Triv_{F_4}$  is the
automorphic realization of the trivial representation in the space
of (square-integrable) automorphic forms.
\item
$\Sigma_{1} = \oplus_{|\mathcal{S}| \text{ is even }}  \pi^{\mathcal{S},1}$.
\item
$\Sigma_{4} = \oplus_{|\mathcal{S}| \text{ is even }}  \pi^{\mathcal{S},4}$.
\item
$\Sigma_{3} = \oplus_{|\mathcal{S}| \neq 1}  \pi^{\mathcal{S},3}$.
\end{itemize}
\end{Thm}     
\subsubsection{Sketch of the proof}
\begin{itemize}
\item
In \Cref{subsection:algo}, we recall how to  determine the set of poles of degenerate Eisenstein series.
\item
\Cref{global::rearange}  explains how to determine the order of each pole.
\item
In \Cref{global::firstreuce}, we show that 
$L^{2}_{deg}(G,K_{\infty}) \subseteq  \bigoplus\limits_{(\para{P},z_0) \in \mathds{P}} \Image \leadingterm_{\para{P}}(z_0)$,
where  $\leadingterm_{\para{P}}(z_0)$ is the leading term of $E_{\para{P}}(z)$ at $z=z_0$ and   
\begin{align*}
\mathds{P} = \set{ (\para{P}_1,1), (\para{P}_3 , \frac{1}{2}) , (\para{P}_{4},\frac{5}{2}), (\para{P}_2,\frac{5}{2}) }.
\end{align*}

\item
In  \Cref{global::triv::rep}, we show that $Triv_{F_4} =\Image \leadingterm_{\para{P}_2}\bk{\frac{5}{2}} \subseteq L^{2}_{deg}(G,K_{\infty}).$ 
\item
\Cref{global::squre} shows that $\bigoplus\limits_{(\para{P},z_0)} \Image \leadingterm_{\para{P}}(z_0) \subset L_{deg}^{2}(G, K_{\infty})$, where 
$$(\para{P},z_0) \in \set{ (\para{P}_1,1), (\para{P}_3 , \frac{1}{2}) , (\para{P}_{4},\frac{5}{2})}.$$

 This shows that 
$$L^{2}_{deg}(G,K_{\infty}) =  \bigoplus_{(\para{P},z_0) \in \mathds{P}} \Image \leadingterm_{\para{P}}(z_0).$$
\item 
\Cref{global::p1p4 image}  shows that 
$\Sigma_{1} =\Image \leadingterm_{\para{P}_1}\bk{1}$ and 
$\Sigma_{4} =\Image \leadingterm_{\para{P}_4}\bk{\frac{5}{2}}$.
\item   
In \Cref{global::p3}, we show that
$\Sigma_{3} =\Image \leadingterm_{\para{P}_3}\bk{\frac{1}{2}}$.
\end{itemize}
\subsection{The poles of $E_{\para{P}_i}(\cdot)$} \label{subsection:algo}

We keep the notations of \Cref{chapter::preimalies}.

\begin{Prop}
Set $\mathcal{P}_i =
\set{  z \in \R \: : \: \exists \gamma \in \Phi_{\G}^{+} \setminus \Phi_{\m_i}^{+} \text{ such that } \inner{\chi_{\para{P},z},\check{\gamma}} \in \set{1,0-1}}$.
Let $f = \otimes_{\nu}f_{\nu}$ be a flat section such that at Archimedean places, $f_{\nu} =  f_{\nu}^{0}$. Assume that $0< z_0 \in \R$  and that $E_{\para{P}_i}(f,z,g)$ has  a pole at $z=z_0$;
then  $z_0\in \mathcal{P}_i$.
\end{Prop}
\begin{proof}
By \Cref{global::eisention::behavior}, the Eisenstein series $E_{\para{P}_i}(f,z,g)$ and
its constant term $E_{\para{P}_i}^{0}(f,z,g)$ share
the same analytic properties. By \eqref{global::conatnttrem_a} and \eqref{global::inter_b}
\begin{align*}
E_{\para{P}_i}^{0}(f,z,g) &=  \sum_{w \in W(\para{P}_i,G)}  \M_{w}(z)f(g) \\
& =\sum_{w \in W(\para{P}_i,G)}C_{w}(z) \bk{\otimes_{\nu  \in \mathcal{S}} \NN_{w,\nu}(z)f_{z,\nu}\otimes_{\nu \not \in \mathcal{S}} f^{0}_{w.z,\nu}}.   
\end{align*}
Thus, 
\begin{align*}
  \set{ \text { Poles of }  E_{\para{P_i}}(\cdot, z, \cdot )} &=
  \set{ \text { Poles of }  E_{\para{P_i}}^{0}(\cdot, z,\cdot )}
\\
&\subset \cup_{w \in W(\para{P}_i,G)} \bk{ 
\set{ \text { Poles of }  C_w(z)  } \cup \set{ \text { Poles of }  \NN_{w,\nu} (z) }} 
\end{align*}
By abuse of notation, let  $C_{w}(z)$
denote the reduced form of $C_{w}(z)$. Namely, write  
$C_{w}(z) = \prod_{i=1}^{k} \frac{\zeta(a_i z + b_i)}{\zeta(c_i z + d_i)}$
for some $a_i,b_i,c_i,d_i \in \C$,
where the numerator and denominator do not have common terms.

A direct computation (see \cite[p. 26]{HeziMScThesis}) shows that in the region $\Real(z) >0$,
$\Real(c_iz +d_i) > 1$ holds for every $i$. 

In particular, the poles of $C_{w}(z)$ come from the poles of its numerator. In other words, suppose that $C_{w}(z)$ admits a pole at $z=z_0$. Then, there exists an $i$ such that $a_i z_0  +  b_i \in \set{0,1}$.

Recall that 
$$ C_{w}(z) = \prod_{\gamma \in R(w)}  \frac{ \zeta\bk{ \inner{\chi_{\para{P}_i,z}, \check{\gamma}}}}{ \zeta\bk{ \inner{\chi_{\para{P}_i,z}, \check{\gamma}}+ 1}}\:\: ,$$
where $\zeta(z)$ is the completed zeta function of $F$, and 
$R(\w) =  \set{\alpha > 0 \: : \: w\alpha <0}$.

Thus, by definition, there exists $\gamma  \in R(w) $ such that
$ \inner{ \chi_{\para{P}_i,z} ,\check{\gamma}} =a_i z +  b_i$. 

In addition, since  $\chi_{\para{P}_i,z} =  z \cdot \fun{_{i}} -  \rho_{\T}^{\MM_i}$, one has 
\begin{align*}
a_i &=  \inner{ \fun{i}, \check{\gamma}},
&
b_i &=  -\inner{ \rho_{\T}^{\MM_i} , \check{\gamma}}. 
\end{align*}

Hence $a_i,b_i \in \R$. Thus, if $a_iz_0 + b_i \in \set{0,1}$, it follows that $z_0 \in \R$ . 

In conclusion, in the region $\Real(z)>0$, one has for every $w \in W(\para{P}_{i},G)$ 
\begin{align*}
\set{ \text { Poles of }  C_w(z)  } \subset
 \set{  z \in \R \: : \: \exists \gamma \in \Phi_{\G}^{+} \setminus \Phi_{\m_i}^{+} \text{ such that } \inner{\chi_{\para{P}_i,z},\check{\gamma}} \in \set{0,1}}.
\end{align*}

Since $f_{z,\nu} =  f^{0}_{z,\nu}$ for every infinite place $\nu$,
for a finite place $\nu$ and a simple reflection $\w_{\alpha}$, if the operator $\NN_{w_{\alpha},\nu}(\lambda)$ is singular at $\lambda$, then $\inner{\lambda, \check{\alpha}} =-1$  
\color{black}{(\Cref{Nor::holo})}.
\color{black}
 Thus, by the same argument as above, one has 
for every $w \in W(\para{P}_{i},G)$ and finite place $\nu$ 
\begin{align*}
\set{ \text { Poles of }  \NN_{w,\nu}(z)  } \subseteq
 \set{  z \in \R \: : \: \exists \gamma \in \Phi_{\G}^{+} \setminus \Phi_{\m_i}^{+} \text{ such that } \inner{\chi_{\para{P},z},\check{\gamma}}=-1}.
\end{align*}
In conclusion, if $z_0>0$  and for a $K_{\infty}$-- fixed section $f$,  the Eisenstein series $E_{\para{P}_i}(f,z,g)$ might have  a pole at $z=z_0$ only for $z_0\in \mathcal{P}_i$.
\end{proof}

\subsection{ Rearranging $E_{\para{P}_i}^{0}(\cdot)$} \label{global::rearange}
There is a possibility of cancellation among the terms of
\eqref{global::conatnttrem_a}. We will regroup the terms
according to their homogeneous properties under the action by $T$ on the left. This naturally leads to the following equivalence relation on
$W(\para{P}_i,G)$.
\begin{Def} \label{global::def:: eq}
Let $u_1,u_2 \in W(\para{P}_i,G)$. We say that $u_1 \sim_{z_0} u_2$, if $u_1 \chi_{\para{P}_i,z_0} = u_2 \chi_{\para{P}_i,z_0}$. In that case, we put 
$\M_{u}^{\#}(z) =  \sum_{w \in [u]_{z_0}}\M_{w}(z)$. 
\end{Def}

\begin{Lem}\label{global:: ord::eq}
Let $k =\max\set{ \ord\limits_{z=z_0}  \M_{u}^{\#}(z) \: : \:  u \in W(\para{P},G) \rmod \sim_{z_0} }$. Denote by   
$$A(k) = \set{ u \in W(\para{P},G) \: : \: \ord\limits_{z=z_0}\M_{u}^{\#}(z)=k}.$$ Then  
\begin{enumerate}
\item
$\ord\limits_{z=z_0}  E_{\para{P}}(z) = k$.
\item
$$\ker (z-z_0)^{k} E_{\para{P}}(z)\res{z=z_0} =  \bigcap_{u \in W(\para{P},G)\rmod \sim z_0} \ker(z-z_0)^{k}  \M_{u}^{\#}(z).$$
\item 
If for every $u \in A(k)$, one has 
$ u \chi_{\para{P},z_0} =  \sum_{\alpha \in \Delta_{G}}  n_{\alpha} \alpha$ where $\Real\bk{ n_{\alpha}} <0$, then 
$\Image (z-z_0)^{k} E_{\para{P}}(z)\res{z=z_0}$ is contained in $L^{2}_{\disc}(G)$.
\end{enumerate}
\end{Lem}
\begin{proof}
Recall that by \Cref{global::eisention::behavior}, the degenerate Eisenstein series and its constant term share the same analytical behavior.  In that case, for every flat section $f$, one has around $z=z_0$
$$ E^{0}_{\para{P}}(f,z,g) = \sum_{w \in W(\para{P},G)} \M_{w}(z)f =
\sum_{u \in W(\para{P},G)\rmod \sim z_0} \M^{\#}_{u}(z)  f.$$
Thus,  $\ord_{z=z_0} E_{\para{P}}(z) \leq k$. 

For any flat section $f$, we let $A(k,f) = \set{u \in W(\para{P},G)\rmod \sim_{z_0} \: : \: \ord_{z=z_0} \M^{\#}_{u}(z)f =  k}$. Then 
$\set{(z-z_0)^{k}\M_{u}^{\#}(z)\res{z=z_0}f }_{u \in A(k,f)}$ is linearly independent. In particular $$\ord_{z=z_0} E_{\para{P}}(z) =k.$$ 

In addition, one has  
$$\ker (z-z_0)^{k} E_{\para{P}}(z)\res{z=z_0} =  \bigcap_{u \in W(\para{P},G)\rmod \sim z_0} \ker(z-z_0)^{k}  \M_{u}^{\#}(z).$$ 
  
The third part is Langlands criterion for square integrability (see \cite[Lemma I.4.11]{EisensteinBook}).
\end{proof}
\color{black}
\begin{Remark}
  Let us specify the condition above in our case.
Let $\lambda = \sum_{i=1}^{4} s_i \fun{i}$.
Recall that 
\begin{align*}
\fun{1} &=  2\alpha_{1}+3\alpha_{2}+ 4 \alpha_{3} + 2 \alpha_{4}, &
\fun{2} &=  3\alpha_{1}+6\alpha_{2}+ 8 \alpha_{3} + 4 \alpha_{4},\\
 \fun{3} &=  2\alpha_{1}+4\alpha_{2}+ 6 \alpha_{3} + 3 \alpha_{4}, & 
\fun{4} &=  \alpha_{1}+2\alpha_{2}+ 3 \alpha_{3} + 4 \alpha_{4}.   
\end{align*}

  Then
$\lambda=\sum_{\alpha\in \Delta_{\G}} a_{\alpha} \alpha,$ where 
\begin{align*}
a_1 &=  2s_1 +3s_2 +  2s_3 +  s_4, &  a_2 &= 3s_1 +6s_2 +  4s_3 +  2s_4,
\\
a_3 &=4s_1 +8s_2 +  6s_3 +  3s_4,  & a_4 &=2s_1 +4s_2 +  3s_3 +  2s_4.   
\end{align*}
In particular, $\Real(a_1),\Real(a_2),\Real(a_3),\Real(a_4)<0$ if and only if 
\begin{align}  
\Real( 2s_1 +3s_2 +  2s_3 +  s_4 ) &<0, \nonumber& 
\Real(3s_1 +6s_2 +  4s_3 +  2s_4    ) &<0,  \nonumber\\
\Real(4s_1 +8s_2 +  6s_3 +  3s_4 ) &<0, &
\Real(2s_1 +4s_2 +  3s_3 +  2s_4  ) &<0.  \label{square::cond}
\end{align}  
\end{Remark}
\color{black}
\subsection{ Dealing with the square integrable case}

\begin{Prop}
  Let $(\para{P},z_0)$ be a pair such that
  $\Image \leadingterm_{\para{P}}(z_0)$ is contained in $L^{2}_{disc}(G)$.
Assume that $\Ind_{M(\A)}^{G(\A)}(z_0)$ admits a maximal semi-simple quotient $\oplus_{j} \Pi_j$,
and denote the quotient map $i_j: \Ind_{M(\A)}^{G(\A)}(z_0)\rightarrow \Pi_j$. 
Then $\Pi_j$ is contained in  $\Image \leadingterm_{\para{P}}(z_0)$
if and only if
 there is a flat section $f$ and an equivalence class $[u]_{z_0}$ such that
\begin{enumerate}
\item
$\iota_j(f)\neq 0$.
\item
$ \lim_{z \to z_0} (z-z_0)^{d_{\para{P}}(z_0)}\M_{u}^{\#}(z) f \neq 0$ , where $d_{\para{P}}(z_0) =\ord_{z=z_0}E_{\para{P}}(z)$.
\end{enumerate}
\end{Prop}

\begin{proof}
Recall that $\Image \leadingterm_{\para{P}}(z_0)$ is a quotient of $\Ind_{M_{i}(\A)}^{G(\A)}(z_0)$. Since  $\Image \leadingterm_{\para{P}}(z_0)$ is square integrable, the image
  $\Image \leadingterm_{\para{P}}(z_0)$ is semi-simple. In particular,  
  $\Image \leadingterm_{\para{P}}(z_0) \subseteq \oplus_{j} \Pi_{j}.$
 
Assume that $\Pi_j \subseteq \Image \leadingterm_{\para{P}}(z_0)$. By definition, there is a flat section $f$ such that 
$$0 \neq \leadingterm_{\para{P}}(z_0)f \in \Pi_{j}.$$
In particular, $\iota_{j}(f) \neq 0$. The existence of an equivalence class is a direct consequence of  \Cref{global:: ord::eq}.     

\end{proof}

\color{black}
The following Proposition is key tool in order to prove our main result. 
\color{black}
\begin{Prop}[First Reduction] \label{global::firstreuce}
$$L^{2}_{deg}(G,K_{\infty}) \subseteq  \bigoplus_{ (\para{P},z_0)} \Image \leadingterm_{\para{P}}(z_0), $$
where 
$(\para{P},z_0) \in \set{(\para{P}_1 ,1) ,(\para{P}_3,\frac{1}{2}), (\para{P}_4,\frac{5}{2}) ,(\para{P}_2, \frac{5}{2}) }$
 and $L^{2}_{deg}(G,K_{\infty})$ is defined in \Cref{defn::l2k}.
\end{Prop}

\begin{proof}

Given an irreducible representation $\Pi = \otimes \pi_{\nu} \subset L^{2}_{deg}(G,K_{\infty})$, we show that $$\Pi \subset   \bigoplus_{ (\para{P},z_0)} \Image \leadingterm_{\para{P}}(z_0).$$ 
\begin{itemize}
\item 
By definition, $\Pi$ is a quotient of $\Ind_{M_{i}(\A)}^{G(\A)}(z_0)$ for some standard Levi subgroup and a positive  point $z_0$. Namely, for every place $\nu$, $\pi_{\nu}$ is a quotient of $\Pi_{i,z_0,\nu}:=\Ind_{M_{i}(F_{\nu})}^{G(F_{\nu})}(z_0)$. 
\item
\color{black}
By \cite{F4}, for every finite place $\nu$ 
and  $(\para{P},z_0) \not \in \set{ (\para{P}_1 ,1) ,(\para{P}_3,\frac{1}{2}), (\para{P}_4,\frac{5}{2})}  $, the representation
$\Pi_{i,z_0,\nu}$
  admits a unique irreducible quotient, which is a spherical
  constituent of $\Pi_{i,z_0,\nu}$.
  Therefore it is generated by the normalized spherical vector. 
Thus, every irreducible $K_{\infty}$- fixed  quotient of $\Ind_{M(\A)} ^{G(\A)}(z_0)$   is generated by the image of the normalized spherical vector.
 As a result, for a $K_{\infty}$--fixed section, one has  
 $$\ord_{z=z_0}  E_{\para{P}}(f,z,g) =
  \ord_{z=z_0}  E_{\para{P}}(f^{0},z,g).$$
\color{black}
\item
 \cite[Subsection 3.3.1]{HeziMScThesis} gives explicit criteria for the square-integrability of   $E_{\para{P}}(f^{0},g,z)$  at  $\Real(z)>0$.   Since we are interested only in the square-integrability cases, we  restrict ourselves only to these points.
\item
 In \cite[Section 6.4]{HeziMScThesis} we applied the  Siegel-Weil identities for $E_{\para{P}}(f^0,z,g)$ and showed that these images can be realized as leading terms  of $E_{\para{P}}(z)$, where  
  $$(\para{P},z_0)  \in \set{ (\para{P}_1 ,1) ,(\para{P}_3,\frac{1}{2}), (\para{P}_4,\frac{5}{2}) ,(\para{P}_2,\frac{5}{2})}.$$
\item
As a result, one can reduce the problem to  the study of the image of the leading terms of
$$(\para{P},z_0)  \in \set{ (\para{P}_1 ,1) ,(\para{P}_3,\frac{1}{2}), (\para{P}_4,\frac{5}{2}) ,(\para{P}_2,\frac{5}{2})}.$$ 
\end{itemize}

\end{proof}


\begin{Remark}
If $
L =  \set{ z m +  \lambda_0 \ :  z \in \C} \subset \mathfrak{a}_{T,\C}^{\ast}$ is a line with a slope $0\neq  m \in \mathfrak{a} _{T,\C}^{\ast}$ passing through $\lambda_0 \in \mathfrak{a} _{T,\C}^{\ast}$, then  
$\NN_{w}^{L}(z)$ or $\NN_{w}(L)$   stands for  the restriction of
$\NN_{w}(\lambda)$ to the line $L$. In addition, 
by abuse of notation, if $L$ is not mentioned, then $\NN_{w}(z)$ stands for the restriction of $\NN^{\chi_{\para{P},z}}_{w}(z)$ to the subspace $\Ind_{M(\A)}^{G(\A)}(z)$.
\end{Remark}

\section{The image of $\leadingterm_{\para{P}_2}(\frac{5}{2})$}
\begin{Prop}\label{global::triv::rep}
Let $(\para{P},z_0) = (\para{P}_2,\frac{5}{2})$. Then the Eisenstein series $E_{\para{P}}(z)$ admits a simple pole at $z=z_0$. In addition, the image of its leading term is square integrable, which gives an automorphic realization of the trivial representation. 
\end{Prop} 
\begin{proof}

Let $w_{p} \in W(\para{P},G)$ be the shortest representative of the longest Weyl element. 

\begin{itemize}
\item
If $w_p \neq  w \in W(\para{P},G)$, then a direct computation shows that 
$\inner{\chi_{\para{P},z_0},\check{\alpha} } >1 $ for every $\alpha \in R(w),$ where 
$R(w) =  \set{\alpha> 0 \: : \:  w\alpha<0}$. In particular, $\NN_{w,\nu}(z)$ is holomorphic at $z=z_0$ for every place $\nu$. In conclusion, $\NN_{w}(z)$ and $C_{w}(z)$ are holomorphic at $z=z_0$.
\item
Otherwise $w=w_p$, and then  a direct computation shows that 
$\inner{\chi_{\para{P},z_0},\check{\alpha} } \geq 1 $  for every $\alpha \in R(w_p)$.  In addition, there is a unique $\alpha \in R(w_{p})$ such that $\inner{\chi_{\para{P},z_0},\check{\alpha}}=1$. Thus, $\ord_{z=z_0} C_{w_{p}}(z) =1$, and  $\NN_{w_p}(z)$ is holomorphic at $z=z_0$.
\item
$|\Stab_{\weyl{G}} (\chi_{\para{P},z_0})|=1$ yields  that 
$\ord_{z=z_0} E_{\para{P}}(z) =1$.
\item
A direct computation shows that $w_p \chi_{\para{P},z_0} =-\sum_{i=1}^{4}{\fun{i}}$. Thus, $\Image \leadingterm_{\para{P}}(z_0)$ is square integrable.
\item
Note that $z_0 \fun{2} = \rho^{M_2}_{T}$. Thus,  the trivial representation is the unique irreducible quotient of $\Ind_{M_2(F_{\nu})}^{G(F_{\nu})}(z_0)$ for every place $\nu$. As a result, one has 
$$\Image \leadingterm_{\para{P}_2}(z_0) =Triv.$$    
\end{itemize}
\vspace{-1cm} \qedhere
\end{proof}

For the rest of the chapter we assume that  $$(\para{P},z_0)  \in \set{ (\para{P}_1 ,1) ,(\para{P}_3,\frac{1}{2}), (\para{P}_4,\frac{5}{2})},$$
and set $\para{P}= \para{P}_i$.

\section{The image of $\leadingterm_{\para{P}}(z_0)$}

Let $[u]_{z_0}$ be an equivalence class. 
A quick look at  \Cref{Table:: P4 ::Global},
\Cref{Table:: P1 ::Global}, \Cref{Table::P3::Global} shows that for all
$w \in [u]_{z_0} $, one has 
$$\ord_{z=z_0} C_{w}(z) =  \ord_{z=z_0}C_{u}(z).$$

The following lemma gives an upper bound for the order of $\M_{u}^{\#}(z)$ at $z=z_0$.

\begin{Lem}\label{global::bounded order}
  Let $[u]_{z_0}$ be an equivalence class. Then for every $K_{\infty}$--fixed section
  $f\in \Ind_{M(\A)}^{G(\A)}(z)$, one has 
$$\ord_{z=z_0} \M_{u}^{\#}(z)f \leq \ord_{z=z_0}C_{u}(z).$$
\end{Lem}

\begin{proof}
Note that for an equivalence class $[u]_{z_0}$, one has 
$$\ord_{z=z_0} \M^{\#}_{u}(z)f =
\ord_{z= z_0} \sum_{w \in [u]_{z_0}} C_{w}(z) \NN_{w}(z)f. $$
Due to our restriction that the section at  Archimedean places is spherical, it follows that the poles of $\NN_{w}(z)$ might come only from the poles at the finite places.
\Cref{local::holo} shows  that the  normalized intertwining operators at finite places are all
holomorphic at $z=z_0$.
Thus, 
$$\ord_{z=z_0} \M_{u}^{\#}(z) \leq \ord_{z=z_0}C_{u}(z).$$
\end{proof}

\begin{Remark}
Combining \Cref{Table:: P4 ::Global}, \Cref{Table:: P1 ::Global}, \Cref{Table::P3::Global} and  \Cref{Table:: P4 ::images}, \Cref{Table:: P1 ::images}, \Cref{Table:: P3 ::images}  shows the following:

Let $[u]_{z_0}$ be an equivalence class. Then for every finite place $\nu$ and for every $w \in [u]_{z_0}$, one has 
$$ \Image \NN_{w,\nu}(z_0) \simeq \Image \NN_{u,\nu}(z_0).$$
\end{Remark}

The last observation paves the way to the classification of the equivalence classes
according to the image $\Image \NN_{w,\nu}(z_0),$
where $\nu$ is a finite place. 
\begin{Def}\label{def:: classes}
 We say that 
$[u]_{z_0}$ is of type:   
\begin{enumerate}
\item[\mylabel{image:sph}{$C_{sph}$}] if $\Image \NN_{u,\nu}(z_0) \simeq \pi_{1,\nu}^{i}$. 
\item[\mylabel{image:C2:first}{$C_{2}$}]  if $\Image \NN_{u,\nu}(z_0) \simeq \pi_{1,\nu}^{i} \oplus \pi_{2,\nu}^{i}$. Later  we  refine the class, splitting it into two sub-classes  \ref{image:even} and \ref{image:third}. \item[\mylabel{image:0}{$C_{0}$}] otherwise. 
\end{enumerate}
\end{Def}

\subsection{ Square-integrability of the image of the leading term }

Let $d_{\para{P}}^{0}(z_0) = \ord_{z=z_0} E_{\para{P}}(f^{0},z,g)$ be the order of the spherical degenerate Eisenstein series. Clearly, 
$$ d_{\para{P}}^{0}(z_0) \leq \ord_{z=z_0} E_{\para{P}}(z) =d_{\para{P}}(z_0).$$
\begin{Prop} \label{global::squre}
Let $(\para{P},z_0) \in \set{ (\para{P}_1,1) , (\para{P}_3 ,\frac{1}{2}), (\para{P}_4,\frac{5}{2})}$. Then $\Image \leadingterm_{\para{P}}(z_0)$ is square integrable.
\end{Prop}
\begin{proof}
We start with the following observation: For an equivalence class $[u]_{z_0}$ labeled by \ref {image:sph} or \ref{image:C2:first} (i.e \ref{image:even} or \ref{image:third} or \ref{image:sph}), one has that  $u \chi_{\para{P},z_0}$ is a sum of simple roots with negative coefficients. Thus, it suffices  to prove that if $[u]_{z_0}$ is labeled by \ref{image:0}, then $\ord_{z=z_0} \M_{u}^{\#}(z) < d_{\para{P}}(z_0)$.

Let $[u]_{z_0}$ be an equivalence class labeled by \ref{image:0}. \Cref{local::p4::holo::arch}, \Cref{local::p1::holo::arch}, \Cref{local::p3::holo::arch} assert that $\NN_{w}(z)$ is holomorphic at $z=z_0$ for every $w \in [u]_{z_0}$. 
Thus, $\ord_{z=z_0} \M_{u}^{\#}(z) \leq \ord_{z=z_0} C_{u}(z)$. 
\begin{itemize}
\item If 
$(\para{P},z_0) \in \set{ (\para{P}_1,1) , (\para{P}_4,\frac{5}{2})}$. 
By \cite[Section 3.3]{HeziMScThesis} one has  $d_{\para{P}}^{0}(z_0) =1$. In addition, \Cref{Table:: P4 ::Global} (resp . \Cref{Table:: P1 ::Global}) shows that for   
$[u]_{z_0}$ labeled by \ref{image:0}, one has  $\ord_{z=z_0}  C_{u}(z) =0$.
Hence the claim follows.
\item
  If $(\para{P},z_0)= (\para{P}_3,\frac{1}{2})$. By \cite[Section 3.3]{HeziMScThesis} one has  $d_{\para{P}}^{0}(z_0) =2$.  Thus, it suffices to show that for every such class, one has $\ord_{z=z_0} \M_{u}^{\#}(z) \leq 1$. If $\ord_{z=z_0} C_{u}(z) \leq 1$, we argue as above. 
  However, there are equivalence classes $[u]_{z_0}$ labeled as \ref{image:0a} such that $\ord_{z=z_0} C_{u}(z) =2$, so that $u\chi_{\para{P},z_0}$ is a sum of simple roots with coefficients that are not necessarily negative. The argument \Cref{global::p3::c0azero}  shows that for such  a class one has $\ord_{z=z_0} \M_{u}^{\#}(z) \leq 1 $. 
\end{itemize}

\end{proof}

\subsection{ The contribution of \ref{image:sph} classes}
 
\begin{Lem}\label{global::sph::lemma}
 Let $[u]_{z_0}$ be an equivalence class  of type \ref{image:sph}. Then, for every $K_{\infty}$--fixed section $f$, one has 
 $$ \ord_{z=z_0}\M_{u}^{\#}(z)f \leq d_{\para{P}}^{0}(z_0).$$
 In addition, 
 $$ \lim_{z \to z_0} (z-z_0)^{d_{\para{P}}^{0}(z_0)}\M^{\#}_{u}(z)f \in \pi^{\emptyset,i}.$$
 \end{Lem}
 \begin{proof}
 Let $f = \otimes_{\nu} f_{\nu}$ be a flat section which is $K_{\infty}$--fixed. Recall that 
 $$ \M_{u}^{\#}(z) = \sum_{w \in [u]_{z_0}} C_{w}(z) \NN_{w}(z).$$
 Let $u$ be the shortest   representative of the equivalence class
 $[u]_{z_0}$.  Each $w \in [u]_{z_0}$ can be written as 
 $ w =  s_{w} u$. Here by definition,
 $$s_{w} \in \Stab_{\weyl{G}}(u \chi_{\para{P},z_0}).$$  
 Since $\NN_{w}(z) =  \NN_{s_{w}}\bk{ u\chi_{\para{P},z}} \circ \NN_{u}(z),$
  one has 
  \begin{align*}
  \sum_{w \in [u]_{z_0}}  C_{w}(z) \NN_{w}(z) &= \bk{ \sum_{w \in [u]_{z_0}} C_{w}(z) \NN_{s_{w}} (u\chi_{\para{P},z}) } \circ \NN_{u}(z).  
  \end{align*}
  Note that $\NN_{w}(z)$ is holomorphic at $z=z_0$ for all $w$.
  Hence, 
  $$\ord_{z=z_0} \M_{u}^{\#}(z)f =
  \ord_{z=z_0} \bk{ \sum_{w \in [u]_{z_0}} C_{w}(z) \NN_{s_w}(u\chi_{\para{P},z}) }\res{ \Image \NN_{u}(z)}.$$     
 By our assumption at $z=z_0$, $\Image \NN_{u}(z_0) \simeq \pi^{\emptyset,i}$. Thus, 
 $$  \ord_{z=z_0} \bk{ \sum_{w \in [u]_{z_0}} C_{w}(z) \NN_{s_w}(u\chi_{\para{P},z}) } f^{0}_{u.z_0}
 =
 \ord_{z=z_0} \bk{ \sum_{w \in [u]_{z_0}} C_{w}(z)} \leq d_{\para{P}}^{0}(z_0)\le d_P(z_0).
 $$
 Hence, 
 $$\lim_{z \to z_0} (z-z_0)^{d_{\para{P}}(z_0)} \M_{u}^{\#}(z)f \in \pi^{\emptyset,i}.$$
 \end{proof}

 \subsection{ Contribution of \ref{image:C2:first} classes }
Note that in \Cref{global::squre}, we already proved that $\Image \leadingterm_{\para{P}}(z_0)$ is square integrable. Hence, it is enough to determine the image, and thus it suffices to prove the following: 

\begin{Thm}\label{global::p3::thm}
 Let $f\in \Ind_{M(\A)}^{G(\A)}(z)$ be a $K_{\infty}$--fixed section.
\begin{enumerate}
\item
  Let $(\para{P},z_0) \in \set{(\para{P}_1,1),(\para{P}_4,\frac{5}{2})}$.
  In this case $d_{\para{P}}^{0}(z_0) =1$.  For any  $[u]_{z_0}$  equivalence class  labeled \ref{image:C2:first}, one has 
$$\lim_{z \to z_0} (z-z_0)\M_{u}^{\#}(z)f \in
\oplus_{|\mathcal{S}| \text{ is even }} \pi^{\mathcal{S}, i}.$$

In addition, there exists an equivalence class $[u]_{z_0}$ labeled by \ref{image:C2:first} such that for any finite set of finite places $\mathcal{S}$ of even cardinality one has 
$$\pi^{\mathcal{S} ,i} \subset \lim_{z \to z_0}\Image(z-z_0) \M_{u}^{\#}(z).$$

\item 
 Let $(\para{P},z_0)= (\para{P}_3,\frac{1}{2})$; in this case $d_{\para{P}}^{0}(z_0) =2$.
 For any equivalence class  $[u]_{z_0}$  labeled  \ref{image:C2:first} we have  
$$\lim_{z \to z_0} (z-z_0)^{2}\M_{u}^{\#}(z)f \in \oplus_{|\mathcal{S}| \neq 1} \pi^{\mathcal{S},3}.$$
In addition, there exists an equivalence class $[u]_{z_0}$ labeled by \ref{image:C2:first} such that for any finite set of finite places $\mathcal{S}$ of $|\mathcal{S}| \neq 1$, one has 
$$\pi^{\mathcal{S} ,i} \subset \lim_{z \to z_0}\Image(z-z_0)^{2} \M_{u}^{\#}(z).$$ 

\end{enumerate}
\end{Thm}

\begin{proof}
Let $[u]_{z_0}$ be an equivalence class labeled  \ref{image:C2:first}. We call it \ref{image:even} if
for every $K_{\infty}$--fixed section $f$, one has 
\begin{align*}
\lim_{z \to z_0} (z-z_0)^{d_{\para{P}}^{0}(z_0)}   \M_{u}^{\#}(z)f \in \underset{|\mathcal{S}|= 0  \mod 2}{\oplus} \pi^{\mathcal{S},i}  \tag{$C_{even}$}\label{image:even}.
\end{align*}
Otherwise, we call it \ref{image:third} (abbreviation of "not one"). In that case, we show that for every $K_{\infty}$--fixed section $f$, one has 
\begin{align*}
\lim_{z \to z_0} \Image (z-z_0)^{d_{\para{P}}^{0}(z_0)}   \M_{u}^{\#}(z)f  \in \underset{|\mathcal{S}| \neq 1}{\oplus} \pi^{\mathcal{S},i}
\tag{$C_{n.o}$}\label{image:third}.
\end{align*}
\subsection{The proof of part 1} \label{global::p1p4 image}
 
In the cases  $ (\para{P}_1,1),(\para{P}_4,\frac{5}{2})$, the proof runs along the same lines.
By \Cref{Table:: P4 ::Global}, \Cref{Table:: P1 ::Global}, there are exactly two equivalence classes which are of type \ref{image:C2:first}. We prove that they are  both of type \ref{image:even}. 
Let $[u]_{z_0}$ be such a class.
It follows from
\Cref{Table:: P4 ::Global}, \Cref{Table:: P1 ::Global} that
\begin{itemize}
\item
$[u]_{z_0} = \set{u,u_2}$ where $l(u) < l(u_2)$.
\item
$\ord_{z=z_0} C_{u}(z)=1$.
\end{itemize}

A direct calculation shows that 
$\Res_{z=z_0} C_{u}(z) =  \Res_{z=z_0} C_{u_2}(z)$.
Furthermore, \Cref{Table:: P4 ::images} and \Cref{Table:: P1 ::images} imply that if  $\nu$ is a  finite place  such that $0\neq \NN_{u,\nu}(z_0)f_{\nu} \in \pi_{2,\nu}^{i}$, one has
 $$\NN_{u_2,\nu}(z_0)f_{\nu} = (-1) \NN_{u,\nu}(z_0)f_{\nu}.$$
 
Thus, if $f =\otimes_{\nu} f_{\nu}$ is a  $K_{\infty}$--fixed section  such that its image under the quotient map belongs to $\pi^{\mathcal{S},i}$,
then one has  $$ \NN_{u_2}(z_0)f = (-1)^{|\mathcal{S}|}\NN_{u}(z_0)f.$$   
Hence 
\begin{align*}
\Res_{z=z_0} \M_{u}^{\#}(z)f &= \Res_{z=z_0}C_{u}(z) \NN_{u}(z_0)f + \Res_{z=z_0}C_{u}(z) \NN_{u_2}(z_0)f\\
&= \Res_{z=z_0} C_{u}(z) \bk { 1 + (-1)^{|\mathcal{S}|} }\NN_{u}(z_0)f,       
\end{align*}

which is non zero if and only if $|\mathcal{S}|$ is even. Namely,
$\pi^{\mathcal{S},i}$ is contained in the image of the
residue of Eisenstein series if and only if $|\mathcal{S}|$ is even.
\end{proof}

\begin{Cor}
Let $(\para{P},z_0) \in \set{ (\para{P}_1,1),(\para{P}_4,\frac{5}{2})}$ and for every  
 $K_{\infty}$--fixed section $f \in \Ind_{M_{i}(\A)}^{G(\A)}(z)$, one has $\ord_{z=z_0} E_{\para{P}}(f,z,g) \leq 1$. In addition, if $\Sigma \subset \Ind_{M_1(\A)}^{G(\A)}(z_0)$ stands for the subrepresentation generated by the $K_{\infty}$--fixed sections, then 
$$\bk{(z-z_0) E_{\para{P}}(z)\res{z=z_0}}\bk{\Sigma} = \oplus_{|\mathcal{S}|\text{ is even} } \pi^{\mathcal{S},i}.$$   
\end{Cor}

\newpage
\section{The proof of the second part}\label{global::p3}

This section is dedicated  to proving the second part of \Cref{global::p3::thm}. We set  $\para{P}= \para{P}_3$, $z_0= \frac{1}{2}$. 
We start by showing that  $\leadingterm_{\para{P}}(z_0)(f)$ is
square integrable.
We have already shown that the equivalence classes labeled by $C_2, C_{sph}$
satisfy \eqref{square::cond}. Among the classes of type $C_0$, there exist
classes that do not satisfy \eqref{square::cond}  yet $\ord_{z=z_0} C_u(z)=2$, thus potentially
violating the square-integrability of the image.
We label  these classes by   
\begin{enumerate*}[label=C\textnormal{\arabic*}]
 \item[\mylabel{image:0a}{$C_{0,a}$}],
 \end{enumerate*}
 and show that for these classes,
 $\ord_{z=z_0} \M^{\#}_{u}(z)< 2 $ and so they do not contribute to the image.

\subsection{  Zero contribution  of \ref{image:0a} classes }
\paragraph*{ Properties \ref{image:0a}.}
The equivalence classes  labeled by \ref{image:0a} have the following properties:
\begin{enumerate}[label=(R\arabic{*}), ref= (R\arabic{*})]
\item
$[u]_{z_0} = \set{ u_1,u_2}$ where $\l(u_1) < l(u_2)$. In addition, there are $w,v \in \weyl{G}$ and a simple reflection $s$ such that 
$u_1 =  w\cdot v$ and $u_2 = w \cdot s \cdot v$.  
\item\label{c0::cw}
$\ord_{z=z_0} C_{u_1}(z) = \ord_{z=z_0} C_{u_2}(z) =2$.  Furthermore, 
$$\lim_{z \to z_0} (z-z_0)^{2}C_{u_1}(z) = - \lim_{z \to z_0} (z-z_0)^{2}C_{u_2}(z).$$ 
\item \label{c0a:holo}
\Cref{local::p3::holo::arch} asserts that both $\NN_{u_1}(z)$ and $\NN_{u_2}(z)$ are holomorphic at $z=z_0$. 
\end{enumerate}    

Thus, in order to show that $\ord_{z=z_0} \M_{u}^{\#}(z) <2$, it suffices to show that $\NN_{u_1}(z_0) = \NN_{u_2}(z_0)$.
\begin{Lem}\label{global::p3::same::c0a}
Let $u_1,u_2$ as above,  $\NN_{u_1}(z_0) =\NN_{u_2}(z_0)$. 
\end{Lem}
\begin{proof} 
It suffices to show that $\NN_{u_1,\nu }(z_0) =\NN_{u_2,\nu}(z_0)$ for every place $\nu$.
Using the properties of the normalized intertwining operator, one has 
\begin{align*}
\NN_{u_1,\nu}(z) &=  \NN_{w,\nu}^{v \chi_{\para{P},z}}(z) \circ \NN_{v,\nu}(z), &
     \NN_{u_2,\nu}(z) &=  \NN_{w,\nu}^{sv \chi_{\para{P},z}}(z) \circ \NN_{s,\nu}^{v \chi_{\para{P},z}}(z) \circ \NN_{v,\nu}(z).
\end{align*}
 It suffices to show that at $z=z_0$, one has
 $\NN_{s,\nu}^{v\chi_{\para{P},z}}(z_0) =\id$
 and 
 $$\NN_{w,\nu}^{v \chi_{\para{P},z}}(z_0) = \NN_{w,\nu}^{sv \chi_{\para{P},z}}(z_0).$$
 
\begin{itemize}
\item
Direct computation shows that $sv \chi_{\para{P},z_0} =  v \chi_{\para{P},z_0}$. In addition, if $\alpha$ is the simple root associated with $s$, then $\inner{v \chi_{\para{P},z_0},\check{\alpha}}=0$. 
Namely, $\NN_{s,\nu}(v \chi_{\para{P},z_0})=\id$, and in particular,
$\NN_{s,\nu}^{v\chi_{\para{P},z}}(z_0) =\id$.
\item
In addition, for every $\beta \in \set{ \alpha > 0 \: : \: w\alpha<0}$
one has $\inner{ v \chi_{\para{P},z_0} \check{\beta}} \geq 0$. 
Thus, the operator $\NN_{w,\nu}(\lambda)$ is holomorphic at $ \lambda = v \chi_{\para{P},z_0}$. 
\item
Since $sv \chi_{\para{P},z_0} =  v \chi_{\para{P},z_0}$, it follows that 
 $\NN_{w,\nu}^{v \chi_{\para{P},z}}(z_0) = \NN_{w,\nu}^{sv \chi_{\para{P},z}}(z_0).$ Thus, the claim follows. 
\end{itemize}
\end{proof}
\begin{Prop}\label{global::p3::c0azero}
For any  $[u]_{z_0}$  labeled by \ref{image:0a}, one has
$\ord_{z=z_0} \M_{u}^{\#}(z)\leq 1.$
\end{Prop}
\begin{proof}
Keeping the above notations, 
\Cref{global::p3::same::c0a} implies that $\NN_{u_1}(z_0) = \NN_{u_2}(z_0)$. In addition, \ref{c0::cw}  asserts that $$\lim_{z \to z_0} (z-z_0)^{2}C_{u_1}(z) = - \lim_{z \to z_0} (z-z_0)^{2}C_{u_2}(z).$$ 
Putting together and using \ref{c0a:holo}, one has  
\begin{align*}
 \lim_{z \to z_0} (z-z_0)^{2} \M^{\#}_{u}(z) &=  \bk{ \lim_{z \to z_0} (z-z_0)^{2} C_{u_1}(z)} \NN_{u_1}(z_0) +  \bk{ \lim_{z \to z_0} (z-z_0)^{2} C_{u_2}(z)} \NN_{u_2}(z_0)\\
 &=  \bk{ \lim_{z \to z_0} (z-z_0)^{2} C_{u_1}(z)} \bk{\NN_{u_1}(z_0) - \NN_{u_2}(z_0) } =0.
\end{align*}

\end{proof}

\subsection{ Contribution  of \ref{image:even} Classes}
In this section we prove the following:
\begin{Prop}
Assume $[u]_{z_0}$ is an equivalence class labeled by \ref{image:even}. Then for every $K_{\infty}$--fixed section $f$, one has 
$\ord_{z=z_0} \M_{u}^{\#}(z) f \leq 2$. Furthermore, 
$$ \lim_{z \to z_0} (z-z_0)^{2} \M_{u}^{\#}(z) f \in \oplus_{|\mathcal{S}|  \text{ is even }} \pi^{\mathcal{S},3}.$$ 
\end{Prop} 
\begin{proof}
\Cref{Table::P3::Global} contains 4 equivalence classes with the label \ref{image:even}. Moreover, these classes can be split into 2 sets (each of two classes) based on 
$\ord_{z=z_0}  C_{u}(z)$. In our cases, $\ord_{z=z_0}  C_{u}(z)  \in\set{2,3}$. 
\begin{itemize}
\item
\Cref{global::p3 ::ceven ::3::a}  deals with the equivalence classes of
$\ord_{z=z_0}  C_{u}(z) =3$.
\item
\Cref{global::p3:even::2} deals with the equivalence classes of
$\ord_{z=z_0}  C_{u}(z) =2$.
\end{itemize}
\end{proof}

\begin{Lem}\label{global::p3 ::ceven ::3::a}
Given a $K_{\infty}$--fixed section $f$ and $[u]_{z_0}$ an equivalence class of label \ref{image:even} such that $\ord_{z=z_0} C_{u}(z) = 3$, then one has 
$\ord_{z=z_0} \M_{u} ^{\#}(z)f \leq 2$. In addition 
$$\bk{\Image (z-z_0)^{2}\M_{u}^{\#}(z)\res{z=z_0}} f \in \oplus_{|\mathcal{S}| \text{ is even} } \pi^{\mathcal{S},3}.$$
\end{Lem}
\begin{proof}
Note that there are exactly two equivalence classes with these properties.
By \Cref{global::bounded order} , one has 
$\ord_{z=z_0} \M_{u}^{\#}(z)f \leq 3$.
\paragraph*{Structure of $[u]_{z_0}$.}
Assume that $u$ is the shortest Weyl element in the equivalence class. Then one has
$[u]_{z_0} = \set{u, w_2u, v, w_2 v}$
for some $v \in W(\para{P},G)$. Rewrite
\begin{align*}
\M_{u}^{\#}(z)  &= \sum_{w \in \set{1,w_2}} C_{wu}(z)\NN_{wu}(z) + \sum_{w \in \set{1,w_2}} C_{wv}(z)\NN_{wv}(z) \\
&=T_{1}(z) C_{u}(z) \NN_{u}(z) + T_{2}(z) C_{v}(z) \NN_{v}(z), 
\end{align*} 
where
\begin{align*}
T_{1}(z) & = \bk{ \id +  C_{w_2}( u \chi_{\para{P},z}) \NN_{w_2}( u \chi_{\para{P},z})}, & T_{2}(z)&=\bk{ \id +  C_{w_2}( v \chi_{\para{P},z}) \NN_{w_2}( v \chi_{\para{P},z})}.
\end{align*}

\paragraph*{Laurent expansion of $T_{i}(z)$ at $z=z_0$.}
\begin{itemize}
\item
By definition $u\chi_{\para{P},z_0} = v\chi_{\para{P},z_0}$.
\item
Note that $\inner{ u \chi_{\para{P},z_0} ,\check{\alpha_2}}=0$. Thus, the operator $\NN_{w_2}(\lambda)$ is holomorphic at $\lambda= u\chi_{\para{P},z_0}$. In addition, $\NN_{w_2}(u\chi_{\para{P},z_0})=\id$. However, $
\inner{ u \chi_{\para{P},z},\check{\alpha}_2} = z- \frac{1}{2}$
and
$\inner{ v \chi_{\para{P},z},\check{\alpha}_2} = 2z -1$. 
\item
Recall that the operator $\NN_{w_2}(\lambda)$ as a meromorphic function depends only on $\inner{\lambda,\check{\alpha_2}}$. Henceforth, the analytical behavior of $\NN_{w_2}^{L}(z)$ depends only on $ \inner{L,\check{\alpha}_{2}}$. In our case, around $z=z_0$, one has 
\begin{align*}
\NN_{w_2}^{u\chi_{\para{P},z}}(z) &= \id + \sum_{i=1}^{\infty} B_{i} (z-z_0)^i,  &
\NN_{w_2}^{v\chi_{\para{P},z}}(z) &= \id + \sum_{i=1}^{\infty} B_{i} \bk{2(z-z_0)}^i,   
\end{align*}
where  $B_{i} \in \operatorname{End} \bk{\Ind_{T}^{G}\bk{u\chi_{\para{P},z_0}}}$.
\item
Moreover, around $z=z_0$, one has 
\begin{align*}
C_{w_2}( u\chi_{\para{P},z}) &=
\frac{\zeta\bk{z-\frac{1}{2}} }{ \zeta\bk{z+\frac{1}{2}}}= 
 -1 + \sum_{j=1}^{\infty} c_{j} (z-z_0)^{j},\\
C_{w_2}( v\chi_{\para{P},z}) &=
\frac{\zeta\bk{2z-1} }{ \zeta\bk{2z}}= 
 -1 + \sum_{j=1}^{\infty} c_{j}\bk{2(z-z_0)}^{j}
\end{align*}
for some $c_i \in \C$.
\item
In conclusion, around $z=z_0$, one has  
\begin{align*}
T_{1}(z) &= 
\bk{ c_{1}\id  -B_{1} }(z-z_0)+ \mathcal{O}\bk{ \bk{z-z_0}^2}, &
\\ T_{2}(z) &= 
2\bk{ c_{1}\id  -B_{1} }(z-z_0)+ \mathcal{O}\bk{ \bk{z-z_0}^2}.
\end{align*}

\end{itemize}
\paragraph*{Proving  $\ord_{z=z_0} \M_{u}^{\#}(z)f\leq 2$.}
\begin{itemize}
\item
For every $K_{\infty}$--fixed section $\NN_{u}(z)f, \NN_{v}(z)f$ are holomorphic at $z=z_0$. 
\item
$\ord_{z=z_0} C_{u}(z) = \ord_{z=z_0}C_{v}(z) =3.$
\item
Thus,
\begin{align*} 
&\ord_{z=z_0} \M^{\#}_{u}(z)f \leq\\ &\max\set{\ord_{z=z_0}\bk{T_{1}(z)C_{u}(z) \NN_{u}(z)f}, \ord_{z=z_0} \bk{ T_{2}(z) C_{v}(z) \NN_{v}(z)f}} \leq 2.  
\end{align*}
\end{itemize}
\paragraph*{Showing  $\bk{\Image (z-z_0)^{2}\M_{u}^{\#}(z)\res{z=z_0}} f \in \oplus_{|\mathcal{S}| \text{ is even} } \pi^{\mathcal{S},3}$.}
\begin{itemize}
\item
By the above argument, the leading term of $\M^{\#}_{u}(z)$ around $z=z_0$ is the sum of the leading terms of $T_{1}(z) C_{u}(z) \NN_{u}(z)$  and $ T_{2}(z) C_{v}(z) \NN_{v}(z)$. 
\item
A direct calculation shows that 
$$c =\lim_{z\to z_0} (z-z_0)^{3} C_{u}(z) =  2\lim_{z\to z_0} (z-z_0)^{3} C_{v}(z).$$
\item  
Thus, around $z=z_0$, one has 
\begin{align*}
\M_{u}^{\#}(z)f &=  \bk {c \bk{c_1 \id  -B_1} }   \bk{ \NN_{u}(z_0) f +
\NN_{v}(z_0) f }  \times \frac{1}{(z-z_0)^{2}} + \mathcal{O}\bk {\frac{1}{z-z_0}}.   
\end{align*}
\item
For every finite place $\nu$, one has $\Image \NN_{u,\nu}(z_0) =\Image \NN_{v,\nu}(z_0) \simeq \pi_{1,\nu}^{3} \oplus \pi_{2,\nu}^{3} $.
In addition, if $f_{\nu} \in \Ind_{M_3(F_{\nu})}^{G(F_{\nu})}(z_0)$ such that $\NN_{u,\nu}(z_0)f_{\nu} \in \pi_{2,\nu}^{3}$, then 
\Cref{Table:: P3 ::images} implies that 
$ \NN_{v,\nu}(z_0) f_{\nu} = (-1) \NN_{u,\nu}(z_0) f_{\nu}$.

\item  
Let $f$ be a $K_{\infty}$--fixed section and $\mathcal{S}$ be a finite set of finite places such that $\NN_{u}(z_0)f \in \pi^{\mathcal{S},3}$. Then   
$ \NN_{v}(z_0) f = (-1)^{|\mathcal{S}|} \NN_{u}(z_0) f$. 
As a result, for such a section, one has 
$$ \bk {c \bk{c_1 \id  -B_1} }   \bk{ \NN_{u}(z_0) f +
\NN_{v}(z_0) f } = \bk{ 1+(-1)^{|\mathcal{S}|}} \bk {c \bk{c_1 \id  -B_1} } \NN_{u}(z_0) f. $$
\item
In conclusion,
$$\bk{(z-z_0)^{2}\M_{u}^{\#}(z)\res{z=z_0}} f \in \oplus_{|\mathcal{S}| \text{ is even} } \pi^{\mathcal{S},3}.$$
\end{itemize}
\end{proof}

\begin{Lem}\label{global::p3:even::2}
For any  $K_{\infty}$--fixed section $f$ and 
  an equivalence class $[u]_{z_0}$ of label \ref{image:even} such that
 $\ord_{z=z_0}C_{u}(z)=2$, one has  $\ord_{z=z_0} \M^{\#}_{u}(z)f \leq2$. 
 Moreover, if $\Sigma$ stands for the subrepresentation of $\Ind_{M(\A)}^{G(\A)}(z_0)$ that is generated by $K_{\infty}$--fixed vectors, then

 $$\bk{ \lim_{z\to z_0}(z-z_0)^{2} \M_{u}^{\#}(z)}\bk{
 \Sigma}=\oplus_{|\mathcal{S}| \: \text{ is even }} \pi^{\mathcal{S},3}.$$ 
 
\end{Lem}
\begin{proof}
There are two equivalence classes that satisfy the condition of the lemma.
For both of them  $[u]_{z_0} =\set{u,v}$ where $l(v) > l(u)$.
\Cref{global::bounded order}  implies that 
 $$\ord_{z=z_0} \M_{u}^{\#}(z)f \leq 2.$$
Note that: 
\begin{itemize}
\item
$\lim\limits_{z \to z_0} (z-z_0)^{2} C_{u}(z) =\lim\limits_{z \to z_0} (z-z_0)^{2} C_{v}(z).$
\item
For every finite place $\nu$, one has $\Image \NN_{u,\nu}(z_0) =\Image \NN_{v,\nu}(z_0) \simeq \pi_{1,\nu}^{3} \oplus \pi_{2,\nu}^{3} $.
In addition, if $f_{\nu} \in \Ind_{M_3(F_{\nu})}^{G(F_{\nu})}(z_0)$ such that $\NN_{u,\nu}(z_0)f_{\nu} \in \pi_{2,\nu}^{3}$, then 
\Cref{Table:: P3 ::images} implies that 
$ \NN_{v,\nu}(z_0) f_{\nu} = (-1) \NN_{u,\nu}(z_0) f_{\nu}$.

\item
Let $f$ be a $K_{\infty}$--fixed section and $\mathcal{S}$ be a finite set of finite places such that $\NN_{u}(z_0)f \in \pi^{\mathcal{S},3}$. Then   
$ \NN_{v}(z_0) f = (-1)^{|\mathcal{S}|} \NN_{u}(z_0) f$. 

Since for such a section, one has 
$$
\lim_{z\to z_0} \M^{\#}_{u}(z)f = \bk{ \lim_{z\to z_0}(z-z_0)^{2}C_{u}(z)} (1 +(-1)^{|\mathcal{S}|}) \NN_{u}(z_0)f.  
$$
Hence, 
the representation $$\pi^{\mathcal{S},3} \subset \bk{ \lim_{z\to z_0}(z-z_0)^{2} \M_{u}^{\#}(z)}\bk{
 \Sigma} $$ 
 if and only if $|\mathcal{S}|$ is even. 
\end{itemize}

\end{proof}

\subsection{ Contribution  of \ref{image:third} Classes}
In this section we prove the following:
\begin{Prop}
Assume $[u]_{z_0}$ is an equivalence class labeled by \ref{image:third}. Then for every $K_{\infty}$--section, one has 
$\ord_{z=z_0} \M_{u}^{\#}(z) f \leq 2$. Furthermore, 
$$ \lim_{z\to z_0} (z-z_0)^{2} \M_{u}^{\#}(z) f \in \oplus_{|\mathcal{S}|  \neq1} \pi^{\mathcal{S},3}.$$

In addition, there exists an equivalence class $[u]_{z_0}$, labeled by \ref{image:third} such that for any finite set of finite places $\mathcal{S}$ such that $|\mathcal{S}| \neq 1$, one has 
$$\pi^{\mathcal{S} ,3} \subset \lim_{z \to z_0}\Image(z-z_0)^{2} \M_{u}^{\#}(z).$$ 

\end{Prop} 
\begin{proof}
\Cref{Table::P3::Global} contains 4 equivalence classes with the label \ref{image:third}. Moreover, these classes can be split into 3 sets according to  
$\ord_{z=z_0}C_{u}(z)$. In our cases, $ \ord_{z=z_0}  C_{u}(z) \in\set{2,3,4}$. 
\begin{itemize}
\item
\Cref{global::p3::cthird::order 4}  deals with the equivalence class such that 
$\ord_{z=z_0}  C_{u}(z) =4$.
\item
\Cref{global::p3::cthird::order 3::a} and \Cref{global::p3::cthird::order 3::b}  deal with the two  equivalence classes such that
$\ord_{z=z_0}  C_{u}(z) =3$.
\item
\Cref{global::p3::third::order 2} deals with the equivalence class such that
$\ord_{z=z_0}  C_{u}(z) =2$.
\end{itemize}
\end{proof}

\begin{Lem}\label{global::p3::cthird::order 4}
Given a $K_{\infty}$--fixed section $f$  and let $u= \w_{2}\w_{3}\w_{1}\w_{2}\w_{3}\w_{4}\w_{3}\w_{1}\w_{2}\w_{3}$.
Then 
$\ord_{z=z_0} \M_{u}^{\#}(z)f\leq 2$. In addition 
$$\bk{(z-z_0)^{2} \M_{u}^{\#}(z)\res{z=z_0}}f \in \oplus_{|\mathcal{S}| \neq 1} \pi^{\mathcal{S},3}.$$
\end{Lem}
\begin{proof}
Note that $u$ is the shortest Weyl element in the equivalence class. Then one has
$$[u]_{z_0} = \set{w\cdot u , w \cdot v  \: : \:  w \in \inner{w_{2},w_4 }}, \quad v  = \w_{2}\w_{3} \w_{4}
 \w_{3}\w_{2}\w_{3}
\w_{1}\w_{2}\w_{3}
\w_{4}\w_{3}\w_{2} 
\w_{3} \w_{1}\w_{2}\w_{3}.$$
 Rewrite
\begin{align*}
\M_{u}^{\#}(z)  &= \sum_{w \in \inner{w_2,w_4}} C_{wu}(z)\NN_{wu}(z) + \sum_{w \in \inner{w_2,w_4}} C_{wv}(z)\NN_{wv}(z) \\
&=T_{1}(z) C_{u}(z) \NN_{u}(z) + T_{2}(z) C_{v}(z) \NN_{v}(z), 
\end{align*} 
where
\begin{align*}
T_{1}(z) & = \bk{ \id +  C_{w_2}( u \chi_{\para{P},z}) \NN_{w_2}( u \chi_{\para{P},z})} \bk{ \id +  C_{w_4}( u \chi_{\para{P},z}) \NN_{w_2}( u \chi_{\para{P},z})}  \\ T_{2}(z)&=\bk{ \id +  C_{w_2}( v \chi_{\para{P},z}) \NN_{w_2}( v \chi_{\para{P},z})}\bk{ \id +  C_{w_4}( v \chi_{\para{P},z}) \NN_{w_4}( v \chi_{\para{P},z})}.
\end{align*}

The above decomposition is available due to $w_2 w_4 =w_4w_2$.
 
As a result of 
$$ \inner{ u \chi_{\para{P},z} ,\check{\alpha}_i} =  
\inner{ v \chi_{\para{P},z} ,\check{\alpha}_i} = z-\frac{1}{2} \quad \text{for}  \quad  i \in \set{2,4},$$
arguing as in \Cref{global::p3 ::ceven ::3::a}  yields that 
$T_{1}(z)$ and $T_{2}(z)$ admits the same Laurent expansion around $z=z_0$. In addition, they both admit  zero of an order of at least two at $z=z_0$.  Write 
$$T_{1}(z)=T_2(z) =  B_{2}\bk{z-z_0}^{2} + \mathcal{O}\bk{ (z-z_0)^{3}},$$ where 
$B_2 \in \operatorname{End}\bk{\Ind_{T(\A)}^{G(\A)}\bk{ u \chi_{\para{P},z_0}}}.$

\paragraph*{Proving  $\ord_{z=z_0} \M_{u}^{\#}(z)f\leq 2$.}
\begin{itemize}
\item
For every $K_{\infty}$--fixed section,  $\NN_{u}(z)f, \NN_{v}(z)f$ are holomorphic at $z=z_0$. 
\item
$\ord_{z=z_0} C_{u}(z) = \ord_{z=z_0}C_{v}(z) =4.$
\item
Thus, 
\begin{align*}
&\ord_{z=z_0} \M^{\#}_{u}(z)f  \leq \\ &\max\set{\ord_{z=z_0}\bk{T_{1}(z)C_{u}(z) \NN_{u}(z)f} , \ord_{z=z_0} \bk{ T_{2}(z) C_{v}(z) \NN_{v}(z)f} } \leq 2.   
\end{align*}
\end{itemize}

\paragraph*{Showing  $\bk{\Image (z-z_0)^{2}\M_{u}^{\#}(z)\res{z=z_0}} f \in \oplus_{|\mathcal{S}|\neq 1 } \pi^{\mathcal{S},3}$.}
\begin{itemize}
\item
By the above argument, the leading term of $\M^{\#}_{u}(z)$ around $z=z_0$ is the sum of the leading terms of $T_{1}(z) C_{u}(z) \NN_{u}(z)$  and $ T_{2}(z) C_{v}(z) \NN_{v}(z)$. 
\item
A direct calculation shows that 
$$c= \lim_{z\to z_0} (z-z_0)^{4} C_{u}(z) =  3\lim_{z\to z_0} (z-z_0)^{4} C_{v}(z).$$
\item  
Thus, around $z=z_0$, one has 
\begin{align*}
\M_{u}^{\#}(z)f &=  \bk {c B_2 \circ \NN_{u}(z_0)f + \frac{c}{3} B_2 \circ \NN_{v}(z_0)f }     \times \frac{1}{(z-z_0)^{2}} + \mathcal{O}\bk {\frac{1}{z-z_0}}.   
\end{align*}
\item
For every finite place $\nu$, one has $\Image \NN_{u,\nu}(z_0) =\Image \NN_{v,\nu}(z_0) \simeq \pi_{1,\nu}^{3} \oplus \pi_{2,\nu}^{3} $.
In addition, if $f_{\nu} \in \Ind_{M_3(F_{\nu})}^{G(F_{\nu})}(z_0)$ such that $\NN_{u,\nu}(z_0)f_{\nu} \in \pi_{2,\nu}^{3}$, then 
\Cref{Table:: P3 ::images} implies that 
$ \NN_{v,\nu}(z_0) f_{\nu} = (-3) \NN_{u,\nu}(z_0) f_{\nu}$.

\item
Let $f$ be a $K_{\infty}$--fixed section and $\mathcal{S}$ be a finite set of finite places such that $\NN_{u}(z_0)f \in \pi^{\mathcal{S},3}$. Thus,   
$ \NN_{v}(z_0) f = (-3)^{|\mathcal{S}|} \NN_{u}(z_0) f$. 

As a result, for such a section, one has 
\begin{align*}
\lim_{z \to z_0} (z-z_0)^{2}\M_{u}^{\#}(z)f &=  c\bk{ 1 + \frac{(-3)^{|\mathcal{S}|}}{3}} B_2 \NN_{u}(z_0)f.      
\end{align*}
\item
In conclusion,
$$\bk{\Image (z-z_0)^{2}\M_{u}^{\#}(z)\res{z=z_0}} f \in \oplus_{|\mathcal{S}| \neq 1 } \pi^{\mathcal{S},3}.$$
\end{itemize}

\paragraph{Remark:}
Arguing as in \Cref{local::p3::holo::arch}, the operators $\NN_{u,\nu}(z)$ and $\NN_{v,\nu}(z)$ are holomorphic at $z=z_0$ at Archimedean places too.
Thus, $\ord_{z=z_0} \M_{u}^{\#}(z)f \leq 2$, for an arbitrary section $f$. 
\end{proof}

\begin{Lem}\label{global::p3::cthird::order 3::a}
Let $f$ be a $K_{\infty}$--fixed section and 
$u= \w_{3} \w_{1} \w_{2}\w_{3} \w_{4} \w_{3} \w_{1}\w_{2}\w_{3} $.
Then 
$\ord_{z=z_0} \M_{u}^{\#}(z)f\leq 2$. In addition 
$$\bk{ (z-z_0)^{2} \M_{u}^{\#}(z)\res{z=z_0}}f \in \oplus_{|\mathcal{S}| \neq 1} \pi^{\mathcal{S},3}.$$
\end{Lem}
\begin{proof}
Note that 
$ [u]_{z_0} =  \set{u, s_1u,s_2u, v, s_1s_2v  },$
where 
\begin{align*}
s_1 &= \w_{2} \w_{3} \w_{2} &
s_2 &= \w_{2} \w_{1} \w_{2} & 
v &= \w_{3} \w_{4} \w_{3} 
\w_{2} \w_{3} \w_{1}
\w_{2} \w_{3} \w_{4}
\w_{3} \w_{2} \w_{3}
\w_{1} \w_{2} \w_{3}.
\end{align*}
\paragraph*{Showing $\ord_{z=z_0} \M^{\#}_{u}(z) f \leq 2$.}
Assume  $\lim\limits_{z \to z_0} \bk{ z-z_0}^{3}C_{u}(z) = c$. Then 
\begin{align*}
\lim_{z \to z_0} (z-z_0)^{3} C_{s_1u}(z) &=-\frac{1}{2}c, & \lim_{z \to z_0} (z-z_0)^{3} C_{s_2u}(z) &= -\frac{1}{2}c,
\\
\lim_{z \to z_0} (z-z_0)^{3} C_{v}(z) & =\frac{c}{6},
&
\lim_{z \to z_0} (z-z_0)^{3} C_{s_1s_2v}(z) & =-\frac{c}{6}.
\end{align*} 
In addition, by \Cref{Table:: P3 ::images}, for every finite place $\nu$, 
the operators $\NN_{w,\nu}(z)$  are holomorphic at $z=z_0$, and 
\begin{align*}
\NN_{u,\nu}(z_0) =  \NN_{s_1u,\nu}(z_0) &= \NN_{s_2u,\nu}(z_0),&
\NN_{v,\nu}(z_0) &=  \NN_{s_1s_2v,\nu}(z_0). 
\end{align*}
Thus, for every $K_{\infty}$--fixed section $f$,  $\NN_{w}(z)f$ for every $w \in[u]_{z_0}$  is holomorphic at $z=z_0$, and one has 
\begin{align*}
\NN_{u}(z_0)f &=  \NN_{s_1u}(z_0)f= \NN_{s_2u}(z_0)f, &
\NN_{v}(z_0)f &=  \NN_{s_1s_2v}(z_0)f.
\end{align*}
In conclusion, 
\begin{align*}
&\lim_{z \to z_0} (z-z_0)^{3} \sum_{w \in [u]_{z_0}} C_{w}(z) \NN_{w}(z)f \\&= \lim_{z \to z_0} (z-z_0)^{3} \sum_{w \in \set{u,s_1u,s_2 u}}  C_{w}(z) \NN_{w}(z)f +  \lim_{z \to z_0} (z-z_0)^{3} \sum_{w \in \set{v,s_1s_2v}} C_{w}(z) \NN_{w}(z)f  
\\&
=\bk{c -\frac{c}{2} -\frac{c}{2}} \NN_{u}(z_0)f + \frac{1}{6}\bk{c -c}\NN_{v}(z_0)f=0.  
\end{align*}
Henceforth,  $\ord_{z=z_0} \M_{u}^{\#}(z)f \leq 2.$

\paragraph*{Showing $\bk{ (z-z_0)^{2} \M_{u}^{\#}(z)\res{z=z_0}}f \in \oplus_{|\mathcal{S}| \neq 1} \pi^{\mathcal{S},3}$.}
$ $

For this purpose, we fix a $K_{\infty}$--fixed section 
such that $f= \otimes_{\nu \in \mathcal{S}}f_{v} \otimes_{\nu \not \in \mathcal{S}} f^{0}_{z,\nu}$, and for every $\nu \in \mathcal{S}$, $ 0 \neq \NN_{u,\nu}(z_0)f_{\nu} \in \pi_{2,\nu}^{3}$. Here, $f^{0}_{z,\nu}$ stands for the normalized spherical vector. Note that this section generated $\pi^{\mathcal{S},3}$.
Thus, in order to prove the claim, it is enough to show that 
if $|\mathcal{S}| =1$ then $$f \in \ker (z-z_0)^{2} \M_{u}^{\#}(z)\res{z=z_0}.$$ 
Write  $\mathcal{S} =\set{\nu_{0}}$. Then 
\begin{align*}
\M_{u}^{\#}(z) f &= \sum_{w \in [u]_{z_0}} C_{w}(z) \NN_{w}(z)f \overset{}{=}    
  \bk{\sum_{w \in [u]_{z_0}} C_{w}(z) \NN_{w,\nu_0}(z)f_{\nu_0} }\otimes_{\nu \neq \nu_{0} } f^{0}_{w.z,\nu}. 
\end{align*}
Thus, it is enough to show 
$$f_{\nu_0} \in \ker(z-z_0)^{2}\bk{\sum_{w \in [u]_{z_0}} C_{w}(z) \NN_{w,\nu_0}(z)}\res{z=z_0}.$$

For every $w \in [u]_{z_0}$, one has around $z=z_0$,  
\begin{align*}
C_{w}(z) &= \frac{c_{w}^{(-3)}}{ \bk{z-z_0}^{3}} + \frac{c_{w}^{(-2)}}{ \bk{z-z_0}^{2}} + \mathcal{O}\bk{ \bk{z-z_0}^{-1}}, \\  
\NN_{w,\nu_0}(z) &= \NN_{w,\nu_0}(z_0) +  \NN_{w,\nu_0}^{'}(z_0) (z-z_0) + \mathcal{O}\bk{ \bk{z-z_0}^{2}}. 
\end{align*}

Note that  $ \ord_{z=z_0}\bk{\sum_{w \in [u]_{z_0}} C_{w}(z) \NN_{w,\nu_0}(z)}f \leq 2$. Thus, 
around $z=z_0$   
$$  \M_{u}^{\#}(z)f = \bk{ \sum_{w \in[u]_{z_0}} c_{w}^{(-3)} \NN_{w,\nu_0}'(z_0)f + \sum_{w \in[u]_{z_0}} c_{w}^{(-2)} \NN_{w,\nu_0}(z_0)f} \cdot \frac{1}{(z-z_0)^{2}}  + \mathcal{O}\bk{ \bk{z-z_0 }^{-1} }. $$

Thus, it is enough to show that $f_{\nu_0}$ belongs to the kernel of each summand. 

For the first summand, we need to show that 
$$ \bk{ c \NN_{u,\nu_0}^{'}(z_0)  -\frac{c}{2}\bk{ \NN_{s_1u,\nu_0}^{'}(z_0) +   
\NN_{s_2u,\nu_0}^{'}(z_0)}
+\frac{c}{6} \bk{ \NN_{v,\nu_0}^{'}(z_0)  - \NN_{s_1 s_2v,\nu_0}^{'}(z_0)  }} f_{\nu_0} =0.  
$$
This is proven in \Cref{local::p3::der}.

For the second summand, we need to show that:
\begin{align}
 & \bk{ c_{u}^{(-2)}\NN_{u,\nu_0}(z_0)  + c_{s_1u}^{(-2)}\NN_{s_1u,\nu_0}(z_0) +   
 c_{s_2u}^{(-2)}\NN_{s_2u,\nu_0}(z_0)}f_{\nu_0}  \nonumber \\ 
&+ \bk{  c_{v}^{(-2)}\NN_{v,\nu_0}(z_0)  +
 c_{s_1s_2v }^{(-2)}\NN_{s_1 s_2v,\nu_0}(z_0)} f_{\nu_{0}}=0.   
\label{p3::cno::3}
\end{align}
For this purpose  we proceed as follows:
\begin{itemize}
\item
By \Cref{Table:: P3 ::images}, one has 
\begin{align*}
\NN_{u,\nu_0}(z_0) f_{\nu_0} = \NN_{s_1u,\nu_0}(z_0) f_{\nu_0} = \NN_{s_2u,\nu_0}(z_0) f_{\nu_0} \\
 \NN_{v,\nu_0}(z_0) f_{\nu_0}= \NN_{s_2s_1v,\nu_0}(z_0) f_{\nu_0} =(-3) \NN_{u,\nu_0}(z_0)f_{\nu_0}.
\end{align*}

Thus, \eqref{p3::cno::3} can be written as 
$$
\underbrace{\bk{ c_{u}^{(-2)} +  c_{s_1u}^{(-2)} + c_{s_2u}^{(-2)} }}_{c_{u}} \NN_{u,\nu_0}(z_0)f_{\nu_0} +    
\underbrace{\bk{ c_{v}^{(-2)} +  c_{s_2s_1v}^{(-2)}}}_{c_{v}} 
\NN_{v,\nu_0}(z_0)f_{\nu_0}.
$$ 

Thus, it is enough to show that
$c_{u} = \frac{1}{3}c_{v}$. 
\item
Using the expansions as in \cite{HeziMScThesis} shows that this is indeed the case. 
\end{itemize}
Thus, the claim follows.
\end{proof}
\begin{Lem}\label{global::p3::cthird::order 3::b}
Given a $K_{\infty}$--fixed section $f$  and let $u= w_{1}w_{2}w_{3}w_{4}w_{2}w_{3}w_{1}w_{2}w_{3}$.
Then 
$\ord_{z=z_0} \M_{u}^{\#}(z)f\leq 2$. In addition 
$$\bk{ (z-z_0)^{2} \M_{u}^{\#}(z)\res{z=z_0}}f \in \oplus_{|\mathcal{S}| \neq 1} \pi^{\mathcal{S},3}.$$
\end{Lem}
\begin{proof}
The proof is similar  to  \Cref{global::p3::cthird::order 3::a} and hence omitted.
\end{proof}

The main theorem will be proved once we 
prove the  following significant Lemma:
\begin{Lem} \label{global::p3::third::order 2}
Let 
$u_1=w_{1}w_{2}w_{3}w_{4}w_{1}w_{2}w_{3}$, and let
$\Sigma$ be the subrepresentation of $\Ind_{M(\A)}^{G(\A)}(z_0)$ that is generated by $K_{\infty}$--fixed vectors. Then, 
$\ord_{z=z_0} \M_{u_1}^{\#}(z) (\Sigma) \leq 2$.
Moreover,
 $$\bk{ \lim_{z\to z_0}(z-z_0)^{2} \M_{u_1}^{\#}(z)}\bk{
 \Sigma}=\oplus_{|\mathcal{S}| \neq 1} \pi^{\mathcal{S},3}.$$ 
\end{Lem}
\begin{proof}
Note that 
$[u_1]_{z_0} =\set{u_1,u_2,u_3}$ where 
\begin{align*} 
 u_2&=w_{1}w_{2}w_{3}w_{4}w_{3}w_{2}w_{3}w_{1}w_{2}w_{3} & 
u_3&=w_{3}w_{2}w_{3}w_{1}w_{2}w_{3}w_{4}w_{3}w_{2}w_{3}w_{1}w_{2}w_{3}.  
\end{align*}
Since for every finite place $\nu$, the operator $\NN_{w,\nu}(z)$ is holomorphic at $z=z_0$, one has 
$ \ord_{z=z_0} \M_{u_1}^{\#}(z)f \leq 2$. A direct calculation shows that if $\lim\limits_{z \to z_0} (z-z_0)^{2}  C_{u_1}(z) = c$, then 
\begin{align*}
\lim_{z \to z_0} (z-z_0)^{2}  C_{u_2}(z) = \lim_{z \to z_0} (z-z_0)^{2}  C_{u_3}(z) = \frac{1}{2} c.
\end{align*}
Hence, for the spherical section, one has 
$(z-z_0)^{2}\M_{u_1}^{\#}(z)\res{z=z_0}f^{0} \neq 0$. Thus, the pole is attained.

\begin{itemize}
\item
By \Cref{local:p3::special::1}, for every finite place $\nu$, there is a vector $f_{\nu}^{1}$ such that $\set{ v_{2,\nu}, v_{3,\nu}}$ is a linearly independent set, where 
\begin{align*}
v_{2,\nu} &=\NN_{u_2,\nu}(z_0)f_{\nu}^{1} \in \pi_{2,\nu}^{3}
&
v_{3,\nu} =\NN_{u_3,\nu}(z_0)f_{\nu}^{1} \in \pi_{2,\nu}^{3}.
\end{align*}
\item
Note that the set 
$\set{ v^{\mathcal{S}'} = \otimes_{\nu \in \mathcal{S^{'}}} v_{2,\nu} \otimes_{\nu \in \mathcal{S} \setminus \mathcal{S^{'}}} v_{3,\nu} \: : \:  \mathcal{S}'\subseteq \mathcal{S}}$ is linearly independent. Here,  $\mathcal{S}$ is a finite set of finite places. 
\item
By \Cref{local:p3::special::2}, one has 
 $$ \NN_{u_1,\nu}(z_0) f_{\nu}^{1} =  -\frac{1}{2}\NN_{u_2,\nu}(z_0) f_{\nu}^1  
 -\frac{1}{2}\NN_{u_3,\nu}(z_0) f_{\nu}^1 =  -\frac{v_{2,\nu}}{2} -\frac{v_{3,\nu}}{2}. 
 $$
\item
Given a finite set of  places $\mathcal{S}$, we consider the following $K_{\infty}$--fixed section $f$ as follows:
\begin{enumerate}
\item
If $\nu \in \mathcal{S}$ then $f_{\nu}=f_{\nu}^{1}$.
\item
Else, $\nu \not \in \mathcal{S}$ then $f_{\nu}=f_{\nu}^{0}$.
\end{enumerate} 
In particular, for such a section, one has 
$$\NN_{u_1}(z_0)f= \frac{1}{(-2)^{\mathcal{|S|}}}\sum_{\mathcal{S}'\subseteq \mathcal{S} } \bk{ v^{\mathcal{S}'} \otimes_{\nu \not \in \mathcal{S}}f^{0}_{u_1.z_0}}. $$
\item
In conclusion, one has
 \begin{align*}
 \lim_{z \to z_0} (z-z_0)^{2} \M_{u_1}^{\#}(z)f &=
  \sum_{\emptyset \neq \mathcal{S}^{'} \subsetneq \mathcal{S}} 
 \frac{c} {(-2)^{|\mathcal{S}|}} v^{\mathcal{S}'}
   \otimes_{\nu \not \in \mathcal{S}} f^{0}_{u_1.z_0}  \\
 &+
 \bk{\frac{c}{2} + \frac{c}{(-2)^{|\mathcal{S}|}}} v^{\mathcal{S}} \otimes_{\nu \not \in \mathcal{S}} f^{0}_{u_1.z_0}\\
&+
 \bk{\frac{c}{2} + \frac{c}{(-2)^{|\mathcal{S}|}}} v^{\mathcal{\emptyset}} \otimes_{\nu \not \in \mathcal{S}} f^{0}_{u_1.z_0}
 \end{align*}
 which is not zero if and only if $|\mathcal{S}| \neq 1$. 
\end{itemize}

It remains to demonstrate that if $f \in \Ind_{M_{3}(A)}^{G(\A)}(z_0) $ such that $\NN_{u}(z_0)f \in \pi^{\mathcal{S},3}$ and $|\mathcal{S}|=1$, then $f \in \ker (z-z_0)^2\M_{u_1}^{\#}(z)\res{z=z_0} $. This is a direct consequence of \Cref{local:p3::special::2}. Thus the claim follows.

\paragraph*{Remark:}
Note that as in \Cref{local::p3::holo::arch}, the operators $\NN_{w,\nu}(z)$ where $\w \in [u_1]_{z_0}$  are holomorphic at $z=z_0$ at Archimedean places too.
Thus, the same proof shows that $\ord_{z=z_0} \M_{u_1}^{\#}(z)f \leq 2$, for an arbitrary section $f$. 

\end{proof}
\begin{Cor}
For any $K_{\infty}$--fixed section $f \in \Ind_{M_{3}(\A)}^{G(\A)}(z)$, one has $\ord_{z=z_0} E_{\para{P}}(f,z,g) \leq 2$. In addition, if $\Sigma \subset \Ind_{M_3(\A)}^{G(\A)}(z_0)$ stands for the subrepresentation generated by the $K_{\infty}$--fixed sections, then 
$$\bk{(z-z_0)^{2} E_{\para{P}}(z)\res{z=z_0}}\bk{\Sigma} = \oplus_{|\mathcal{S}|\neq 1} \pi^{\mathcal{S},3}.$$   
\end{Cor}

\newpage
\appendix
\chapter{Zampera's argument}\label{local::section::Zampera}
\setcounter{Thm}{0}
Let $\chi$ be a point in $\mathfrak{a}_{T,\C}^{\ast}$. 
A local intertwining operator $\NN_{u}(\lambda),$ being a
meromorphic operator, might not be defined at $\chi$.
However, it may be holomorphic at $\chi$ when $\lambda$
approaches $\chi$ along a given line $L$ through this point. Below we describe the setting in which this phenomenon occurs. 
In addition, when $u(\chi)=\chi$, the operator $\NN_{u}(\lambda)|_L,$
at $\lambda=\chi$ defines an endomorphism of $i^G_T(\chi)$.
We describe the situation in which it is diagonalizable with two different eigenvalues
and defines the direct sum decomposition of $i^G_T(\chi)$. 

These results will be used repeatedly throughout the thesis. 

We  first learned this argument from the paper
by S. Zampera \cite{Zampera1997} and hence we call it here  \textit{Zampera's argument}.

\begin{Prop}\label{Zampera ::lemma_holo}
Let $ u =\s{\alpha} w \s{\alpha}$
such that: 
\begin{enumerate}
\item
$\alpha \in \Delta_{G} ,\: \inner{\chi,\check{\alpha}}=1$.
\item
$w \in W({\s{\alpha}\chi}) = \inner{\s{\beta}
\: : \: \beta \in \Delta_{G} \text{ such that } \inner{ \w_{\alpha} \chi,\check{\beta}}=0}$.
\item $l(u) =  2 +l(w)$.
\end{enumerate}

Let $L=\set{ zm +\chi \: : \: z \in \C } $
be a line in $ \mathfrak{a}_{T,\C}^{\ast}$ through $\chi$,
such that $\inner{um,\check{\alpha}}\neq 0$.

Then the restriction $\NN_{u}(\lambda)\res{L}$ to the line $L$
is holomorphic at $\chi$. 
\end{Prop} 
 
\begin{proof}
Before starting the proof, we first elaborate why the operator
$\NN_{u}(\lambda)$ is not defined at $\lambda=\chi$. 

By the properties of the normalized intertwining operator for an arbitrary  $\lambda \in \mathfrak{a}_{T,\C}^{\ast}$, one has
$$\NN_{u}(\lambda)= \NN_{\s{\alpha}}(w\s{\alpha}\lambda) \circ \NN_{w}(\s{\alpha} \lambda) \circ \NN_{\s{\alpha}}(\lambda).$$

Note that the operator $\NN_{\s{\alpha}}(w\s{\alpha}\lambda)$ is not holomorphic at $\lambda =\chi$ since 
$$\inner{w\s{\alpha}\chi ,\check{\alpha}} =\inner{\s{\alpha}w\s{\alpha}\chi ,\s{\alpha}\check{\alpha}} = -\inner{\chi,\check{\alpha}}=-1.
$$

Hence, the point $\chi$ belongs to the intersection of two hyperplanes.
On one of them, $H_{\alpha}^{1}(\lambda),$
the operator has a kernel, 
and on the other
$H_{\s{\alpha} w^{-1} \alpha}^{-1}(\lambda),$ it has a singularity.
Here
$$    H_{\beta}^{\epsilon}(\lambda) = \set{\lambda \in \mathfrak{a}_{T,\C}^{\ast} \: : \: \inner{\lambda,\check{\beta}}=\epsilon}$$
for any  $\beta \in \Phi_{G}$.

Thus, the operator $\NN_{u}(\lambda)$ is not defined at $\lambda =\chi$.
However, we are able to  define its restriction to a line $L$ and then to take
the value at $\chi$. 
The restriction to each line will define a different operator,
and for almost all of the lines, the restriction will be holomorphic at $\chi$. 
We start with some simple observations:
\begin{itemize}
\item
$w\s{\alpha} \chi =  \s{\alpha} \chi$ i.e $w \in \Stab_{\weyl{G}}(\s{\alpha}\chi)$.
\item
The set $\set{\check{\alpha}, \check{\s{\alpha}w^{-1}\alpha}}$ is linear independent.
For this, it suffices to show that the set $\set{\check{\alpha}, w\check{\alpha}}$
is linear independent. Note that $w\alpha > 0,$ since $\s{\alpha} \not \in W(\s{\alpha}\chi)$, and $w\alpha \neq \alpha$; otherwise, $\s{\alpha} w \s{\alpha} \alpha = \alpha$ and hence $l(u) \neq 2 +l(w)$.   
\item 
Let $\set{\check{\gamma}_1,\check{\gamma}_2 ,\dots , \check{\gamma}_{n}}$ be a linearly independent set of elements of the coroots space, where $\check{\gamma}_1=\check{\alpha}$ and $\gamma_2 = (\s{\alpha} w^{-1})\check{\alpha}$. It exists  by  the previous item.
Define the set of affine functionals $l_i$ on $\mathfrak{a}_{T,\C}^{\ast}$ by
 $$l_i(\lambda) = \inner{\lambda,\check{\gamma}_i} -\inner{\chi,\check{\gamma}_i}.$$
 The set $\set{ l_{i}}_{i=1}^{n}$ is  linearly independent and defines a coordinate system for
 $\mathfrak{a}_{T,\C}^{\ast}$ with an origin at $\chi$. 
\item
The operator $\NN_{\s{\alpha}}(\chi)$ has a non-trivial kernel.
\item
The operator $\NN_{w}(\lambda)$ is holomorphic at $\lambda =\s{\alpha}\chi$. Moreover, $\NN_{w}(\s{\alpha}\chi)=\id$. It suffices to show that for each $\s{\beta} \in W(\s{\alpha}\chi)$, the operator $\NN_{\s{\beta}}(\lambda)$ is holomorphic at $\lambda =\s{\alpha} \chi$ and  acts as an identity. 

Note that by induction in stages $$\Ind_{T}^{G}(\s{\alpha}\chi) \simeq \Ind_{M_{\beta}}^{G} \left( \left(\Ind_{T}^{M_{\beta}} \Id \right) \otimes \s{\alpha}\chi  \right).$$ 
The representation $\Ind_{T}^{M_{\beta}} \Id$ is irreducible. By the properties of the normalized intertwining operator, the action of $\NN_{\s{\beta}}(\s{\alpha}\chi)$ factors through its action on $\Ind_{T}^{M_{\beta}} \Id$. Since  $\NN_{\s{\beta}}(0)
\in \operatorname{End}(\Ind_{T}^{M_{\beta}} \Id ) $ is holomorphic on the representation $\Ind_{T}^{M_{\beta}} \Id$, it fixes the spherical vector
and hence acts  as an identity.
It follows that $\NN_{w}(\s{\alpha}(\lambda))$ is holomorphic at $\lambda =\chi$ and acts as an identity, as a composition of such operators.
\end{itemize}
After these preparations, we are ready to prove the claim. For this purpose, we use the following factorization 
\begin{equation}\label{zampera::factor}
\NN_{u}(\lambda) =  \NN_{\s{\alpha}}(w\s{\alpha} \lambda) \circ \NN_{w}( \s{\alpha} \lambda) \circ \NN_{\s{\alpha}}(\lambda).
\end{equation}

For each term, we write the Laurent expansion. The operator $\NN_{\s{\alpha}}(\lambda)$  is a meromorphic function and  depends only on $\inner{\lambda,\check{\alpha}}$. We have the following expansion at $\lambda = \chi$:
\begin{equation}\label{zampera::term1}
\NN_{\s{\alpha}}(\lambda) =  A_{0} + l_{1}(\lambda) A_{1}  + \sum_{i_1=2}^{\infty}  l_{1}^{i_1}(\lambda) A_{i_1}
\end{equation}     
where $\Image A_{0} = \Ind_{M_{\alpha}}^{G}(\chi  -\frac{1}{2}\alpha)$.

We have already shown that $\NN_{w}(\s{\alpha} \chi)$ is holomorphic and acts as an identity.
One has the following expansion of $\NN_{w}(\s{\alpha}\lambda)$ at $\lambda=\chi$:
\begin{equation}\label{zampera::term2}
\NN_{w}(\s{\alpha}\lambda)= \id + \sum_{|I| \geq 1} \underbar{l}(\lambda)^{I} B_{I}
\end{equation}
where $I=(i_1,\dots, i_n)  \in \N^{n}$, $|I| =  \sum_{j=1}^{n} i_j$ and 
$\underbar{l}(\lambda) = \prod_{j=1}^{n} l_{j}^{i_j}(\lambda)$.

Finally, for $\NN_{\s{\alpha}}(w\s{\alpha} \lambda)$, the hyperplane $l_{2}(w\s{\alpha} \lambda)$ is singular. Hence 
\begin{equation}\label{zampera::term3}
\NN_{\s{\alpha}}(w\s{\alpha} \lambda) = \frac{C_{-1}}{l_{2}(\lambda)} + C_{0} +  \sum_{i_2=1}^{\infty} l_{2}^{i_2}(\lambda)C_{i_2}.
\end{equation}
Plugging  all together in \eqref{zampera::factor}, one has 
\begin{align}
\NN_{u}(\lambda) &=  \frac{1}{l_{2}(\lambda)} C_{-1}A_{0} +C_{0}A_{0}   \nonumber \\
&+ \sum_{j>2}^{n} \frac{l_{j}(\lambda)}{l_{2}(\lambda)} C_{-1}  B_{e_{j}} A_{0} + \frac{l_1(\lambda)}{l_2(\lambda)} \left( C_{-1} A_{1} +  C_{-1}  B_{e_1} A_0 \right) \nonumber \\
&+ \sum_{ |I|\geq 2}l(\lambda)^{I} D_{I}.   \label{zampera :: plug_all}
\end{align}

Note that for every $j \neq 2$ one has  
\begin{align*}
a_{j} = \frac{l_{j}(\lambda)}{l_{2}(\lambda)}\res{L} =  \frac{(z-z_0) \inner{m, \check{\gamma}_j}}{(z-z_0) \inner{m, \check{\gamma}_2}} = \frac{ \inner{m, \check{\gamma}_j}}{ \inner{m, \check{\gamma}_2}}. 
\end{align*}

By the functional equation for  the normalized intertwining operator
$\NN_{\s{\alpha}}(\lambda)$, it follows that for every $\lambda'$ one has 
\begin{equation}\label{zampera:: funtional}
 \NN_{\s{\alpha}}(\lambda') \circ \NN_{\s{\alpha}}(\s{\alpha}\lambda')  =\id.
\end{equation}

The Laurent expansion of $\NN_{\s{\alpha}}(\s{\alpha}\lambda')$ for $\lambda' =  w\s{\alpha}\lambda$ at $\lambda=\chi$ is 
 
$$ \NN_{\s{\alpha}}(\s{\alpha}\lambda^{'}) =  A_{0} + \sum_{i_1=1}(-1)^{i_{1}} l_{2}(\lambda)^{i_1} A_{i_1}$$ 
since  $\inner{\s{\alpha} \lambda^{'},\check{\alpha}}= - \inner{\lambda^{'},\check{\alpha}}
$.

Plugging in \eqref{zampera:: funtional} we obtain 

$$ \left( \frac{C_{-1}}{l_{2}(\lambda)} + C_{0} + \sum_{i=1}^{\infty}l_{2}^{i}(\lambda) \right)
\left(
A_{0} + \sum_{i=1}^{\infty}(-1)^{i}l_{2}(\lambda)^{i}A_{i} 
\right) =\id,
 $$ 
i.e.,
\begin{eqnarray}
C_{-1}A_{0} &= 0.\\
C_0A_0 -  C_{-1}A_1 &=\id.
\end{eqnarray}
Substituting in  \eqref{zampera :: plug_all}, restricting to the line $L$ and evaluating at $\chi$ one has 
\begin{equation} \label{formula::E}
E(\chi) = \NN_{u}(\lambda) \res{L} =  \left(C_{0}A_{0}+a_1 \left( C_{0}A_{0}-\id \right)  \right)   + C_{-1}\left(\sum_{j=1}a_{j} B_{e_{j}}\right) A_{0}.  
\end{equation}
In particular, $\NN_{u}(\lambda) \res{L}$ is holomorphic at $\chi$.
\end{proof}

For $u=\s{\alpha}w \s{\alpha}, \chi,L$ as in \Cref{Zampera ::lemma_holo}.
Define
$a_1 = \frac{\inner{m,\check{\alpha}}}{\inner{m,\s{\alpha}w^{-1}\check{\alpha}}}$
and assume $L$ is chosen such that $a_1\neq -1$.
Let $E = E(\chi)$ be the operator $\NN_{u}(\lambda) \res{L}$ at $\lambda=\chi$.
The following proposition concerns the action of $E(\chi)$ on $\Ind_{T}^{G}(\chi)$.
\begin{Remark}
Note that  $E\in \operatorname{End}_{G}(\Ind_{T}^{G}(\chi))$.
\end{Remark}
\begin{Prop} \label{Zampera::action}
Assume $a_1 \neq -1$, then 
\begin{enumerate}[ref=\Cref{Zampera::action}.(\arabic{*})  ]
\item
$\Ind_{T}^{G}(\chi) =  V_1 \oplus V_{-a_1}$ where 
\begin{align*}
E\res{V_1} &=\id & E\res{V_{-a_1}} &=-a_1\id.
\end{align*}
In addition, $\ker \NN_{\s{\alpha}}(\chi) \subseteq V_{-a_1}$. 
\item\label{Zampera::action:2}
In addition, if  $\Ind_{T}^{G}(\s{\alpha} \chi)$ admits a unique irreducible subrepresentation,
then $$V_1\simeq \Image \NN_{\s{\alpha}}(\chi) \quad \text { and } \quad   V_{-a_1} =  \ker \NN_{\s{\alpha}}(\chi).$$
\end{enumerate}
\end{Prop}
\begin{proof} 
In the proof of \Cref{{Zampera ::lemma_holo}},  we found that 
$$E= \left(C_{0}A_{0}+a_1 \left( C_{0}A_{0}-\id \right)  \right)   + C_{-1}\left(\sum_{j=1}a_{j} B_{e_{j}}\right) A_{0}.  $$
Thus  it suffices to show that
\begin{equation} \label{zampera:: minimal_poly}
(E+a_{1}\id)(E-\id)=0 \iff E^{2} =(1-a_1)E +a_1\id. 
\end{equation}

Since  $a_1\neq -1$, the above implies  that $E$ is a diagonalizable operator
with at most two eigenvalues $1,-a_1$.  
 Showing \eqref{zampera:: minimal_poly} requires the following identities:
 \begin{eqnarray}
 A_0C_{-1} =& 0, \\
 A_0C_0 -  A_1C_{-1} =& \id, \\ 
 C_0A_0C_0A_0  =&  C_0 A_0, \\
 C_{-1} B_{e_j} A_0 C_0 A_0 =& C_{-1} B_{e_j} A_0.  
 \end{eqnarray}

Observe that if  $a_1 \neq -1$, then $P =  \frac{1}{a_1+1} \bk{ \id - E   }$ is a projection, which gives the decomposition of $\Ind_{T}^{G}(\chi)$ to a direct sum.

The operator $E$ fixes the spherical vector, and hence $1$ is an eigenvalue. 
 On the other hand, $E\res{\ker \NN_{\s{\alpha}}(\chi)} = -a_1\id$, which implies also that the $-a_1$-
 eigenspace is not empty. 
Thus $$\Ind_{T}^{G}(\chi) =  V_1 \oplus V_{-a_1},$$ 
which finishes the first part.

For the second part, we already know that $\Ind_{T}^{G}(\chi) = V_{1} \oplus V_{-a_1}$, as we showed that $\ker \NN_{\s{\alpha}}(\chi) \subset  V_{-a_1}$.
Assume that $V_{-a_1} \neq \ker \NN_{\s{\alpha}}(\chi)$. Then the natural projection of $\Image \NN_{\s{\alpha}}(\chi)$ on $V_1$ and $V_{-a_1}$ has at least two irreducible subrepresentations.
However, by our assumption, $\Image \NN_{\s{\alpha}}(\chi)$ admits a unique irreducible subrepresentation. Hence  $V_{-a_1} = \ker \NN_{\s{\alpha}}(\chi)$ and $V_1 \simeq \Image \NN_{\s{\alpha}}(\chi)$.    
\end{proof}   
Keeping the notations as in \Cref{Zampera::action}, one has 
\begin{Prop}\label{Zampera::length::two}
If $\inner{\w_{\alpha} \chi ,\check{\beta}} \leq 0$ for every $\beta \in \Delta_{G}$, then $\Ind_{T}^{G}(\chi)$ admits a maximal semi-simple subrepresentation of length two.   
\end{Prop}
\begin{proof}
By Langland's subrepresentation theorem,
the representation $\Ind_{T}^{G}(\w_{\alpha} \chi)$ admits a unique irreducible subrepresentation. In particular, $V_{1}$ admits a  unique irreducible subrepresentation. Thus, it is enough to prove
that $V_{-a_1}$ admits a unique irreducible subrepresentation. By \ref{Zampera::action:2}, $V_{-a_1} = \ker \NN_{\w_{\alpha}}(\chi)$. We show that $\ker \NN_{\alpha}(\chi)$ is a standard module and
hence  admits a unique irreducible subrepresentation. Recall that 
$$\Ker \NN_{\w_{\alpha}}(\chi) =  \Ind_{M_{\alpha}}^{G}(St, \chi -\frac{1}{2}\alpha).$$

Note that $$ \inner{\chi -\frac{1}{2}\alpha,\check{\beta}} \leq 0 \quad \forall \beta \in \Delta_{G} \setminus \set{\alpha}.$$
Indeed, 
\begin{align}
\inner{\chi -\frac{1}{2} \alpha,\check{\beta}} &=\inner{\w_{\alpha}\bk{\chi -\frac{1}{2} \alpha},\w_{\alpha}\bk{\check{\beta}}} \nonumber \\ 
&=  \inner{\w_{\alpha}\chi,\check{\beta} - \inner{\alpha,\check{\beta}}\check{\alpha}} + \frac{1}{2} \inner{\alpha,
\check{\beta} - \inner{\alpha,\check{\beta}}\check{\alpha} 
 } \nonumber\\
 &= \inner{\w_{\alpha} \chi ,\check{\beta}} - \inner{\alpha,\check{\beta}} \underbrace{\inner{\w_{\alpha}\chi,\alpha}}_{=-1} -\inner{\alpha,\check{\beta}} +\frac{1}{2}\inner{\alpha,\check{\beta}} \nonumber \\
 &= \inner{w_{\alpha}\chi,\check{\beta}}+\frac{1}{2}\inner{\alpha,\check{\beta}} \leq 0.
\end{align}
In addition, if $\inner{\alpha,\check{\beta}} <0$, it follows that 
$ \inner{\chi-\frac{1}{2}\alpha, \check{\beta}} \neq 0$. Otherwise,
$$  \inner{\w_{\alpha}\chi ,\check{\beta}} = -\frac{1}{2}\inner{\alpha,\check{\beta}}.$$ 
However, the LHS is a non-positive number, while the RHS is a positive number, which leads  to a contradiction. 

Set $\Theta = \set{\beta \in \Delta_{G}: \: \:  \inner{\chi -\frac{1}{2}\alpha , \check{\beta}} =0 }.$ 
Then , induction in stages implies that  
$$\Ind_{M_{\alpha}}^{G}(St, \chi -\frac{1}{2}\alpha) \simeq  \Ind_{M_{\alpha} \times M_{\Theta}}^{G}(St_{M_{\alpha}}\boxtimes\Id_{M_{\Theta}}, \chi -\frac{1}{2}\alpha ).$$

The representation $\Id_{M_{\Theta}}$ is tempered and irreducible. Hence, $St_{M_{\alpha}}\boxtimes\Id_{M_{\Theta}}$ is a tempered and irreducible representation of $M_{\alpha} \times M_{\Theta}$. Thus, $V_{-a_1}$ is a standard module.     
\end{proof}

\newpage 
\chapter{ Branching rule triples}
\label{App:knowndata}
In this section, we list the various branching triples used in this thesis.
We  make a list of various branching triples associated with Levi subgroups of $G$, organized by the type of the Levi subgroup. Most of these rules arise from irreducible degenerate principal series of the Levi subgroup, while some are proven by other methods.

Note that it is convenient to encode the branching triples in terms of the action of Weyl elements. Here, $H$ is the derived subgroup of a Levi subgroup of $G$.

\section{Different types of triples}
\subsection*{Orthogonality rule}
We recall the Orthogonality Rule from \cite[A.1]{E6}. Let $\lambda \in \mathfrak{a}_{T,\C}^{\ast}$ and set $\Theta_{\lambda}= \set{\alpha \: : \:  \alpha \in \Delta_{G}, \; \inner{\lambda,\check{\alpha}}=0}$. Then 
\begin{framed}
\begin{equation}\tag{OR}
\label{Eq:OR}
\lambda \leq\jac{G}{T}{\pi} \Longrightarrow \Card{W_{M_{\Theta_\lambda}}} \big \vert \mult{\lambda}{\jac{G}{T}{\pi}}.
\end{equation}
\end{framed}
\begin{framed}
\begin{equation}\label{Eq:A1}
\tag{$A_1$}
\lambda\leq \jac{H}{T}{\pi},\ \gen{\lambda,\check{\alpha}} \neq \pm 1 \Longrightarrow \coset{\lambda}+\coset{s_\alpha\lambda} \leq \coset{\jac{H}{T}{\pi}} .
\end{equation}
\end{framed}

\subsection*{Triple of type $A_2$}
We label the Dynkin diagram of a group of type $A_2$ as follows:
$$
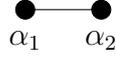
\begin{tikzpicture}[scale=0.5]
\draw (-1,0) node[anchor=east]{};
\draw (0 cm,0) -- (2cm,0);
\draw[fill=black] (0 cm, 0 cm) circle (.25cm) node[below=4pt]{$\alpha_1$};
\draw[fill=black] (2 cm, 0 cm) circle (.25cm) node[below=4pt]{$\alpha_2$};
\end{tikzpicture}
$$  

We recall the branching triple of type $A_2$ \cite[A.3]{E6}.

\begin{framed}
	\begin{equation}\tag{$A_2$} \label{Eq:A2}  
	\lambda\leq \jac{H}{T}{\pi},\ \lambda \in \set{\pm \fun{1}} \Longrightarrow 2\times\coset{\lambda}+\coset{\s{\alpha_1}\lambda} \leq \coset{\jac{H}{T}{\pi}} .
	\end{equation}
\end{framed}

\section{Triples of type $C_n$}
\subsection{Triples of type $C_2$}

We label the Dynkin diagram of a group of type $C_2$ as follows:
$$
\begin{tikzpicture}[scale=0.5]
\draw (0 cm,0) -- (0 cm,0);
\draw (0 cm, 0.1 cm) -- +(2 cm,0);
\draw (0 cm, -0.1 cm) -- +(2 cm,0);
\draw[shift={(1, 0)}, rotate=180] (135 : 0.45cm) -- (0,0) -- (-135 : 0.45cm);
\draw[fill=black] (0 cm, 0 cm) circle (.25cm) node[below=4pt]{$\alpha_1$};
\draw[fill=black] (2 cm, 0 cm) circle (.25cm) node[below=4pt]{$\alpha_2$};
\end{tikzpicture}
$$

\begin{Thm} $ $
\begin{enumerate}
\item
There is a unique irreducible representation $\pi$ of $H$ having $\lambda_{a.d}^{1} \leq \jac{H}{T}{\pi},$ where $\lambda_{a.d}^{1}=-\fun{2}$. In that case 
$$\coseta{\jac{H}{T}{\pi}}  =  2\times\coset{\lambda_{a.d}^{1}}+\coset{\s{\alpha_2}\lambda_{a.d}^{1}}
	+\coset{\s{\alpha_1}\s{\alpha_2}\lambda_{a.d}^{1}}. $$
\item
There is a unique irreducible representation $\pi$ of $H$ having $\lambda_{a.d}^{2} \leq \jac{H}{T}{\pi},$ where $\lambda_{a.d}^{2}=-\fun{1}$. In that case
	$$\coseta{\jac{H}{T}{\pi}} = 2\times\coset{\lambda_{a.d}^{2}}+\coset{\s{\alpha_1}\lambda_{a.d}^{2}}. 
	 $$
\end{enumerate}
\end{Thm}
\begin{proof}
In both cases, $\pi$ is an irreducible constituent of $\Pi_i=\Ind_{T}^{H}(\lambda_{a.d}^{i})$ according to the central character argument. 
\begin{enumerate}
\item
Let $\pi_1 =  \Ind_{M_1}^{H}(1)$. According to \cite[Theorem 4.2.1]{MR2611917}, $\pi_1$ is irreducible. Geometric Lemma  implies that 
$$\coseta{\jac{H}{T}{\Pi_1}}  =  2\times\coset{\lambda_{a.d}^{1}}+\coset{\s{\alpha_2}\lambda_{a.d}^{1}}
	+\coset{\s{\alpha_1}\s{\alpha_2}\lambda_{a.d}^{1}}. $$
The uniqueness follows from the fact that $\mult{\lambda_{a.d}^{1}}{\jac{H}{T}{\Pi_1}} =
\mult{\lambda_{a.d}^{1}}{\jac{H}{T}{\pi_1}}$.
\item
Let $\pi_2 =\Ind_{M_1}^{H}(0)$. According to \cite[Theorem 4.2.1]{MR2611917}, $\pi_2$  is of length two.  Note that $\pi_2 \hookrightarrow \Ind_{T}^{H}(\w_{\alpha_1}\lambda_{a.d}^{2})$. Geometric Lemma implies that 
$$ \coseta{\jac{H}{T}{\pi_2}} = 2\times \lambda_{a.d}^{2} + 2 \times \w_{\alpha_1}\lambda_{a.d}^{2}.$$
Let $\sigma_0$ be an irreducible constituent of $\pi_2$ having $\lambda_{a.d}^{2}\leq \jac{H}{T}{\sigma_0}$. \eqref{Eq:OR} implies that $ 2 \vert  \mult{\lambda_{a.d}^{2}}{\jac{H}{T}{\sigma_0}} \leq 2$. Hence, $$\mult{\lambda_{a.d}^{2}}{\jac{H}{T}{\sigma_0}} = 2.$$  
Since $\pi_2$ is unitary, it is completely reducible. Frobenius  reciprocity asserts that  each summand $\tau$ of $\pi_2$  has $ \w_{\alpha_1}\lambda_{a.d}^{2} \leq \jac{H}{T}{\tau}$. Hence, 
$$\coseta{\jac{H}{T}{\sigma_0}} \leq 2 \times \lambda_{a.d}^{2} + \w_{\alpha_1}\lambda_{a.d}^{2}.$$ Now, by comparing Jacquet modules, we deduce that 
$$ \coseta{\jac{H}{T}{\sigma_0}} =2\times \lambda_{a.d}^{2} +  \w_{\alpha_1}\lambda_{a.d}^{2}.$$
 The uniqueness is as before.
\end{enumerate}

\end{proof}

 We conclude that we have the following triples of type $C_2$: 
\begin{framed}
	\begin{equation}\tag{$C_2(a)$}\label{Eq:B2a}
	\lambda\leq \jac{H}{T}{\pi},\ \lambda \in \set{\pm \fun{2}} \Longrightarrow 2\times\coset{\lambda}+\coset{\s{\alpha_2}\lambda}
	+\coset{\s{\alpha_1}\s{\alpha_2}\lambda} 
	 \leq \coset{\jac{H}{T}{\pi}} .
	\end{equation}
\end{framed}

\begin{framed}
	\begin{equation}\label{Eq::C2b}\tag{$C_2(b)$}
	\lambda\leq \jac{H}{T}{\pi},\ \lambda \in \set{\pm \fun{1}} \Longrightarrow 2\times\coset{\lambda}+\coset{\s{\alpha_1}\lambda} 
	 \leq \coset{\jac{H}{T}{\pi}} .
	\end{equation}
\end{framed} 

 \subsubsection*{Triples  of type $C_3$}
In this subsection we list some branching triples for the group $H =  Sp_{6}$. For this purpose, we label the Dynkin diagram of $H$ as follows:
 $$
 \begin{tikzpicture}[scale=0.5]
 \draw (-1,0) node[anchor=east] {};
 \draw (0 cm,0) -- (2 cm,0);
 \draw (2 cm, 0.1 cm) -- +(2 cm,0);
 \draw (2 cm, -0.1 cm) -- +(2 cm,0);
 \draw[shift={(2.8, 0)}, rotate=180] (135 : 0.45cm) -- (0,0) -- (-135 : 0.45cm);
 \draw[fill=black] (0 cm, 0 cm) circle (.25cm) node[below=4pt]{$\alpha_1$};
 \draw[fill=black] (2 cm, 0 cm) circle (.25cm) node[below=4pt]{$\alpha_2$};
 \draw[fill=black] (4 cm, 0 cm) circle (.25cm) node[below=4pt]{$\alpha_3$};
 \end{tikzpicture}
 $$

In \Cref{app:C3:}, we describe the semi-simplification of $\Pi =  \Ind_{T}^{H}(-\fun{2})$ and the Jacquet module of each irreducible constituent. We derive the relevant branching triples from this proposition.
Let us specify elements in $\weyl{H} \cdot \bk{-\fun{2}}$. Set
\begin{align*}
\lambda_{0} &=  -\fun{2}, & \lambda_{1} &=  \w_{\alpha_3}\w_{\alpha_2}\lambda_{a.d}= -\fun{1}-\fun{2}+\fun{3},  \\
 \lambda_{2} &=\fun{1} -\fun{3}.  &\lambda_3 &= 2\fun{1}-\fun{2}.   
\end{align*}

\begin{Prop}\label{app:C3:}
Let $\Pi = \Ind_{T}^{H}\bk{-\fun{2}}$. 
\begin{enumerate}
\item
$D_{H}(\Pi) = \Pi$, where $D_{H}(\cdot)$ stands for the Aubert involution.
\item
There exist irreducible representations $\sigma_{i}$ for $0 \leq i \leq 4$ such that:
\begin{itemize}
\item
$\lambda_i \leq \jac{H}{T}{\sigma_j}$ if and only if $j=i$ for $i\in \set{0,2,3}$.
\item
$\jac{H}{T}{\sigma_1} =\lambda_1$.
\item
$\coseta{\jac{H}{T}{\sigma_3}} = \coseta{\jac{H}{T}{\sigma_4}}$.
\item
$(\Pi)_{s.s} =  \tau \oplus D_{H}(\tau) \oplus \sigma_3 \oplus \sigma_4$ where $\tau =  \sigma_0 \oplus \sigma_1 \oplus 2 \times \sigma_2$. Moreover, $\Pi$ is of length $10$.
\end{itemize}  
In \Cref{Table::C3:: a} we describe the Jacquet module of each irreducible constituent.
\end{enumerate}
\end{Prop} 
\begin{proof}
The first part follows from \eqref{aub::prinipalseries}. For the second part, we start by making a distinction between different constituents of $\Pi$. For this purpose, we recall that for every $\lambda \in \set{\lambda_{0},\lambda_1,\lambda_2,\lambda_3}$, one has $\bk{\Ind_{T}^{H}(\lambda)}_{s.s} =(\Pi)_{s.s}$, due to the fact that these elements belong to the same Weyl-orbit. In order to show the existence of such irreducible constituents, we proceed as follows: we consider $\Sigma_1,\Sigma_2,\Sigma_3$ representations of $H$ such that $(\Sigma_i)_{s.s} \leq (\Pi)_{s.s}$. These representations help us to make a distinction between different constituents of $\Pi$. For this purpose,   
we fix the following representations:
\begin{align*}
\Sigma_1 &= \Ind_{M_3}^{H}(0) \hookrightarrow \Ind_{T}^{H}\left(\lambda_1 \right),  \\ 
\Sigma_2 &= \Ind_{M_2}^{H}(\frac{1}{2}\fun{2}) \hookrightarrow \Ind_{T}^{H}(-\fun{1} + 2\fun{2} - \fun{3}), \\
  \Sigma_3 &= \Ind_{M_{\alpha_3}}^H(\fun{1}) \hookrightarrow \Ind_{T}^{H}(-\lambda_1).
\end{align*}      

\begin{Lem}\label{app::c3::lemma1} The representation
$\Sigma_1 =  \sigma_{0} \oplus \sigma_1$, and $\bk{\Sigma_{2}}_{s.s}=\sigma_{0} \oplus \sigma_{2}$. Moreover, 
\begin{align*}
 \coseta{\jac{H}{T}{\sigma_{0}}}&= 4 \times\lambda_{0} + 2 \times 
 \s{\alpha_2}\lambda_{0} + \s{\alpha_3}\s{\alpha_2}\lambda_{0}, \\ 
 \coseta{\jac{H}{T}{\sigma_{1}}}&=  \s{\alpha_3}\s{\alpha_2}\lambda_{0} = \lambda_{1},\\
 \coseta{\jac{H}{T}{\sigma_{2}}}& = 2\times \lambda_{2} + \w_{\alpha_3}\lambda_2 + \w_{\alpha_2}\w_{\alpha_3}\lambda_2 + \w_{\alpha_1}  \lambda_2. 
 \end{align*}
\end{Lem}
\begin{proof}
Using the Geometric Lemma, one has 
\begin{align}
\coseta{\jac{H}{T}{\Sigma_1}} &=4 \times \lambda_{0} + 2 \times 
 \s{\alpha_2}\lambda_{0}+ 2\times \s{\alpha_3}\s{\alpha_2}\lambda_{0},
 \label{C3::Sigma1::jac}  \\
\coseta{\jac{H}{T}{\Sigma_2}} &=4 \times \lambda_{0} + 3 \times 
 \s{\alpha_2}\lambda_{a.d}+ \s{\alpha_3}\s{\alpha_2}\lambda_{0} 
 \label{C3::Sigma2::jac} \\
 &+2\times \lambda_{2} + \w_{\alpha_3}\lambda_2 + \w_{\alpha_2}\w_{\alpha_3}\lambda_2. \nonumber  
\end{align}
In order to show that $\Sigma_1,\Sigma_2$ are both reducible, we apply  Tadic's criterion \cite{Tadic}   with the arguments $(\Pi,\Sigma_1,\Sigma_2,\lambda_{0})$. Indeed, \eqref{C3::Sigma1::jac},\eqref{C3::Sigma2::jac} yield  $\jac{H}{T}{\Sigma_1} \not \leq \jac{H}{T}{\Sigma_2}$ and $\jac{H}{T}{\Sigma_2} \not \leq \jac{H}{T}{\Sigma_1}$. Moreover, since
\begin{equation}\label{C3::multi::sigma_anti}
\mult{\lambda_{0}}{\jac{H}{T}{\Pi}}=\mult{\lambda_{0}}{\jac{H}{T}{\Sigma_1}}=\mult{\lambda_{0}}{\jac{H}{T}{\Sigma_2}}=4
\end{equation} 
the reducibility follows.
 In particular, each of $\Sigma_1,\Sigma_2$ is of a length of at least two. We continue by showing that $\Sigma_1$ is of length two. Since
$\Sigma_1 \hookrightarrow \Ind_{T}^{H}(\lambda_1)$, Frobenius reciprocity implies that  
for each subrepresentation $\pi$  of $\Sigma_1$, one has
\begin{equation}\label{C3:Sigma1 :propoety}
\lambda_1 \leq \jac{H}{T}{\pi}.
\end{equation} 

 The representation $\Sigma_1$ is unitary and hence completely reducible; consequently, each irreducible constituent of $\Sigma_1$ satisfies \eqref{C3:Sigma1 :propoety}. Let $\sigma_{0}$ be an irreducible constituent of $\Sigma_{1}$ such that $\lambda_{a.d} \leq \jac{H}{T}{\sigma_{0}}$.
 Applying \eqref{Eq:OR}, \eqref{Eq:A2} and \eqref{C3:Sigma1 :propoety} yields
 \begin{eqnarray}\label{eq::sigma_c3::anti1}
  4 \times\lambda_{0} + 2 \times 
  \s{\alpha_2}\lambda_{0}+ \underbrace{\s{\alpha_3}\s{\alpha_2}\lambda_{0}}_{\lambda_1}
  \leq \coseta{\jac{H}{T}{\sigma_{0}}}.
 \end{eqnarray}

Let $\sigma_1$ denote another irreducible subrepresentation of $\Sigma_1$. By \eqref{C3:Sigma1 :propoety} one has $\lambda_1 \leq \jac{H}{T}{\sigma_1}$. Comparing  the Jacquet modules   $$\jac{H}{T}{\Sigma_1}_{s.s} =\jac{H}{T}{\sigma_{0}}_{s.s} +  \jac{H}{T}{\sigma_1}_{s.s}.$$ Thus, there is an equality in \eqref{eq::sigma_c3::anti1}.
Hence $\Sigma_1 =  \sigma_0 \oplus \sigma_1$. In particular,  
\begin{eqnarray}\label{eq::sigma_c3::anti}
  \coseta{\jac{H}{T}{\sigma_{0}}}=
  4 \times\lambda_{0} + 2 \times 
  \s{\alpha_2}\lambda_{0}+ \underbrace{\s{\alpha_3}\s{\alpha_2}\lambda_{0}}_{\lambda_1} .
 \end{eqnarray}

By \eqref{C3::multi::sigma_anti}, $\sigma_0$ is the unique irreducible constituent $\pi$ of $\Pi$ having $\lambda_{0} \leq \jac{H}{T}{\pi}$. Thus, $\sigma_0$ is an irreducible constituent of $\Sigma_2$.  By \eqref{eq::sigma_c3::anti}, $\lambda_{2} \not \leq \jac{H}{T}{\sigma_0}$. Hence, there exists  an irreducible constituent $\sigma_2$ of $\Sigma_2$ having $\lambda_2 \leq \jac{H}{T}{\sigma_2}$. Applying \eqref{Eq:OR}, \eqref{Eq:A2},\eqref{Eq:B2a} yields 
\begin{equation}\label{eq::sigma_c3::sigma_2}
2\times \lambda_{2} + \w_{\alpha_3}\lambda_2 + \w_{\alpha_2}\w_{\alpha_3}\lambda_2 + \underbrace{\w_{\alpha_1}  \lambda_2}_{\w_{\alpha_2}\lambda_{0}}  \leq \coseta{\jac{H}{T}{\sigma_2}}.
\end{equation}
Again, by Jacquet module considerations, it follows that $\bk{\Sigma_2}_{s.s} = \sigma_{0} \oplus \sigma_{2}$, and there is an equality in \eqref{eq::sigma_c3::sigma_2}.
 \end{proof}
 
 \begin{Lem}\label{C3::sigma2::multi::2} 
 The representation 
 $\sigma_2$ appears with multiplicity two in $\Pi$. Moreover, $\sigma_2$ is the unique irreducible constituent of $\Pi$ having $\lambda_2$ in its Jacquet module. 
 \end{Lem}
 \begin{proof}
 Put $\lambda =\s{\alpha_2}\lambda_{a.d}$. By \Cref{app::c3::lemma1}, one has $$\mult{\lambda}{\jac{H}{T}{\sigma_{0}}} =2, \quad \mult{\lambda}{\jac{H}{T}{\sigma_{2}}} =1. $$
 Since $\mult{\lambda}{\jac{H}{T}{\Pi}}=4$, it follows that there exists another irreducible constituent $\sigma$ of $\Pi$ having $\mult{\lambda}{\jac{H}{T}{\sigma}}=1$. If $\sigma \not \simeq \sigma_2, \sigma \not \simeq \sigma_{0}$, then  by the central character argument, one has 
$$\sigma_{0} \oplus \sigma_{2} \oplus \sigma \hookrightarrow \Ind_{T}^{H}(\lambda).$$
This contradicts  the fact that $\Ind_{T}^{H}(\lambda)$ admits a maximal semi-simple subrepresentation of length two. The last assertion follows from  an application of \Cref{Zampera::length::two} with the arguments 
 \begin{align*}
 \chi &=\lambda, &  u&= \w_{\alpha_2}\w_{\alpha_1}\w_{\alpha_3}\w_{\alpha_2}, & w &=  \w_{\alpha_1}\w_{\alpha_3},& m&=\fun{2}.  
 \end{align*}

 Hence, $\sigma \simeq \sigma_2$ or  $\sigma \simeq \sigma_{0}$. Since $\mult{\lambda_{0}}{\jac{H}{T}{\sigma_{0}}} = \mult{\lambda_{0}}{\jac{H}{T}{\Pi}}=4$, by Jacquet module considerations, $\sigma \not \simeq \sigma_{0}$. Thus, $\sigma  \simeq \sigma_2$. In particular, $\sigma_2$ appears with multiplicity two in $\Pi$. Since
 $\mult{\lambda_2}{\jac{H}{T}{\Pi}}=4$,
  $\mult{\lambda_2}{\jac{H}{T}{\sigma_2}}=2$ (by \eqref{eq::sigma_c3::sigma_2}) and $\sigma_2$ appears with multiplicity two, the second part follows.   
 \end{proof}

\begin{Lem}
There is  an irreducible constituent $\sigma_3$ of $\Pi$ such that  
$$\coseta{\jac{H}{T}{\sigma_3}} =		
2 \times \lambda_3 + 2 \times \s{\alpha_1}\lambda_3 + \s{\alpha_2}\lambda_3 + \s{\alpha_2}\s{\alpha_1}\lambda_3.$$
\end{Lem}
\begin{proof}
Let $\sigma_3$ be an irreducible constituent of $\Pi$ having $m= \mult{\lambda_3}{\jac{H}{T}{\sigma_3}}\neq 0$.
By \eqref{Eq:OR} it follows that $2 \big \vert m$. Since $\mult{\lambda_3}{\jac{H}{T}{\Pi}}=4$, it follows that $m \in \set{2,4}$. We start by showing that $m=2$, i.e., there are exactly two irreducible constituents of $\Pi$ having $\lambda_3$ in their Jacquet module. For this purpose, we consider the representation $\Sigma_3 =\Ind_{M_{\alpha_3}}^H(\fun{1}) \hookrightarrow \Ind_{T}^{H}(-\lambda_1) $. By the Geometric Lemma, one has 
\begin{align*}
\coseta{\jac{H}{T}{\Sigma_3}} &= \mathbf{4 \times\lambda_{0} + 4 \times \lambda_{2} + 2 \times \lambda_3 }\\
 &+  2 \times  \w_{\alpha_3}\lambda_2 + 2 \times  \w_{\alpha_2}\w_{\alpha_3}\lambda_2\\
 &+ 4 \times 
  \s{\alpha_2}\lambda_{0} + 2 \times \s{\alpha_1}\lambda_3 + 2 \times  \s{\alpha_2}\lambda_3 +2 \times  \s{\alpha_2}\s{\alpha_1}\lambda_3.  
\end{align*}
Since $\mult{\lambda_3}{\jac{H}{T}{\Sigma_3}}=2$, it follows that there exist exactly two irreducible constituents of $\Pi$ having $\lambda_3$ in their Jacquet module.
Recall that  $\sigma_0$ is the unique irreducible constituent of $\Pi$ having $\lambda_{0}$ in its Jacquet module, in particular $\mult{\sigma_0}{\Sigma_3}=1$. Moreover, since $\mult{\lambda_2}{\jac{H}{T}{\Sigma_3}}=4$,   
\Cref{C3::sigma2::multi::2} yields  $\mult{\sigma_2}{\Sigma_3}=2$.
 Thus,  
$$\sigma_0 + 2\times \sigma_2\leq (\Sigma_3)_{s.s}.$$
Since $\lambda_3 \not \leq \jac{H}{T}{\sigma_{0}}$ and $\lambda_3 \not \leq \jac{H}{T}{\sigma_{2}}$, there exists an irreducible constituent   $\sigma_3$ of $\Sigma_3$ having $\lambda_3 \leq \jac{H}{T}{\sigma_3}$. Applying \eqref{Eq:OR},\eqref{Eq:A1}, \eqref{Eq::C2b} yields  
 \begin{equation}\label{app::c3::sigma3}
 \coseta{\jac{H}{T}{\sigma_3}} \geq		
 2 \times \lambda_3 + 2 \times \s{\alpha_1}\lambda_3 + \s{\alpha_2}\lambda_3 + \s{\alpha_2}\s{\alpha_1}\lambda_3.
 \end{equation}
Moreover, by Jacquet module considerations, it follows that $$\bk{\Sigma_3}_{s.s} =  \sigma_{0} + 2 \times \sigma_2 + \sigma_3.$$ 
and there is an equality in \eqref{app::c3::sigma3}. 
\end{proof}

Since 
\begin{align*}
\mult{\lambda_3}{\jac{H}{T}{\Pi}}&=4, & \mult{\lambda_3}{\jac{H}{T}{\sigma_3}}&=2
\end{align*}
and 
$$\mult{\lambda_3}{\jac{H}{T}{\sigma_0}} =\mult{\lambda_3}{\jac{H}{T}{\sigma_1}}= \mult{\lambda_3}{\jac{H}{T}{\sigma_1}}=0,$$
there exists an irreducible constituent $\sigma_4$ of $\Pi$ such that $\mult{\lambda_3}{\jac{H}{T}{\sigma_4}}\neq 0$. As in \eqref{app::c3::sigma3}, one has
\begin{equation}\label{app::c3::sigma4}
 \coseta{\jac{H}{T}{\sigma_4}} \geq		
 2 \times \lambda_3 + 2 \times \s{\alpha_1}\lambda_3 + \s{\alpha_2}\lambda_3 + \s{\alpha_2}\s{\alpha_1}\lambda_3.
\end{equation}

Recall that $D_{H}(\cdot)$ takes irreducible representations to irreducible representations. Moreover, in that case, one has  
$$ r^{H}_{T} \circ D_{H} =  w \circ D_{T} \circ  r_{T}^{H},$$
where $w$ is the longest Weyl element in $\weyl{H}$. 
 
In our case, the longest Weyl element acts as $-1$ on each $\lambda$. Thus, one has for each irreducible constituent $\pi$, 
$$\coseta{\jac{H}{T}{D_{H}(\pi)}}=  \oplus_{\lambda \leq \jac{H}{T}{\pi}} m_{\lambda}\times(-\lambda).$$ 

Set $\tau = \sigma_0 \oplus \sigma_1 \oplus 2 \times \sigma_2$. Then it follows that 
$$ \coseta{\jac{H}{T}{\Pi}} \geq \coset{\jac{H}{T}{\tau}} \oplus \coseta{\jac{H}{T}{D_{H}(\tau)}} \oplus \coseta{\jac{H}{T}{\sigma_3}} \oplus \coseta{\jac{H}{T}{\sigma_4}} \geq \coseta{\jac{H}{T}{\Pi}}.$$
In particular, there is an equality in \eqref{app::c3::sigma4}. Note that $\coseta{\jac{H}{T}{\sigma_3}} = \coseta{\jac{H}{T}{\sigma_4}}$.   
We conclude the proof by summarizing the Jacquet module of each irreducible constituent of $\Pi$. The description appears in \Cref{Table::C3:: a} below.  
\newpage
\begin{longtable}{|c||c|c|c|c|c|c|c|c|} \hline
Exp & ${\Pi}$& $\sigma_{0}$ & $D_{H}(\sigma_{0})$ & $\sigma_2$ &
$D_{H}(\sigma_2)$ & $\sigma_3 ,\sigma_4$ & $\sigma_1$ & $D_{H}(\sigma_1)$\\
\hline \hline
$\left(0, -1, 0\right) $ & $4$ & $4$ & $0$ & $0$ & $0$ & $0$ & $0$ & $0$ \\ \hline
$\left(-1, 1, -1\right)$ & $4$ & $2$ & $0$ & $1$ & $0$ & $0$ & $0$ & $0$ \\ \hline
$\left(-1, -1, 1\right)$ & $4$ & $1$ & $0$ & $0$ & $0$ & $1$ & $0$ & $1$ \\ \hline
$\left(2, -1, 0\right)$ & $4$ & $0$ & $0$ & $0$ & $0$ & $2$ & $0$ & $0$ \\ \hline
$\left(-1, 2, -1\right)$ & $4$ & $0$ & $0$ & $1$ & $1$ & $0$ & $0$ & $0$ \\ \hline
$\left(1, 0, -1\right)$ & $4$ & $0$ & $0$ & $2$ & $0$ & $0$ & $0$ & $0$ \\ \hline
$\left(0, 1, 0\right)$ & $4$ & $0$ & $4$ & $0$ & $0$ & $0$ & $0$ & $0$ \\ \hline
$\left(1, -1, 1\right)$ & $4$ & $0$ & $2$ & $0$ & $1$ & $0$ & $0$ & $0$ \\ \hline
$\left(1, 1, -1\right)$ & $4$ & $0$ & $1$ & $0$ & $0$ & $1$ & $1$ & $0$ \\ \hline
$\left(-2, 1, 0\right)$ & $4$ & $0$ & $0$ & $0$ & $0$ & $2$ & $0$ & $0$ \\ 
\hline
$\left(1, -2, 1\right)$ & $4$ & $0$ & $0$ & $1$ & $1$ & $0$ & $0$ & $0$ \\ \hline
$\left(-1, 0, 1\right)$ & $4$ & $0$ & $0$ & $0$ & $2$ & $0$ & $0$ & $0$ \\ \hline
\caption{ Jacquet Module of $\Ind_{T}^{H}\bk{-\fun{2}}$ }
\label{Table::C3:: a}
\end{longtable}

\end{proof}

\begin{Prop}\label{app:Lemma:c3::b}
Let $\lambda_{a.d} = -\fun{1} -\fun{2}$. There exists a unique irreducible representation $\pi$ of $H$ having 
\begin{equation}\label{app::c3::b}
\lambda_{a.d} \leq \jac{H}{T}{\pi}.
\end{equation} 
Moreover,
$$\coseta{\jac{H}{T}{\pi}} = 2 \times \lambda_{a.d} +  \w_{\alpha_2} \lambda_{a.d} + \w_{\alpha_1}\w_{\alpha_2}\lambda_{a.d}.$$ 
\end{Prop}

\begin{proof}
Note that \eqref{app::c3::b} implies two things: 
\begin{itemize}
\item
By the central character argument $\pi \hookrightarrow \Ind_{T}^{H}(\lambda_{a.d})$. Namely, 
$\pi$ is an irreducible constituent of $\Pi =  \Ind_{T}^{H}(\lambda_{a.d})$.
\item
 Application of  \eqref{Eq:OR},\eqref{Eq::C2b} and \eqref{Eq:A1} implies that
\begin{equation}\label{app:c3::c}
\coseta{\jac{H}{T}{\pi}} \geq 2 \times \lambda_{a.d} + \w_{\alpha_2}\lambda_{a.d} + \w_{\alpha_1}\w_{\alpha_2}\lambda_{a.d}.
\end{equation}

\end{itemize}

Note that using the Geometric Lemma, one has 
$$2=\mult{\lambda_{a.d}}{\jac{H}{T}{\Pi}}\geq  \mult{\lambda_{a.d}}{\jac{H}{T}{\pi}}\geq 2.$$
Hence,  $\mult{\lambda_{a.d}}{\jac{H}{T}{\pi}}=2$, and there exists a unique irreducible constituent of $\Pi$ that satisfies \eqref{app::c3::b}.

Consider the following two representations of $H$.
\begin{align*}
\Sigma_1 &= \Ind_{M_{1}}^{H}(0) \hookrightarrow \Ind_{T}^{H}(\w_{\alpha_1}\w_{\alpha_2} \lambda_{a.d}),\\
\Sigma_2 &= \Ind_{M_3}^{H}(\fun{3}) \hookrightarrow \Ind_{T}^{H}(w\lambda_{a.d}),   \text{ where } \: w=  \w_{\alpha_3}\w_{\alpha_2}\w_{\alpha_1}\w_{\alpha_3}\w_{\alpha_2}.
\end{align*}
Thus, $\bk{\Sigma_1}_{s.s}, \bk{\Sigma_2}_{s.s} \leq \bk{\Pi}_{s.s}$. Using Geometric Lemma, it follows: 
\begin{eqnarray}
\coseta{\jac{H}{T}{\Sigma_1}} =  2 \times \lambda_{a.d} +  2 \times \w_{\alpha_2}\lambda_{a.d} +  2 \times \w_{\alpha_1}\w_{\alpha_2}\lambda_{a.d}.
\end{eqnarray}
Since $\mult{\lambda_{a.d}}{\jac{H}{T}{\Sigma_1}}=2$, it follows that $\mult{\pi}{\Sigma_1} =1$. Hence, if $\lambda \leq \jac{H}{T}{\pi}$, then $\lambda \in \set{\lambda_{a.d},\w_{\alpha_2}\lambda_{a.d},\w_{\alpha_1}\w_{\alpha_2}\lambda_{a.d}}$. 

On the other hand, Geometric Lemma yields: 
\begin{align*}
\mult{\lambda_{a.d}}{\jac{H}{T}{\Sigma_2}}=& 2, &
\mult{\w_{\alpha_2}\lambda_{a.d}}{\jac{H}{T}{\Sigma_2}}=& 1, &
\mult{\w_{\alpha_1}\w_{\alpha_2}\lambda_{a.d}}{\jac{H}{T}{\Sigma_2}}=& 1. 
\end{align*}
In particular $\mult{\pi}{\Sigma_2}=1$, and 
$\mult{\lambda}{\jac{H}{T}{\pi}}  \leq 1$ for $\lambda \in \set{\w_{\alpha_2}\lambda_{a.d}, \w_{\alpha_1}\w_{\alpha_2}\lambda_{a.d}}$. 

Thus, there is an equality in \eqref{app:c3::c}. 

\end{proof}

We end this part by deriving the branching triples from  \Cref{Table::C3:: a} and    \Cref{app:Lemma:c3::b}. These triples are presented in the frames below:

\begin{framed}
	\begin{equation}\tag{$C_3(a)$} \label{Eq::C3a}
	\lambda\leq \jac{H}{T}{\pi},\ \lambda \in \set{\pm \fun{2}} \Longrightarrow 4\times\lambda+2\times \coset{\s{\alpha_2}\lambda}
	+\coset{\s{\alpha_3}\s{\alpha_2}\lambda} 
	 \leq \coseta{\jac{H}{T}{\pi}} .
	\end{equation}
\end{framed}

\begin{framed}
	\begin{equation}\tag{$C_3(b)$} \label{Eq::C3b}
	\begin{split}
	\lambda\leq \jac{H}{T}{\pi},\ \lambda \in \set{\pm\bk{ \fun{1} -\fun{3}}}\\ \Longrightarrow
	2 \times \coset{\lambda} + \coset{\s{\alpha_1}\lambda} + \coset{\s{\alpha_3}\lambda} + \coset{\s{\alpha_2}\s{\alpha_3}\lambda} \leq \coseta{\jac{H}{T}{\pi}}.   
		\end{split}
	\end{equation}
\end{framed}

\begin{framed}
	\begin{equation}\tag{$C_3(c)$} \label{Eq::C3c}
	\begin{split}
		\lambda\leq \jac{H}{T}{\pi},\ \lambda \in \set{\pm\bk{ 2\fun{1} -\fun{2}}}\\ \Longrightarrow
2 \times \coset{\lambda} + 2 \times \coset{\s{\alpha_1}\lambda} + \coset{\s{\alpha_2}\lambda} + \coset{\s{\alpha_2}\s{\alpha_1}\lambda}\leq   \coseta{\jac{H}{T}{\pi}}.   
\end{split}
	\end{equation}
\end{framed}

\begin{framed}
	\begin{equation}\tag{$C_3(d)$} \label{Eq::C3d}
	\begin{split}
		\lambda\leq \jac{H}{T}{\pi},\ \lambda \in \set{\pm\bk{ -\fun{1} -\fun{2}}}\\ \Longrightarrow
2 \times \coset{\lambda} + \coset{\s{\alpha_2}\lambda} +  \coset{\s{\alpha_1}\s{\alpha_2}\lambda}\leq   \coseta{\jac{H}{T}{\pi}}.   
\end{split}
	\end{equation}
\end{framed}

\section{Triples of type $B_3$}
In this subsection, we list some branching triples for the group $H =  Spin_{7}$. For this purpose, we label the Dynkin diagram of $H$ as follows:
 $$
 
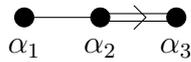
\begin{tikzpicture}[scale=0.5]
 \draw (-1,0) node[anchor=east] {};
 \draw (0 cm,0) -- (2 cm,0);
 \draw (2 cm, 0.1 cm) -- +(2 cm,0);
 \draw (2 cm, -0.1 cm) -- +(2 cm,0);
\draw[shift={(3.2, 0)}, rotate=0] (135 : 0.45cm) -- (0,0) -- (-135 : 0.45cm);
 \draw[fill=black] (0 cm, 0 cm) circle (.25cm) node[below=4pt]{$\alpha_1$};
 \draw[fill=black] (2 cm, 0 cm) circle (.25cm) node[below=4pt]{$\alpha_2$};
 \draw[fill=black] (4 cm, 0 cm) circle (.25cm) node[below=4pt]{$\alpha_3$};
 \end{tikzpicture}
 $$
 Here we need to consider two principal series, $\Pi_1= \Ind_{T}^{H}(-\fun{3}), \Pi_2= \Ind_{T}^{H}(-\fun{1} -\fun{3})$. The description of their  irreducible constituents gives rise to branching triples of type $B_3$.  
 
\subsubsection{The structure of $\Pi_1$}

\begin{Prop} Let $\Pi_1 = \Ind_{T}^{H}(-\fun{3})$.
\begin{enumerate}
\item
$D_{H}(\Pi_1) = \Pi_1$. 
\item
$(\Pi_1)_{s.s} =\tau_{a.d}^{1} \oplus \tau_{d}^{1} \oplus \tau_{1}^{1} \oplus\tau_2^{1} \oplus \tau_3^{1} \oplus \tau_4^{1}$, where \cref{Table::B3:: a} determines their Jacquet modules.
\end{enumerate}
\end{Prop}
\begin{proof}
The first part follows from \eqref{aub::prinipalseries}. For the second part, we proceed as follows: We specify certain elements in $\weyl{H}\cdot (-\fun{3})$: 
\begin{align*}
\lambda_{a.d} &= -\fun{3},  &  \lambda_1 &=-\fun{1} + \fun{2} -\fun{3}, &
\lambda_2 &=-\fun{1} +\fun{3}.   
\end{align*}
As before, we fix certain subrepresentations of $\Ind_{T}^{H}(\lambda) $ for $\lambda \in \weyl{H}\cdot (-\fun{3})$.
Let 
\begin{align*}
\Sigma_1 &= \Ind_{M_2}^{H}(0) \hookrightarrow \Ind_{T}^{H}(\lambda_1),
\\ \Sigma_2 &= \Ind_{M_{\alpha_1} \times M_{\alpha_3}}^{H}(Triv_{M_{\alpha_1}} \times St_{M_{\alpha_3}}) \hookrightarrow \Ind_{T}^{H}(\lambda_2).   
\end{align*}

\begin{Lem}
The representation $\Sigma_1$ is irreducible.
\end{Lem}
\begin{proof}
Using the Geometric Lemma, one has 
\begin{align*}
\coseta{\jac{H}{T}{\Sigma_1}} =  6\times \lambda_{a.d} +  4 \times \w_{\alpha_3}\lambda_{a.d} +  2 \times \w_{\alpha_2}\w_{\alpha_3}\lambda_{a.d}.
\end{align*}
Let $\tau_{a.d}^{1}$ be an irreducible constituent of $\Sigma_1$ having $\lambda_{a.d} \leq \jac{H}{T}{\tau_{a.d}^{1}}$. Applying \eqref{Eq:OR},\eqref{Eq::C2b} implies that 
$$ \coseta{\jac{H}{T}{\tau_{a.d}^{1}}} \geq 6 \times \lambda_{a.d} + 3 \times \w_{\alpha_3}\lambda_{a.d}.$$

By \eqref{Eq:OR}, if $\mult{\w_{\alpha_3}\lambda_{a.d}}{\jac{H}{T}{\pi}}\neq 0$, then $2 \vert \mult{\w_{\alpha_3}\lambda_{a.d}}{\jac{H}{T}{\pi}}$. In particular,
$$4 \leq \mult{\w_{\alpha_3}\lambda_{a.d}}{\jac{H}{T}{\tau_{a.d}^{1}}}  \leq \mult{\w_{\alpha_3}\lambda_{a.d}}{\jac{H}{T}{\Sigma_1}} =4.$$   
Hence, $\mult{\w_{\alpha_3}\lambda_{a.d}}{\jac{H}{T}{\tau_{a.d}^{1}}}=4$.  Applying \eqref{Eq:A2} implies that $$\mult{\w_{\alpha_2}\w_{\alpha_3}\lambda_{a.d}}{\jac{H}{T}{\tau_{a.d}^{1}}}\geq 2.$$
Let us summarize: 
\begin{equation}\label{eq::b3::tau1a.d}
\coseta{\jac{H}{T}{\Sigma_1}} \geq \coseta{\jac{H}{T}{\tau_{a.d}^{1}}} \geq 
6 \times \lambda_{a.d} +  4 \times \w_{\alpha_3}\lambda_{a.d} +  2 \times \w_{\alpha_2} \w_{\alpha_3}\lambda_{a.d}  = \coseta{\jac{H}{T}{\Sigma_1}}.  
\end{equation}  
Hence, there is an equality in \eqref{eq::b3::tau1a.d}, and $\Sigma_1$ is irreducible.
\end{proof}

\begin{Prop}
One has $\Sigma_2=\tau_1^{1} \oplus \tau_2^{1}$ where 
\begin{align*}
\coseta{\jac{H}{T}{\tau_1^{1}}} & = 2 \times \lambda_2 +  \w_{\alpha_1}\lambda_2 + \w_{\alpha_3}\lambda_2 +2 \times \w_{\alpha_3}\w_{\alpha_1}\lambda_2,   \\ 
\coseta{\jac{H}{T}{\tau_2^{1}}} & = 2 \times \lambda_2 +  \w_{\alpha_1}\lambda_2 + \w_{\alpha_3}\lambda_2 +2 \times \w_{\alpha_2}\w_{\alpha_3}\lambda_2.  
\end{align*}
\end{Prop}
\begin{proof}
Using the Geometric Lemma, one has
\begin{equation}
\coseta{\jac{H}{T}{\Sigma_2}} =  4 \times \lambda_{2} +  2 \times \w_{\alpha_1}\lambda_2 +  2 \times \w_{\alpha_3}\lambda_2 +  2 \times \underbrace{\w_{\alpha_{3}}\w_{\alpha_1} \lambda_2}_{-\lambda_2} +  2 \times \w_{\alpha_2} \w_{\alpha_3} \lambda_2.
\end{equation}
Note that $\w_{\alpha_{3}}\w_{\alpha_1} \lambda_2=-\lambda_2$. Assume that $\Sigma_2$ is irreducible. Thus, by Aubert, $D_{H}(\Sigma_2)$ is also irreducible. In that case, one has 
\begin{align*}
\coseta{\jac{H}{T}{D_{H}(\Sigma_2)}} &=  4 \times \bk{-\lambda_{2}} +  2 \times \bk{-\w_{\alpha_1}\lambda_2} + \\&+ 2 \times \bk{- \w_{\alpha_3}\lambda_2} +  2 \times \bk{-\w_{\alpha_{3}}\w_{\alpha_1} \lambda_2} +  2 \times \bk{-\w_{\alpha_2} \w_{\alpha_3} \lambda_2}.
\end{align*} 
Since $\coseta{\jac{H}{T}{\Sigma_2}} \neq \coseta{\jac{H}{T}{D_{H}(\Sigma_2)}}$, it follows that $\Sigma_2 \not \simeq D_{H}(\Sigma_2)$.   Notice that
\begin{align*}
\mult{\lambda_2}{\jac{H}{T}{\Sigma_2}} &= 4, & 
\mult{\lambda_2}{\jac{H}{T}{D_{H}(\Sigma_2)}} &= 2, & \mult{\lambda_2}{\jac{H}{T}{\Pi_1}}&=6.
\end{align*}
Thus, $\Sigma_2, D_{H}(\Sigma_2)$ are the only irreducible constituents of $\Pi_1$ having $\lambda_2$ in their Jacquet module.
Notice that 
\begin{align*}
\mult{\lambda_2}{\jac{H}{T}{\Sigma_2}}&= 4, &  \mult{-\lambda_2}{\jac{H}{T}{\Sigma_2}}&= 2, \\
\mult{\lambda_2}{\jac{H}{T}{D_{H}(\Sigma_2)}}&= 2, &  \mult{-\lambda_2}{\jac{H}{T}{D_{H}(\Sigma_2)}}&= 4. 
\end{align*}
Consider the following two representations of $G=F_4$, $\sigma_1 =\Ind_{M_{3}}^{G}(\frac{1}{2}), \sigma_2 =\Ind_{M_{2}}^{G}(\frac{1}{2})$. According to \Cref{local::p3::pi1::common}, $\sigma_1,\sigma_2$ share a common irreducible constituent $\pi_1$. Moreover,
\begin{eqnarray}
\mult{\fun{1}^{F_4}-\fun{3}^{F_4}+\fun{4}^{F_4}}{\jac{G}{T}{\pi_1}}&=2, \label{eq::F4::B3::1} \\
\mult{\w_{\alpha_3}\w_{\alpha_1}\bk{\fun{1}^{F_4}-\fun{3}^{F_4}+\fun{4}^{F_4}}}{\jac{G}{T}{\pi_1}}&\leq2. \label{Sigma_2::cont}
\end{eqnarray}

According to \eqref{f4::parabolic}, the group $G$ has a Levi subgroup $M_4$ whose derived subgroup is $H$. Let $\lambda^{G}= \sum_{i=1}^{4} z_i \fun{\alpha_i}^{G}$ be a character of the torus of $F_4$. We set  $$P_r\bk{\lambda^{G}}=\sum_{i=1}^{3}z_i \fun{\alpha_i}^{G}.$$ Note that $P_r\bk{\lambda^{G}}$ can be thought of as a character of the torus of $H$. 

Let $\lambda^{G} \leq \jac{G}{T}{\sigma_1}$. Then there is exactly one $\lambda^{G}_2\leq \jac{G}{T}{\sigma_1}$ such that $P_r(\lambda^{G}) =  \lambda_2$. In that case, $\lambda_2^{G}= \fun{1}^{G}-\fun{3}^{G}+\fun{4}^{G} $.  Furthermore, there is exactly one $\lambda^{G}_3\leq \jac{G}{T}{\sigma_1}$ such that $P_r(\lambda_3^{G}) =  \w_{\alpha_3}\w_{\alpha_1}\lambda_2=\lambda_3$. In that case, $\lambda_3^{G}=\w_{\alpha_3}\w_{\alpha_1} \lambda^{G}_{2}$. 
Hence, $$\mult{\lambda_i}{\jac{H}{T}{Res_{H}^{M_4}(\sigma_1)}}= \mult{\lambda_2^{G}}{\jac{G}{T}{\sigma_1}}\quad i\in \set{2,3}.$$

By  \eqref{eq::F4::B3::1}, \eqref{Sigma_2::cont}, it follows that  
\begin{equation}\label{B3::cont}
\mult{\lambda_i}{\jac{H}{T}{Res_{H}^{M_4}(\pi_1)}}= 
\mult{\lambda_2^{G}}{\jac{G}{T}{\pi_1}}\leq 2 \quad i\in \set{2,3}.
\end{equation}
By \cite[Lemma 2.1]{Tadic},  $Res_{H}^{M_4}(\jac{G}{M_4}{\pi_1}$ is a direct sum of representations of $H$. Hence, 
\begin{align*}
\mult{\Sigma_2}{ Res_{H}^{M_4}(\jac{G}{M_4}{\pi_1})} &\geq 1,  &
\mult{D_{H}(\Sigma_2)}{ Res_{H}^{M_4}(\jac{G}{M_4}{\pi_1})} &= 0.
\end{align*}
If  $\Sigma_2$ is irreducible, then  $\mult{\lambda_3}{\jac{H}{T}{\Sigma_3}}\geq 4$. This contradicts \eqref{B3::cont}. Thus, $\Sigma_2$ is reducible.

Since $\Sigma_2$ is unitary, it is completely reducible. By Frobenius reciprocity, each irreducible summand $\sigma$ of $\Sigma_2$ satisfies  $\mult{\lambda_2}{\jac{H}{T}{\sigma}}\neq 0$, and by \eqref{Eq:OR} 
$$2 \vert \mult{\lambda_2}{\jac{H}{T}{\sigma}}.$$
Since $\Sigma_2$ is reducible and $\mult{\lambda_2}{\jac{H}{T}{\Sigma_2}}=4$, the length of $\Sigma_2$ is two. In particular, for every irreducible summand $\sigma$ of $\Sigma_2$, one has 
\begin{equation}\label{B3::sigma2::sigma}
\mult{\lambda_2}{\jac{H}{T}{\sigma}}=2.
\end{equation} 
   
Let 
$\tau_1^{1}$ be an irreducible constituent of $\Sigma_2$ such that 
$\w_{\alpha_3}\w_{\alpha_1} \lambda_2 \leq \jac{H}{T}{\tau_1^1}$. 
Application of \eqref{Eq:OR}, \eqref{Eq:A2},\eqref{Eq::C2b}, \eqref{B3::sigma2::sigma} implies that 
\begin{equation}
2 \times \w_{\alpha_3}\w_{\alpha_1} \lambda_2 
+ \w_{\alpha_1} \lambda_2 + \w_{\alpha_3} \lambda_2 + 2 \times \lambda_2  
\leq \coseta{\jac{H}{T}{\tau_1^{1}}}.
\end{equation}

Let $\tau_2^{1}$ be an irreducible constituent of $\Sigma_2$ such that 
$\w_{\alpha_2}\w_{\alpha_3} \lambda_2 \leq \jac{H}{T}{\tau_2^{1}}$. 
Application of \eqref{Eq:OR}, \eqref{Eq:A2}, \eqref{B3::sigma2::sigma} implies that 
\begin{equation}
2 \times \w_{\alpha_2}\w_{\alpha_3} \lambda_2 
+ \w_{\alpha_3} \lambda_2 + 2 \times \lambda_2  
\leq \coseta{\jac{H}{T}{\tau_2^{1}}}.
\end{equation}
Application of 
\eqref{Eq:A2} on $\lambda_2$ implies that 
\begin{equation*}
2 \times  \lambda_2 
+ \w_{\alpha_1} \lambda_2 + 
\leq \coseta{\jac{H}{T}{\tau_2^{1}}}.
\end{equation*}
In conclusion,
\begin{equation}
2 \times \w_{\alpha_2}\w_{\alpha_3} \lambda_2 
+ \w_{\alpha_3} \lambda_2 +\w_{\alpha_1}\lambda_2 + 2 \times \lambda_2  
\leq \coseta{\jac{H}{T}{\tau_2^{1}}}.
\end{equation}
Note that $\coseta{\jac{H}{T}{\tau_1^{1}}} + \coseta{\jac{H}{T}{\tau_2^1}}=\coseta{\jac{H}{T}{\Sigma_2}}$. Thus, the claim follows.
\end{proof}

Let $\tau_{d}^{1} =  D_{H}(\tau_{a.d}^{1}), \tau_3^{1}=D_{H}(\tau_2^{1})$. 
By comparing the Jacquet modules of each constituent in 
$\set{\tau_{a.d}^{1} ,\tau_{d}^{1}, \tau_1^{1},\tau_2^{1},\tau_3 ^{1}}$, we derive that they are pairwise non-isomorphic constituents of $\Pi_1$. Moreover, 
$$\coseta{ \jac{H}{T}{\tau_{a.d}^{1}}} + \coseta{ \jac{H}{T}{\tau_{d}^{1}}}+
\coseta{ \jac{H}{T}{\tau_{1}^{1}}}+ \coseta{ \jac{H}{T}{\tau_{2}^{1}}}+ \coseta{ \jac{H}{T}{\tau_{3}^{1}}} \neq \coseta{ \jac{H}{T}{\Pi_1}}.$$

Let $\tau_4^{1}$ be a representation of $H$ such that
$$\bk{\Pi_1}_{s.s}=  \tau_{a.d}^{1} \oplus \tau_{d}^{1} \oplus \tau_1^{1} \oplus \tau_2^{1} \oplus \tau_3^{1} \oplus \tau_4^{1}.$$
In that case, one has 
$$\coseta{\jac{H}{T}{\tau_4^{1}}} =  2 \times \lambda_2 + \w_{\alpha_3}\lambda_2 + \w_{\alpha_1}\lambda_2 + 2 \times \w_{\alpha_1}\w_{\alpha_3}\lambda_{2}.$$

Let $\tau \leq \tau_{4}^{1}$ be an irreducible constituent of  $\tau_{4}^{1}$. In order to show that $\tau_{4}^{1}$ is irreducible, it suffices to prove that $\coseta{\tau_{4}^{1}}=\coseta{\tau}$. 
For this purpose, it is enough to show that 
\begin{align*}
\mult{\lambda_2}{\jac{H}{T}{\tau}}  \times \mult{\w_{\alpha_3}\w_{\alpha_1}\lambda_2}{\jac{H}{T}{\tau}}\neq 0.
\end{align*}
Assume by contradiction that there exist two irreducible constituents $\tau_4^{a},\tau_{4}^{b}$   of $\tau_4^{1}$ such that 
\begin{align*}
\mult{\lambda_2}{\jac{H}{T}{\tau_{4}^{a}}} &\neq 0, & 
\mult{\lambda_2}{\jac{H}{T}{\tau_{4}^{b}}} &= 0, \\
\mult{\w_{\alpha_3}\w_{\alpha_1}\lambda_2}{\jac{H}{T}{\tau_{4}^{a}}} &= 0, & 
\mult{\w_{\alpha_3}\w_{\alpha_1}\lambda_2}{\jac{H}{T}{\tau_{4}^{b}}} &\neq 0. 
\end{align*} 
By \eqref{Eq:OR} it follows that 
\begin{align*}
\mult{\lambda_2}{\jac{H}{T}{\tau_{4}^{a}}} &=2,  & 
\mult{\w_{\alpha_3}\w_{\alpha_1}\lambda_2}{\jac{H}{T}{\tau_{4}^{b}}} &= 2 .
\end{align*} 
Applying \eqref{Eq:A2} on $\lambda_2$ and \eqref{Eq::C2b} on $\w_{\alpha_3}\w_{\alpha_1}\lambda_2$ yields that 
$\w_{\alpha_1}\lambda_2$ appears in the Jacquet module of $\tau_{4}^{a},\tau_{4}^{b}$. However, $\mult{\w_{\alpha_1}\lambda_2}{\jac{H}{T}{\tau_4^1}}=1$ implies that there is exactly one irreducible constituent of $\tau_4^{1}$ containing $\w_{\alpha_1}\lambda_2$ in its Jacquet module. This raises a  contradiction. Thus, $\tau_4^{1}$ is irreducible. 
We summarize by describing the Jacquet module of each irreducible constituent of $\Pi_1$.

\begin{longtable}{|c||c|c|c|c|c|c|cH|} \hline
$Exp$ & $\Pi$& $\tau_{a.d}^{1}$ & $\tau_{d}^{1}$ & $\tau_{1}^{1}$ & $\tau_{2}^{1}$ & $\tau_{3}^{1}$ & $\tau_{4}^{1}$ &\\
\hline \hline 
$\left(0, 0, -1\right)$ & $6$ & $6$ & $0$ & $0$ & $0$ & $0$ & $0$ & START1 \\ \hline

$\left(0, -1, 1\right)$ & $6$ & $4$ & $0$ & $0$ & $2$ & $0$ & $0$ & START3 \\ \hline

$\left(-1, 1, -1\right)$ & $6$ & $2$ & $0$ & $1$ & $1$ & $1$ & $1$ & $0$ \\ \hline

$\left(1, 0, -1\right)$ & $6$ & $0$ & $0$ & $2$ & $0$ & $2$ & $2$ & START5 \\ \hline

$\left(0, 0, 1\right)$ & $6$ & $0$ & $6$ & $0$ & $0$ & $0$ & $0$ & START2 \\ \hline

$\left(0, 1, -1\right)$ & $6$ & $0$ & $4$ & $0$ & $0$ & $2$ & $0$ & START4 \\ \hline

$\left(1, -1, 1\right)$ & $6$ & $0$ & $2$ & $1$ & $1$ & $1$ & $1$ & $0$ \\ \hline

$\left(-1, 0, 1\right)$ & $6$ & $0$ & $0$ & $2$ & $2$ & $0$ & $2$ & START6 \\ \hline
\caption{ Jacquet Module of $\Ind_{T}^{H}\bk{-\fun{3}}$ }
\label{Table::B3:: a}
\end{longtable}

\end{proof}    
\begin{Cor} \label{App:B3::1}
Let $H$ be a group of type $B_3$, and let $\pi$ be an irreducible representation of $H$ having 
$\coset{\lambda} \leq \coset{\jac{H}{T}{\pi}},$ where $\lambda = [1,0,-1]$.
\begin{enumerate}
\item
$\pi \in \set{\tau_1^{1} ,\tau_3^{1},\tau_4^{1}}$.
\item
If  $\coset{\s{\alpha_2}\s{\alpha_3}\lambda} \not \leq \coset{\jac{H}{T}{\pi}}$, then $\pi \in \set{\tau_{1}^{1},\tau^{1}_{4}}$. In that case
\begin{equation}\label{Eq::B3::c} \tag{$B_3(c)$}
\coseta{\jac{H}{T}{\pi}} = 2\times \coset{\lambda} +  \coset{\s{\alpha_1} \lambda} + \coset{\s{\alpha_3} \lambda} + 2 \times \coset{\s{\alpha_1}\s{\alpha_3}\lambda}.
\end{equation}

\end{enumerate}  
\end{Cor}

\begin{Cor} \label{App:B3::2}
Let $H$ be a group of type $B_3$, and let $\pi$ be an irreducible representation of $H$ having 
$\coset{\lambda} \leq \coset{\jac{H}{T}{\pi}}$, where $\lambda = [0,-1,1]$.
\begin{enumerate}
\item
$\pi \in \set{\tau^{1}_{a.d} ,\tau^{1}_{2}}$.
\item
If  $\coset{\s{\alpha_3}\lambda} \not \leq \coset{\jac{H}{T}{\pi}}$, then $\pi =\tau_{2}^{1}$. In that case
\begin{equation}\label{Eq::B3::d} \tag{$B_3(d)$}
\coseta{\jac{H}{T}{\pi}} = 2\times \coset{\lambda} +  \coset{\s{\alpha_2} \lambda} + \coset{\s{\alpha_3}\s{\alpha_2} \lambda} +  2 \times \coset{\s{\alpha_1}\s{\alpha_3}\s{\alpha_2} \lambda}. 
\end{equation}

\end{enumerate}  
\end{Cor}

\subsubsection{The structure of $\Pi_2$}
\begin{Prop} Let $\Pi_2 = \Ind_{T}^{H}(-\fun{2})$.
\begin{enumerate}
\item
 $D_{H}(\Pi_2) = \Pi_2$.
\item
$(\Pi_2)_{s.s} =\tau_{a.d}^{2} \oplus \tau_{d}^{2} \oplus 2 \times  \tau_{1}^{2}$, where \cref{Table::B3:: b} determines their Jacquet module.
\end{enumerate}
\end{Prop}
\begin{proof}
As before, we specify the following exponents of $\Pi_2$. Let 
\begin{align*}
\lambda_{a.d} &=  -\fun{2},   & \lambda_1 &= -\fun{1} +\fun{2}-2\fun{3}, \\ 
\lambda_2 &=  2\fun{1} -\fun{2},  & \lambda_3 &=\w_{\alpha_3}\w_{\alpha_2}\lambda_{a.d} =-\fun{1}-\fun{2}+2\fun{3}, &\lambda_4&=-\fun{1}+2\fun{3}.  
\end{align*}
We also consider the following representations:
\begin{align*}
\Sigma_1 &=  \Ind_{M_{3}}^{H}(0) \hookrightarrow \Ind_{T}^{H}(\lambda_3), \\
\Sigma_2 &= \Ind_{M_{\alpha_2}}^{H}\bk{Triv_{M_{\alpha_2}},\frac{3}{2}\fun{1}-\fun{3}} \hookrightarrow \Ind_{T}^{H}(\lambda_2).
\end{align*}

\begin{Lem}\label{B3::sigma_a.d::2}
The representation $\Sigma_1$ is irreducible. Moreover, 
$$\coseta{\jac{H}{T}{\Sigma_1}} =  4 \times \lambda_{a.d} + 2 \times \w_{\alpha_2} \lambda_{a.d} +  2 \times \w_{\alpha_3} \w_{\alpha_2} \lambda_{a.d}.$$
In addition, $\Sigma_1$ is the unique irreducible constituent of $\Pi_2$ having $\lambda_{a.d}$ in its Jacquet module.
\end{Lem}
\begin{proof}
By the Geometric Lemma, one has 
\begin{align*}
\coseta{\jac{H}{T}{\Sigma_1}} &= 4 \times \lambda_{a.d} +2 \times \w_{\alpha_2}\lambda_{a.d} +  2 \times \w_{\alpha_3}\w_{\alpha_2}\lambda_2.  
\end{align*}
Let $\tau_{a.d}^{2}$ be an irreducible constituent of $\Pi_2$ having $\lambda_{a.d} \leq \jac{H}{T}{\tau_{a.d}^{2}}$. An application of \eqref{Eq:OR},\eqref{Eq:A2},\eqref{Eq:A1} yields:
\begin{equation}\label{App:B3::tau2a.d}
\coseta{\jac{H}{T}{\tau_{a.d}^{2}}} \geq  4 \times \lambda_{a.d} +2 \times \w_{\alpha_2}\lambda_{a.d} +  2 \times \w_{\alpha_3}\w_{\alpha_2}\lambda_2.  
\end{equation}
A comparison of $\coset{\jac{H}{T}{\Sigma_1}}$ and  $\coset{\jac{H}{T}{\tau_{a.d}^{2}}}$ implies that there is an equality in \eqref{App:B3::tau2a.d}. In particular, $\Sigma_1$ is irreducible. Since $\mult{\lambda_{a.d}}{\jac{H}{T}{\Pi_2}} = \mult{\lambda_{a.d}}{\jac{H}{T}{\Sigma_1}}=4$, it follows that $\Sigma_1$ is the unique irreducible constituent of $\Pi_2$ having $\lambda_2$ in its Jacquet module. 
\end{proof}
\begin{Lem}
The representation $\Sigma_2$ is of length two.  
Moreover, $\bk{\Sigma_2}_{s.s} =  \Sigma_1 \oplus \tau_1^{2}$ where 
\begin{align*}
\coseta{\jac{H}{T}{\tau_1^{2}}} &=   \w_{\alpha_2}\lambda_{a.d}+ \w_{\alpha_3}\w_{\alpha_2}\lambda_{a.d} 
 + 2 \times \lambda_2 +2 \times \w_{\alpha_1}\lambda_2 + \w_{\alpha_2}\lambda_2+ \w_{\alpha_3}\w_{\alpha_2}\lambda_2 \nonumber \\
&+ 2 \times \lambda_{4} +2\times \w_{\alpha_3}\lambda_4 + 2 \times \w_{\alpha_2}\w_{\alpha_3}\lambda_4 + 2 \times \w_{\alpha_3}\w_{\alpha_2}\w_{\alpha_3}\lambda_4. \nonumber
\end{align*}
\end{Lem}
\begin{proof}
Geometric Lemma implies that 
\begin{align*}
\coseta{\jac{H}{T}{\Sigma_2}} &= \mathbf{4 \times \lambda_{a.d} +  3 \times \w_{\alpha_2}\lambda_{a.d}}  +  3\times \w_{\alpha_3}\w_{\alpha_2}\lambda_{a.d} \\
& + 2 \times \lambda_2 +2 \times \w_{\alpha_1}\lambda_2 + \w_{\alpha_2}\lambda_2+ \w_{\alpha_3}\w_{\alpha_2}\lambda_2\\
&+ 2 \times \lambda_{4} +2\times \w_{\alpha_3}\lambda_4 + 2 \times \w_{\alpha_2}\w_{\alpha_3}\lambda_4 + 2 \times \w_{\alpha_3}\w_{\alpha_2}\w_{\alpha_3}\lambda_4.
\end{align*}
Since $\mult{\lambda_{a.d}}{\jac{H}{T}{\Sigma_2}}=4$, by \Cref{B3::sigma_a.d::2}, $\mult{\Sigma_1}{\Sigma_2}=1$. Let $\tau_1^{2}$  such that $(\Sigma_2)_{s.s} =  \Sigma_1 \oplus \tau_1^{2}$. By comparing Jacquet modules, one has 
\begin{eqnarray}\label{B3::tau1::2}
\coseta{\jac{H}{T}{\tau_1^{2}}} &=   \w_{\alpha_2}\lambda_{a.d}+ \w_{\alpha_3}\w_{\alpha_2}\lambda_{a.d} 
 + 2 \times \lambda_2 +2 \times \w_{\alpha_1}\lambda_2 + \w_{\alpha_2}\lambda_2+ \w_{\alpha_3}\w_{\alpha_2}\lambda_2 \nonumber \\
&+ 2 \times \lambda_{4} +2\times \w_{\alpha_3}\lambda_4 + 2 \times \w_{\alpha_2}\w_{\alpha_3}\lambda_4 + 2 \times \w_{\alpha_3}\w_{\alpha_2}\w_{\alpha_3}\lambda_4 \nonumber.
\end{eqnarray}
Let $\tau^{a}$ be  an irreducible  constituent of $\tau_1^{2}$. In order to demonstrate that  $\tau_1^{2}$ is irreducible, it is enough to show that
\begin{equation}\label{B3::b:irr tau_1}
\mult{\lambda_4}{\jac{H}{T}{\tau^{a}}} \times \mult{\lambda_2}{\jac{H}{T}{\tau^{a}}}\neq 0.
\end{equation}
Assume that \eqref{B3::b:irr tau_1} does not hold. Namely, there are two irreducible constituents $\tau^{a},\tau^{b}$ of $\tau_{1}^{2}$ with the following properties
\begin{align*}
\mult{\lambda_2}{\jac{H}{T}{\tau^{a}}}&\neq 0, & \mult{\lambda_2}{\jac{H}{T}{\tau^{b}}}&=0,\\
\mult{\lambda_4}{\jac{H}{T}{\tau^{a}}}&=0, & \mult{\lambda_4}{\jac{H}{T}{\tau^{b}}}&\neq 0.
\end{align*} 
By \eqref{Eq:OR} it follows that 
 \begin{align*}
 \mult{\lambda_2}{\jac{H}{T}{\tau^{a}}}&= 2, & \mult{\lambda_4}{\jac{H}{T}{\tau^{b}}}&= 2.
 \end{align*}
 Applying \eqref{Eq:A1},\eqref{Eq:B2a},\eqref{Eq:A2} on $\lambda_2$ and then \eqref{Eq:A1} on $\w_{\alpha_1}\w_{\alpha_2}\lambda_2$,  implies that 
 $$\coseta{\jac{H}{T}{\tau^{a}}} \geq 2\times \lambda_2 +  2 \times \w_{\alpha_1}\lambda_2 +  w_{\alpha_2}\lambda_2 + \underbrace{\w_{\alpha_2} \w_{\alpha_1} \lambda_2}_{\w_{\alpha_3}\w_{\alpha_2}\lambda_{a.d}} + \w_{\alpha_3} \w_{\alpha_2} \lambda_2 +  \w_{\alpha_3} \w_{\alpha_2} \w_{\alpha_1} \lambda_2. $$  
On the other hand, an application of a sequence of  \eqref{Eq:A1} rules and \eqref{Eq:A2} on $\lambda_4$ implies that 
$$\coseta{\jac{H}{T}{\tau^{b}}} \geq 2\times \lambda_4 + 2\times \w_{\alpha_3}\lambda_4 + 2\times \w_{\alpha_2}\w_{\alpha_3}\lambda_4 + 2\times\w_{\alpha_3} \w_{\alpha_2}\w_{\alpha_3}\lambda_4 +  \underbrace{\w_{\alpha_1}\lambda_4}_{\w_{\alpha_3}\w_{\alpha_2}\lambda_{2}}.$$
However, \eqref{B3::tau1::2} shows that there is exactly one irreducible constituent of $\tau_{1}^{2}$ having $\w_{\alpha_3}\w_{\alpha_2}\lambda_2$ in its Jacquet module, which is a contradiction. Hence, $\tau^{a},\tau^{b}$ are the same irreducible constituent of $\tau_1^{2}$. Let $\tau \leq \tau^{2}_1$ be an irreducible constituent of $\tau^{2}_1$ having $\lambda_2,\lambda \leq \jac{H}{T}{\tau}$. The above calculation shows   

\begin{align*}
\coseta{\jac{H}{T}{\tau}} &\geq   2\times \lambda_2 +  2 \times \w_{\alpha_1}\lambda_2 +  w_{\alpha_2}\lambda_2 + \w_{\alpha_2} \w_{\alpha_1} \lambda_2 + \w_{\alpha_3} \w_{\alpha_2} \lambda_2 +  \w_{\alpha_3} \w_{\alpha_2} \w_{\alpha_1} \lambda_2 \\
& + 2 \times \lambda_{4} +2\times \w_{\alpha_3}\lambda_4 + 2 \times \w_{\alpha_2}\w_{\alpha_3}\lambda_4 + 2 \times \w_{\alpha_3}\w_{\alpha_2}\w_{\alpha_3}\lambda_4.
\end{align*}
By comparing Jacquet modules one has $\coseta{\jac{H}{T}{\tau_1^{2}}} = \coseta{\jac{H}{T}{\tau}}$. Hence , $\tau_1^{2}$ is irreducible. 
\end{proof}

\begin{Lem}
The representation $\tau_1^{2}$ appears with multiplicity two in $\Pi_2$.
\end{Lem}
\begin{proof}
The proof is similar to the proof of \Cref{C3::sigma2::multi::2}. 
\end{proof}
\end{proof}

\begin{longtable}{|c||c|c|c|cH|} \hline
 $Exp$ & $\Pi$ & $\tau_{a.d}^{2}$& $\tau_{d}^{2}$& $\tau_{1}^{2}$& $\tau_{2}^{2}$ \\ \hline \hline 
$\left(0, -1, 0\right)$ & $4$ & $4$ & $0$ & $0$ & $0$ \\ \hline

$\left(-1, 1, -2\right)$ & $4$ & $2$ & $0$ & $1$ & $1$\\ \hline

$\left(-1, -1, 2\right)$ & $4$ & $2$ & $0$ & $1$ & $1$\\ \hline

$\left(2, -1, 0\right)$ & $4$ & $0$ & $0$ & $2$ & $2$\\ \hline

$\left(-1, 2, -2\right)$ & $4$ & $0$ & $0$ & $2$ & $2$ \\ \hline

$\left(1, 0, -2\right)$ & $4$ & $0$ & $0$ & $2$ & $2$\\ \hline

$\left(0, 1, 0\right)$ & $4$ & $0$ & $4$ & $0$ & $0$\\ \hline

$\left(1, -1, 2\right)$ & $4$ & $0$ & $2$ & $1$ & $1$\\ \hline

$\left(1, 1, -2\right)$ & $4$ & $0$ & $2$ & $1$ & $1$\\ \hline 

$\left(-2, 1, 0\right)$ & $4$ & $0$ & $0$ & $2$ & $2$\\ \hline 

$\left(1, -2, 2\right)$ & $4$ & $0$ & $0$ & $2$  & $2$\\ \hline

$\left(-1, 0, 2\right)$ & $4$ & $0$ & $0$ & $2$ & $2$\\ \hline

\caption{ Jacquet Module of $\Ind_{T}^{H}\bk{-\fun{2}}$ }
\label{Table::B3:: b}
\end{longtable}

We end this section by describing the relevant branching triples  of type $B_3$.
\begin{framed}
	\begin{equation}\tag{$B_3(a)$} \label{Eq::B3a}
	\begin{split}
	\lambda\leq \jac{H}{T}{\pi},\ \lambda \in \set{\pm\bk{ \fun{3}}}\\ \Longrightarrow
	6 \times \coset{\lambda} +
	4 \times \coset{\s{\alpha_3}\lambda}+
	2 \times \coset{\s{\alpha_2}\s{\alpha_3}\lambda}& \leq \coset{\jac{H}{T}{\pi}}.  
		\end{split}
	\end{equation}
\end{framed}

\begin{framed}
	\begin{equation}\tag{$B_3(b)$} \label{Eq::B3b}
	\begin{split}
	\lambda\leq \jac{H}{T}{\pi},\ \lambda \in \set{\pm\bk{ -\fun{1} +2\fun{2}-2\fun{3}}}\\ \Longrightarrow
	2 \times \coset{\lambda} +
	2 \times \coset{\s{\alpha_2}\lambda}+
	2 \times \coset{\s{\alpha_3}\lambda}&\\
	+\coset{\s{\alpha_1}\lambda}	
	+2 \times \coset{\s{\alpha_2}\s{\alpha_1}\lambda}
	+\coset{\s{\alpha_1}\s{\alpha_2}\lambda} &\\
	+2 \times \coset{\s{\alpha_3}\s{\alpha_2}\lambda}
	+\coset{\s{\alpha_1}\s{\alpha_3}\lambda}
	+2 \times \coset{\s{\alpha_1}\s{\alpha_2}\s{\alpha_1}\lambda}
	+\coset{\s{\alpha_1}\s{\alpha_3}\s{\alpha_2}\lambda}& \leq \coset{\jac{H}{T}{\pi}}.  
		\end{split}
	\end{equation}
\end{framed}

\newpage
\newpage
\chapter{Computational issues}\label{app:compute}
 \setcounter{Thm}{0}
Let $G$ be a split group of type $F_4$.
Let $F$ be a local field of characteristic $0$ with a residue field of cardinality $q$, which is a power of prime.  The Iwahori-Hecke algebra of the group $G(F)$
will be denoted as $\Hecke(q)$. 

One of the main tasks in this thesis is to find  the dimension of the image of
$n_w(z,q)$ in $\Hecke_0(q)$, the finite Iwahori--Hecke algebra.  In this appendix we shall prove the following: 
\begin{Prop} \label{app:main::prop} 
Let $(\para{P}_i,z_0) \in \set{ (\para{P}_1,1),(\para{P}_3,\frac{1}{2}), (\para{P}_4,\frac{5}{2})}$. Then for every $w \in W(\para{P}_i, G)$, 
the dimension of  $\Image n_w(z_0,q)$ does not depend on $q$. 
\end{Prop}

The proposition can be deduced from the category equivalence described in \cite[Section 4]{MR2869018}. However, for the sake of simplicity, we derive it directly below.

\section{Preparations}
For any $s\in \Q$, denote by $e_s: \Q[x]\rightarrow \Q$
the evaluation map, sending $x$ to $s$.  For any  finite dimensional vector space $V$  over $\Q$,  this gives rise to evaluation maps 
$e_s: \Q[x]\otimes V\rightarrow V.$

Let $S$ be a subset of the set of integer numbers.
We say that $p(x)\in \Q[x]$ is \textbf{$\mathbf{S}$-non-vanishing} if
$p(s)\neq 0$ for all $s\in S$.

Given a polynomial $p(x)=\sum_{k=0}^n a_k x^k\in \Z[x]$ such that $p(0) \neq 0$,
and a set  $S$ of integer numbers, it is easy to find  whether
$p(x)$ is $S$-non-vanishing. It is enough to
check whether $p(x)$  vanishes on the finite set of divisors of $a_0$. 
The set of all $S$-non-vanishing polynomials constitutes a multiplicative
system $\Q[x](S)$.
The localization with respect to it will be denoted $\Q[x]_S$. 

A matrix $A\in M_n(\Q[x])$ is called \textbf{$\mathbf{S}$-invertible}
if it is invertible in $M_n(\Q[x]_S)$. Equivalently, 
the  determinant $\det(A)\in \Q[x]$ is $S$-non-vanishing.

\subsection{Cramer's rule} 
Let $V$ be a vector space of dimension $n$ with a basis $B$.
For any ordered set $T$ of $k$ vectors $\{v_1, \ldots, v_k\}$,
define $[T]_B$ to be the matrix of size $n\times k$
whose columns are $[v_1]_B, \ldots [v_k]_B$, respectively.

The set $T$ is linearly independent if and only if there exists
an invertible $k\times k$ submatrix $A_T$ of $[T]_B$. Fix  such a matrix $A_T$.
For such $T$ and a vector $v_{k+1}$, define the ordered sets
$T_i$ for  $1\leq i \leq k$ that are obtained by replacing the vector $v_i$
by $v_{k+1}$. The submatrix $A_{T_i}$ is the submatrix  of
$[T_i]_B$ that corresponds to the submatrix $A_T$ of $[T]_B$,
i.e., it is supported on the same rows.

\begin{Prop} (Cramer's rule)
For a linearly independent set $T=\{v_1, \ldots, v_k\},$
the vector $v_{k+1}\in Span(T)$ if and only if
$v_{k+1}=\sum \frac{\det(A_{T_i})}{\det(A_T)} v_i$.
\end{Prop}

Let $V$ be a vector space of dimension $n$ with a basis $B$. 
For a given set of vectors in $T=\{T_i\in \Q[x]\otimes V\}_{i=1}^k$,
we can consider the matrix $[T]_B \in M_{n\times k}(\Q[x]).$

The set $T$ is  called  \textbf{$\mathbf{S}$-linearly independent} if the matrix $[T]_B$
contains an $S$-invertible  $k\times k$ submatrix $A_T$.
Note that if $\{T_i\}$ is $S$-linearly independent, then
the set $\{e_s(T_i)\}\subset V$
is  linearly independent for all $s\in S$.  

\begin{Lem}\label{compute::lem}
Let $\{T_1, \ldots, T_k\}\subset \Q[x] \otimes V$ be a $S$-linearly independent set. Denote by $A$ any  $k\times k$ submatrix of $[T]_B$ with
$\det(A)\in Q[x](S)$, i.e, $det(A)$ is $S$-non-vanishing. 
An element  $T$ is an element in $Span_{Q[x]_S}\{T_1, \ldots, T_k\}$
if and only if  $\det(A)  T=\sum det(A_i) T_i$.
\end{Lem}

\section{The proof of  Proposition C.0.1}
Let $S$ be the set of all prime powers. 
We recall some key facts about the Iwahori-Hecke algebra. More details  can be found in \Cref{loacl::section:IH}.  
\begin{itemize}
\item
Given an unramified principal series $\Pi = \Ind_{T}^{H}(\lambda)$, we let $\Hecke(\lambda,q)$ be  the left $\Hecke$--module, corresponding to it by the equivalence of categories of \cite{Borel1976}. As a vector space  $\Hecke(\lambda,q) \simeq \Span\set{T_{u} \: : \:  u \in \weyl{G}}$. Note that the dependence in $q$ enters in the multiplication of two elements $T_{u_1} \cdot T_{u_2}$. 

\item
Given $\lambda  \in \mathfrak{a}_{T,\C}^{\ast}$ and a simple reflection $w= \w_{\alpha}$, the normalized intertwining operator 
$\NN_{w_{\alpha}}(\lambda) \: : \: \Ind_{T}^{G}(\lambda) \to \Ind_{T}^{G}(w_{\alpha} \lambda)$   induces a map between the corresponding left $\Hecke$-modules $\Hecke(\lambda,q)$ and $\Hecke(\w_{\alpha}\lambda,q)$.
By \cite[Section 2]{MR2250034}, the action of $\NN_{\w_{\alpha}}(\lambda)\res{\Pi^{\iwhaori}}$ for $\alpha \in \Delta_{G}$ is given by
\textbf{right}-multiplication by the following element
$$n_{\w_\alpha}(\lambda,q)  = \frac{q-1}{q^{z+1}-1} T_{e} + \frac{q^{z}-1}{q^{z+1}-1}T_{\s{\alpha}} \in \Hecke_{0},\quad \text{where} \:\:  z=  \inner{\lambda,\check{\alpha}}. $$ 

In addition, if $w =u_1 u_2$, then 
$n_{u_1u_2}(\lambda,q) =n_{u_2}(\lambda,q)n_{u_1}(u_2\lambda,q)$.

\item
Let $\para{P}_i$ be the maximal parabolic subgroup of $G$ associated with the root $\alpha_i$. For the unramified degenerate principal series $\pi = \Ind_{M_{i}}^{G}(z_0 \fun{i})$, we associated a left $\Hecke$-module $\Hecke_{\para{P}}(z_0,q)$ by the  equivalence of categories of \cite{Borel1976}. As a vector space $ \Hecke_{\para{P}_i}(z_0,q) \simeq \Span \set{  T_{u} \cdot T_{\para{P}_i} \: : \:  u \in w(\para{P}_i,G)}$, where $T_{\para{P}_i} =  \sum_{w \in \weyl{M_i} } T_{w}$.

\item 
By induction in stages, for every $z \in \C$, 
the degenerate principal series $\Ind_{M_i}^{G}(z\fun{i})
 \hookrightarrow \Ind_{T}^{G}(\chi_{\para{P}_i,z})$. Thus, we set   
   $\NN_{w}(z) \: : \:  \Ind_{M_i}^{G}(z\fun{i}) \to \Ind_{T}^{G}(w \chi_{\para{P}_i})$ to be the restriction of $\NN_{w}(\chi_{\para{P}_i,z})$ to the subrepresentation $\pi_{z}=\Ind_{M_{i}}^{G}(z \fun{i})$. Namely, we consider $\NN_{w}(z)$ to be the restriction of $\NN_{w}(\lambda)$ to the line $\chi_{\para{P}_i,z}$ and then to the subspace $\pi_{z}$.  In the same fashion, we set   $n_{w,\para{P}_i}(z,q) \: : \Hecke_{\para{P}_i}(z,q) \to \Hecke(\w\chi_{\para{P}_i,z})$ to be the restriction of $n_{w}(\lambda,q)$ to the line $\chi_{\para{P}_i,z}$ and to the subspace $\Hecke_{\para{P}_i}(z,q)$. 
\end{itemize}

For $(\para{P}_i,z_0) \in \set{ (\para{P}_1,1),(\para{P}_3,\frac{1}{2}), (\para{P}_4,\frac{5}{2})}$  one has :
\begin{itemize}
\item
For every $\beta \in \Phi_{G}^{+} \setminus \Phi^{+}_{M}$, one has $\inner{\chi_{\para{P}_i,z_0},\check{\beta}} \in \Z$. As in \Cref{local::p3::der}, we write the expansion of $n_{w,\para{P}}(z,q)$ around $z=z_0$. By \Cref{local::holo} , each $\NN_{w}(z)$ is holomorphic at $z=z_0$. This means that as a function of $z$, the element 
$$v(z,q,w) = T_{\para{P}_i}\cdot  n_{w,\para{P}}(z,q)$$  is holomorphic at $z=z_0$. A direct calculation shows that
$v(z_0,q,w) = \sum_{u \in \weyl{G}}  \frac{p_{1,u,w}(q)}{p_{2,u,w}(q)}T_{u}$ for some $p_{1,u,w}(x),p_{2,u,w}(x) \in \Q[x]$. 
\item  
Set $p_{w}(x) =  lcm(p_{2,u,w}(x))_{u \in \weyl{G}}$. 
An application of \Cref{alg:nonvanishing}  shows that $p_{w}(x)$ is $S$-non-vanishing. This implies that $v(z_0,q,w) \in \Hecke_{0} \otimes \Q[x]_{S}$.
\item
Since $T_{\para{P}_i} \cdot T_{\para{P}_i} = c(q) \cdot T_{\para{P}_i}$ where $c(x)\in \Q[x]_{S}$, it follows that 
as operators, the image of $n_{w,\para{P}}(z_0,q)$ and the image of right multiplication by $v(z_0,q,w)$ are the same when restricted to $\Hecke_{\para{P}_i}(z_0,q)$. Since $p_w(x)$ is $S$-non-vanishing one has
\begin{equation} \label{coumute::eq1}
\Image n_{w}(z_0,q,w) =\Span\set { T_u \cdot v_1(z_0,q,w) \: : \: u \in W(\para{P}_i,G) },
\end{equation}
where  $v_1(z_0,q,w) = p_{w}(q) v(z_0,q,w)$.
 
\item 
By abuse of notation, we let $v_{1}(z_0,q,w) \: : \: \Hecke_{\para{P}}(z_0,q)  \to \Hecke(\w \chi_{\para{P},z_0},q)$ be  the operator given by right multiplication with $v_{1}(z_0,q,w)$.  
Since  $v(z_0,q,w)$ is  
an  element in $\Hecke_{0} \otimes \Q[x]_{S}$, it is enough to show that $\dim \Image v_{1}(z_0,q,w)$ is not dependent on $q$. In addition, since $v_{1}(z_0,q,w)\in \Span_{\Q}\set{ T_{u} \: : \: u \in \weyl{G}}$ for each $q \in S$, it suffices to calculate $\dim_{\Q} \Image v_{1}(z_0,q,w)$.
 
\item 
Let $\Lambda_w(z_0,q)$ be the matrix whose columns are $[T_{u} \cdot v_{1}(z_0,q,w)]_{B}$, where $u \in W(\para{P}_i,G)$, and $B$ is the basis $\set{T_{u_1} \: : \:  u_1 \in \weyl{G}}$ of $\Hecke_{0}$. Clearly, for every $q \in S$ one has $$\dim \Image v_{1}(z_0,q,w) =  \operatorname{Rank}  \Lambda_w(z_0,q).$$
 
In addition, since $v_{1}(z_0,q,w) \in \Q[q] \otimes \Hecke_{0}$, one has 
 $\Lambda_w(z_0,q) \in M_{n\times m}(\Q[q])$, where $m = |W(\para{P}_i,G)|$ and $n =|\weyl{G}|$. By abuse of notation, we consider $q$ a variable. 
 \item
 Note that 
 $\dim_{\Q (q)} \Image v_{1}(z_0,q,w)\ge \dim_{\Q} \Image v_{1}(z_0,q_1,w)$ for all $q_1\in S$.
 \item 
 Let $m=\dim \Image v_{1}(z_0,2,w)$. Select the elements $u_1, \dots ,u_{m} \in W(\para{P}_i,G)$ such that the set $ \set{  T_{u_i} \cdot v_{1}(z_0,2,w)}_{i=1}^{m}$ is linearly independent. Denote by   $\tilde{T}_{i}(q)= T_{u_i}\cdot v_{1}(z_0,q,w)$. Note that the set  $\set{ \tilde{T}_{i}(q)}_{i=1}^{m}$ is $\Q(q)$- linearly independent.

 \begin{Lem}
 The  set $\{\tilde T_i(q)\}_{i=1}^{m}$ 
 is  $S$-linearly independent. In addition  $\{\tilde T_i(q)\}_{i=1}^{m}$ is a basis 
 of $\Image v_{1}(z_0, q, w)\subset \Hecke_{0}(w\chi_{\para{P}_i,z_0},q ) \otimes \Q[q]_{S}$.
 \end{Lem}
 Assume the Lemma is proved.  For any $q_1\in S$, the set
 $\set{e_{q_1}(\tilde T_i(q))} \subset \Image n_w(z_0, q_1)$ is linearly independent.
 Thus, $$m \leq \dim_{\Q} \Image n_w(z_0,q_1)\le \dim_{\Q[q]_{S}} \Image n_w(z_0,q)\le m, \quad \forall q_1 \in S.$$
 
 In other words, $\dim \Image n_w(z_0,q_1)=m$ for any $q_1 \in S$, and the proposition is proved. 
 
 It remains to prove the Lemma.  We start by showing that $\set{\tilde T_i(q)}_{i=1}^{m}$ is $S$--linearly independent.
 
 Since $\set{\tilde T_i(2)}_{i=1}^{m}$ is linearly independent, there  exists a $m \times m$  submatrix  $A_{T}(2)$ of $\Lambda_{w}(z_0, 2)$ which is invertible. Let $A_{T}(q)$ be the pullback of $A_{T}(2)$.
 
 Using \Cref{alg:det}, we find $p(q) = \det (A_{T}(q)) \in \Q[q]$.  In order to show that $\set{\tilde T_i(q)}_{i=1}^{m}$ is $S$-linearly independent, it suffices to show that for every $s \in S$, one has $p(s) \neq 0$. For this purpose, we may assume that $p(q) \in \Z[x]$ and with a non-zero constant coefficient. Otherwise, we multiply by $c \in \Z$ and divide by some power of $q$. Checking that $p(q)$ is $S$ non-vanishing is accomplished through \Cref{alg:nonvanishing}. This shows that the set $\tilde T_i(q)$ is  $S$-linearly independent. 
 
 To show that it is a basis of the image, it is enough to show that for any $u \in W(\para{P}_i,G)$,  one has 
 $$T_{u} \cdot v_{1}(z_0,q,w) \in \Span_{\Q[q]_{S}}\set{ \tilde{T}_{i}(q) \: : \:  i=1, \dots m}.$$

 By \Cref{compute::lem}  it is the same as checking that $$T_u\cdot  v_{1}(z_0,q,w) =\sum \frac{\det A_{T_{u,i}}(q)}{\det A_{T}(q)} \tilde T_i(q),$$
 where $A_{T_{u,i}}(q)$ is obtained by replacing the column
 $[\tilde{T}_i(q)]$ of $A_{T}(q)$ by the column $[T_u \cdot v_{1}(z_0,q,w)]$. 
 
 \end{itemize}

\begin{algorithm}[H] 
\caption{:  Find $\det A(x)$}
\label{alg:det}
\begin{algorithmic}[1]
\Require{$A(x)\in M_{n}(\Q[x])$ such that
 $\deg a_{ij}(x)\le d$} 
\Ensure{$\det A(x)$ }
  \State {Compute $\det(e_{q}(A(x)))$ at $nd+1$ different values of
     $x=q$}
    \State {Use Lagrange interpolation method to find the polynomial
       $\det(A(x))$}
    \State \Return {$\det(A(x))$}
\end{algorithmic}
\end{algorithm}

\begin{algorithm}[H] 
\caption{$S$  non-vanishing}
\label{alg:nonvanishing}
\begin{algorithmic}[1]
\Require{$p(x)=\sum_{i=0}^{k}\frac{a_i}{b_i}x^{i} \in \Q[x]$, with  $a_0 \neq 0$ and $S \subset \Z$ } 
\Ensure{$p(x)$ is $S$ invertible}
\Statex
  \State {$c$ $\gets$ {$\operatorname{lcm\set{b_1,\dots b_k}}$}}
    \State {$p_{1}(x) $ $\gets$ {$c p(x)$}}

    \State {$possible\_roots $ $\gets$ {$\set{m \: : \:  m\vert  p_1(0) }$}}
    \For{each s in $possible\_roots \cap S$}
            \If{$p_1(s)=0$ }
                \State \Return {false}
            \EndIf
    \EndFor
    \State \Return {$true$}
\end{algorithmic}
\end{algorithm}


\chapter{Local information}\label{app:localtable}
\color{black}
In this appendix we present  the tables used along this thesis.  
The appendix contains two kinds of tables:
\begin{itemize}
\item
Tables of the first type summarize the information concerning  the semi-simplification of the Jacquet module $\jac{G}{T}{\Pi}$, where $\Pi = \Ind_{M}^{G}(z)$ and its irreducible constituents. Here $Exp$ represents  the  coordinate vector of  an exponent $\lambda$ in the basis of fundamental weights.    
\item
Tables of the second type  summarize the information about the normalized intertwining operators. These tables should be read in the following way:

Recall that $\Ind_{M}^{G}(z_0) \hookrightarrow \Ind_{T}^{G}(\lambda_0)$, where $\lambda_0 = \jac{M}{T}{z_0 \fun{M}}$.  
For each word $w$ in $W(\para{P},G)$ we write the following data: 
\begin{itemize}
\item
In the column $Exp$ we write the coordinate vector of $w\lambda_0$ in the basis of fundamental weights.
\item
In the column $\dim \Sigma_{w}^{\iwhaori}$ we write the dimension of the Iwahori-fixed vectors of $\Image \NN_{w}(z_0)\res{\Pi}$.    
\item
In the  column \textit{type} we give the reason why the normalized intertwining operator $\NN_{w}(z)$ is holomorphic at $z=z_0$. Recall that  in \Cref{local::holo} we defined  holomorphic operators of three types:  H,Z and  S. We write $\NN_{w}(z)$ as a composition of these type e.g. suppose $w = u_1 u_2 u_3$ which is of type $HZS$ then $$\NN_{w}(z) = \NN_{u_3}^{u_2 u_1 \chi_{\para{P},z}}(u_2 u_1 \chi_{\para{P},z}) \circ   
\NN_{u_2}^{u_1 \chi_{\para{P},z}}(u_1 \chi_{\para{P},z}) \circ \NN_{u_1}(z).$$

Here $\NN_{u_3}^{u_2 u_1 \chi_{\para{P},z}}(u_2 u_1 \chi_{\para{P},z})$ is operator of type H, 
$\NN_{u_2}^{u_1 \chi_{\para{P},z}}(u_1 \chi_{\para{P},z})$ is operator of type $Z$
and $\NN_{u_1}(z)$ is operator of type  S at $z=z_0$.
\item
The column $Const$ contains information only for $w$ satisfying:
$\Sigma_w\simeq \pi_1\oplus \pi_2$  and $\Sigma_u=\Sigma_w$  for each $u\in [w]_{z_0}.$ 
Such $w$ can be written as $w=su$ where $u$ is the shortest element in $[w]_{z_0}$, which always exists, 
and $s\in \Stab_{\weyl{G}} (u\chi_{\para{P},z_0})$.  The operator $\NN_s\in \operatorname{End}(\Sigma_u)$  
acts as identity on $\pi_1$ and by a constant $c$ on $\pi_2$. The constant $c$ is tabulated in the column $Const$.          
\end{itemize}
The tables of the second type are organized by $\dim$ and then by $Exp$.
\end{itemize}

 \color{black}           
 \begin{longtable}{|c|c|c|c|cH|} 
\hline
Exp & $\mult{\lambda}{\jac{G}{T}{\Pi}}$ & $\mult{\lambda}{\jac{G}{T}{\pi_1}}$   & $\mult{\lambda}{\jac{G}{T}{\pi_2}}$  & $\mult{\lambda}{\jac{G}{T}{\tau}}$ &  \\
\hline
\endhead
$\left[-1, 0, -1, -1\right]$ & $2$ & $2$ & $0$ & $0$ & START1 \\ \hline
$\left[-1, -1, 1, -2\right]$ & $2$ & $1$ & $1$ & $0$ & $0$ \\ \hline
$\left[-1, -1, -1, 2\right]$ & $2$ & $1$ & $1$ & $0$ & START2 \\ \hline
$\left[4, -1, -1, -3\right]$ & $1$ & $0$ & $0$ & $1$ & $0$ \\ \hline
$\left[4, -1, -4, 3\right]$ & $1$ & $0$ & $0$ & $1$ & $0$ \\ \hline
$\left[4, -5, 4, -1\right]$ & $1$ & $0$ & $0$ & $1$ & $0$ \\ \hline
$\left[1, -1, -1, -1\right]$ & $1$ & $1$ & $0$ & $0$ & $0$ \\ \hline
$\left[1, -2, 1, -2\right]$ & $1$ & $0$ & $0$ & $1$ & $0$ \\ \hline
$\left[1, -2, -1, 2\right]$ & $1$ & $0$ & $0$ & $1$ & $0$ \\ \hline
$\left[-1, 5, -6, -1\right]$ & $1$ & $0$ & $0$ & $1$ & $0$ \\ \hline
$\left[-1, 2, -3, -2\right]$ & $1$ & $0$ & $0$ & $1$ & $0$ \\ \hline
$\left[-1, 2, -5, 2\right]$ & $1$ & $0$ & $0$ & $1$ & $0$ \\ \hline
$\left[-1, -1, 6, -7\right]$ & $1$ & $0$ & $0$ & $1$ & $0$ \\ \hline
$\left[-1, -1, 3, -5\right]$ & $1$ & $0$ & $0$ & $1$ & $0$ \\ \hline
$\left[-1, -1, -1, 7\right]$ & $1$ & $0$ & $0$ & $1$ & START3 \\ \hline
$\left[-1, -1, -2, 5\right]$ & $1$ & $0$ & $0$ & $1$ & $0$ \\ \hline
$\left[-1, -3, 5, -3\right]$ & $1$ & $0$ & $0$ & $1$ & $0$ \\ \hline
$\left[-1, -3, 2, 3\right]$ & $1$ & $0$ & $0$ & $1$ & $0$ \\ \hline
$\left[-4, 3, -1, -3\right]$ & $1$ & $0$ & $0$ & $1$ & $0$ \\ \hline
$\left[-4, 3, -4, 3\right]$ & $1$ & $0$ & $0$ & $1$ & $0$ \\ \hline
$\left[-4, -1, 4, -1\right]$ & $1$ & $0$ & $0$ & $1$ & $0$ \\ \hline
\caption{$\Ind_{M_4}^{G}(\frac{5}{2})$ Jacquet Module }
\label{Table::P4}

 \end{longtable}
\newpage
\begin{longtable}{|r|r|c|Hc|c|} \hline
 word  & Type  &  $ \dim \Sigma_{w}^{\iwhaori}$  &   $ \Sigma_{w}$  &  Exp  \\  \hline     
 \endhead
  $\bk{w_{3}w_{2}w_{3}w_{4}w_{3}w_{2}w_{3}w_{1}w_{2}w_{3}w_{4}}$  & 
 H  &  $5$  &   $\pi_1$  &  $\left[1, -1, -1, -1\right]$  \\  \hline     
  $\bk{w_{3}w_{1}w_{2}w_{3}w_{4}w_{3}w_{2}w_{3}w_{1}w_{2}w_{3}w_{4}}$  & 
 H  &  $5$  &   $\pi_1$  &  $\left[-1, 0, -1, -1\right]$  \\  \hline     
  $\bk{w_{2}w_{3}w_{1}w_{2}w_{3}w_{4}w_{3}w_{2}w_{3}w_{1}w_{2}w_{3}w_{4}}$  & 
 H &  $5$  &   $\pi_1$  &  $\left[-1, 0, -1, -1\right]$  \\  \hline     
  $\bk{w_{1}w_{2}w_{3}w_{4}w_{2}w_{3}w_{1}w_{2}w_{3}w_{4}}$  & 
 H &  $7$  &   $\pi_1 \oplus \pi_2$  &  $\left[-1, -1, -1, 2\right]$  \\  \hline     
  $\bk{w_{4}}\bk{w_{3}w_{2}w_{3}}\bk{w_{1}w_{2}w_{3}w_{4}w_{3}w_{2}w_{3}w_{1}w_{2}w_{3}w_{4}}$  & 
 HZH &  $7$  &   $\pi_1 \oplus \pi_2$  &  $\left[-1, -1, -1, 2\right]$  \\  \hline     
  $\bk{w_{1}w_{2}w_{3}w_{4}w_{3}w_{2}w_{3}w_{1}w_{2}w_{3}w_{4}}$  & 
 H &  $7$  &   $\pi_1 \oplus \pi_2$  &  $\left[-1, -1, 1, -2\right]$  \\  \hline     
  $\bk{w_{3}w_{2}w_{3}}\bk{w_{1}w_{2}w_{3}w_{4}w_{3}w_{2}w_{3}w_{1}w_{2}w_{3}w_{4}}$  & 
 ZH &  $7$  &   $\pi_1 \oplus \pi_2$  &  $\left[-1, -1, 1, -2\right]$  \\  \hline     
  $1$  & 
 H &  $24$  &   $\Pi$  &  $\left[-1, -1, -1, 7\right]$  \\  \hline     
  $\bk{w_{4}}$  & 
 H &  $24$  &   $\Pi$  &  $\left[-1, -1, 6, -7\right]$  \\  \hline     
  $\bk{w_{3}w_{4}}$  & 
 H &  $24$  &   $\Pi$  &  $\left[-1, 5, -6, -1\right]$  \\  \hline     
  $\bk{w_{2}w_{3}w_{4}}$  & 
 H &  $24$  &   $\Pi$  &  $\left[4, -5, 4, -1\right]$  \\  \hline     
  $\bk{w_{3}w_{2}w_{3}w_{4}}$  & 
 H &  $24$  &   $\Pi$  &  $\left[4, -1, -4, 3\right]$  \\  \hline     
  $\bk{w_{1}w_{2}w_{3}w_{4}}$  & 
 H &  $24$  &   $\Pi$  &  $\left[-4, -1, 4, -1\right]$  \\  \hline     
  $\bk{w_{3}w_{1}w_{2}w_{3}w_{4}}$  & 
 H &  $24$  &   $\Pi$  &  $\left[-4, 3, -4, 3\right]$  \\  \hline     
  $\bk{w_{2}w_{3}w_{1}w_{2}w_{3}w_{4}}$  & 
 H &  $24$  &   $\Pi$  &  $\left[-1, -3, 2, 3\right]$  \\  \hline     
  $\bk{w_{3}w_{2}w_{3}w_{1}w_{2}w_{3}w_{4}}$  & 
 H &  $24$  &   $\Pi$  &  $\left[-1, -1, -2, 5\right]$  \\  \hline     
  $\bk{w_{4}w_{3}w_{2}w_{3}w_{4}}$  & 
 H &  $24$  &   $\Pi$  &  $\left[4, -1, -1, -3\right]$  \\  \hline     
  $\bk{w_{4}w_{3}w_{1}w_{2}w_{3}w_{4}}$  & 
 H &  $24$  &   $\Pi$  &  $\left[-4, 3, -1, -3\right]$  \\  \hline     
  $\bk{w_{4}w_{2}w_{3}w_{1}w_{2}w_{3}w_{4}}$  & 
 H &  $24$  &   $\Pi$  &  $\left[-1, -3, 5, -3\right]$  \\  \hline     
  $\bk{w_{3}w_{4}w_{2}w_{3}w_{1}w_{2}w_{3}w_{4}}$  & 
 H &  $24$  &   $\Pi$  &  $\left[-1, 2, -5, 2\right]$  \\  \hline     
  $\bk{w_{2}w_{3}w_{4}w_{2}w_{3}w_{1}w_{2}w_{3}w_{4}}$  & 
 H &  $24$  &   $\Pi$  &  $\left[1, -2, -1, 2\right]$  \\  \hline     
  $\bk{w_{4}w_{3}w_{2}w_{3}w_{1}w_{2}w_{3}w_{4}}$  & 
 H &  $24$  &   $\Pi$  &  $\left[-1, -1, 3, -5\right]$  \\  \hline     
  $\bk{w_{3}w_{4}w_{3}w_{2}w_{3}w_{1}w_{2}w_{3}w_{4}}$  & 
 H &  $24$  &   $\Pi$  &  $\left[-1, 2, -3, -2\right]$  \\  \hline     
  $\bk{w_{2}w_{3}w_{4}w_{3}w_{2}w_{3}w_{1}w_{2}w_{3}w_{4}}$  & 
 H &  $24$  &   $\Pi$  &  $\left[1, -2, 1, -2\right]$  \\  \hline     
\caption{Images of intertwining operators of $P_4$}
\label{Table:: P4 ::images}
\end{longtable}


\begin{longtable}{|c|c|c|c|cH|} 
\hline
 Exp & $\mult{\lambda}{\jac{G}{T}{\pi}}$ & $\mult{\lambda}{\jac{G}{T}{\pi_1}}$   & $\mult{\lambda}{\jac{G}{T}{\pi_2}}$  & $\mult{\lambda}{\jac{G}{T}{\sigma}}$ &  \\ \hline
$\left[-1, 0, -1, 0\right] $ & $4$ & $4$ & $0$ & $0$ & START1 \\ \hline
$\left[2, -1, -1, -1\right] $ & $2$ & $1$ & $1$ & $0$ & $0$ \\ \hline
$\left[1, -1, -1, 0\right] $ & $2$ & $2$ & $0$ & $0$ & $0$ \\ \hline
$\left[-1, -1, 1, -1\right] $ & $2$ & $2$ & $0$ & $0$ & $0$ \\ \hline
$\left[-2, 1, -1, -1\right] $ & $2$ & $1$ & $1$ & $0$ & START2 \\ \hline
$\left[4, -1, -1, -1\right] $ & $1$ & $0$ & $0$ & $1$ & START3 \\ \hline
$\left[1, 1, -3, -1\right] $ & $1$ & $0$ & $0$ & $1$ & $0$ \\ \hline
$\left[1, -2, 3, -4\right] $ & $1$ & $0$ & $0$ & $1$ & $0$ \\ \hline
$\left[1, -2, 1, -1\right] $ & $1$ & $1$ & $0$ & $0$ & $0$ \\ \hline
$\left[1, -2, -1, 4\right] $ & $1$ & $0$ & $0$ & $1$ & $0$ \\ \hline
$\left[-1, 2, -1, -4\right] $ & $1$ & $0$ & $0$ & $1$ & $0$ \\ \hline
$\left[-1, 2, -3, -1\right] $ & $1$ & $1$ & $0$ & $0$ & $0$ \\ \hline
$\left[-1, 2, -5, 4\right] $ & $1$ & $0$ & $0$ & $1$ & $0$ \\ \hline
$\left[-1, -1, 3, -4\right] $ & $1$ & $1$ & $0$ & $0$ & $0$ \\ \hline
$\left[-1, -1, -1, 4\right] $ & $1$ & $1$ & $0$ & $0$ & $0$ \\ \hline
$\left[-1, -3, 5, -1\right] $ & $1$ & $0$ & $0$ & $1$ & $0$ \\ \hline
$\left[-4, 3, -1, -1\right] $ & $1$ & $0$ & $0$ & $1$ & $0$ \\ \hline
\caption{$\Ind_{M_1}^{G}(1)$ Jacquet Module }
\label{Table::P1}

 \end{longtable}
\begin{longtable}{|r|r|c|Hc|c|} \hline
  word  & Type  &  $ \dim \Sigma_{w}^{\iwhaori}$  &   $ \Sigma_{w}$  &  Exp  \\  \hline \endhead     
  $\bk{w_{1}w_{2}w_{3}w_{2}w_{1}}$  &  
 H &  $14$  &   $\Pi$  &  $\left[-1, -1, -1, 4\right]$  \\  \hline     
  $\bk{w_{4}w_{1}w_{2}w_{3}w_{2}w_{1}}$  & 
 H &  $14$  &   $\Pi$  &  $\left[-1, -1, 3, -4\right]$  \\  \hline     
  $\bk{w_{3}w_{4}w_{1}w_{2}w_{3}w_{2}w_{1}}$  & 
 H &  $14$  &   $\Pi$  &  $\left[-1, 2, -3, -1\right]$  \\  \hline     
  $\bk{w_{2}w_{3}w_{4}w_{1}w_{2}w_{3}w_{2}w_{1}}$  & 
 H &  $14$  &   $\Pi$  &  $\left[1, -2, 1, -1\right]$  \\  \hline     
  $\bk{w_{3}w_{2}w_{3}w_{4}w_{1}w_{2}w_{3}w_{2}w_{1}}$  & 
 H &  $14$  &   $\Pi$  &  $\left[1, -1, -1, 0\right]$  \\  \hline     
  $\bk{w_{4}w_{3}w_{2}w_{3}w_{4}w_{1}w_{2}w_{3}w_{2}w_{1}}$  & 
 H &  $14$  &   $\Pi$  &  $\left[1, -1, -1, 0\right]$  \\  \hline     
  $\bk{w_{1}w_{2}w_{3}w_{4}w_{1}w_{2}w_{3}w_{2}w_{1}}$  & 
 H &  $14$  &   $\Pi$  &  $\left[-1, -1, 1, -1\right]$  \\  \hline     
  $\bk{w_{3}w_{4}w_{2}w_{3}}\bk{w_{1}w_{2}w_{3}w_{4}w_{1}w_{2}w_{3}w_{2}w_{1}}$  & 
 ZH &  $14$  &   $\Pi$  &  $\left[-1, -1, 1, -1\right]$  \\  \hline     
  $\bk{w_{3}w_{1}w_{2}w_{3}w_{4}w_{1}w_{2}w_{3}w_{2}w_{1}}$  & 
 H &  $14$  &   $\Pi$  &  $\left[-1, 0, -1, 0\right]$  \\  \hline     
  $\bk{w_{2}w_{3}w_{1}w_{2}w_{3}w_{4}w_{1}w_{2}w_{3}w_{2}w_{1}}$  & 
 H &  $14$  &   $\Pi$  &  $\left[-1, 0, -1, 0\right]$  \\  \hline     
  $\bk{w_{4}w_{3}w_{1}w_{2}w_{3}w_{4}w_{1}w_{2}w_{3}w_{2}w_{1}}$  & 
 H &  $14$  &   $\Pi$  &  $\left[-1, 0, -1, 0\right]$  \\  \hline     
  $\bk{w_{4}w_{2}w_{3}w_{1}w_{2}w_{3}w_{4}w_{1}w_{2}w_{3}w_{2}w_{1}}$  & 
 H &  $14$  &   $\Pi$  &  $\left[-1, 0, -1, 0\right]$  \\  \hline     
  $\bk{w_{2}w_{3}w_{4}w_{2}w_{3}w_{2}w_{1}}$  & 
 H &  $16$  &   $\Pi$  &  $\left[2, -1, -1, -1\right]$  \\  \hline     
  $w_{1}w_{2}w_{3}w_{4}w_{2}w_{3}w_{1}w_{2}w_{3}w_{4}w_{1}w_{2}w_{3}w_{2}w_{1}$  & 
 HSH &  $16$  &   $\Pi$  &  $\left[2, -1, -1, -1\right]$  \\  \hline     
  $\bk{w_{1}w_{2}w_{3}w_{4}w_{2}w_{3}w_{2}w_{1}}$  & 
 H &  $16$  &   $\Pi$  &  $\left[-2, 1, -1, -1\right]$  \\  \hline     
  $w_{2}w_{3}w_{4}w_{2}w_{3}w_{1}w_{2}w_{3}w_{4}w_{1}w_{2}w_{3}w_{2}w_{1}$  & 
 SH &  $16$  &   $\Pi$  &  $\left[-2, 1, -1, -1\right]$  \\  \hline     
  $1$  & 
 H &  $24$  &   $\Pi$  &  $\left[4, -1, -1, -1\right]$  \\  \hline     
  $\bk{w_{1}}$  & 
 H &  $24$  &   $\Pi$  &  $\left[-4, 3, -1, -1\right]$  \\  \hline     
  $\bk{w_{2}w_{1}}$  & 
 H &  $24$  &   $\Pi$  &  $\left[-1, -3, 5, -1\right]$  \\  \hline     
  $\bk{w_{3}w_{2}w_{1}}$  & 
 H &  $24$  &   $\Pi$  &  $\left[-1, 2, -5, 4\right]$  \\  \hline     
  $\bk{w_{2}w_{3}w_{2}w_{1}}$  & 
 H &  $24$  &   $\Pi$  &  $\left[1, -2, -1, 4\right]$  \\  \hline     
  $\bk{w_{4}w_{3}w_{2}w_{1}}$  & 
 H &  $24$  &   $\Pi$  &  $\left[-1, 2, -1, -4\right]$  \\  \hline     
  $\bk{w_{4}w_{2}w_{3}w_{2}w_{1}}$  & 
 H &  $24$  &   $\Pi$  &  $\left[1, -2, 3, -4\right]$ H \\  \hline     
  $\bk{w_{3}w_{4}w_{2}w_{3}w_{2}w_{1}}$  & 
 H &  $24$  &   $\Pi$  &  $\left[1, 1, -3, -1\right]$  \\  \hline     
 \caption{Images of intertwining operators of $P_1$}
 \label{Table:: P1 ::images}
\end{longtable}

  \newpage
 
 \begin{landscape} 
 \renewcommand*{\arraystretch}{1.5} 
 \fontsize{10}{12}
  \begin{longtable}{|R|L|c|c|c|}
  \hline 
  word  & Type  &   $\dim \Sigma_{w}^{\iwhaori}$   &   Exp  &  Const \\  \hline  
\endhead 
\hline
  $w_{3}w_{4}w_{2}w_{3}w_{1}w_{2}w_{3}w_{4}w_{3}w_{1}w_{2}w_{3}$  &  $H$  &   42   &   $ \left[0, 0, -1, 0\right] $  &  -- \\   \hline     
  $w_{3}w_{4}w_{2}w_{3}w_{1}w_{2}w_{3}w_{4}w_{2}w_{3}w_{1}w_{2}w_{3} $  &  $H$  &   42   &   $ \left[0, 0, -1, 0\right] $   &  --\\    \hline     
  $w_{3}w_{4}w_{3}w_{2}w_{3}w_{1}w_{2}w_{3}w_{4}w_{3}w_{1}w_{2}w_{3} $  &  $H$  &   42   &   $ \left[0, 0, -1, 0\right] $  &  --\\    \hline     
  $w_{2}w_{3}w_{4}w_{2}w_{3}w_{1}w_{2}w_{3}w_{4}w_{3}w_{1}w_{2}w_{3} $  &  $H$  &   42   &   $ \left[0, 0, -1, 0\right] $  &  -- \\   \hline     
  $w_{3}w_{4}w_{3}w_{2}w_{3}w_{1}w_{2}w_{3}w_{4}w_{2}w_{3}w_{1}w_{2}w_{3} $  &  $H$  &   42   &   $ \left[0, 0, -1, 0\right] $  &  -- \\    \hline     
  $w_{2}w_{3}w_{4}w_{2}w_{3}w_{1}w_{2}w_{3}w_{4}w_{2}w_{3}w_{1}w_{2}w_{3} $  &   $H$  &   42   &   $ \left[0, 0, -1, 0\right] $  &  -- \\   \hline     
  $w_{2}w_{3}w_{4}w_{3}w_{2}w_{3}w_{1}w_{2}w_{3}w_{4}w_{3}w_{1}w_{2}w_{3} $  &   $H$  &   42   &   $ \left[0, 0, -1, 0\right] $  &  -- \\    \hline     
  $w_{1}w_{2}w_{3}w_{4}w_{2}w_{3}w_{1}w_{2}w_{3}w_{4}w_{3}w_{1}w_{2}w_{3} $  &   $H$  &   42   &   $ \left[0, 0, -1, 0\right] $  &  -- \\    \hline     
  $w_{2}w_{3}w_{4}w_{3}w_{2}w_{3}w_{1}w_{2}w_{3}w_{4}w_{2}w_{3}w_{1}w_{2}w_{3} $  &  $H$  &   42   &   $ \left[0, 0, -1, 0\right] $  &  -- \\    \hline     
  $w_{1}w_{2}w_{3}w_{4}w_{2}w_{3}w_{1}w_{2}w_{3}w_{4}w_{2}w_{3}w_{1}w_{2}w_{3} $  &  $H$  &   42   &   $ \left[0, 0, -1, 0\right] $   &  --\\    \hline     
  $w_{1}w_{2}w_{3}w_{4}w_{3}w_{2}w_{3}w_{1}w_{2}w_{3}w_{4}w_{3}w_{1}w_{2}w_{3} $  &  $H$  &   42   &   $ \left[0, 0, -1, 0\right] $   &  --\\    \hline     
  $w_{1}w_{2}w_{3}w_{4}w_{3}w_{2}w_{3}w_{1}w_{2}w_{3}w_{4}w_{2}w_{3}w_{1}w_{2}w_{3} $  &  $H$  &   42   &   $ \left[0, 0, -1, 0\right] $  &  -- \\    \hline     
  $w_{4}w_{2}w_{3}w_{1}w_{2}w_{3}w_{4}w_{3}w_{1}w_{2}w_{3} $  &  $H$  &   42   &   $ \left[0, -1, 1, -1\right] $   &  --\\   \hline     
  $w_{4}w_{2}w_{3}w_{1}w_{2}w_{3}w_{4}w_{2}w_{3}w_{1}w_{2}w_{3} $  &  $H$  &   42   &   $ \left[0, -1, 1, -1\right] $  &  -- \\    \hline     
  $w_{4}w_{3}w_{2}w_{3}w_{1}w_{2}w_{3}w_{4}w_{3}w_{1}w_{2}w_{3} $  &   $H$  &   42   &   $ \left[0, -1, 1, -1\right] $  &  -- \\    \hline     
  $w_{2}w_{3}w_{4}w_{2}w_{3}w_{1}w_{2}w_{3}w_{4}w_{1}w_{2}w_{3} $  &  $H$  &   42   &   $ \left[0, -1, 1, -1\right] $  &  -- \\    \hline     
  $w_{4}w_{3}w_{2}w_{3}w_{1}w_{2}w_{3}w_{4}w_{2}w_{3}w_{1}w_{2}w_{3} $  &  $H$  &   42   &   $ \left[0, -1, 1, -1\right] $  &  -- \\    \hline     
  $w_{1}w_{2}w_{3}w_{4}w_{2}w_{3}w_{1}w_{2}w_{3}w_{4}w_{1}w_{2}w_{3} $  &   $H$  &   42   &   $ \left[0, -1, 1, -1\right] $  &  -- \\    \hline     
  $ \left(w_{2}w_{3}w_{4}w_{2}w_{3}w_{1}w_{2}w_{3}w_{4} \right) \cdot \left(w_{3}w_{2}w_{3}\right) 
   \cdot \left(w_{1}w_{2}w_{3}\right) $  &  
 $HZH$  &   42   &   $ \left[0, -1, 1, -1\right] $  &  -- \\    \hline     
  $\left(w_{1}w_{2}w_{3}w_{4}w_{2}w_{3}w_{1}w_{2}w_{3}w_{4}\right) 
   \cdot \left(w_{3}w_{2}w_{3}\right) 
   \cdot \left(w_{1}w_{2}w_{3} \right)$  &  
 $HZH$  &   42   &   $ \left[0, -1, 1, -1\right] $  &  -- \\    \hline     
  $ \left(w_{3}w_{2}w_{3}\right) \cdot \left(w_{4}w_{3}w_{2}w_{3}w_{1}w_{2}w_{3}w_{4}w_{2}w_{3}w_{1}w_{2}w_{3} \right) $  &   $ZH$  &   42   &   $ \left[0, -1, 1, -1\right] $  &  -- \\    \hline     
  $ \left(w_{3}w_{1}w_{2}w_{3}\right)\cdot \left(w_{4}w_{3}w_{2}w_{3}w_{1}w_{2}w_{3}w_{4}w_{2}w_{3}w_{1}w_{2}w_{3} \right) $  &  
 $ZH$  &   42   &   $ \left[0, -1, 1, -1\right] $  &  -- \\    \hline     
  $ w_{4}w_{3}w_{1}w_{2}w_{3}w_{4}w_{3}w_{1}w_{2}w_{3} $  &  
 $H$  &   42   &   $ \left[-1, 1, -1, -1\right] $  &  -- \\    \hline     
  $ w_{4}w_{3}w_{1}w_{2}w_{3}w_{4}w_{2}w_{3}w_{1}w_{2}w_{3} $  &  
 $H$  &   42   &   $ \left[-1, 1, -1, -1\right] $  &  -- \\    \hline     
  $ w_{3}w_{4}w_{2}w_{3}w_{1}w_{2}w_{3}w_{4}w_{1}w_{2}w_{3} $  &  
 $H$  &   42   &   $ \left[-1, 1, -1, -1\right] $  &  -- \\    \hline     
  $ w_{4}w_{3}w_{2}w_{3}w_{1}w_{2}w_{3}w_{4}w_{3}w_{2}w_{3} $  &  
 $H$  &   42   &   $ \left[-1, 1, -1, -1\right] $  &  --\\    \hline     
  $ \left(w_{3}w_{4}w_{2}w_{3}w_{1}w_{2}w_{3}w_{4}\right) \cdot \left(w_{3}w_{2}w_{3}\right) 
   \cdot \left(w_{1}w_{2}w_{3} \right)$  &  
 $HZH$  &   42   &   $ \left[-1, 1, -1, -1\right] $  &  -- \\    \hline     
  $ w_{2} w_{3}w_{1}w_{2}w_{3} w_{4}w_{3}w_{2}w_{3}w_{1}w_{2}w_{3}w_{4}w_{2}w_{3}w_{1}w_{2}w_{3} $  & 
 $HZHZHZH$  &   42   &   $ \left[-1, 1, -1, -1\right] $   &  --\\    \hline     
  $ w_{4}w_{3}w_{2}w_{3}w_{4}w_{3}w_{1}w_{2}w_{3} $  &  
 $H$  &   42   &   $ \left[1, 0, -1, -1\right] $  &  -- \\    \hline     
  $ w_{4}w_{3}w_{2}w_{3}w_{4}w_{2}w_{3}w_{1}w_{2}w_{3} $  &  
 $H$  &   42   &   $ \left[1, 0, -1, -1\right] $  &  -- \\    \hline     
  $ w_{2}w_{3}w_{1}w_{2}w_{3}w_{4}w_{3}w_{1}w_{2}w_{3} $  &  
 $H$  &   61   &   $ \left[0, -1, 0, 1\right] $  &  $1$ \\    \hline     
  $ w_{2}w_{3}w_{1}w_{2}w_{3}w_{4}w_{2}w_{3}w_{1}w_{2}w_{3} $  &  
 $H$  &   61   &   $ \left[0, -1, 0, 1\right] $  &  $1$ \\    \hline     
  $ w_{3}w_{2}w_{3}w_{1}w_{2}w_{3}w_{4}w_{3}w_{1}w_{2}w_{3} $  &  
 $H$  &   61   &   $ \left[0, -1, 0, 1\right] $  &  $1$ \\    \hline     
  $ w_{3}w_{2}w_{3}w_{1}w_{2}w_{3}w_{4}w_{2}w_{3}w_{1}w_{2}w_{3} $  &  
 $H$  &   61   &   $ \left[0, -1, 0, 1\right] $  &  $1$ \\    \hline     
  $ \left( w_{2}w_{3}w_{4}\right) \cdot\left(w_{3}w_{2}w_{3}\right) 
   \cdot \left(w_{1}w_{2}w_{3}w_{4}\right) 
   \cdot \left(w_{3}w_{2}w_{3}\right) 
   \cdot \left(w_{1}w_{2}w_{3}\right) $  &  
 $HZHZH$  &   61   &   $ \left[0, -1, 0, 1\right] $  &  $-3$ \\    \hline     
  $ \left(w_{1}w_{2}w_{3}w_{4}\right) \cdot\left(w_{3}w_{2}w_{3}\right) 
   \cdot \left(w_{1}w_{2}w_{3}w_{4}\right) 
   \cdot \left(w_{3}w_{2}w_{3}\right) 
   \cdot \left(w_{1}w_{2}w_{3} \right) $  &  
 $HZHZH$  &   61   &   $ \left[0, -1, 0, 1\right] $  & $-3$ \\    \hline     
  $ \left( w_{3}w_{2}w_{3}w_{4}\right) \cdot\left(w_{3}w_{2}w_{3}\right) 
   \cdot \left(w_{1}w_{2}w_{3}w_{4}\right) 
   \cdot \left(w_{3}w_{2}w_{3}\right) 
   \cdot \left(w_{1}w_{2}w_{3} \right) $  &  
 $HZHZH$  &   61   &   $ \left[0, -1, 0, 1\right] $  &  $-3$ \\    \hline     
  $ \left(w_{3}w_{1}w_{2}w_{3}w_{4}\right) \cdot\left(w_{3}w_{2}w_{3}\right) 
   \cdot \left(w_{1}w_{2}w_{3}w_{4}\right) 
   \cdot \left(w_{3}w_{2}w_{3}\right) 
   \cdot \left(w_{1}w_{2}w_{3}\right) $  &  
 $HZHZH$  &   61   &   $ \left[0, -1, 0, 1\right] $  &  $-3$\\    \hline     
  $ w_{3}w_{1}w_{2}w_{3}w_{4}w_{3}w_{1}w_{2}w_{3} $  &  
 $H$  &   61   &   $ \left[-1, 1, -2, 1\right] $  & $1$ \\    \hline     
  $ w_{3}w_{1}w_{2}w_{3}w_{4}w_{2}w_{3}w_{1}w_{2}w_{3} $  &  
 $H$  &   61   &   $ \left[-1, 1, -2, 1\right] $  &  $1$\\    \hline     
  $ w_{3}w_{2}w_{3}w_{1}w_{2}w_{3}w_{4}w_{3}w_{2}w_{3} $  &  
 $H$  &   61   &   $ \left[-1, 1, -2, 1\right] $ &  $1$ \\    \hline     
  $ 
 \left(
 w_{3}w_{4}\right) \cdot\left(w_{3}w_{2}w_{3}\right) 
   \cdot \left(w_{1}w_{2}w_{3}w_{4}\right) 
   \cdot \left(w_{3}w_{2}w_{3}\right) 
   \cdot \left(w_{1}w_{2}w_{3}\right) $  &  
 $HZHZH$  &   61   &   $ \left[-1, 1, -2, 1\right] $  &  $-3$ \\   \hline     
  $ \left(w_{2}w_{3}w_{1}w_{2}\right) \cdot\left(w_{3}w_{4}\right) \cdot\left(w_{3}w_{2}w_{3}\right) 
   \cdot \left(w_{1}w_{2}w_{3}w_{4}\right) 
   \cdot \left(w_{3}w_{2}w_{3}\right) 
   \cdot \left(w_{1}w_{2}w_{3}\right) $  &  
 $ZHZHZH$  &   61   &   $ \left[-1, 1, -2, 1\right] $  & $-3$\\    \hline     
  $ w_{4}w_{3}w_{1}w_{2}w_{3}w_{4}w_{1}w_{2}w_{3} $  &  
 $H$  &   61   &   $ \left[-1, 0, 1, -2\right] $   & $1$\\    \hline     
  $ w_{4}w_{2}w_{3}w_{1}w_{2}w_{3}w_{4}w_{1}w_{2}w_{3} $  &  
 $H$  &   61   &   $ \left[-1, 0, 1, -2\right] $  & $1$\\    \hline     
  $ \left(w_{4}w_{3}w_{1}w_{2}w_{3}w_{4}\right) \cdot \left(w_{3}w_{2}w_{3}\right) 
   \cdot \left(w_{1}w_{2}w_{3}\right) $  &  
 $HZH$  &   61   &   $ \left[-1, 0, 1, -2\right] $  & $-1$\\    \hline     
  $ \left(w_{4}w_{2}w_{3}w_{1}w_{2}w_{3}w_{4}\right) \cdot \left(w_{3}w_{2}w_{3}\right) 
   \cdot \left(w_{1}w_{2}w_{3}\right) $  &  
 $HZH$  &   61   &   $ \left[-1, 0, 1, -2\right] $  & $-1$\\    \hline     
  $ w_{3}w_{1}w_{2}w_{3}w_{4}w_{1}w_{2}w_{3} $  &  
 $H$  &   61   &   $ \left[-1, 0, -1, 2\right] $  & $1$\\    \hline     
  $ w_{2}w_{3}w_{1}w_{2}w_{3}w_{4}w_{1}w_{2}w_{3} $  &  
 $H$  &   61   &   $ \left[-1, 0, -1, 2\right] $   & $1$\\    \hline     
  $ \left(w_{3}w_{1}w_{2}w_{3}w_{4}\right) \cdot \left(w_{3}w_{2}w_{3}\right) 
   \cdot \left(w_{1}w_{2}w_{3}\right) $  &  
 $HZH$  &   61   &   $ \left[-1, 0, -1, 2\right] $  & $-1$\\    \hline     
  $ \left(w_{2}w_{3}w_{1}w_{2}w_{3}w_{4}\right) \cdot \left(w_{3}w_{2}w_{3}\right) 
   \cdot \left(w_{1}w_{2}w_{3}\right) $  &  
 $HZH$  &   61   &   $ \left[-1, 0, -1, 2\right] $  & $-1$\\    \hline     
  $ w_{1}w_{2}w_{3}w_{4}w_{3}w_{1}w_{2}w_{3} $  &  
 $H$  &   61   &   $ \left[-1, -1, 2, -1\right] $  &  $1$\\    \hline     
  $ w_{1}w_{2}w_{3}w_{4}w_{2}w_{3}w_{1}w_{2}w_{3} $  &  
 $H$  &   61   &   $ \left[-1, -1, 2, -1\right] $  & $1$ \\    \hline     
  $ w_{2}w_{3}w_{1}w_{2}w_{3}w_{4}w_{3}w_{2}w_{3} $  &  
 $H$  &   61   &   $ \left[-1, -1, 2, -1\right] $  & $1$ \\    \hline     
  $ \left(w_{4}\right) \cdot\left(w_{3}w_{2}w_{3}\right) 
   \cdot \left(w_{1}w_{2}w_{3}w_{4}\right) 
   \cdot \left(w_{3}w_{2}w_{3}\right) 
   \cdot \left(w_{1}w_{2}w_{3}\right) $  &  
 $HZHZH$  &   61   &   $ \left[-1, -1, 2, -1\right] $  & $-3$\\    \hline     
  $ \left(w_{3}\right) 
   \cdot \left(w_{2}w_{3}
 w_{1}w_{2}\right) 
   \cdot \left(w_{3}w_{4}\right) \cdot\left(w_{3}w_{2}w_{3}\right) 
  \cdot \left(w_{1}w_{2}w_{3}w_{4}\right) 
  \cdot \left(w_{3}w_{2}w_{3}\right) 
  \cdot \left(w_{1}w_{2}w_{3}\right) $  &  
 $HZHZHZH$  &   61   &   $ \left[-1, -1, 2, -1\right] $  & $-3$\\    \hline     
  $ w_{4}w_{3}w_{2}w_{3}w_{4}w_{1}w_{2}w_{3} $  &  
 $H$  &   61   &   $ \left[1, -1, 1, -2\right] $  &  $1$\\    \hline     
  $ \left(w_{4}w_{3}w_{2}w_{3}w_{4}\right) \cdot \left(w_{3}w_{2}w_{3}\right) 
  \cdot \left(w_{1}w_{2}w_{3}\right) $  &  
 $HZH$  &   61   &   $ \left[1, -1, 1, -2\right] $  & $-1$\\   \hline     
  $ w_{3}w_{2}w_{3}w_{4}w_{1}w_{2}w_{3} $  &  
 $H$  &   61   &   $ \left[1, -1, -1, 2\right] $  & $1$ \\    \hline     
  $\left( w_{3}w_{2}w_{3}w_{4}\right) \cdot \left(w_{3}w_{2}w_{3}\right) 
  \cdot \left(w_{1}w_{2}w_{3}\right) $  &  
 $HZH$  &   61   &   $ \left[1, -1, -1, 2\right] $  & $-1$\\    \hline     
  $ w_{1}w_{2}w_{3}w_{4}w_{1}w_{2}w_{3} $  &  
 $H$  &   61   &   $ \left[-1, -1, 1, 1\right] $  &  -- \\    \hline     
  $ \left(w_{1}w_{2}w_{3}w_{4}\right) \cdot \left(w_{3}w_{2}w_{3}\right) 
  \cdot \left(w_{1}w_{2}w_{3}\right) $  &  
 $HZH$  &   61   &   $ \left[-1, -1, 1, 1\right] $  &  -- \\    \hline     
  $ \left( w_{3}w_{2}w_{3}\right) 
  \cdot \left(w_{1}w_{2}w_{3}w_{4}\right) 
  \cdot \left(w_{3}w_{2}w_{3}\right) 
  \cdot \left(w_{1}w_{2}w_{3}\right)  $  &  
 $ZHZH$  &   61   &   $ \left[-1, -1, 1, 1\right] $  &  -- \\    \hline     
  $ w_{2}w_{3}w_{4}w_{3}w_{2}w_{3} $  &  
 $H$  &   86   &   $ \left[2, -1, 0, -1\right] $  &  --\\    \hline     
  $ w_{3}w_{2}w_{3}w_{4}w_{3}w_{2}w_{3} $  &  
 $H$  &   86   &   $ \left[2, -1, 0, -1\right] $  &  -- \\    \hline     
  $ w_{3}w_{2}w_{3}w_{4}w_{3}w_{1}w_{2}w_{3} $  &  
 $H$  &   86   &   $ \left[1, 0, -2, 1\right] $  &  -- \\    \hline     
  $ w_{3}w_{2}w_{3}w_{4}w_{2}w_{3}w_{1}w_{2}w_{3} $  &  
 $H$  &   86   &   $ \left[1, 0, -2, 1\right] $  &  -- \\    \hline     
  $ w_{2}w_{3}w_{4}w_{3}w_{1}w_{2}w_{3} $  &  
 $H$  &   86   &   $ \left[1, -2, 2, -1\right] $  &  -- \\    \hline     
  $ w_{2}w_{3}w_{4}w_{2}w_{3}w_{1}w_{2}w_{3} $  &  
 $H$  &   86   &   $ \left[1, -2, 2, -1\right] $  &  -- \\    \hline     
  $ w_{3}w_{4}w_{3}w_{1}w_{2}w_{3} $  &  
 $H$  &   86   &   $ \left[-1, 2, -2, -1\right] $  &  --  \\    \hline     
  $ w_{3}w_{4}w_{2}w_{3}w_{1}w_{2}w_{3} $  &  
 $H$  &   86   &   $ \left[-1, 2, -2, -1\right] $  &  -- \\    \hline     
  $ w_{4}w_{3}w_{1}w_{2}w_{3} $  &  
 $H$  &   86   &   $ \left[-1, 0, 2, -3\right] $  &  -- \\    \hline     
  $ w_{4}w_{2}w_{3}w_{1}w_{2}w_{3} $  &  
 $H$  &   86   &   $ \left[-1, 0, 2, -3\right] $  &  -- \\    \hline     
  $ w_{3}w_{1}w_{2}w_{3} $  &  
 $H$  &   86   &   $ \left[-1, 0, -1, 3\right] $  &  --\\    \hline     
  $ w_{2}w_{3}w_{1}w_{2}w_{3} $  &  
 $H$  &   86   &   $ \left[-1, 0, -1, 3\right] $  &  -- \\    \hline     
  $ w_{1}w_{2}w_{3}w_{4}w_{3}w_{2}w_{3} $  &  
 $H$  &   86   &   $ \left[-2, 1, 0, -1\right] $  &  -- \\    \hline     
  $ w_{3}w_{1}w_{2}w_{3}w_{4}w_{3}w_{2}w_{3} $  &  
 $H$  &   86   &   $ \left[-2, 1, 0, -1\right] $  &  -- \\    \hline     
  $ w_{2}w_{3}w_{4}w_{2}w_{3} $  &  
 $H$  &   86   &   $ \left[2, -1, -1, 1\right] $  &  -- \\    \hline     
  $ w_{3}w_{4}w_{3}w_{2}w_{3} $  &  
 $H$  &   86   &   $ \left[1, 1, -2, -1\right] $  &  -- \\    \hline     
  $ w_{4}w_{3}w_{2}w_{3} $  &  
 $H$  &   86   &   $ \left[1, -1, 2, -3\right] $  &  -- \\    \hline     
  $ w_{3}w_{2}w_{3} $  &  
 $H$  &   86   &   $ \left[1, -1, -1, 3\right] $  &  -- \\    \hline     
  $ w_{1}w_{2}w_{3}w_{4}w_{2}w_{3} $  &  
 $H$  &   86   &   $ \left[-2, 1, -1, 1\right] $  &  -- \\    \hline     
  $ w_{2}w_{3}w_{4}w_{1}w_{2}w_{3} $  &  
 $H$  &   90   &   $ \left[1, -2, 1, 1\right] $  &  -- \\    \hline     
  $ \left(w_{2}w_{3}w_{4}\right) \cdot \left(w_{3}w_{2}w_{3}\right) 
  \cdot \left(w_{1}w_{2}w_{3}\right) $  &  
 $HZH$  &   90   &   $ \left[1, -2, 1, 1\right] $  &  -- \\    \hline     
  $ w_{3}w_{4}w_{1}w_{2}w_{3} $  &  
 $H$  &   90   &   $ \left[-1, 2, -3, 1\right] $  &  -- \\    \hline     
  $\left( w_{3}w_{4}\right) \cdot \left(w_{3}w_{2}w_{3}\right) 
  \cdot \left(w_{1}w_{2}w_{3}\right) $  &  
 $HZH$  &   90   &   $ \left[-1, 2, -3, 1\right] $  &  -- \\    \hline     
  $ w_{4}w_{1}w_{2}w_{3} $  &  
 $H$  &   90   &   $ \left[-1, -1, 3, -2\right] $  &  --\\   \hline     
  $ \left(w_{4}\right) \cdot \left(w_{3}w_{2}w_{3}\right) 
  \cdot \left(w_{1}w_{2}w_{3}\right) $  &  
 $HZH$  &   90   &   $ \left[-1, -1, 3, -2\right] $  &  --\\   \hline     
  $ w_{1}w_{2}w_{3} $  &  
 $H$  &   90   &   $ \left[-1, -1, 1, 2\right] $  &  -- \\    \hline     
  $ \left(w_{3}w_{2}w_{3}\right) 
 \cdot \left( w_{1}w_{2}w_{3}\right) $  &  
 $ZH$  &   90   &   $ \left[-1, -1, 1, 2\right] $  &  --\\    \hline     
  $ w_{3}w_{4}w_{2}w_{3} $  &  
 $H$  &   96   &   $ \left[1, 1, -3, 1\right] $  &  -- \\    \hline     
  $ w_{4}w_{2}w_{3} $  &  
 $H$  &   96   &   $ \left[1, -2, 3, -2\right] $  &  -- \\    \hline     
  $ w_{2}w_{3} $  &  
 $H$  &   96   &   $ \left[1, -2, 1, 2\right] $  &  -- \\    \hline     
  $ w_{4}w_{3} $  &  
 $H$  &   96   &   $ \left[-1, 2, -1, -2\right] $  &  --\\   \hline     
  $ w_{3} $  &  
 $H$  &   96   &   $ \left[-1, 2, -3, 2\right] $  &  -- \\   \hline     
  $ 1 $  &  
 $H$  &   96   &   $ \left[-1, -1, 3, -1\right] $  &  --  \\    \hline
  \caption{Images of intertwining operators of $P_3$}
  \label{Table:: P3 ::images} 
  \end{longtable}
  \end{landscape}


\chapter{Global Information}\label{app: globaltable}
\color{black}
In this appendix we summarize  the global information which is  relevant to this thesis.

Recall that for $u,w \in W(\para{P}_i,G)$ we say that $ u \sim_{z_0} w $ if $u\chi_{\para{P}_i,z_0} = w\chi_{\para{P}_i,z_0}$. We organize the tables according to this equivalence relation. The tables contain the following information:
\begin{itemize}
\item
 The column $w$ stands for words in $W(\para{P}_i,G)$. 
\item
The column  Order denotes $\ord_{z=z_0} C_{w}(z)$ where $C_{w}(z)$ is the Gindikin-Karpelevich factor ( see \ref{global:gid}).  
\item
The column Exp denotes the coordinate vector of  $w \chi_{\para{P}_i,z_0}$ with respect to fundamental weights.
\item
The column Label contains the labels we attach to each equivalence class in \Cref{def:: classes} and \Cref{global::p3::thm}.
\end{itemize}  

    \color{black}
    \newpage
 \begin{longtable}{|Hr|c|r|c|c|}
 \hline 
Point  
 &  $w$  &  Order  &  Factor  &   Exp  &  Label \\\hline 
 \hline \hline
 \endhead
  \rowcolor{LightCyan} 
$ \frac{5}{2} $  
 &  $w_{3}w_{1}w_{2}w_{3}w_{4}w_{3}w_{2}w_{3}w_{1}w_{2}w_{3}w_{4}$  &  $ 2$  &  $\frac{ \zeta ( 2 \, z ) \zeta ( z - \frac{3}{2} ) ^{ 2 } } { \zeta ( 2 \, z + 1 ) \zeta ( z + \frac{5}{2} ) \zeta ( z + \frac{11}{2} ) }$  &   $\left[-1, 0, -1, -1\right]$  &  \ref{image:sph}\\  \rowcolor{LightCyan} 
$ \frac{5}{2} $  
 &  $w_{2}w_{3}w_{1}w_{2}w_{3}w_{4}w_{3}w_{2}w_{3}w_{1}w_{2}w_{3}w_{4}$  &  $ 2$  &  $\frac{ \zeta ( 2 \, z ) \zeta ( z - \frac{5}{2} ) \zeta ( z - \frac{3}{2} ) } { \zeta ( 2 \, z + 1 ) \zeta ( z + \frac{5}{2} ) \zeta ( z + \frac{11}{2} ) }$  &   $\left[-1, 0, -1, -1\right]$  &  \\ \hline 
 \hline 
  \rowcolor{ashgrey} 
$ \frac{5}{2} $  
 &  $w_{3}w_{2}w_{3}w_{4}w_{3}w_{2}w_{3}w_{1}w_{2}w_{3}w_{4}$  &  $ 1$  &  $\frac{ \zeta ( 2 \, z ) \zeta ( z - \frac{1}{2} ) \zeta ( z - \frac{3}{2} ) } { \zeta ( 2 \, z + 1 ) \zeta ( z + \frac{5}{2} ) \zeta ( z + \frac{11}{2} ) }$  &   $\left[1, -1, -1, -1\right]$  &  \ref{image:sph} \\ \hline 
 \hline 
  \rowcolor{ashgrey} 
$ \frac{5}{2} $  
 &  $w_{1}w_{2}w_{3}w_{4}w_{2}w_{3}w_{1}w_{2}w_{3}w_{4}$  &  $ 1$  &  $\frac{ \zeta ( 2 \, z ) \zeta ( z - \frac{3}{2} ) \zeta ( z + \frac{1}{2} ) } { \zeta ( 2 \, z + 1 ) \zeta ( z + \frac{5}{2} ) \zeta ( z + \frac{11}{2} ) }$  &   $\left[-1, -1, -1, 2\right]$  &   \ref{image:even}\\  \rowcolor{ashgrey} 
$ \frac{5}{2} $  
 &  $w_{4}w_{3}w_{2}w_{3}w_{1}w_{2}w_{3}w_{4}w_{3}w_{2}w_{3}w_{1}w_{2}w_{3}w_{4}$  &  $ 1$  &  $\frac{ \zeta ( 2 \, z ) \zeta ( z - \frac{3}{2} ) \zeta ( z - \frac{9}{2} ) } { \zeta ( 2 \, z + 1 ) \zeta ( z + \frac{5}{2} ) \zeta ( z + \frac{11}{2} ) }$  &   $\left[-1, -1, -1, 2\right]$  &  \\ \hline 
 \hline 
  \rowcolor{ashgrey} 
$ \frac{5}{2} $  
 &  $w_{1}w_{2}w_{3}w_{4}w_{3}w_{2}w_{3}w_{1}w_{2}w_{3}w_{4}$  &  $ 1$  &  $\frac{ \zeta ( 2 \, z ) \zeta ( z - \frac{3}{2} ) \zeta ( z - \frac{1}{2} ) } { \zeta ( 2 \, z + 1 ) \zeta ( z + \frac{5}{2} ) \zeta ( z + \frac{11}{2} ) }$  &   $\left[-1, -1, 1, -2\right]$  &   \ref{image:even}\\  \rowcolor{ashgrey} 
$ \frac{5}{2} $  
 &  $w_{3}w_{2}w_{3}w_{1}w_{2}w_{3}w_{4}w_{3}w_{2}w_{3}w_{1}w_{2}w_{3}w_{4}$  &  $ 1$  &  $\frac{ \zeta ( 2 \, z ) \zeta ( z - \frac{3}{2} ) \zeta ( z - \frac{7}{2} ) } { \zeta ( 2 \, z + 1 ) \zeta ( z + \frac{5}{2} ) \zeta ( z + \frac{11}{2} ) }$  &   $\left[-1, -1, 1, -2\right]$  &  \\ \hline 
 \hline 
  \rowcolor{WHITE} 
$ \frac{5}{2} $  
 &  $1$  &  $ 0$  &  $1$  &   $\left[-1, -1, -1, 7\right]$  &  \ref{image:0}\\ \hline 
 \hline 
  \rowcolor{WHITE} 
$ \frac{5}{2} $  
 &  $w_{4}$  &  $ 0$  &  $\frac{ \zeta ( z + \frac{9}{2} ) } { \zeta ( z + \frac{11}{2} ) }$  &   $\left[-1, -1, 6, -7\right]$  &  \ref{image:0}\\ \hline 
 \hline 
  \rowcolor{WHITE} 
$ \frac{5}{2} $  
 &  $w_{3}w_{4}$  &  $ 0$  &  $\frac{ \zeta ( z + \frac{7}{2} ) } { \zeta ( z + \frac{11}{2} ) }$  &   $\left[-1, 5, -6, -1\right]$  &  \ref{image:0} \\ \hline 
 \hline 
  \rowcolor{WHITE} 
$ \frac{5}{2} $  
 &  $w_{2}w_{3}w_{4}$  &  $ 0$  &  $\frac{ \zeta ( z + \frac{5}{2} ) } { \zeta ( z + \frac{11}{2} ) }$  &   $\left[4, -5, 4, -1\right]$  &  \ref{image:0}\\ \hline 
 \hline 
  \rowcolor{WHITE} 
$ \frac{5}{2} $  
 &  $w_{1}w_{2}w_{3}w_{4}$  &  $ 0$  &  $\frac{ \zeta ( z + \frac{3}{2} ) } { \zeta ( z + \frac{11}{2} ) }$  &   $\left[-4, -1, 4, -1\right]$   &  \ref{image:0} \\ \hline 
 \hline 
  \rowcolor{WHITE} 
$ \frac{5}{2} $  
 &  $w_{3}w_{2}w_{3}w_{4}$  &  $ 0$  &  $\frac{ \zeta ( z + \frac{3}{2} ) } { \zeta ( z + \frac{11}{2} ) }$  &   $\left[4, -1, -4, 3\right]$  &   \ref{image:0}\\ \hline 
 \hline 
  \rowcolor{WHITE} 
$ \frac{5}{2} $  
 &  $w_{3}w_{1}w_{2}w_{3}w_{4}$  &  $ 0$  &  $\frac{ \zeta ( z + \frac{3}{2} ) ^{ 2 } } { \zeta ( z + \frac{5}{2} ) \zeta ( z + \frac{11}{2} ) }$  &   $\left[-4, 3, -4, 3\right]$  &  \ref{image:0} \\ \hline 
 \hline 
  \rowcolor{WHITE} 
$ \frac{5}{2} $  
 &  $w_{4}w_{3}w_{2}w_{3}w_{4}$  &  $ 0$  &  $\frac{ \zeta ( z + \frac{1}{2} ) } { \zeta ( z + \frac{11}{2} ) }$  &   $\left[4, -1, -1, -3\right]$  &  \ref{image:0} \\ \hline 
 \hline 
  \rowcolor{WHITE} 
$ \frac{5}{2} $  
 &  $w_{2}w_{3}w_{1}w_{2}w_{3}w_{4}$  &  $ 0$  &  $\frac{ \zeta ( z + \frac{3}{2} ) \zeta ( z + \frac{1}{2} ) } { \zeta ( z + \frac{5}{2} ) \zeta ( z + \frac{11}{2} ) }$  &   $\left[-1, -3, 2, 3\right]$  &   \ref{image:0}\\ \hline 
 \hline 
  \rowcolor{WHITE} 
$ \frac{5}{2} $  
 &  $w_{4}w_{3}w_{1}w_{2}w_{3}w_{4}$  &  $ 0$  &  $\frac{ \zeta ( z + \frac{3}{2} ) \zeta ( z + \frac{1}{2} ) } { \zeta ( z + \frac{5}{2} ) \zeta ( z + \frac{11}{2} ) }$  &   $\left[-4, 3, -1, -3\right]$  &  \ref{image:0} \\ \hline 
 \hline 
  \rowcolor{WHITE} 
$ \frac{5}{2} $  
 &  $w_{3}w_{2}w_{3}w_{1}w_{2}w_{3}w_{4}$  &  $ 0$  &  $\frac{ \zeta ( z + \frac{3}{2} ) \zeta ( z - \frac{1}{2} ) } { \zeta ( z + \frac{5}{2} ) \zeta ( z + \frac{11}{2} ) }$  &   $\left[-1, -1, -2, 5\right]$  &  \ref{image:0} \\ \hline 
 \hline 
  \rowcolor{WHITE} 
$ \frac{5}{2} $  
 &  $w_{4}w_{2}w_{3}w_{1}w_{2}w_{3}w_{4}$  &  $ 0$  &  $\frac{ \zeta ( z + \frac{1}{2} ) ^{ 2 } } { \zeta ( z + \frac{5}{2} ) \zeta ( z + \frac{11}{2} ) }$  &   $\left[-1, -3, 5, -3\right]$  &   \ref{image:0}\\ \hline 
 \hline 
  \rowcolor{WHITE} 
$ \frac{5}{2} $  
 &  $w_{4}w_{3}w_{2}w_{3}w_{1}w_{2}w_{3}w_{4}$  &  $ 0$  &  $\frac{ \zeta ( 2 \, z ) \zeta ( z + \frac{3}{2} ) \zeta ( z - \frac{1}{2} ) } { \zeta ( 2 \, z + 1 ) \zeta ( z + \frac{5}{2} ) \zeta ( z + \frac{11}{2} ) }$  &   $\left[-1, -1, 3, -5\right]$  &  \ref{image:0} \\ \hline 
 \hline 
  \rowcolor{WHITE} 
$ \frac{5}{2} $  
 &  $w_{3}w_{4}w_{2}w_{3}w_{1}w_{2}w_{3}w_{4}$  &  $ 0$  &  $\frac{ \zeta ( 2 \, z ) \zeta ( z + \frac{1}{2} ) ^{ 2 } } { \zeta ( 2 \, z + 1 ) \zeta ( z + \frac{5}{2} ) \zeta ( z + \frac{11}{2} ) }$  &   $\left[-1, 2, -5, 2\right]$   &  \ref{image:0} \\ \hline 
 \hline 
  \rowcolor{WHITE} 
$ \frac{5}{2} $  
 &  $w_{3}w_{4}w_{3}w_{2}w_{3}w_{1}w_{2}w_{3}w_{4}$  &  $ 0$  &  $\frac{ \zeta ( 2 \, z ) \zeta ( z + \frac{1}{2} ) \zeta ( z - \frac{1}{2} ) } { \zeta ( 2 \, z + 1 ) \zeta ( z + \frac{5}{2} ) \zeta ( z + \frac{11}{2} ) }$  &   $\left[-1, 2, -3, -2\right]$  &  \ref{image:0} \\ \hline 
 \hline 
  \rowcolor{WHITE} 
$ \frac{5}{2} $  
 &  $w_{2}w_{3}w_{4}w_{2}w_{3}w_{1}w_{2}w_{3}w_{4}$  &  $ 0$  &  $\frac{ \zeta ( 2 \, z ) \zeta ( z - \frac{1}{2} ) \zeta ( z + \frac{1}{2} ) } { \zeta ( 2 \, z + 1 ) \zeta ( z + \frac{5}{2} ) \zeta ( z + \frac{11}{2} ) }$  &   $\left[1, -2, -1, 2\right]$ &  \ref{image:0}  \\ \hline 
 \hline 
  \rowcolor{WHITE} 
$ \frac{5}{2} $  
 &  $w_{2}w_{3}w_{4}w_{3}w_{2}w_{3}w_{1}w_{2}w_{3}w_{4}$  &  $ 0$  &  $\frac{ \zeta ( 2 \, z ) \zeta ( z - \frac{1}{2} ) ^{ 2 } } { \zeta ( 2 \, z + 1 ) \zeta ( z + \frac{5}{2} ) \zeta ( z + \frac{11}{2} ) }$  &   $\left[1, -2, 1, -2\right]$  &  \ref{image:0}\\ \hline  

\caption{Global Information of  $P_4$}
\label{Table:: P4 ::Global}
\end{longtable}

\begin{small}

\begin{longtable}{|Hr|c|r|c|c|} 
\hline 
 Point  
 &  $w$  & Order  &  Factor  &   Exp   &  Label \\ \hline \hline
 \endhead 
  \rowcolor{LightCyan}
$ 1 $  
 &  $w_{3}w_{1}w_{2}w_{3}w_{4}w_{1}w_{2}w_{3}w_{2}w_{1}$  &  $ 3$  &  $\frac{ \zeta ( z ) ^{ 2 } \zeta ( 2 \, z - 1 ) } { \zeta ( 2 \, z + 4 ) \zeta ( z + 1 ) \zeta ( z + 4 ) }$  &   $\left[-1, 0, -1, 0\right]$  &   \ref{image:sph} \\  \rowcolor{LightCyan} 
$ 1 $  
 &  $w_{2}w_{3}w_{1}w_{2}w_{3}w_{4}w_{1}w_{2}w_{3}w_{2}w_{1}$  &  $ 3$  &  $\frac{ \zeta ( z ) \zeta ( z - 1 ) \zeta ( 2 \, z - 1 ) } { \zeta ( 2 \, z + 4 ) \zeta ( z + 1 ) \zeta ( z + 4 ) }$  &   $\left[-1, 0, -1, 0\right]$  &   \\  \rowcolor{LightCyan} 
$ 1 $  
 &  $w_{4}w_{3}w_{1}w_{2}w_{3}w_{4}w_{1}w_{2}w_{3}w_{2}w_{1}$  &  $ 3$  &  $\frac{ \zeta ( z ) ^{ 2 } \zeta ( 2 \, z - 2 ) } { \zeta ( 2 \, z + 4 ) \zeta ( z + 1 ) \zeta ( z + 4 ) }$  &   $\left[-1, 0, -1, 0\right]$  &  \\  \rowcolor{LightCyan} 
$ 1 $  
 &  $w_{4}w_{2}w_{3}w_{1}w_{2}w_{3}w_{4}w_{1}w_{2}w_{3}w_{2}w_{1}$  &  $ 3$  &  $\frac{ \zeta ( z ) \zeta ( z - 1 ) \zeta ( 2 \, z - 2 ) } { \zeta ( 2 \, z + 4 ) \zeta ( z + 1 ) \zeta ( z + 4 ) }$  &   $\left[-1, 0, -1, 0\right]$  &  \\ \hline 
 \hline 
  \rowcolor{LightCyan} 
$ 1 $  
 &  $w_{1}w_{2}w_{3}w_{4}w_{1}w_{2}w_{3}w_{2}w_{1}$  &  $ 2$  &  $\frac{ \zeta ( 2 \, z ) \zeta ( z ) ^{ 2 } } { \zeta ( 2 \, z + 4 ) \zeta ( z + 1 ) \zeta ( z + 4 ) }$  &   $\left[-1, -1, 1, -1\right]$  &  \ref{image:sph} \\  \rowcolor{LightCyan} 
$ 1 $  
 &  $w_{3}w_{4}w_{2}w_{3}w_{1}w_{2}w_{3}w_{4}w_{1}w_{2}w_{3}w_{2}w_{1}$  &  $ 2$  &  $\frac{ \zeta ( 2 \, z - 3 ) \zeta ( z ) \zeta ( z - 1 ) } { \zeta ( 2 \, z + 4 ) \zeta ( z + 1 ) \zeta ( z + 4 ) }$  &   $\left[-1, -1, 1, -1\right]$  &  \\ \hline 
 \hline 
  \rowcolor{LightCyan} 
$ 1 $  
 &  $w_{3}w_{2}w_{3}w_{4}w_{1}w_{2}w_{3}w_{2}w_{1}$  &  $ 2$  &  $\frac{ \zeta ( z ) \zeta ( 2 \, z - 1 ) } { \zeta ( 2 \, z + 4 ) \zeta ( z + 4 ) }$  &   $\left[1, -1, -1, 0\right]$  &  \ref{image:sph}\\  \rowcolor{LightCyan} 
$ 1 $  
 &  $w_{4}w_{3}w_{2}w_{3}w_{4}w_{1}w_{2}w_{3}w_{2}w_{1}$  &  $ 2$  &  $\frac{ \zeta ( z ) \zeta ( 2 \, z - 2 ) } { \zeta ( 2 \, z + 4 ) \zeta ( z + 4 ) }$  &   $\left[1, -1, -1, 0\right]$  &   \\ \hline 
 \hline 
  \rowcolor{ashgrey} 
$ 1 $  
 &  $w_{1}w_{2}w_{3}w_{2}w_{1}$  &  $ 1$  &  $\frac{ \zeta ( 2 \, z + 3 ) \zeta ( z ) } { \zeta ( 2 \, z + 4 ) \zeta ( z + 4 ) }$  &   $\left[-1, -1, -1, 4\right]$  &  \ref{image:sph}\\ \hline 
 \hline 
  \rowcolor{ashgrey} 
$ 1 $  
 &  $w_{4}w_{1}w_{2}w_{3}w_{2}w_{1}$  &  $ 1$  &  $\frac{ \zeta ( z ) \zeta ( 2 \, z + 2 ) } { \zeta ( 2 \, z + 4 ) \zeta ( z + 4 ) }$  &   $\left[-1, -1, 3, -4\right]$  &  \ref{image:sph}\\ \hline 
 \hline 
  \rowcolor{ashgrey} 
$ 1 $  
 &  $w_{3}w_{4}w_{1}w_{2}w_{3}w_{2}w_{1}$  &  $ 1$  &  $\frac{ \zeta ( z ) \zeta ( 2 \, z + 1 ) } { \zeta ( 2 \, z + 4 ) \zeta ( z + 4 ) }$  &   $\left[-1, 2, -3, -1\right]$  &  \ref{image:sph} \\ \hline 
 \hline 
  \rowcolor{ashgrey} 
$ 1 $  
 &  $w_{2}w_{3}w_{4}w_{1}w_{2}w_{3}w_{2}w_{1}$  &  $ 1$  &  $\frac{ \zeta ( 2 \, z ) \zeta ( z ) } { \zeta ( 2 \, z + 4 ) \zeta ( z + 4 ) }$  &   $\left[1, -2, 1, -1\right]$  &  \ref{image:sph}\\ \hline 
 \hline 
  \rowcolor{ashgrey} 
$ 1 $  
 &  $w_{2}w_{3}w_{4}w_{2}w_{3}w_{2}w_{1}$  &  $ 1$  &  $\frac{ \zeta ( z ) \zeta ( 2 \, z + 1 ) } { \zeta ( 2 \, z + 4 ) \zeta ( z + 4 ) }$  &   $\left[2, -1, -1, -1\right]$  &  \ref{image:even}  \\  \rowcolor{ashgrey} 
$ 1 $  
 &  $w_{1}w_{2}w_{3}w_{4}w_{2}w_{3}w_{1}w_{2}w_{3}w_{4}w_{1}w_{2}w_{3}w_{2}w_{1}$  &  $ 1$  &  $\frac{ \zeta ( 2 \, z - 3 ) \zeta ( z - 3 ) \zeta ( z ) } { \zeta ( 2 \, z + 4 ) \zeta ( z + 1 ) \zeta ( z + 4 ) }$  &   $\left[2, -1, -1, -1\right]$  &   \\ \hline 
 \hline 
  \rowcolor{ashgrey} 
$ 1 $  
 &  $w_{1}w_{2}w_{3}w_{4}w_{2}w_{3}w_{2}w_{1}$  &  $ 1$  &  $\frac{ \zeta ( 2 \, z ) \zeta ( z ) } { \zeta ( 2 \, z + 4 ) \zeta ( z + 4 ) }$  &   $\left[-2, 1, -1, -1\right]$  &  \ref{image:even} \\  \rowcolor{ashgrey} 
$ 1 $  
 &  $w_{2}w_{3}w_{4}w_{2}w_{3}w_{1}w_{2}w_{3}w_{4}w_{1}w_{2}w_{3}w_{2}w_{1}$  &  $ 1$  &  $\frac{ \zeta ( 2 \, z - 3 ) \zeta ( z - 2 ) \zeta ( z ) } { \zeta ( 2 \, z + 4 ) \zeta ( z + 1 ) \zeta ( z + 4 ) }$  &   $\left[-2, 1, -1, -1\right]$  & \\ \hline 
 \hline 
  \rowcolor{WHITE} 
$ 1 $  
 &  $1$  &  $ 0$  &  $\frac{ } {}$  &   $\left[4, -1, -1, -1\right]$   &  \ref{image:0}\\ \hline 
 \hline 
  \rowcolor{WHITE} 
$ 1 $  
 &  $w_{1}$  &  $ 0$  &  $\frac{ \zeta ( z + 3 ) } { \zeta ( z + 4 ) }$  &   $\left[-4, 3, -1, -1\right]$  &  \ref{image:0} \\ \hline 
 \hline 
  \rowcolor{WHITE} 
$ 1 $  
 &  $w_{2}w_{1}$  &  $ 0$  &  $\frac{ \zeta ( z + 2 ) } { \zeta ( z + 4 ) }$  &   $\left[-1, -3, 5, -1\right]$   &  \ref{image:0} \\ \hline 
 \hline 
  \rowcolor{WHITE} 
$ 1 $  
 &  $w_{3}w_{2}w_{1}$  &  $ 0$  &  $\frac{ \zeta ( 2 \, z + 3 ) \zeta ( z + 2 ) } { \zeta ( 2 \, z + 4 ) \zeta ( z + 4 ) }$  &   $\left[-1, 2, -5, 4\right]$  &  \ref{image:0} \\ \hline 
 \hline 
  \rowcolor{WHITE} 
$ 1 $  
 &  $w_{2}w_{3}w_{2}w_{1}$  &  $ 0$  &  $\frac{ \zeta ( 2 \, z + 3 ) \zeta ( z + 1 ) } { \zeta ( 2 \, z + 4 ) \zeta ( z + 4 ) }$  &   $\left[1, -2, -1, 4\right]$  &  \ref{image:0} \\ \hline 
 \hline 
  \rowcolor{WHITE} 
$ 1 $  
 &  $w_{4}w_{3}w_{2}w_{1}$  &  $ 0$  &  $\frac{ \zeta ( z + 2 ) \zeta ( 2 \, z + 2 ) } { \zeta ( 2 \, z + 4 ) \zeta ( z + 4 ) }$  &   $\left[-1, 2, -1, -4\right]$  &  \ref{image:0}\\ \hline 
 \hline 
  \rowcolor{WHITE} 
$ 1 $  
 &  $w_{4}w_{2}w_{3}w_{2}w_{1}$  &  $ 0$  &  $\frac{ \zeta ( z + 1 ) \zeta ( 2 \, z + 2 ) } { \zeta ( 2 \, z + 4 ) \zeta ( z + 4 ) }$  &   $\left[1, -2, 3, -4\right]$  &  \ref{image:0}  \\ \hline 
 \hline 
  \rowcolor{WHITE} 
$ 1 $  
 &  $w_{3}w_{4}w_{2}w_{3}w_{2}w_{1}$  &  $ 0$  &  $\frac{ \zeta ( z + 1 ) \zeta ( 2 \, z + 1 ) } { \zeta ( 2 \, z + 4 ) \zeta ( z + 4 ) }$  &   $\left[1, 1, -3, -1\right]$  &  \ref{image:0}  \\ \hline  
\caption{Global Information of  $P_1$}
\label{Table:: P1 ::Global}
 \end{longtable}
 \end{small}

\begin{landscape}
\renewcommand*{\arraystretch}{1} 
\fontsize{10}{12}
 \begin{longtable}{|Hr|c|r|c|c|} 
\hline 
Point  
 &  $w$  &  Order  &  Factor  &   Exp  &  Label \\ \hline \hline \hline
 \endhead 
  \rowcolor{LightCyan} 
$ \frac{1}{2} $  
 &  $w_{3}w_{4}w_{2}w_{3}w_{1}w_{2}w_{3}w_{4}w_{3}w_{1}w_{2}w_{3}$  &  $ 6$  &  $\frac{ \zeta ( 3 \, z - \frac{1}{2} ) \zeta ( z + \frac{1}{2} ) ^{ 3 } \zeta ( 2 \, z ) ^{ 2 } } { \zeta ( z + \frac{3}{2} ) \zeta ( z + \frac{5}{2} ) \zeta ( z + \frac{7}{2} ) \zeta ( 2 \, z + 1 ) \zeta ( 3 \, z + \frac{3}{2} ) \zeta ( 2 \, z + 3 ) }$  &   $\left[0, 0, -1, 0\right]$  &  \ref{image:sph} \\  \rowcolor{LightCyan} 
$ \frac{1}{2} $  
 &  $w_{3}w_{4}w_{2}w_{3}w_{1}w_{2}w_{3}w_{4}w_{2}w_{3}w_{1}w_{2}w_{3}$  &  $ 6$  &  $\frac{ \zeta ( 3 \, z - \frac{1}{2} ) \zeta ( z + \frac{1}{2} ) ^{ 2 } \zeta ( 2 \, z ) ^{ 2 } \zeta ( z - \frac{1}{2} ) } { \zeta ( z + \frac{3}{2} ) \zeta ( z + \frac{5}{2} ) \zeta ( z + \frac{7}{2} ) \zeta ( 2 \, z + 1 ) \zeta ( 3 \, z + \frac{3}{2} ) \zeta ( 2 \, z + 3 ) }$  &   $\left[0, 0, -1, 0\right]$  &  \\  \rowcolor{LightCyan} 
$ \frac{1}{2} $  
 &  $w_{3}w_{4}w_{3}w_{2}w_{3}w_{1}w_{2}w_{3}w_{4}w_{3}w_{1}w_{2}w_{3}$  &  $ 6$  &  $\frac{ \zeta ( 3 \, z - \frac{1}{2} ) \zeta ( z + \frac{1}{2} ) ^{ 2 } \zeta ( 2 \, z ) ^{ 2 } \zeta ( z - \frac{1}{2} ) } { \zeta ( z + \frac{3}{2} ) \zeta ( z + \frac{5}{2} ) \zeta ( z + \frac{7}{2} ) \zeta ( 2 \, z + 1 ) \zeta ( 3 \, z + \frac{3}{2} ) \zeta ( 2 \, z + 3 ) }$  &   $\left[0, 0, -1, 0\right]$  &  \\  \rowcolor{LightCyan} 
$ \frac{1}{2} $  
 &  $w_{2}w_{3}w_{4}w_{2}w_{3}w_{1}w_{2}w_{3}w_{4}w_{3}w_{1}w_{2}w_{3}$  &  $ 6$  &  $\frac{ \zeta ( 3 \, z - \frac{1}{2} ) \zeta ( z + \frac{1}{2} ) ^{ 2 } \zeta ( 2 \, z ) ^{ 2 } \zeta ( z - \frac{1}{2} ) } { \zeta ( z + \frac{3}{2} ) \zeta ( z + \frac{5}{2} ) \zeta ( z + \frac{7}{2} ) \zeta ( 2 \, z + 1 ) \zeta ( 3 \, z + \frac{3}{2} ) \zeta ( 2 \, z + 3 ) }$  &   $\left[0, 0, -1, 0\right]$  &  \\  \rowcolor{LightCyan} 
$ \frac{1}{2} $  
 &  $w_{3}w_{4}w_{3}w_{2}w_{3}w_{1}w_{2}w_{3}w_{4}w_{2}w_{3}w_{1}w_{2}w_{3}$  &  $ 6$  &  $\frac{ \zeta ( 3 \, z - \frac{1}{2} ) \zeta ( z + \frac{1}{2} ) \zeta ( 2 \, z ) ^{ 2 } \zeta ( z - \frac{1}{2} ) ^{ 2 } } { \zeta ( z + \frac{3}{2} ) \zeta ( z + \frac{5}{2} ) \zeta ( z + \frac{7}{2} ) \zeta ( 2 \, z + 1 ) \zeta ( 3 \, z + \frac{3}{2} ) \zeta ( 2 \, z + 3 ) }$  &   $\left[0, 0, -1, 0\right]$  &  \\  \rowcolor{LightCyan} 
$ \frac{1}{2} $  
 &  $w_{2}w_{3}w_{4}w_{2}w_{3}w_{1}w_{2}w_{3}w_{4}w_{2}w_{3}w_{1}w_{2}w_{3}$  &  $ 6$  &  $\frac{ \zeta ( 3 \, z - \frac{1}{2} ) \zeta ( z + \frac{1}{2} ) ^{ 2 } \zeta ( 2 \, z - 1 ) \zeta ( 2 \, z ) \zeta ( z - \frac{1}{2} ) } { \zeta ( z + \frac{3}{2} ) \zeta ( z + \frac{5}{2} ) \zeta ( z + \frac{7}{2} ) \zeta ( 2 \, z + 1 ) \zeta ( 3 \, z + \frac{3}{2} ) \zeta ( 2 \, z + 3 ) }$  &   $\left[0, 0, -1, 0\right]$  &  \\  \rowcolor{LightCyan} 
$ \frac{1}{2} $  
 &  $w_{2}w_{3}w_{4}w_{3}w_{2}w_{3}w_{1}w_{2}w_{3}w_{4}w_{3}w_{1}w_{2}w_{3}$  &  $ 6$  &  $\frac{ \zeta ( 3 \, z - \frac{1}{2} ) \zeta ( z + \frac{1}{2} ) \zeta ( 2 \, z ) ^{ 2 } \zeta ( z - \frac{1}{2} ) ^{ 2 } } { \zeta ( z + \frac{3}{2} ) \zeta ( z + \frac{5}{2} ) \zeta ( z + \frac{7}{2} ) \zeta ( 2 \, z + 1 ) \zeta ( 3 \, z + \frac{3}{2} ) \zeta ( 2 \, z + 3 ) }$  &   $\left[0, 0, -1, 0\right]$  &  \\  \rowcolor{LightCyan} 
$ \frac{1}{2} $  
 &  $w_{1}w_{2}w_{3}w_{4}w_{2}w_{3}w_{1}w_{2}w_{3}w_{4}w_{3}w_{1}w_{2}w_{3}$  &  $ 6$  &  $\frac{ \zeta ( 3 \, z - \frac{1}{2} ) \zeta ( z + \frac{1}{2} ) ^{ 2 } \zeta ( 2 \, z - 1 ) \zeta ( 2 \, z ) \zeta ( z - \frac{1}{2} ) } { \zeta ( z + \frac{3}{2} ) \zeta ( z + \frac{5}{2} ) \zeta ( z + \frac{7}{2} ) \zeta ( 2 \, z + 1 ) \zeta ( 3 \, z + \frac{3}{2} ) \zeta ( 2 \, z + 3 ) }$  &   $\left[0, 0, -1, 0\right]$  &  \\  \rowcolor{LightCyan} 
$ \frac{1}{2} $  
 &  $w_{2}w_{3}w_{4}w_{3}w_{2}w_{3}w_{1}w_{2}w_{3}w_{4}w_{2}w_{3}w_{1}w_{2}w_{3}$  &  $ 6$  &  $\frac{ \zeta ( 3 \, z - \frac{1}{2} ) \zeta ( z + \frac{1}{2} ) \zeta ( 2 \, z - 1 ) \zeta ( 2 \, z ) \zeta ( z - \frac{1}{2} ) ^{ 2 } } { \zeta ( z + \frac{3}{2} ) \zeta ( z + \frac{5}{2} ) \zeta ( z + \frac{7}{2} ) \zeta ( 2 \, z + 1 ) \zeta ( 3 \, z + \frac{3}{2} ) \zeta ( 2 \, z + 3 ) }$  &   $\left[0, 0, -1, 0\right]$  &  \\  \rowcolor{LightCyan} 
$ \frac{1}{2} $  
 &  $w_{1}w_{2}w_{3}w_{4}w_{2}w_{3}w_{1}w_{2}w_{3}w_{4}w_{2}w_{3}w_{1}w_{2}w_{3}$  &  $ 6$  &  $\frac{ \zeta ( 3 \, z - \frac{1}{2} ) \zeta ( z + \frac{1}{2} ) \zeta ( 2 \, z - 1 ) \zeta ( 2 \, z ) \zeta ( z - \frac{1}{2} ) ^{ 2 } } { \zeta ( z + \frac{3}{2} ) \zeta ( z + \frac{5}{2} ) \zeta ( z + \frac{7}{2} ) \zeta ( 2 \, z + 1 ) \zeta ( 3 \, z + \frac{3}{2} ) \zeta ( 2 \, z + 3 ) }$  &   $\left[0, 0, -1, 0\right]$  &  \\  \rowcolor{LightCyan} 
$ \frac{1}{2} $  
 &  $w_{1}w_{2}w_{3}w_{4}w_{3}w_{2}w_{3}w_{1}w_{2}w_{3}w_{4}w_{3}w_{1}w_{2}w_{3}$  &  $ 6$  &  $\frac{ \zeta ( 3 \, z - \frac{1}{2} ) \zeta ( z + \frac{1}{2} ) \zeta ( 2 \, z - 1 ) \zeta ( 2 \, z ) \zeta ( z - \frac{1}{2} ) ^{ 2 } } { \zeta ( z + \frac{3}{2} ) \zeta ( z + \frac{5}{2} ) \zeta ( z + \frac{7}{2} ) \zeta ( 2 \, z + 1 ) \zeta ( 3 \, z + \frac{3}{2} ) \zeta ( 2 \, z + 3 ) }$  &   $\left[0, 0, -1, 0\right]$  &  \\  \rowcolor{LightCyan} 
$ \frac{1}{2} $  
 &  $w_{1}w_{2}w_{3}w_{4}w_{3}w_{2}w_{3}w_{1}w_{2}w_{3}w_{4}w_{2}w_{3}w_{1}w_{2}w_{3}$  &  $ 6$  &  $\frac{ \zeta ( 3 \, z - \frac{1}{2} ) \zeta ( 2 \, z - 1 ) \zeta ( 2 \, z ) \zeta ( z - \frac{1}{2} ) ^{ 3 } } { \zeta ( z + \frac{3}{2} ) \zeta ( z + \frac{5}{2} ) \zeta ( z + \frac{7}{2} ) \zeta ( 2 \, z + 1 ) \zeta ( 3 \, z + \frac{3}{2} ) \zeta ( 2 \, z + 3 ) }$  &   $\left[0, 0, -1, 0\right]$  &  \\ \hline 
 \hline 
  \rowcolor{LightCyan} 
$ \frac{1}{2} $  
 &  $w_{4}w_{2}w_{3}w_{1}w_{2}w_{3}w_{4}w_{3}w_{1}w_{2}w_{3}$  &  $ 5$  &  $\frac{ \zeta ( z + \frac{1}{2} ) ^{ 3 } \zeta ( 2 \, z ) ^{ 2 } \zeta ( 3 \, z + \frac{1}{2} ) } { \zeta ( z + \frac{3}{2} ) \zeta ( z + \frac{5}{2} ) \zeta ( z + \frac{7}{2} ) \zeta ( 2 \, z + 1 ) \zeta ( 3 \, z + \frac{3}{2} ) \zeta ( 2 \, z + 3 ) }$  &   $\left[0, -1, 1, -1\right]$  &  \ref{image:sph} \\  \rowcolor{LightCyan} 
$ \frac{1}{2} $  
 &  $w_{4}w_{2}w_{3}w_{1}w_{2}w_{3}w_{4}w_{2}w_{3}w_{1}w_{2}w_{3}$  &  $ 5$  &  $\frac{ \zeta ( z + \frac{1}{2} ) ^{ 2 } \zeta ( 2 \, z ) ^{ 2 } \zeta ( z - \frac{1}{2} ) \zeta ( 3 \, z + \frac{1}{2} ) } { \zeta ( z + \frac{3}{2} ) \zeta ( z + \frac{5}{2} ) \zeta ( z + \frac{7}{2} ) \zeta ( 2 \, z + 1 ) \zeta ( 3 \, z + \frac{3}{2} ) \zeta ( 2 \, z + 3 ) }$  &   $\left[0, -1, 1, -1\right]$ &  \\  \rowcolor{LightCyan} 
$ \frac{1}{2} $  
 &  $w_{4}w_{3}w_{2}w_{3}w_{1}w_{2}w_{3}w_{4}w_{3}w_{1}w_{2}w_{3}$  &  $ 5$  &  $\frac{ \zeta ( 3 \, z - \frac{1}{2} ) \zeta ( z + \frac{1}{2} ) ^{ 2 } \zeta ( 2 \, z ) \zeta ( z - \frac{1}{2} ) } { \zeta ( z + \frac{3}{2} ) \zeta ( z + \frac{5}{2} ) \zeta ( z + \frac{7}{2} ) \zeta ( 3 \, z + \frac{3}{2} ) \zeta ( 2 \, z + 3 ) }$  &   $\left[0, -1, 1, -1\right]$  &  \\  \rowcolor{LightCyan} 
$ \frac{1}{2} $  
 &  $w_{2}w_{3}w_{4}w_{2}w_{3}w_{1}w_{2}w_{3}w_{4}w_{1}w_{2}w_{3}$  &  $ 5$  &  $\frac{ \zeta ( 3 \, z - \frac{1}{2} ) \zeta ( 2 \, z ) ^{ 2 } \zeta ( z + \frac{1}{2} ) \zeta ( z - \frac{1}{2} ) } { \zeta ( z + \frac{5}{2} ) \zeta ( z + \frac{7}{2} ) \zeta ( 2 \, z + 1 ) \zeta ( 3 \, z + \frac{3}{2} ) \zeta ( 2 \, z + 3 ) }$  &   $\left[0, -1, 1, -1\right]$  &  \\  \rowcolor{LightCyan} 
$ \frac{1}{2} $  
 &  $w_{4}w_{3}w_{2}w_{3}w_{1}w_{2}w_{3}w_{4}w_{2}w_{3}w_{1}w_{2}w_{3}$  &  $ 5$  &  $\frac{ \zeta ( 3 \, z - \frac{1}{2} ) \zeta ( z + \frac{1}{2} ) \zeta ( 2 \, z ) \zeta ( z - \frac{1}{2} ) ^{ 2 } } { \zeta ( z + \frac{3}{2} ) \zeta ( z + \frac{5}{2} ) \zeta ( z + \frac{7}{2} ) \zeta ( 3 \, z + \frac{3}{2} ) \zeta ( 2 \, z + 3 ) }$  &   $\left[0, -1, 1, -1\right]$  &  \\  \rowcolor{LightCyan} 
$ \frac{1}{2} $  
 &  $w_{1}w_{2}w_{3}w_{4}w_{2}w_{3}w_{1}w_{2}w_{3}w_{4}w_{1}w_{2}w_{3}$  &  $ 5$  &  $\frac{ \zeta ( 3 \, z - \frac{1}{2} ) \zeta ( 2 \, z - 1 ) \zeta ( 2 \, z ) \zeta ( z + \frac{1}{2} ) \zeta ( z - \frac{1}{2} ) } { \zeta ( z + \frac{5}{2} ) \zeta ( z + \frac{7}{2} ) \zeta ( 2 \, z + 1 ) \zeta ( 3 \, z + \frac{3}{2} ) \zeta ( 2 \, z + 3 ) }$  &   $\left[0, -1, 1, -1\right]$  &   \\  \rowcolor{LightCyan} 
$ \frac{1}{2} $  
 &  $w_{2}w_{3}w_{4}w_{2}w_{3}w_{1}w_{2}w_{3}w_{4}w_{3}w_{2}w_{3}w_{1}w_{2}w_{3}$  &  $ 5$  &  $\frac{ \zeta ( 3 \, z - \frac{1}{2} ) \zeta ( z - \frac{3}{2} ) \zeta ( z + \frac{1}{2} ) ^{ 2 } \zeta ( 2 \, z - 1 ) \zeta ( 2 \, z ) } { \zeta ( z + \frac{3}{2} ) \zeta ( z + \frac{5}{2} ) \zeta ( z + \frac{7}{2} ) \zeta ( 2 \, z + 1 ) \zeta ( 3 \, z + \frac{3}{2} ) \zeta ( 2 \, z + 3 ) }$  &   $\left[0, -1, 1, -1\right]$  &  \\  \rowcolor{LightCyan} 
$ \frac{1}{2} $  
 &  $w_{1}w_{2}w_{3}w_{4}w_{2}w_{3}w_{1}w_{2}w_{3}w_{4}w_{3}w_{2}w_{3}w_{1}w_{2}w_{3}$  &  $ 5$  &  $\frac{ \zeta ( 3 \, z - \frac{1}{2} ) \zeta ( z - \frac{3}{2} ) \zeta ( z - \frac{1}{2} ) \zeta ( z + \frac{1}{2} ) \zeta ( 2 \, z - 1 ) \zeta ( 2 \, z ) } { \zeta ( z + \frac{3}{2} ) \zeta ( z + \frac{5}{2} ) \zeta ( z + \frac{7}{2} ) \zeta ( 2 \, z + 1 ) \zeta ( 3 \, z + \frac{3}{2} ) \zeta ( 2 \, z + 3 ) }$  &   $\left[0, -1, 1, -1\right]$  &  \\  \rowcolor{LightCyan} 
$ \frac{1}{2} $  
 &  $w_{3}w_{2}w_{3}w_{4}w_{3}w_{2}w_{3}w_{1}w_{2}w_{3}w_{4}w_{2}w_{3}w_{1}w_{2}w_{3}$  &  $ 5$  &  $\frac{ \zeta ( 3 \, z - \frac{1}{2} ) \zeta ( z + \frac{1}{2} ) \zeta ( 2 \, z ) \zeta ( z - \frac{1}{2} ) ^{ 2 } \zeta ( 2 \, z - 2 ) } { \zeta ( z + \frac{3}{2} ) \zeta ( z + \frac{5}{2} ) \zeta ( z + \frac{7}{2} ) \zeta ( 2 \, z + 1 ) \zeta ( 3 \, z + \frac{3}{2} ) \zeta ( 2 \, z + 3 ) }$  &   $\left[0, -1, 1, -1\right]$   &  \\  \rowcolor{LightCyan} 
$ \frac{1}{2} $  
 &  $w_{3}w_{1}w_{2}w_{3}w_{4}w_{3}w_{2}w_{3}w_{1}w_{2}w_{3}w_{4}w_{2}w_{3}w_{1}w_{2}w_{3}$  &  $ 5$  &  $\frac{ \zeta ( 3 \, z - \frac{1}{2} ) \zeta ( 2 \, z ) \zeta ( z - \frac{1}{2} ) ^{ 3 } \zeta ( 2 \, z - 2 ) } { \zeta ( z + \frac{3}{2} ) \zeta ( z + \frac{5}{2} ) \zeta ( z + \frac{7}{2} ) \zeta ( 2 \, z + 1 ) \zeta ( 3 \, z + \frac{3}{2} ) \zeta ( 2 \, z + 3 ) }$  &   $\left[0, -1, 1, -1\right]$ &   \\ \hline 
 \hline 
  \rowcolor{LightCyan} 
$ \frac{1}{2} $  
 &  $w_{4}w_{3}w_{1}w_{2}w_{3}w_{4}w_{3}w_{1}w_{2}w_{3}$  &  $ 4$  &  $\frac{ \zeta ( z + \frac{1}{2} ) ^{ 3 } \zeta ( 3 \, z + \frac{1}{2} ) \zeta ( 2 \, z ) } { \zeta ( z + \frac{3}{2} ) \zeta ( z + \frac{5}{2} ) \zeta ( z + \frac{7}{2} ) \zeta ( 3 \, z + \frac{3}{2} ) \zeta ( 2 \, z + 3 ) }$  &   $\left[-1, 1, -1, -1\right]$ &  \ref{image:sph} \\  \rowcolor{LightCyan} 
$ \frac{1}{2} $  
 &  $w_{4}w_{3}w_{1}w_{2}w_{3}w_{4}w_{2}w_{3}w_{1}w_{2}w_{3}$  &  $ 4$  &  $\frac{ \zeta ( z + \frac{1}{2} ) \zeta ( 2 \, z ) ^{ 2 } \zeta ( z - \frac{1}{2} ) \zeta ( 3 \, z + \frac{1}{2} ) } { \zeta ( z + \frac{5}{2} ) \zeta ( z + \frac{7}{2} ) \zeta ( 2 \, z + 1 ) \zeta ( 3 \, z + \frac{3}{2} ) \zeta ( 2 \, z + 3 ) }$  &   $\left[-1, 1, -1, -1\right]$  &  \\  \rowcolor{LightCyan} 
$ \frac{1}{2} $  
 &  $w_{3}w_{4}w_{2}w_{3}w_{1}w_{2}w_{3}w_{4}w_{1}w_{2}w_{3}$  &  $ 4$  &  $\frac{ \zeta ( 3 \, z - \frac{1}{2} ) \zeta ( z + \frac{1}{2} ) \zeta ( z - \frac{1}{2} ) \zeta ( 2 \, z ) } { \zeta ( z + \frac{5}{2} ) \zeta ( z + \frac{7}{2} ) \zeta ( 3 \, z + \frac{3}{2} ) \zeta ( 2 \, z + 3 ) }$  &   $\left[-1, 1, -1, -1\right]$  &  \\  \rowcolor{LightCyan} 
$ \frac{1}{2} $  
 &  $w_{4}w_{3}w_{2}w_{3}w_{1}w_{2}w_{3}w_{4}w_{3}w_{2}w_{3}$  &  $ 4$  &  $\frac{ \zeta ( 3 \, z - \frac{1}{2} ) \zeta ( z + \frac{1}{2} ) \zeta ( 2 \, z ) \zeta ( z - \frac{1}{2} ) } { \zeta ( z + \frac{5}{2} ) \zeta ( z + \frac{7}{2} ) \zeta ( 3 \, z + \frac{3}{2} ) \zeta ( 2 \, z + 3 ) }$  &   $\left[-1, 1, -1, -1\right]$  &  \\  \rowcolor{LightCyan} 
$ \frac{1}{2} $  
 &  $w_{3}w_{4}w_{2}w_{3}w_{1}w_{2}w_{3}w_{4}w_{3}w_{2}w_{3}w_{1}w_{2}w_{3}$  &  $ 4$  &  $\frac{ \zeta ( 3 \, z - \frac{1}{2} ) \zeta ( z - \frac{3}{2} ) \zeta ( z + \frac{1}{2} ) \zeta ( 2 \, z - 1 ) \zeta ( 2 \, z ) } { \zeta ( z + \frac{5}{2} ) \zeta ( z + \frac{7}{2} ) \zeta ( 2 \, z + 1 ) \zeta ( 3 \, z + \frac{3}{2} ) \zeta ( 2 \, z + 3 ) }$  &   $\left[-1, 1, -1, -1\right]$  &  \\  \rowcolor{LightCyan} 
$ \frac{1}{2} $  
 &  $w_{2}w_{3}w_{1}w_{2}w_{3}w_{4}w_{3}w_{2}w_{3}w_{1}w_{2}w_{3}w_{4}w_{2}w_{3}w_{1}w_{2}w_{3}$  &  $ 4$  &  $\frac{ \zeta ( 3 \, z - \frac{1}{2} ) \zeta ( 2 \, z ) \zeta ( z - \frac{1}{2} ) ^{ 2 } \zeta ( z - \frac{3}{2} ) \zeta ( 2 \, z - 2 ) } { \zeta ( z + \frac{3}{2} ) \zeta ( z + \frac{5}{2} ) \zeta ( z + \frac{7}{2} ) \zeta ( 2 \, z + 1 ) \zeta ( 3 \, z + \frac{3}{2} ) \zeta ( 2 \, z + 3 ) }$  &   $\left[-1, 1, -1, -1\right]$  &  \\ \hline 
 \hline 
  \rowcolor{LightCyan} 
$ \frac{1}{2} $  
 &  $w_{2}w_{3}w_{1}w_{2}w_{3}w_{4}w_{3}w_{1}w_{2}w_{3}$  &  $ 4$  &  $\frac{ \zeta ( z + \frac{1}{2} ) ^{ 3 } \zeta ( 2 \, z ) \zeta ( 3 \, z + \frac{1}{2} ) } { \zeta ( z + \frac{3}{2} ) \zeta ( z + \frac{5}{2} ) \zeta ( z + \frac{7}{2} ) \zeta ( 3 \, z + \frac{3}{2} ) \zeta ( 2 \, z + 3 ) }$  &   $\left[0, -1, 0, 1\right]$  &  \ref{image:third} \\  \rowcolor{LightCyan} 
$ \frac{1}{2} $  
 &  $w_{2}w_{3}w_{1}w_{2}w_{3}w_{4}w_{2}w_{3}w_{1}w_{2}w_{3}$  &  $ 4$  &  $\frac{ \zeta ( z + \frac{1}{2} ) ^{ 2 } \zeta ( 2 \, z ) \zeta ( z - \frac{1}{2} ) \zeta ( 3 \, z + \frac{1}{2} ) } { \zeta ( z + \frac{3}{2} ) \zeta ( z + \frac{5}{2} ) \zeta ( z + \frac{7}{2} ) \zeta ( 3 \, z + \frac{3}{2} ) \zeta ( 2 \, z + 3 ) }$  &   $\left[0, -1, 0, 1\right]$  &  \\  \rowcolor{LightCyan} 
$ \frac{1}{2} $  
 &  $w_{3}w_{2}w_{3}w_{1}w_{2}w_{3}w_{4}w_{3}w_{1}w_{2}w_{3}$  &  $ 4$  &  $\frac{ \zeta ( z + \frac{1}{2} ) ^{ 2 } \zeta ( 2 \, z ) \zeta ( 3 \, z + \frac{1}{2} ) \zeta ( z - \frac{1}{2} ) } { \zeta ( z + \frac{3}{2} ) \zeta ( z + \frac{5}{2} ) \zeta ( z + \frac{7}{2} ) \zeta ( 3 \, z + \frac{3}{2} ) \zeta ( 2 \, z + 3 ) }$  &   $\left[0, -1, 0, 1\right]$  &  \\  \rowcolor{LightCyan} 
$ \frac{1}{2} $  
 &  $w_{3}w_{2}w_{3}w_{1}w_{2}w_{3}w_{4}w_{2}w_{3}w_{1}w_{2}w_{3}$  &  $ 4$  &  $\frac{ \zeta ( z + \frac{1}{2} ) \zeta ( 2 \, z ) \zeta ( z - \frac{1}{2} ) ^{ 2 } \zeta ( 3 \, z + \frac{1}{2} ) } { \zeta ( z + \frac{3}{2} ) \zeta ( z + \frac{5}{2} ) \zeta ( z + \frac{7}{2} ) \zeta ( 3 \, z + \frac{3}{2} ) \zeta ( 2 \, z + 3 ) }$  &   $\left[0, -1, 0, 1\right]$   & \\  \rowcolor{LightCyan} 
$ \frac{1}{2} $  
 &  $w_{2}w_{3}w_{4}w_{3}w_{2}w_{3}w_{1}w_{2}w_{3}w_{4}w_{3}w_{2}w_{3}w_{1}w_{2}w_{3}$  &  $ 4$  &  $\frac{ \zeta ( 3 \, z - \frac{1}{2} ) \zeta ( z - \frac{3}{2} ) \zeta ( z + \frac{1}{2} ) ^{ 2 } \zeta ( 2 \, z ) \zeta ( 2 \, z - 2 ) } { \zeta ( z + \frac{3}{2} ) \zeta ( z + \frac{5}{2} ) \zeta ( z + \frac{7}{2} ) \zeta ( 2 \, z + 1 ) \zeta ( 3 \, z + \frac{3}{2} ) \zeta ( 2 \, z + 3 ) }$  &   $\left[0, -1, 0, 1\right]$  &   \\  \rowcolor{LightCyan} 
$ \frac{1}{2} $  
 &  $w_{1}w_{2}w_{3}w_{4}w_{3}w_{2}w_{3}w_{1}w_{2}w_{3}w_{4}w_{3}w_{2}w_{3}w_{1}w_{2}w_{3}$  &  $ 4$  &  $\frac{ \zeta ( 3 \, z - \frac{1}{2} ) \zeta ( z - \frac{3}{2} ) \zeta ( z - \frac{1}{2} ) \zeta ( z + \frac{1}{2} ) \zeta ( 2 \, z ) \zeta ( 2 \, z - 2 ) } { \zeta ( z + \frac{3}{2} ) \zeta ( z + \frac{5}{2} ) \zeta ( z + \frac{7}{2} ) \zeta ( 2 \, z + 1 ) \zeta ( 3 \, z + \frac{3}{2} ) \zeta ( 2 \, z + 3 ) }$  &   $\left[0, -1, 0, 1\right]$  &  \\  \rowcolor{LightCyan} 
$ \frac{1}{2} $  
 &  $w_{3}w_{2}w_{3}w_{4}w_{3}w_{2}w_{3}w_{1}w_{2}w_{3}w_{4}w_{3}w_{2}w_{3}w_{1}w_{2}w_{3}$  &  $ 4$  &  $\frac{ \zeta ( 3 \, z - \frac{1}{2} ) \zeta ( z - \frac{3}{2} ) \zeta ( z - \frac{1}{2} ) \zeta ( z + \frac{1}{2} ) \zeta ( 2 \, z ) \zeta ( 2 \, z - 2 ) } { \zeta ( z + \frac{3}{2} ) \zeta ( z + \frac{5}{2} ) \zeta ( z + \frac{7}{2} ) \zeta ( 2 \, z + 1 ) \zeta ( 3 \, z + \frac{3}{2} ) \zeta ( 2 \, z + 3 ) }$  &   $\left[0, -1, 0, 1\right]$  &  \\  \rowcolor{LightCyan} 
$ \frac{1}{2} $  
 &  $w_{3}w_{1}w_{2}w_{3}w_{4}w_{3}w_{2}w_{3}w_{1}w_{2}w_{3}w_{4}w_{3}w_{2}w_{3}w_{1}w_{2}w_{3}$  &  $ 4$  &  $\frac{ \zeta ( 3 \, z - \frac{1}{2} ) \zeta ( z - \frac{3}{2} ) \zeta ( z - \frac{1}{2} ) ^{ 2 } \zeta ( 2 \, z ) \zeta ( 2 \, z - 2 ) } { \zeta ( z + \frac{3}{2} ) \zeta ( z + \frac{5}{2} ) \zeta ( z + \frac{7}{2} ) \zeta ( 2 \, z + 1 ) \zeta ( 3 \, z + \frac{3}{2} ) \zeta ( 2 \, z + 3 ) }$  &   $\left[0, -1, 0, 1\right]$  &  \\ \hline 
 \hline 
  \rowcolor{LightCyan} 
$ \frac{1}{2} $  
 &  $w_{1}w_{2}w_{3}w_{4}w_{3}w_{1}w_{2}w_{3}$  &  $ 3$  &  $\frac{ \zeta ( z + \frac{1}{2} ) ^{ 3 } \zeta ( 2 \, z + 1 ) } { \zeta ( z + \frac{3}{2} ) \zeta ( z + \frac{5}{2} ) \zeta ( z + \frac{7}{2} ) \zeta ( 2 \, z + 3 ) }$  &   $\left[-1, -1, 2, -1\right]$  &  \ref{image:third}\\  \rowcolor{LightCyan} 
$ \frac{1}{2} $  
 &  $w_{1}w_{2}w_{3}w_{4}w_{2}w_{3}w_{1}w_{2}w_{3}$  &  $ 3$  &  $\frac{ \zeta ( z + \frac{1}{2} ) \zeta ( 2 \, z ) \zeta ( z - \frac{1}{2} ) \zeta ( 3 \, z + \frac{1}{2} ) } { \zeta ( z + \frac{3}{2} ) \zeta ( z + \frac{7}{2} ) \zeta ( 3 \, z + \frac{3}{2} ) \zeta ( 2 \, z + 3 ) }$  &   $\left[-1, -1, 2, -1\right]$  &  \\  \rowcolor{LightCyan} 
$ \frac{1}{2} $  
 &  $w_{2}w_{3}w_{1}w_{2}w_{3}w_{4}w_{3}w_{2}w_{3}$  &  $ 3$  &  $\frac{ \zeta ( z + \frac{1}{2} ) \zeta ( 2 \, z ) \zeta ( z - \frac{1}{2} ) } { \zeta ( z + \frac{5}{2} ) \zeta ( z + \frac{7}{2} ) \zeta ( 2 \, z + 3 ) }$  &   $\left[-1, -1, 2, -1\right]$  &  \\  \rowcolor{LightCyan} 
$ \frac{1}{2} $  
 &  $w_{4}w_{3}w_{2}w_{3}w_{1}w_{2}w_{3}w_{4}w_{3}w_{2}w_{3}w_{1}w_{2}w_{3}$  &  $ 3$  &  $\frac{ \zeta ( 3 \, z - \frac{1}{2} ) \zeta ( z - \frac{3}{2} ) \zeta ( z + \frac{1}{2} ) \zeta ( 2 \, z ) \zeta ( 2 \, z - 2 ) } { \zeta ( z + \frac{3}{2} ) \zeta ( z + \frac{7}{2} ) \zeta ( 2 \, z + 1 ) \zeta ( 3 \, z + \frac{3}{2} ) \zeta ( 2 \, z + 3 ) }$  &   $\left[-1, -1, 2, -1\right]$  &  \\  \rowcolor{LightCyan} 
$ \frac{1}{2} $  
 &  $w_{3}w_{2}w_{3}w_{1}w_{2}w_{3}w_{4}w_{3}w_{2}w_{3}w_{1}w_{2}w_{3}w_{4}w_{3}w_{2}w_{3}w_{1}w_{2}w_{3}$  &  $ 3$  &  $\frac{ \zeta ( 3 \, z - \frac{1}{2} ) \zeta ( z - \frac{3}{2} ) \zeta ( z - \frac{1}{2} ) \zeta ( 2 \, z ) \zeta ( 2 \, z - 2 ) \zeta ( z - \frac{5}{2} ) } { \zeta ( z + \frac{3}{2} ) \zeta ( z + \frac{5}{2} ) \zeta ( z + \frac{7}{2} ) \zeta ( 2 \, z + 1 ) \zeta ( 3 \, z + \frac{3}{2} ) \zeta ( 2 \, z + 3 ) }$  &   $\left[-1, -1, 2, -1\right]$   &  \\ \hline 
 \hline 
  \rowcolor{LightCyan} 
$ \frac{1}{2} $  
 &  $w_{3}w_{1}w_{2}w_{3}w_{4}w_{1}w_{2}w_{3}$  &  $ 3$  &  $\frac{ \zeta ( z + \frac{1}{2} ) ^{ 2 } \zeta ( 2 \, z ) } { \zeta ( z + \frac{5}{2} ) \zeta ( z + \frac{7}{2} ) \zeta ( 2 \, z + 3 ) }$  &   $\left[-1, 0, -1, 2\right]$  &  \ref{image:even} \\  \rowcolor{LightCyan} 
$ \frac{1}{2} $  
 &  $w_{2}w_{3}w_{1}w_{2}w_{3}w_{4}w_{1}w_{2}w_{3}$  &  $ 3$  &  $\frac{ \zeta ( z + \frac{1}{2} ) \zeta ( z - \frac{1}{2} ) \zeta ( 2 \, z ) } { \zeta ( z + \frac{5}{2} ) \zeta ( z + \frac{7}{2} ) \zeta ( 2 \, z + 3 ) }$  &   $\left[-1, 0, -1, 2\right]$  &  \\  \rowcolor{LightCyan} 
$ \frac{1}{2} $  
 &  $w_{3}w_{1}w_{2}w_{3}w_{4}w_{3}w_{2}w_{3}w_{1}w_{2}w_{3}$  &  $ 3$  &  $\frac{ \zeta ( z - \frac{3}{2} ) \zeta ( z + \frac{1}{2} ) \zeta ( 2 \, z ) ^{ 2 } \zeta ( 3 \, z + \frac{1}{2} ) } { \zeta ( z + \frac{3}{2} ) \zeta ( z + \frac{7}{2} ) \zeta ( 2 \, z + 1 ) \zeta ( 3 \, z + \frac{3}{2} ) \zeta ( 2 \, z + 3 ) }$  &   $\left[-1, 0, -1, 2\right]$  &  \\  \rowcolor{LightCyan} 
$ \frac{1}{2} $  
 &  $w_{2}w_{3}w_{1}w_{2}w_{3}w_{4}w_{3}w_{2}w_{3}w_{1}w_{2}w_{3}$  &  $ 3$  &  $\frac{ \zeta ( z - \frac{3}{2} ) \zeta ( z + \frac{1}{2} ) \zeta ( 2 \, z - 1 ) \zeta ( 2 \, z ) \zeta ( 3 \, z + \frac{1}{2} ) } { \zeta ( z + \frac{3}{2} ) \zeta ( z + \frac{7}{2} ) \zeta ( 2 \, z + 1 ) \zeta ( 3 \, z + \frac{3}{2} ) \zeta ( 2 \, z + 3 ) }$  &   $\left[-1, 0, -1, 2\right]$  &  \\ \hline 
 \hline 
  \rowcolor{LightCyan} 
$ \frac{1}{2} $  
 &  $w_{4}w_{3}w_{2}w_{3}w_{4}w_{3}w_{1}w_{2}w_{3}$  &  $ 3$  &  $\frac{ \zeta ( z + \frac{1}{2} ) ^{ 2 } \zeta ( 3 \, z + \frac{1}{2} ) \zeta ( 2 \, z ) } { \zeta ( z + \frac{5}{2} ) \zeta ( z + \frac{7}{2} ) \zeta ( 3 \, z + \frac{3}{2} ) \zeta ( 2 \, z + 3 ) }$  &   $\left[1, 0, -1, -1\right]$  &  \ref{image:sph} \\  \rowcolor{LightCyan} 
$ \frac{1}{2} $  
 &  $w_{4}w_{3}w_{2}w_{3}w_{4}w_{2}w_{3}w_{1}w_{2}w_{3}$  &  $ 3$  &  $\frac{ \zeta ( z + \frac{1}{2} ) \zeta ( z - \frac{1}{2} ) \zeta ( 3 \, z + \frac{1}{2} ) \zeta ( 2 \, z ) } { \zeta ( z + \frac{5}{2} ) \zeta ( z + \frac{7}{2} ) \zeta ( 3 \, z + \frac{3}{2} ) \zeta ( 2 \, z + 3 ) }$  &   $\left[1, 0, -1, -1\right]$  &  \\ \hline 
 \hline 
  \rowcolor{LightCyan} 
$ \frac{1}{2} $  
 &  $w_{3}w_{1}w_{2}w_{3}w_{4}w_{3}w_{1}w_{2}w_{3}$  &  $ 3$  &  $\frac{ \zeta ( z + \frac{1}{2} ) ^{ 3 } \zeta ( 2 \, z + 1 ) \zeta ( 3 \, z + \frac{1}{2} ) } { \zeta ( z + \frac{3}{2} ) \zeta ( z + \frac{5}{2} ) \zeta ( z + \frac{7}{2} ) \zeta ( 3 \, z + \frac{3}{2} ) \zeta ( 2 \, z + 3 ) }$  &   $\left[-1, 1, -2, 1\right]$  &  \ref{image:third}\\  \rowcolor{LightCyan} 
$ \frac{1}{2} $  
 &  $w_{3}w_{1}w_{2}w_{3}w_{4}w_{2}w_{3}w_{1}w_{2}w_{3}$  &  $ 3$  &  $\frac{ \zeta ( z + \frac{1}{2} ) \zeta ( 2 \, z ) \zeta ( z - \frac{1}{2} ) \zeta ( 3 \, z + \frac{1}{2} ) } { \zeta ( z + \frac{5}{2} ) \zeta ( z + \frac{7}{2} ) \zeta ( 3 \, z + \frac{3}{2} ) \zeta ( 2 \, z + 3 ) }$  &   $\left[-1, 1, -2, 1\right]$  &  \\  \rowcolor{LightCyan} 
$ \frac{1}{2} $  
 &  $w_{3}w_{2}w_{3}w_{1}w_{2}w_{3}w_{4}w_{3}w_{2}w_{3}$  &  $ 3$  &  $\frac{ \zeta ( z + \frac{1}{2} ) \zeta ( 2 \, z ) \zeta ( 3 \, z + \frac{1}{2} ) \zeta ( z - \frac{1}{2} ) } { \zeta ( z + \frac{5}{2} ) \zeta ( z + \frac{7}{2} ) \zeta ( 3 \, z + \frac{3}{2} ) \zeta ( 2 \, z + 3 ) }$  &   $\left[-1, 1, -2, 1\right]$  &  \\  \rowcolor{LightCyan} 
$ \frac{1}{2} $  
 &  $w_{3}w_{4}w_{3}w_{2}w_{3}w_{1}w_{2}w_{3}w_{4}w_{3}w_{2}w_{3}w_{1}w_{2}w_{3}$  &  $ 3$  &  $\frac{ \zeta ( 3 \, z - \frac{1}{2} ) \zeta ( z - \frac{3}{2} ) \zeta ( z + \frac{1}{2} ) \zeta ( 2 \, z ) \zeta ( 2 \, z - 2 ) } { \zeta ( z + \frac{5}{2} ) \zeta ( z + \frac{7}{2} ) \zeta ( 2 \, z + 1 ) \zeta ( 3 \, z + \frac{3}{2} ) \zeta ( 2 \, z + 3 ) }$  &   $\left[-1, 1, -2, 1\right]$  &  \\  \rowcolor{LightCyan} 
$ \frac{1}{2} $  
 &  $w_{2}w_{3}w_{1}w_{2}w_{3}w_{4}w_{3}w_{2}w_{3}w_{1}w_{2}w_{3}w_{4}w_{3}w_{2}w_{3}w_{1}w_{2}w_{3}$  &  $ 3$  &  $\frac{ \zeta ( 3 \, z - \frac{1}{2} ) \zeta ( z - \frac{3}{2} ) ^{ 2 } \zeta ( z - \frac{1}{2} ) \zeta ( 2 \, z ) \zeta ( 2 \, z - 2 ) } { \zeta ( z + \frac{3}{2} ) \zeta ( z + \frac{5}{2} ) \zeta ( z + \frac{7}{2} ) \zeta ( 2 \, z + 1 ) \zeta ( 3 \, z + \frac{3}{2} ) \zeta ( 2 \, z + 3 ) }$  &   $\left[-1, 1, -2, 1\right]$   &  \\ \hline 
 \hline 
  \rowcolor{LightCyan} 
$ \frac{1}{2} $  
 &  $w_{4}w_{3}w_{1}w_{2}w_{3}w_{4}w_{1}w_{2}w_{3}$  &  $ 3$  &  $\frac{ \zeta ( z + \frac{1}{2} ) ^{ 2 } \zeta ( 3 \, z + \frac{1}{2} ) \zeta ( 2 \, z ) } { \zeta ( z + \frac{5}{2} ) \zeta ( z + \frac{7}{2} ) \zeta ( 3 \, z + \frac{3}{2} ) \zeta ( 2 \, z + 3 ) }$  &   $\left[-1, 0, 1, -2\right]$  &  \ref{image:even} \\  \rowcolor{LightCyan} 
$ \frac{1}{2} $  
 &  $w_{4}w_{2}w_{3}w_{1}w_{2}w_{3}w_{4}w_{1}w_{2}w_{3}$  &  $ 3$  &  $\frac{ \zeta ( z + \frac{1}{2} ) \zeta ( z - \frac{1}{2} ) \zeta ( 3 \, z + \frac{1}{2} ) \zeta ( 2 \, z ) } { \zeta ( z + \frac{5}{2} ) \zeta ( z + \frac{7}{2} ) \zeta ( 3 \, z + \frac{3}{2} ) \zeta ( 2 \, z + 3 ) }$  &   $\left[-1, 0, 1, -2\right]$  &  \\  \rowcolor{LightCyan} 
$ \frac{1}{2} $  
 &  $w_{4}w_{3}w_{1}w_{2}w_{3}w_{4}w_{3}w_{2}w_{3}w_{1}w_{2}w_{3}$  &  $ 3$  &  $\frac{ \zeta ( z - \frac{3}{2} ) \zeta ( z + \frac{1}{2} ) \zeta ( 2 \, z ) ^{ 2 } \zeta ( 3 \, z + \frac{1}{2} ) } { \zeta ( z + \frac{5}{2} ) \zeta ( z + \frac{7}{2} ) \zeta ( 2 \, z + 1 ) \zeta ( 3 \, z + \frac{3}{2} ) \zeta ( 2 \, z + 3 ) }$  &   $\left[-1, 0, 1, -2\right]$  &  \\  \rowcolor{LightCyan} 
$ \frac{1}{2} $  
 &  $w_{4}w_{2}w_{3}w_{1}w_{2}w_{3}w_{4}w_{3}w_{2}w_{3}w_{1}w_{2}w_{3}$  &  $ 3$  &  $\frac{ \zeta ( z - \frac{3}{2} ) \zeta ( z + \frac{1}{2} ) \zeta ( 2 \, z - 1 ) \zeta ( 2 \, z ) \zeta ( 3 \, z + \frac{1}{2} ) } { \zeta ( z + \frac{5}{2} ) \zeta ( z + \frac{7}{2} ) \zeta ( 2 \, z + 1 ) \zeta ( 3 \, z + \frac{3}{2} ) \zeta ( 2 \, z + 3 ) }$  &   $\left[-1, 0, 1, -2\right]$  &  \\ \hline 
 \hline 
  \rowcolor{ashgrey} 
$ \frac{1}{2} $  
 &  $w_{3}w_{2}w_{3}w_{4}w_{1}w_{2}w_{3}$  &  $ 2$  &  $\frac{ \zeta ( z + \frac{1}{2} ) \zeta ( z + \frac{3}{2} ) \zeta ( 2 \, z ) } { \zeta ( z + \frac{5}{2} ) \zeta ( z + \frac{7}{2} ) \zeta ( 2 \, z + 3 ) }$  &   $\left[1, -1, -1, 2\right]$  &   \ref{image:even}
 \\  \rowcolor{ashgrey} 
$ \frac{1}{2} $  
 &  $w_{3}w_{2}w_{3}w_{4}w_{3}w_{2}w_{3}w_{1}w_{2}w_{3}$  &  $ 2$  &  $\frac{ \zeta ( z - \frac{3}{2} ) \zeta ( z + \frac{1}{2} ) \zeta ( 3 \, z + \frac{1}{2} ) \zeta ( 2 \, z ) } { \zeta ( z + \frac{3}{2} ) \zeta ( z + \frac{7}{2} ) \zeta ( 3 \, z + \frac{3}{2} ) \zeta ( 2 \, z + 3 ) }$  &   $\left[1, -1, -1, 2\right]$  &  \\ \hline 
 \hline 
  \rowcolor{ashgrey} 
$ \frac{1}{2} $  
 &  $w_{1}w_{2}w_{3}w_{4}w_{1}w_{2}w_{3}$  &  $ 2$  &  $\frac{ \zeta ( 2 \, z + 1 ) \zeta ( z + \frac{1}{2} ) ^{ 2 } } { \zeta ( z + \frac{5}{2} ) \zeta ( z + \frac{7}{2} ) \zeta ( 2 \, z + 3 ) }$  &   $\left[-1, -1, 1, 1\right]$  &  \ref{image:third} \\  \rowcolor{ashgrey} 
$ \frac{1}{2} $  
 &  $w_{1}w_{2}w_{3}w_{4}w_{3}w_{2}w_{3}w_{1}w_{2}w_{3}$  &  $ 2$  &  $\frac{ \zeta ( z - \frac{3}{2} ) \zeta ( z + \frac{1}{2} ) \zeta ( 2 \, z ) \zeta ( 3 \, z + \frac{1}{2} ) } { \zeta ( z + \frac{3}{2} ) \zeta ( z + \frac{7}{2} ) \zeta ( 3 \, z + \frac{3}{2} ) \zeta ( 2 \, z + 3 ) }$  &   $\left[-1, -1, 1, 1\right]$  &  \\  \rowcolor{ashgrey} 
$ \frac{1}{2} $  
 &  $w_{3}w_{2}w_{3}w_{1}w_{2}w_{3}w_{4}w_{3}w_{2}w_{3}w_{1}w_{2}w_{3}$  &  $ 2$  &  $\frac{ \zeta ( z - \frac{3}{2} ) \zeta ( z + \frac{1}{2} ) \zeta ( 2 \, z ) \zeta ( 3 \, z + \frac{1}{2} ) \zeta ( 2 \, z - 2 ) } { \zeta ( z + \frac{3}{2} ) \zeta ( z + \frac{7}{2} ) \zeta ( 2 \, z + 1 ) \zeta ( 3 \, z + \frac{3}{2} ) \zeta ( 2 \, z + 3 ) }$  &   $\left[-1, -1, 1, 1\right]$  &  \\ \hline 
 \hline 
  \rowcolor{ashgrey} 
$ \frac{1}{2} $  
 &  $w_{4}w_{3}w_{2}w_{3}w_{4}w_{1}w_{2}w_{3}$  &  $ 2$  &  $\frac{ \zeta ( z + \frac{1}{2} ) \zeta ( z + \frac{3}{2} ) \zeta ( 3 \, z + \frac{1}{2} ) \zeta ( 2 \, z ) } { \zeta ( z + \frac{5}{2} ) \zeta ( z + \frac{7}{2} ) \zeta ( 3 \, z + \frac{3}{2} ) \zeta ( 2 \, z + 3 ) }$  &   $\left[1, -1, 1, -2\right]$  &  \ref{image:even} \\  \rowcolor{ashgrey} 
$ \frac{1}{2} $  
 &  $w_{4}w_{3}w_{2}w_{3}w_{4}w_{3}w_{2}w_{3}w_{1}w_{2}w_{3}$  &  $ 2$  &  $\frac{ \zeta ( z - \frac{3}{2} ) \zeta ( z + \frac{1}{2} ) \zeta ( 3 \, z + \frac{1}{2} ) \zeta ( 2 \, z ) } { \zeta ( z + \frac{5}{2} ) \zeta ( z + \frac{7}{2} ) \zeta ( 3 \, z + \frac{3}{2} ) \zeta ( 2 \, z + 3 ) }$  &   $\left[1, -1, 1, -2\right]$  &  \\ \hline 
 \hline 
  \rowcolor{WHITE} 
$ \frac{1}{2} $  
 &  $\bk{w_{3}w_{1}w_{2}w_{3}}$  &  $ 2$  &  $\frac{ \zeta ( z + \frac{1}{2} ) ^{ 2 } } { \zeta ( z + \frac{3}{2} ) \zeta ( z + \frac{7}{2} ) }$  &   $\left[-1, 0, -1, 3\right]$ 
 &  \ref{image:0a}
\\  \rowcolor{WHITE} 
$ \frac{1}{2} $  
 &  $w_{2}\bk{w_{3}w_{1}w_{2}w_{3}}$  &  $ 2$  &  $\frac{ \zeta ( z + \frac{1}{2} ) \zeta ( z - \frac{1}{2} ) } { \zeta ( z + \frac{3}{2} ) \zeta ( z + \frac{7}{2} ) }$  &   $\left[-1, 0, -1, 3\right]$  &   \\ \hline 
 \hline 
  \rowcolor{WHITE} 
$ \frac{1}{2} $  
 &  $\bk{w_{4}} \bk{ w_{3}w_{1}w_{2}w_{3}}$  &  $ 2$  &  $\frac{ \zeta ( z + \frac{1}{2} ) ^{ 2 } \zeta ( 2 \, z + 2 ) } { \zeta ( z + \frac{3}{2} ) \zeta ( z + \frac{7}{2} ) \zeta ( 2 \, z + 3 ) }$  &   $\left[-1, 0, 2, -3\right]$  &  \ref{image:0a} \\  \rowcolor{WHITE} 
$ \frac{1}{2} $  
 &  $\bk{w_{4}}\bk{w_{2}}\bk{w_{3}w_{1}w_{2}w_{3}}$  &  $ 2$  &  $\frac{ \zeta ( z + \frac{1}{2} ) \zeta ( z - \frac{1}{2} ) \zeta ( 2 \, z + 2 ) } { \zeta ( z + \frac{3}{2} ) \zeta ( z + \frac{7}{2} ) \zeta ( 2 \, z + 3 ) }$  &   $\left[-1, 0, 2, -3\right]$  & \\ \hline 
 \hline 
  \rowcolor{WHITE} 
$ \frac{1}{2} $  
 &  $\bk{w_{3}w_{4}}\bk{w_{3}w_{1}w_{2}w_{3}}$  &  $ 2$  &  $\frac{ \zeta ( z + \frac{1}{2} ) ^{ 2 } \zeta ( 2 \, z + 2 ) } { \zeta ( z + \frac{5}{2} ) \zeta ( z + \frac{7}{2} ) \zeta ( 2 \, z + 3 ) }$  &   $\left[-1, 2, -2, -1\right]$  &  \ref{image:0a} \\  \rowcolor{WHITE} 
$ \frac{1}{2} $  
 &  $\bk{w_{3}w_{4}}\bk{w_{2}}\bk{w_{3}w_{1}w_{2}w_{3}}$  &  $ 2$  &  $\frac{ \zeta ( z + \frac{1}{2} ) \zeta ( z - \frac{1}{2} ) \zeta ( 3 \, z + \frac{1}{2} ) \zeta ( 2 \, z + 2 ) } { \zeta ( z + \frac{3}{2} ) \zeta ( z + \frac{7}{2} ) \zeta ( 3 \, z + \frac{3}{2} ) \zeta ( 2 \, z + 3 ) }$  &   $\left[-1, 2, -2, -1\right]$  &   \\ \hline 
 \hline 
  \rowcolor{WHITE} 
$ \frac{1}{2} $  
 &  $\bk{w_{2}w_{3}w_{4}w_{3}w_{2}w_{3}}$  &  $ 2$  &  $\frac{ \zeta ( z + \frac{1}{2} ) ^{ 2 } \zeta ( 2 \, z + 2 ) } { \zeta ( z + \frac{5}{2} ) \zeta ( z + \frac{7}{2} ) \zeta ( 2 \, z + 3 ) }$  &   $\left[2, -1, 0, -1\right]$  &  \ref{image:0a}\\  \rowcolor{WHITE} 
$ \frac{1}{2} $  
 &  $\bk{w_{3}}\bk{w_{2}w_{3}w_{4}w_{3}w_{2}w_{3}}$  &  $ 2$  &  $\frac{ \zeta ( z + \frac{1}{2} ) \zeta ( 2 \, z + 2 ) \zeta ( z - \frac{1}{2} ) } { \zeta ( z + \frac{5}{2} ) \zeta ( z + \frac{7}{2} ) \zeta ( 2 \, z + 3 ) }$  &   $\left[2, -1, 0, -1\right]$  &   \\ \hline 
 \hline 
  \rowcolor{WHITE} 
$ \frac{1}{2} $  
 &  $\bk{w_{2}w_{3}w_{4}}\bk{w_{3}w_{1}w_{2}w_{3}}$  &  $ 2$  &  $\frac{ \zeta ( z + \frac{1}{2} ) ^{ 2 } \zeta ( 2 \, z + 1 ) } { \zeta ( z + \frac{5}{2} ) \zeta ( z + \frac{7}{2} ) \zeta ( 2 \, z + 3 ) }$  &   $\left[1, -2, 2, -1\right]$  &  \ref{image:0a} \\  \rowcolor{WHITE} 
$ \frac{1}{2} $  
 &  $\bk{w_{2}w_{3}w_{4}}\bk{w_{2}}\bk{w_{3}w_{1}w_{2}w_{3}}$  &  $ 2$  &  $\frac{ \zeta ( z + \frac{1}{2} ) \zeta ( 2 \, z + 1 ) \zeta ( z - \frac{1}{2} ) \zeta ( 3 \, z + \frac{1}{2} ) } { \zeta ( z + \frac{3}{2} ) \zeta ( z + \frac{7}{2} ) \zeta ( 3 \, z + \frac{3}{2} ) \zeta ( 2 \, z + 3 ) }$  &   $\left[1, -2, 2, -1\right]$  & \\ \hline 
 \hline 
  \rowcolor{WHITE} 
$ \frac{1}{2} $  
 &  $\bk{w_{1}}\bk{w_{2}w_{3}w_{4}w_{3}w_{2}w_{3}}$  &  $ 2$  &  $\frac{ \zeta ( z + \frac{1}{2} ) ^{ 2 } \zeta ( 2 \, z + 1 ) } { \zeta ( z + \frac{5}{2} ) \zeta ( z + \frac{7}{2} ) \zeta ( 2 \, z + 3 ) }$  &   $\left[-2, 1, 0, -1\right]$  &  \ref{image:0a}\\  \rowcolor{WHITE} 
$ \frac{1}{2} $  
 &  $\bk{w_{1}}\bk{w_{3}}\bk{w_{2}w_{3}w_{4}w_{3}w_{2}w_{3}}$  &  $ 2$  &  $\frac{ \zeta ( z + \frac{1}{2} ) \zeta ( 2 \, z + 1 ) \zeta ( z - \frac{1}{2} ) } { \zeta ( z + \frac{5}{2} ) \zeta ( z + \frac{7}{2} ) \zeta ( 2 \, z + 3 ) }$  &   $\left[-2, 1, 0, -1\right]$  &  \\ \hline 
 \hline 
  \rowcolor{WHITE} 
$ \frac{1}{2} $  
 &  $\bk{w_{3}w_{2}w_{3}w_{4}}\bk{w_{3}w_{1}w_{2}w_{3}}$  &  $ 2$  &  $\frac{ \zeta ( z + \frac{1}{2} ) ^{ 2 } \zeta ( 2 \, z + 1 ) \zeta ( 3 \, z + \frac{1}{2} ) } { \zeta ( z + \frac{5}{2} ) \zeta ( z + \frac{7}{2} ) \zeta ( 3 \, z + \frac{3}{2} ) \zeta ( 2 \, z + 3 ) }$  &   $\left[1, 0, -2, 1\right]$  &  \ref{image:0a}\\  \rowcolor{WHITE} 
$ \frac{1}{2} $  
 &  $\bk{w_{3}w_{2}w_{3}w_{4}}\bk{w_{2}}\bk{w_{3}w_{1}w_{2}w_{3}}$  &  $ 2$   &  $\frac{ \zeta ( z + \frac{1}{2} ) \zeta ( 2 \, z + 1 ) \zeta ( z - \frac{1}{2} ) \zeta ( 3 \, z + \frac{1}{2} ) } { \zeta ( z + \frac{5}{2} ) \zeta ( z + \frac{7}{2} ) \zeta ( 3 \, z + \frac{3}{2} ) \zeta ( 2 \, z + 3 ) }$  &   $\left[1, 0, -2, 1\right]$  &  \\ \hline 
 \hline 
  \rowcolor{WHITE} 
$ \frac{1}{2} $  
 &  $w_{3}w_{2}w_{3}$  &  $ 1$  &  $\frac{ \zeta ( z + \frac{1}{2} ) } { \zeta ( z + \frac{7}{2} ) }$  &   $\left[1, -1, -1, 3\right]$   &   \ref{image:0}\\ \hline 
 \hline 
  \rowcolor{WHITE} 
$ \frac{1}{2} $  
 &  $\bk{w_{1}w_{2}w_{3}}$  &  $ 1$  &  $\frac{ \zeta ( z + \frac{1}{2} ) } { \zeta ( z + \frac{7}{2} ) }$  &   $\left[-1, -1, 1, 2\right]$  &   \ref{image:0} \\  \rowcolor{WHITE} 
$ \frac{1}{2} $  
 &  $\bk{w_{3}w_{2}w_{3}}\bk{w_{1}w_{2}w_{3}}$  &  $ 1$  &  $\frac{ \zeta ( z - \frac{3}{2} ) \zeta ( z + \frac{1}{2} ) } { \zeta ( z + \frac{3}{2} ) \zeta ( z + \frac{7}{2} ) }$  &   $\left[-1, -1, 1, 2\right]$  &   \\ \hline 
 \hline 
  \rowcolor{WHITE} 
$ \frac{1}{2} $  
 &  $w_{4}w_{3}w_{2}w_{3}$  &  $ 1$  &  $\frac{ \zeta ( z + \frac{1}{2} ) \zeta ( 2 \, z + 2 ) } { \zeta ( z + \frac{7}{2} ) \zeta ( 2 \, z + 3 ) }$  &   $\left[1, -1, 2, -3\right]$   &   \ref{image:0} \\ \hline 
 \hline 
  \rowcolor{WHITE} 
$ \frac{1}{2} $  
 &  $\bk{w_{4}}\bk{w_{1}w_{2}w_{3}}$  &  $ 1$  &  $\frac{ \zeta ( z + \frac{1}{2} ) \zeta ( z + \frac{3}{2} ) } { \zeta ( z + \frac{5}{2} ) \zeta ( z + \frac{7}{2} ) }$  &   $\left[-1, -1, 3, -2\right]$  &  \ref{image:0} \\  \rowcolor{WHITE} 
$ \frac{1}{2} $  
 &  $\bk{w_4}\bk{w_{3}w_{2}w_{3}}\bk{ w_{1}w_{2}w_{3}}$  &  $ 1$  &  $\frac{ \zeta ( z - \frac{3}{2} ) \zeta ( z + \frac{1}{2} ) \zeta ( 3 \, z + \frac{1}{2} ) } { \zeta ( z + \frac{3}{2} ) \zeta ( z + \frac{7}{2} ) \zeta ( 3 \, z + \frac{3}{2} ) }$  &   $\left[-1, -1, 3, -2\right]$  &  \\ \hline 
 \hline 
  \rowcolor{WHITE} 
$ \frac{1}{2} $  
 &  $w_{3}w_{4}w_{3}w_{2}w_{3}$  &  $ 1$  &  $\frac{ \zeta ( z + \frac{1}{2} ) \zeta ( z + \frac{3}{2} ) \zeta ( 2 \, z + 2 ) } { \zeta ( z + \frac{5}{2} ) \zeta ( z + \frac{7}{2} ) \zeta ( 2 \, z + 3 ) }$  &   $\left[1, 1, -2, -1\right]$  &   \ref{image:0} \\ \hline 
 \hline 
  \rowcolor{WHITE} 
$ \frac{1}{2} $  
 &  $w_{2}w_{3}w_{4}w_{2}w_{3}$  &  $ 1$  &  $\frac{ \zeta ( z + \frac{3}{2} ) \zeta ( z + \frac{1}{2} ) \zeta ( 2 \, z + 2 ) } { \zeta ( z + \frac{5}{2} ) \zeta ( z + \frac{7}{2} ) \zeta ( 2 \, z + 3 ) }$  &   $\left[2, -1, -1, 1\right]$  &  \ref{image:0} \\ \hline 
 \hline 
  \rowcolor{WHITE} 
$ \frac{1}{2} $  
 &  $\bk{w_{3}w_{4}}\bk{w_{1}w_{2}w_{3}}$  &  $ 1$  &  $\frac{ \zeta ( z + \frac{1}{2} ) \zeta ( z + \frac{3}{2} ) \zeta ( 2 \, z + 2 ) } { \zeta ( z + \frac{5}{2} ) \zeta ( z + \frac{7}{2} ) \zeta ( 2 \, z + 3 ) }$  &   $\left[-1, 2, -3, 1\right]$  &  \ref{image:0} \\  \rowcolor{WHITE} 
$ \frac{1}{2} $  
 &  $\bk{w_{3}w_{4}}\bk{w_{3}w_{2}w_{3}}\bk{w_{1}w_{2}w_{3}}$  &  $ 1$  &  $\frac{ \zeta ( z - \frac{3}{2} ) \zeta ( z + \frac{1}{2} ) \zeta ( 3 \, z + \frac{1}{2} ) \zeta ( 2 \, z + 2 ) } { \zeta ( z + \frac{3}{2} ) \zeta ( z + \frac{7}{2} ) \zeta ( 3 \, z + \frac{3}{2} ) \zeta ( 2 \, z + 3 ) }$  &   $\left[-1, 2, -3, 1\right]$  &  \\  \hline 
 \hline 
  \rowcolor{WHITE} 
$ \frac{1}{2} $  
 &  $w_{1}w_{2}w_{3}w_{4}w_{2}w_{3}$  &  $ 1$  &  $\frac{ \zeta ( 2 \, z + 1 ) \zeta ( z + \frac{3}{2} ) \zeta ( z + \frac{1}{2} ) } { \zeta ( z + \frac{5}{2} ) \zeta ( z + \frac{7}{2} ) \zeta ( 2 \, z + 3 ) }$  &   $\left[-2, 1, -1, 1\right]$  &   \ref{image:0} \\ \hline 
 \hline 
  \rowcolor{WHITE} 
$ \frac{1}{2} $  
 &  $\bk{w_{2}w_{3}w_{4}} \bk{w_{1}w_{2}w_{3}}$  &  $ 1$  &  $\frac{ \zeta ( 2 \, z + 1 ) \zeta ( z + \frac{1}{2} ) \zeta ( z + \frac{3}{2} ) } { \zeta ( z + \frac{5}{2} ) \zeta ( z + \frac{7}{2} ) \zeta ( 2 \, z + 3 ) }$  &   $\left[1, -2, 1, 1\right]$  &   \ref{image:0} \\  \rowcolor{WHITE} 
$ \frac{1}{2} $  
 &  $\bk{w_{2}w_{3}w_{4}}\bk{w_{3}w_{2}w_{3}}\bk{w_{1}w_{2}w_{3}}$  &  $ 1$  &  $\frac{ \zeta ( z - \frac{3}{2} ) \zeta ( z + \frac{1}{2} ) \zeta ( 2 \, z + 1 ) \zeta ( 3 \, z + \frac{1}{2} ) } { \zeta ( z + \frac{3}{2} ) \zeta ( z + \frac{7}{2} ) \zeta ( 3 \, z + \frac{3}{2} ) \zeta ( 2 \, z + 3 ) }$  &   $\left[1, -2, 1, 1\right]$  &  \\ \hline 
 \hline 
  \rowcolor{WHITE} 
$ \frac{1}{2} $  
 &  $1$  &  $ 0$  &  $\frac{ } {}$  &   $\left[-1, -1, 3, -1\right]$   &   \ref{image:0} \\ \hline 
 \hline 
  \rowcolor{WHITE} 
$ \frac{1}{2} $  
 &  $w_{3}$  &  $ 0$  &  $\frac{ \zeta ( z + \frac{5}{2} ) } { \zeta ( z + \frac{7}{2} ) }$  &   $\left[-1, 2, -3, 2\right]$  &  \ref{image:0} \\ \hline 
 \hline 
  \rowcolor{WHITE} 
$ \frac{1}{2} $  
 &  $w_{2}w_{3}$  &  $ 0$  &  $\frac{ \zeta ( z + \frac{3}{2} ) } { \zeta ( z + \frac{7}{2} ) }$  &   $\left[1, -2, 1, 2\right]$   &  \ref{image:0} \\ \hline 
 \hline 
  \rowcolor{WHITE} 
$ \frac{1}{2} $  
 &  $w_{4}w_{3}$  &  $ 0$  &  $\frac{ \zeta ( z + \frac{3}{2} ) } { \zeta ( z + \frac{7}{2} ) }$  &   $\left[-1, 2, -1, -2\right]$  &  \ref{image:0} \\ \hline 
 \hline 
  \rowcolor{WHITE} 
$ \frac{1}{2} $  
 &  $w_{4}w_{2}w_{3}$  &  $ 0$  &  $\frac{ \zeta ( z + \frac{3}{2} ) ^{ 2 } } { \zeta ( z + \frac{5}{2} ) \zeta ( z + \frac{7}{2} ) }$  &   $\left[1, -2, 3, -2\right]$  &  \ref{image:0}  \\ \hline 
 \hline 
  \rowcolor{WHITE} 
$ \frac{1}{2} $  
 &  $w_{3}w_{4}w_{2}w_{3}$  &  $ 0$  &  $\frac{ \zeta ( z + \frac{3}{2} ) ^{ 2 } \zeta ( 2 \, z + 2 ) } { \zeta ( z + \frac{5}{2} ) \zeta ( z + \frac{7}{2} ) \zeta ( 2 \, z + 3 ) }$  &   $\left[1, 1, -3, 1\right]$  &   \ref{image:0} \\ \hline 

\caption{Global Information of $\para{P}_3$}
 \label{Table::P3::Global}
 \end{longtable}
 
 \end{landscape}
 \renewcommand*{\arraystretch}{1.5}


\label{Bibliography}
\backmatter
\phantomsection 
\addcontentsline{toc}{chapter}{Bibliography} 
\bibliographystyle{alpha}
\bibliography{bib}

\end{document}